\documentclass[a4paper,11pt,oneside]{book}
\usepackage{amsmath}
\usepackage{amssymb}
\usepackage{latexsym}
\usepackage{theorem}
\usepackage{mathrsfs}
\usepackage[all]{xy}
\usepackage{makeidx}
\usepackage{dave1}
\usepackage{stmaryrd}

\CompileMatrices

\makeindex

\begin{document}

\frontmatter

\begin{titlepage}
\vspace*{\stretch{1}}
\begin{center}\bf
{\LARGE Rational Equivariant Spectra}\\
\vspace{1.2cm}
{\LARGE David James Barnes}\\
\vspace{1cm}
{\Large Submitted for the degree of PhD}\\
\bigskip
{\Large Department of Pure Mathematics}\\
\bigskip
{\large January 2008}
\end{center}
\vspace*{\stretch{2.5}}
\begin{center}
\textbf{\large University of Sheffield}
\end{center}
\end{titlepage}

\tableofcontents

\cleardoublepage
\chapter{Introduction}

\section{Rational Equivariant Cohomology Theories}

Cohomology theories provide information about the abstract nature of
topological spaces and allow us to distinguish between and work with these spaces.
To each space $X$, a cohomology theory $E^*$ gives a graded group $E^*(X)$.
For a map of spaces $f \co X \to Y$ we have a map of graded groups
$E^*(f) \co E^*(Y) \to E^*(X)$ and this structure satisfies many useful axioms.
Since spaces with a group action abound in all areas of mathematics,
an understanding of how these spaces differ and behave is of particular importance.
A $G$-equivariant cohomology theory is a special kind of cohomology theory
that is designed to be used on spaces with a $G$-action.

Particular examples include: equivariant $K$-theory \cite{segal68} (which is constructed in
the same way as $K$-theory but using $G$-vector bundles) and
equivariant cobordism \cite{tdieck72}.
The Borel construction takes any
cohomology theory $E^*$ and makes a $G$-equivariant cohomology theory from it:
$E_G^*(X) = E^*(EG_+ \smashprod_G X)$
($EG_+$ is the universal free $G$-space with an additional $G$-fixed point
adjoined). A rational $G$-equivariant cohomology theory
satisfies the additional condition that each
group $E^*(X)$ is a rational vector space. This extra assumption simplifies
the area so that it becomes amenable to study and calculations
(spectral sequences) become much easier once one works over a field.

The result below, which only applies to finite groups $G$,
follows from \cite[Appendix A]{gremay95}. Let
$g\qq W_G H \leftmod$ denote the category of graded left
modules over the rational group ring of $W_G H$, the Weyl group of
$H$ in $G$.
\begin{thm}
The category of rational $G$-equivariant cohomology theories
is equivalent to the category
$\prod_{(H) \leqslant G} g\qq W_G H \leftmod$, where the
product runs over the collection of conjugacy classes
of subgroups of $G$.
\end{thm}
This shows that one can classify the collection of rational
equivariant cohomology theories in terms of an algebraic model
that is easy to understand.
Such classifications are performed in \cite{gre99} for the circle group
and \cite{gre98a} for $O(2)$ via
spectral sequence calculations.
These calculations become almost impossible
when the rank of the group is greater than 1
(that is, the group contains a torus). Thus
it would be advantageous to have an
alternative method to classify rational equivariant
cohomology theories.

\section{Classifying Cohomology Theories with Model Categories}

To study the category of cohomology theories
as a whole, one works with their
representing objects -- spectra. The basic idea is that
an element of $E^*(X)$, the $E^*$-cohomology of the space
$X$, is a homotopy class of maps from
$\Sigma^\infty X$ to an equivariant spectrum $E$. We say that $E$
represents $E^*$ and to understand
the cohomology theory $E^*$ we study the spectrum $E$.
Thus we want to understand the homotopy category of
equivariant spectra.
The categorical foundations of equivariant spectra
have developed substantially in the past twenty-five years.
The encyclopedic \cite{lms86}
constructed a good category of equivariant spectra and put the
subject area on firm footing. Unfortunately the smash product
constructed there is only a monoidal product up to homotopy.
Models for spectra with point-set level associative
smash products are constructed in
\cite{EKMM97}, \cite{hss00} and \cite{mmss01}.
This construction was soon adapted to an equivariant setting in
\cite{may96} and \cite{mm02}.
These newer categories come with model structures
that can help us to study their homotopy categories.

Very roughly, a model category has a collection of weak equivalences
that one formally makes into isomorphisms to create the homotopy category.
So, one inverts the collection of weak homotopy equivalences
in the category of spaces to obtain the category of $CW$-complexes
with morphisms given by homotopy classes of maps.
In particular, the homotopy category of
a model category of $G$-spectra
(there are several choices for such a
model category, all giving the same homotopy category)
is the category of $G$-equivariant cohomology theories.
A Quillen equivalence between model categories is an adjoint pair
of functors such that these functors create an
equivalence of the homotopy categories.
The precise definition of model categories and Quillen equivalences
give specific criteria to check to see that an
adjoint pair is a Quillen equivalence.
It is generally accepted that the notion of Quillen equivalence
is the correct way to say that two
categories have the same homotopy theory.

Since a Quillen equivalence is defined in terms of an adjoint pair,
one often has a Quillen equivalence $A \overrightarrow{\leftarrow} B$
and a Quillen equivalence $B \overleftarrow{\rightarrow} C$
(with left adjoints on top), but one cannot combine these adjoint pairs
to get an adjunction between $A$ and $C$.
Instead one says that $A$ and $C$ are Quillen equivalent by a zig-zag of
Quillen equivalences. In this example one only has two Quillen equivalences,
but in general one can have any finite number of Quillen equivalences in
a zig-zag.

Model categories can be used to encode more information into
the homotopy category. One may be interested in
model categories with a monoidal product,  such as the tensor product
of modules over a commutative ring or the smash product of spaces
or spectra. When this product is compatible with the model category
we have a monoidal model category. A monoidal Quillen equivalence
(roughly) is a Quillen equivalence which identifies the monoidal product in
the homotopy categories. Hence results about the monoidal product structure in one
homotopy category apply equally well to the other in a
monoidal Quillen equivalence.

Our algebraic model for $G$-spectra, $dg \mathcal{A}(G)$,
is now expected to have the structure of a model category.
Thus, one can now ask for
a zig-zag of Quillen equivalences between a model category
of $G$-equivariant spectra and a model category $dg \mathcal{A}(G)$.
This Quillen equivalence
tells us that if one is interested in homotopy level
information one only needs to work with $dg \mathcal{A}(G)$.
If we can make this zig-zag from a series of monoidal
Quillen functors then we can model the smash product
of spectra by a tensor product operation in $dg \mathcal{A}(G)$.
This allows us to model more complicated structures in spectra by analogous
structures in $dg \mathcal{A}(G)$, such as understanding the collection of
modules over a ring spectrum by considering modules over a ring object in
$dg \mathcal{A}(G)$.

\section{Existing Work}

The paper \cite{shi02} proves that the category of rational
$SO(2)$-spectra and $dg \mathcal{A}(SO(2))$ (as constructed in \cite{gre99})
are Quillen equivalent using information about Massey products for
differential graded rings with many objects. This paper does not
consider monoidal structures and relies on the ease of calculation
for the circle group.

The classification of rational $G$-spectra
for finite $G$ is extended to a Quillen equivalence in
\cite[Example 5.1.2]{ss03stabmodcat}. This paper provides a first blueprint
for our new method of classifying rational
$G$-cohomology theories.
Starting from $G \sscr_\qq$, a category of rational $G$-spectra,
one considers $\gcal_{top}$, a collection of generators
(for the homotopy category). By using the good
properties of $G \sscr_\qq$ we can construct
$\ecal_{top}(\sigma_1, \sigma_2)$, a symmetric spectrum of functions
for each pair of pair of generators $(\sigma_1,\sigma_2)$.
This collection has a composition rule,
$\ecal_{top}(\sigma_2, \sigma_3) \smashprod \ecal_{top}(\sigma_1, \sigma_2)
\to \ecal_{top}(\sigma_1, \sigma_3).$

Thus we have created an enriched category which we call
$\ecal_{top}$, it has object set $\gcal_{top}$ and the subscript
$top$ indicates that this category is of topological
origin. We can consider the category of enriched functors
from $\ecal_{top}$ to symmetric spectra, we call such
a functor a right $\ecal_{top}$-module
and denote the category of such by $\rightmod \ecal_{top}$.
If $M$ is one of these enriched functors then
for each pair $\sigma_1$, $\sigma_2$ in $\gcal_{top}$
we have symmetric spectra
$M(\sigma_1)$ and $M(\sigma_2)$ with an action map
$M(\sigma_2) \smashprod \ecal_{top}(\sigma_1, \sigma_2)
\to M(\sigma_1).$

This category of modules is referred to as the collection
of `topological Mackey functors' in \cite{ss03stabmodcat}.
The categories $G \sscr_\qq$ and $\rightmod \ecal_{top}$
are Quillen equivalent by \cite[Theorem 3.3.3]{ss03stabmodcat}.
Since $G$ is finite and we are working
rationally, the homotopy groups of $\ecal_{top}(\sigma_1, \sigma_2)$
are concentrated in degree zero where they take value
$\underline{A}(\sigma_1, \sigma_2)$,
a $\qq$-module. Hence $\ecal_{top}(\sigma_1, \sigma_2)$
it is weakly equivalent to an Eilenberg-Mac Lane spectrum
$\h \underline{A}(\sigma_1, \sigma_2)$.

From the collection of spectra $\h \underline{A}(\sigma_1, \sigma_2)$,
we construct a category $\h \underline{A}$, which is enriched
over symmetric spectra and one replaces $\rightmod \ecal_{top}$
by the Quillen equivalent category $\rightmod \h \underline{A}$.
The collection $\underline{A}(\sigma_1 , \sigma_2)$
for $\sigma_1, \sigma_2 \in \gcal_{top}$
can be thought of as a category enriched over $dg \qq \leftmod$
and thus we have a model category $\rightmod \underline{A}$.
The category of rational Mackey functors is the
collection of additive functors $\pi_0 \ecal_{top} \to \qq$.
Since $\underline{A}$ is equal to $\pi_0 \ecal_{top}$,
$\rightmod \underline{A}$ is the category of rational differential
graded Mackey functors.
There is a zig-zag of Quillen equivalences
between $\rightmod \h \underline{A}$
and $\rightmod \underline{A}$.
Thus rational $G$-spectra have been classified
in terms of an algebraic category.
Since these categories are rational
the homotopy category of $\rightmod \underline{A}$
is equivalent to the category of graded rational Mackey functors,
which is equivalent to
$\prod_{(H) \leqslant G} g\qq W_G H \leftmod$.
Thus this paper recovers
the results of \cite[Appendix A]{gremay95}.
This classification
does not consider monoidal structures and requires the assumption that
the homotopy groups of $\ecal_{top}(\sigma_1, \sigma_2)$
are concentrated in degree zero.

The preprint \cite{greshi} is intended to be a
combination of \cite[Example 5.1.2]{ss03stabmodcat} and the paper
\cite{shiHZ}, it will classify rational
torus equivariant spectra in terms of an algebraic category.
Perhaps more importantly it provides a basis for the
classification process in general
as it doesn't have the strict requirements of the first two methods.
This preprint is currently under substantial review and expansion.
We have worked from an enhanced version,
but our references agree with the publicly available version.
As we will comment upon later
(Section \ref{subsec:monoidality}), the method of \cite{greshi}
is much more compatible with the monoidal structures.
We now outline the method of this paper.

Once again one begins with $G \sscr_\qq$
and uses the Quillen equivalence between
$G \sscr_\qq$ and $\rightmod \ecal_{top}$.
One can now apply the results of
\cite{shiHZ} to construct a category $\ecal_t$ from
$\ecal_{top}$. This new category will be
enriched over rational chain complexes with its set
of objects given by $\gcal_{top}$. The $t$ indicates that we have come from
the topological side but are now working in an algebraic setting.
We can consider enriched functors from $\ecal_t$ to
rational chain complexes, this category will be
denoted $\rightmod \ecal_t$.

Now we begin our work from the other end, assuming
that we have a suitable candidate, $dg \mathcal{A}(G)$,
for the algebraic model.
We choose a generating set $\gcal_a$ for
$dg \mathcal{A}(G)$, we require that this set
has a specified isomorphism to $\gcal_{top}$.
Analogously to the topological setting,
the set $\gcal_a$ is the object set for a category $\ecal_a$,
which is enriched over rational chain complexes.
We can then replace $dg \mathcal{A}(G)$
by $\rightmod \ecal_a$.
The $a$ indicates we have come from the algebraic model.
The notation $\ecal_{top}$, $\ecal_t$ and $\ecal_a$
is taken from \cite{greshi}, which uses $top$, $t$ and $a$
to indicate whether a particular object is topological,
algebraic but from the topological side or purely algebraic.

So far this process has been formal,
now one must use some specific information about
$\ecal_t$ and $\ecal_a$ to achieve a comparison between them.
The comparison we will use is the notion of
a quasi-isomorphism of categories
enriched over rational chain complexes.
Given two such categories $\mathcal{C}$ and $\mathcal{D}$
with isomorphic object sets,
an enriched functor $F \co \mathcal{C} \to \mathcal{D}$
is a quasi-isomorphism if each
$F(\sigma_1, \sigma_2) \co
\mathcal{C}(\sigma_1, \sigma_2) \to \mathcal{D}(F\sigma_1, F\sigma_2)$,
is a homology isomorphism
($F(\sigma_1, \sigma_2)$ is a map in the
category of rational chain complexes).

One shows by calculation that
$\ecal_t(\sigma_1, \sigma_2)$ and $\ecal_a(\sigma_1, \sigma_2)$
(which are rational chain complexes)
have the same homology for each pair $(\sigma_1, \sigma_2)$.
Then one proves that this homology is intrinsically formal,
that is, any two such enriched categories
with this homology must be quasi-isomorphic.
Now one returns to formal considerations and shows
that since $\ecal_t$ and $\ecal_a$ are quasi-isomorphic,
there is a zig-zag of Quillen equivalences between
$\rightmod \ecal_t$ and $\rightmod \ecal_a$.
Putting all of this together gives the desired result,
a zig-zag of Quillen equivalences between $G$-spectra
and our algebraic model.

\section{Contents of this Thesis}
\subsection*{The Splitting}

We will study the case of finite groups and the case $G=O(2)$.
A new ingredient to the method outlined above is the notion
of splitting the category of rational $G$-spectra $G \sscr_\qq$.
The language of Bousfield localisations is used here,
it is a method of altering the homotopy category of
a model category.
Given a model category $M$, the homotopy category $\ho M$
is formed by inverting the weak equivalences. Hence
if we change the weak equivalences without changing the
objects and morphisms of $M$ we obtain a new
model category $M'$ and a new homotopy category $\ho M'$.
A spectrum $E$ defines a homology theory $E_*$,
from this we have the notion of
$E$-equivalences, those maps $f$ such that
$E_*f$ is an isomorphism. The Bousfield localisations
that we will use are called $E$-localisations,
where the new weak equivalences are the $E$-equivalences.
We write the $E$-localisation of $G\sscr_\qq$ as
$L_E G \sscr_\qq$.

\begin{thm}[\ref{thm:generalsplitting}]
Let $\{ E_i \}_{i \in I}$ be a finite collection $G$-spectra. If
$E_i \smashprod E_j$ is rationally acyclic for $i \neq j$
and $\bigvee_{i \in I} E_i$ is rationally equivalent to $\sphspec$
then we have a monoidal Quillen equivalence
$$\Delta : G \sscr_\qq
\overrightarrow{\longleftarrow}
\prod_{i \in I} L_{E_i} G \sscr_\qq : \prod .$$
\end{thm}

Now we assume that for each $i$
we have a model category $dg \mathcal{A}(G)_i$
which is supposed to model $L_{E_i} G \sscr_\qq$.
Fix some $E_i$ and apply the method outlined above to
get to the stage where we must compare $\ecal_t(i)$
and $\ecal_a(i)$. Here is where we see the advantage of
the splitting, it simplifies the categories
$\ecal_t(i)$ and $\ecal_a(i)$ and thus makes it easier to
prove they are equivalent.
We describe how this works in the finite case and the
$O(2)$ case.

\subsection*{The Finite Case.}
We apply our splitting theorem and the method of
\cite{greshi} to reprove the result of
\cite[Example 5.1.2]{ss03stabmodcat}.
This will be a good introduction to the
general method and once the monoidality issue
has been resolved it will be a monoidal classification.
It should also be helpful in seeing how to
proceed in a classification of the category
of dihedral $O(2)$-spectra,
which we shall define later.
Recall that $A(G) \otimes \qq \cong \prod_{(H)} \qq$
by tom-Dieck's isomorphism,
where the product runs over the set of conjugacy classes
of subgroups of $G$. Thus for each conjugacy class
of subgroups $(H)$ we have an idempotent $e_H$
and we set $E_H=e_H \sphspec$.
Our splitting theorem then states that
the category of $G$-spectra is monoidally equivalent to
$\prod_{(H)} L_{E_H} G \sscr_\qq$.
So we work through the method outlined above
for each $L_{E_H} G \sscr_\qq$ individually.
In fact, we can use the same
argument for each conjugacy class $(H)$.
Each $L_{E_H} G \sscr_\qq$
is generated by a single element,
$G/H_+ \smashprod e_H \sphspec$. Thus
the set of objects for $\ecal_{top}(H)$
will be the collection of smash products
of $G/H_+ \smashprod e_H \sphspec$.

Following the previous work on the finite case
an obvious candidate for the
algebraic model of the $(H)$-part is
$dg \qq W_G H \leftmod$, chain complexes
of modules over the rational group ring
of the Weyl group of $H$.
Note that the homotopy category of
$dg \qq W_G H \leftmod$
is the category $g \qq W_G H \leftmod$.
Calculation shows that $\ecal_a(H)$
has a trivial differential and hence is
equal to its homology, furthermore
the homology of $\ecal_t(H)$
is isomorphic to $\ecal_a(H)$
and is concentrated in degree zero.
The intrinsic formality we now use is the fact
that a chain complex with homology concentrated
in degree zero is equivalent to its homology.
Thus by the formal method outlined above, we
have the following conclusion.

\begin{thm}[\ref{cor:finiteclassification}]
There is a zig-zag of Quillen equivalences between
$L_{E_H} G \sscr_\qq$ and $dg\qq W_G H \leftmod$.
Hence there is a zig-zag of
Quillen equivalences between
$G \sscr_\qq$ and
$dg \mathcal{A}(G) := \prod_{(H) \leqslant G} dg\qq W_G H \leftmod$.
\end{thm}

The difference between the above method
and that of \cite{ss03stabmodcat} is the ordering of the work.
In our method, we split the category, move to Mackey functors,
translate to algebra and then apply a formality result,
\cite{ss03stabmodcat} goes to Mackey functors first,
applies formality, moves to algebra and then splits
the category.
Our formality is that $\ecal_t$ has homology concentrated
in degree zero, \cite{ss03stabmodcat} uses the fact that
$\ecal_{top}$ has homotopy concentrated in degree zero
to replace $\ecal_{top}$ by $\h \underline{A}$.
This is the same information, just in different contexts.

The result above could have been proved
without using the splitting theorem.
In which case we would have to use a much
larger set of generators:
the collection of all
smash products of terms $G/H_+ \smashprod e_H \sphspec$
as $H$ runs over the conjugacy classes of subgroups of $G$.
Many of these terms would have been rationally contractible,
since $G/H_+ \smashprod e_H \sphspec \smashprod
G/K_+ \smashprod e_K \sphspec$ is rationally contractible whenever
$H$ and $K$ are not conjugate. The splitting removes
these extra terms,
making the result easier to prove and understand.
Thus one of the general advantages of the splitting result
is that one can reduce the size of the object set
of $\ecal_{top}$ (and hence of $\ecal_t$ and $\ecal_a$).

\subsection*{The $O(2)$ Case.}
We begin in the same place as
\cite{gre98a},  we fix $W$ as the group
or order two and consider the cofibre sequence
$EW_+ \to S^0 \to E \widetilde{W}$. We show that
$EW_+$ and $E \widetilde{W}$ satisfy the assumptions
of the splitting theorem. We define
$\cscr \sscr_\qq$, the model category of cyclic spectra,
as the $E W_+$-localisation
of $O(2) \sscr_\qq$.
The model category of
dihedral spectra, $\dscr \sscr_\qq$, is defined
as the $E \widetilde{W}$-localisation
of $O(2) \sscr_\qq$.
Define $\cscr$, the set of cyclic subgroups of $O(2)$, to be
the closed subgroups of $O(2)$ which are contained in $SO(2)$.
All other closed subgroups contain a reflection
and we call this set the collection
of dihedral subgroups, $\dscr$.
The reason behind the name cyclic spectra is that
the homotopy category of $\cscr \sscr_\qq$ is
the homotopy category of $O(2)$-spectra
made from $O(2)$-cells of the form
$O(2)/H_+ \smashprod S^n$ with $H$ a cyclic
subgroup of $O(2)$ (so $H \leqslant SO(2)$).
Hence we call the remainder, $\dscr \sscr_\qq$,
the category of dihedral spectra.

\begin{thm}[\ref{thm:O2splitting}]
There is a strong monoidal Quillen equivalence
$$\Delta : O(2) \mcal _\qq \overrightarrow{\longleftarrow}
\cscr \sscr_\qq \times \dscr \sscr_\qq : \prod.$$
In particular, we have the following natural isomorphism for any $G$-spectra $X$ and $Y$
$$[X,Y]^{O(2)}_\qq \cong
[X \smashprod E W_+, Y \smashprod E W_+]^{O(2)}_\qq
\oplus
[X \smashprod E \widetilde{W}, Y \smashprod E \widetilde{W}]^{O(2)}_\qq.$$
\end{thm}

We expect, by looking at the calculations of \cite{gre98a},
that these two parts behave quite differently and will be
classified by quite different methods. The splitting
theorem allows us to deal with each part in turn.
We expect the dihedral part, $\dscr \sscr_\qq$, to behave somewhat
like the finite case, though we do not make this
precise or study the dihedral part any further.
We have concentrated on the cyclic part,  $\cscr \sscr_\qq$.
We could try to use the general method to classify  $\cscr \sscr_\qq$,
but we would still have to find some intrinsic formality argument
in order to understand $\ecal_t(\cscr)$ in terms of some algebraic model.
Instead we make precise the relation between
$\cscr \sscr_\qq$ and $SO(2) \sscr_\qq$ using the notion
of a category with involution. We then intend to
make use of the classification of $SO(2) \sscr_\qq$ in \cite{greshi}
to understand $\ecal_t(\cscr)$.
The motivation here is the relation
between $O(2)$-spaces and $SO(2)$-spaces: an $O(2)$-space $X$
is an $SO(2)$-space with a map $f \co X \to X$, such that
$f^2(x) = x$ and $f(tx) = t^{-1} f(x)$ for all $t \in SO(2)$ and
$x \in X$ (so $f$ is like a reflection).
We formalise this into categorical language, then
investigate how this notion can be applied to spectra.

Take a category $\ccal$ and a functor $\sigma \co \ccal \to \ccal$
such that $\sigma^2=1$, we call $\sigma$ an involution and
$(\ccal, \sigma)$ a category with involution.
In our examples all our categories have
underlying sets, so it makes sense to use equality. In general
one could replace this equality by a fixed natural isomorphism
which would form part of the structure.
The skewed category of $(\ccal, \sigma)$, denoted $\sigma \# \ccal$,  has objects
the maps $f \co X \to \sigma X$ such that $\sigma f \circ f = \id_X$.
Morphisms are pairs $(\alpha, \sigma \alpha)$ making the
obvious square commute, for $\alpha$ a morphism of $\ccal$.
We prove that if $\ccal$ is a cofibrantly generated model category
and $\sigma$ is a left (and hence right) Quillen functor then
$\sigma \# \ccal$ has a cofibrantly generated model structure.
This model structure is defined by the condition that a map
$(\alpha, \sigma \alpha)$ is a weak equivalence or fibration
exactly when $\alpha$ is in $\ccal$.

The simplest example is the case when $\sigma$ is the identity
functor and $\ccal$ is any category. Here $\id \# \ccal$ is the category
of $W$-objects and $W$-maps in $\ccal$ ($W$ is the group of order two).
Next consider the category of based $SO(2)$-equivariant spaces $SO(2) \tscr_*$.
The involution is $\sigma=j^*$, pullback along the group homomorphism
$j \co SO(2) \to SO(2)$ which takes $t$ to $t^{-1}$.
The skewed category, $j^* \# SO(2) \tscr_*$, is then the category of $O(2)$-spaces.
Note that the model structure on $j^* \# SO(2) \tscr_*$ is not the
usual model structure on $O(2)$-spaces, instead a map $f \co X \to Y$
is a weak equivalence if and only if each $f^H \co X^H \to Y^H$
is a weak homotopy equivalence of spaces for all $H \leqslant SO(2)$.
With some work we can construct a similar functor on $SO(2)$-spectra,
so that we have a category with involution $(SO(2) \sscr_\qq, \tau)$.

\begin{thm}[\ref{thm:so2involution}]
There is a monoidal Quillen equivalence
$\tau \# SO(2) \sscr_\qq
\overrightarrow{\longleftarrow}
\cscr \sscr_\qq.$
\end{thm}

We prove that there is an involution $\rho$
on $\rightmod \ecal_t(SO(2))$ and show that the series of
Quillen equivalences between $SO(2) \sscr_\qq$ and
$\ecal_t(SO(2))$ preserves this involution.
We can then conclude the following.
\begin{thm}[Corollary \ref{cor:cyclicsummary}]
There is a zig-zag of Quillen equivalences between
the categories
$\tau \# SO(2) \sscr_\qq$ and $\rho \# \rightmod \ecal_t(SO(2))$.
\end{thm}

\section{Applications and Further Work}

\subsection*{Cyclic spectra}
It remains to complete the classification of
$\cscr \sscr_\qq$, which we have identified with
the skewed category $\rho \# \rightmod \ecal_t(SO(2))$.
From this point the idea is that there should be
an involution $\Upsilon$ on $dg \mathcal{A}(SO(2))$
(see Remark \ref{rmk:nextsteps}).
The algebraic model for $\cscr \sscr_\qq$
will then be given by $\Upsilon \# dg \mathcal{A}(SO(2))$.
Then if we can show that the zig-zag between
$\ecal_t(SO(2))$ and $\ecal_a(SO(2))$ of \cite{greshi}
respects the involutions
it should follow formally that
$\cscr \sscr_\qq$ and $\Upsilon \# dg \mathcal{A}(SO(2))$
are Quillen equivalent.

\subsection*{Dihedral Spectra}
The dihedral part of the $O(2)$ case must still be
classified in terms of an algebraic model. The model
suggested by \cite{gre98a} is a category of sheaves
over a topological space. Once this is understood
one can consider the more general setting of a
split short exact sequence
$0 \to SO(2)^n \to G \to F \to 1$ and try understand
the $F$ part of the category of $G$-spectra.

\subsection*{The Homotopy Category of a Skewed Category}
For $(\ccal,\sigma)$ an involutary model category,
we would like to be able to prove an equivalence
$\ho (\sigma \# \ccal) \simeq \sigma \# \ho(\ccal)$.
We expect that such a result would require
$\ho (\ccal)$ to be rational (so $[X,Y]_\ccal$ is a
rational vector space for each $X$ and $Y$ in $\ccal$).
We would also like to investigate
the conditions necessary for the involution to untwist,
that is, when is there an equivalence
$\sigma \# \ho(\ccal) \simeq \ho(\ccal)^{W}$?
Where the right hand side is the homotopy category of $\ccal$
with a homotopy action of $W$, the group of order two.
In Proposition \ref{prop:relatetoO2cohom}
we give \cite[Proposition 3.1 and Corollary 3.2]{gre98a},
which prove that $\ho(\tau \# \cscr \sscr_\qq)$
is equivalent to $\ho(\cscr \sscr_\qq)^{W}$, which
is our reason for expecting these results.
With suitable assumptions on the nature of $\ccal$
we hope to prove the results above by
following \cite{gre98a}.

\subsection*{Monoidality}\label{subsec:monoidality}
The classification outlined above for
\cite{greshi}, the finite case and $O(2)$ is much more
compatible with monoidal structures than
the work of \cite{shi02} and
\cite[Example 5.1.2]{ss03stabmodcat}.
We go through the classification once more
and point out which equivalences are monoidal.
Begin with $G \sscr_\qq$ and make a more
careful choice of generators $\gcal_{top}$
(the set of generators should be closed under $\smashprod$),
then the Quillen equivalence between
$G \sscr_\qq$ and $\rightmod \ecal_{top}$
will be symmetric monoidal.
Equally, the equivalence between
$\rightmod \ecal_{a}$ and
$dg \mathcal{A}(G)$ will be symmetric monoidal
provided that the generators $\gcal_{a}$
are closed under $\otimes$.
The results of \cite{shiHZ} allow us to
construct a category $\ecal_t$ and
$\rightmod \ecal_t$ is Quillen equivalent to
$\rightmod \ecal_{top}$. However,
due to a technical issue the category
$\rightmod \ecal_t$ is not a monoidal category
(see Remark \ref{rmk:monoidalissue}).
Hence the comparison between $\rightmod \ecal_t$ and
$\rightmod \ecal_{top}$ cannot be monoidal, thus, neither
can the whole classification.
This is the only point of the comparison
where monoidality fails; it is probable that this issue can be
resolved.

\subsection*{Equivariant Categories}
As mentioned in Remark \ref{rmk:extendG}
we identify the notion of a category with involution
with a category with an action of order two on it.
Thus there is an obvious generalisation to a
category with a general group action.
So consider a group $G$, where $G$ is part of a split short
exact sequence $0 \to SO(2)^n \to G \to F \to 1$ with
$F$ a finite group.
Then we can split the category of $G$-spectra into a
part corresponding to $F$ and a part
$G \sscr_\qq (SO(2)^n) $, corresponding to $SO(2)^n$.
It should then be possible to describe
$G \sscr_\qq (SO(2)^n)$ in terms of the category
$SO(2)^n \sscr_\qq$ with an action of $F$.
The theory of these equivariant categories
should allow us to classify $G \sscr_\qq (SO(2)^n)$
in terms of $dg \mathcal{A}(SO(2)^n)$ with an action
of $F$ on the category.
Thus, in this situation,
we can extend known classifications and
generate new algebraic models from existing ones.

\subsection*{Continuous Quaternions}
The next group of interest would be the continuous
quaternion group. There is a short
exact sequence $1 \to SO(2) \to Q_{cts} \to W \to 1$
($W$ the group of order two), but now this sequence
is not split. An understanding of
how to classify $Q_{cts}\sscr_\qq$ should
give a reasonable idea of how to
understand $G$-spectra for a non-split
short exact sequence
$0 \to (SO(2))^n \to G \to F \to 1$.

\newpage

\section{Organisation of the Thesis}

We have divided the thesis into three parts, the first consists
of basic notions and the splitting for a general
compact Lie group $G$,
the second classifies rational $G$-spectra for finite $G$
and the third specialises to the case $G=O(2)$ and examines
how the splitting theorem and the work of \cite{greshi}
can be used to study this case.

\paragraph*{Part \ref{part:Gspec}} We begin Chapter \ref{chp:Gspectra}
with some of the basic notions of model categories,
before moving on to describe some of the more technical conditions,
especially those which ensure that a monoidal product will behave well
on the homotopy category. We give a brief description of the
model categories that are used in this thesis and
then go into more detail on the categories of equivariant spectra
in Section \ref{sec:GIS}.
We introduce homotopy colimits in Section \ref{sec:hocolim}
and then use this to construct $S^0 \qq$, a rational sphere
spectrum, in Section \ref{sec:rationalsphere}.
Chapter \ref{chp:locals} introduces Bousfield localisations
of spectra, which we use to make a category of rational spectra
by localising at $S^0 \qq$.
We also construct an equivalent category of rational spectra
by considering modules over the ring spectrum $S_\qq$
(constructed in Section \ref{sec:rationalsphere}).
We make further use of
Bousfield localisations to prove the splitting theorem for
equivariant orthogonal spectra in Chapter \ref{chp:splitting}.
We show that this implies the corresponding splitting result
for equivariant EKMM spectra and the category of $S_\qq$-modules.
We then consider a particular kind of splitting
in Section \ref{sec:idemsplit} that we will use for $O(2)$-spectra
in Chapter \ref{chp:O2spec}.

\paragraph*{Part \ref{part:finite}}
Chapter \ref{chp:fingroups} classifies rational $G$-spectra
for finite $G$. We begin with the algebraic model
which is particularly simple.
We then apply the splitting theorem to
the category of rational $G$-spectra
and identify the $(H)$-piece of this splitting with
modules over a ring spectrum $S_H$ (there is one piece
for each conjugacy class of subgroups $(H)$).
Since we now have a category with every object fibrant,
we can proceed through the method of \cite{greshi}.
Once we have performed the formal parts of this method,
we must specialise to our particular case and
prove that $\rightmod \ecal_t^H$ is equivalent to
$\rightmod \ecal_a^H$ for each conjugacy class of subgroups
$(H) \subseteq G$. This is done by studying
the structure of the $dg \qq \leftmod$-enriched categories
$\ecal_t$ and $\ecal_a$.
Since this is the first time we
use the notion of right modules over an enriched category,
we introduce the language and theory of this machinery in
Chapter \ref{chp:enrichcat}.

\paragraph*{Part \ref{part:O(2)}}
We consider the group $O(2)$ and see how our splitting theorem
can be used to study rational $O(2)$-spectra.
Chapter \ref{chp:O2spec} splits the category into
cyclic spectra and dihedral spectra, using the results of
Section \ref{sec:idemsplit}.
We concentrate on the model category of cyclic spectra,
we need to understand how this category is related to
$SO(2)$-spectra, so that we can use the work of \cite{greshi}
to classify this category in terms of an algebraic model.
In Chapter \ref{chp:catwithinv} we have
abstracted the relation between
$O(2)$-spaces and $SO(2)$-spaces
to define a category with involution and its associated
skewed category. We apply this to
cyclic spectra to describe this category in terms of
the skewed category of rational $SO(2)$-spectra in
Chapter \ref{chp:cyclicspectra}.
In Chapter \ref{chp:WskewSQmod} we examine
the zig-zag of equivalences
between rational $SO(2)$-spectra and
$\rightmod \ecal_t(SO(2))$ from \cite{greshi}.
We prove that this induces a zig-zag of
Quillen equivalences between the skewed category of
$\rightmod \ecal_t(SO(2))$ and
the category of cyclic spectra. This chapter
requires us to use $S_\qq$-modules as our
category of rational $SO(2)$-spectra, since every object
of the category must be fibrant to apply the work of \cite{greshi}.

We include an appendix listing all of the model categories
that we use.

\section*{Acknowledgements}

I would like to thank my supervisor, John Greenlees,
for all the help and advice he has given me during
my time at Sheffield. I would also like to express my gratitude to
Brooke Shipley and Simon Willerton, with whom I have had many
useful conversations regarding model categories
and category theory.

\mainmatter
\part{$G$-spectra and Localisations}\label{part:Gspec}
\chapter{$G$-Spectra}\label{chp:Gspectra}
We introduce the basic notions necessary to
study cohomology theories in a modern setting.
We begin with model categories, which make the
construction of homotopy categories rigourous
and allow us to prove that two homotopy categories
are equivalent by checking a small list of conditions.
In Section \ref{sec:spacesandspectra} we give
brief details on the categories that we will use in this thesis.
We focus upon the categories of $G$-spectra in Section \ref{sec:GIS}
and go into some details on the properties of these categories.
We prove a few well-known results and show some technical
model category conditions that we have not been able to
find explicitly in the literature.
We construct homotopy pushouts and telescopes for $G$-spectra in
Section \ref{sec:hocolim} so that we will have definite
constructions for our later work.
We then make a rational sphere spectrum in
Section \ref{sec:rationalsphere} which will be used to define
rational $G$-spectra in Chapter \ref{chp:splitting}.
Much of this chapter is definitions and results
from other sources, especially so for the first two sections.
The new content of
Sections \ref{sec:GIS} and \ref{sec:hocolim}
is mainly in the proofs, which are often considered
too standard to be included in the usual sources.
The final section is mostly new, though not surprising.

\section{Model Categories}\label{sec:modcat}
Many of the results of this thesis are phrased
in terms of model categories, which are a
general framework for homotopy theory.
If one wishes to invert a collection of maps in a
category (i.e. formally make them into isomorphisms),
one can not always be sure that the
result will be a category, model categories
are a solution to this problem.
Model categories were first introduced in \cite{quil67},
an excellent modern account is \cite{ds95},
but we take most of our definitions from
Section 1.1 of the comprehensive book \cite{hov99}.
We let $d$ and $c$ be the domain and codomain
functors from $\text{Map} \ccal$ to $\ccal$,
which exist for any category $\ccal$.
We give \cite[Definitions 1.1.1 -- 1.1.4]{hoverr} in order.
\begin{definition}\label{def:funcfact}
Suppose $\ccal$ is a category
\begin{enumerate}
\item A map $f$ in $\ccal$ is a \textbf{retract}\index{Retract}
of a map $g \in \ccal$ if and only if there is a commutative
diagram of the following form
$$\xymatrix{
A \ar[r] \ar[d]_f & C \ar[r] \ar[d]_g & A \ar[d]_f \\
B \ar[r] & D \ar[r] & D  \\
}$$
where the horizontal composites are identities.
\item A \textbf{functorial factorisation}\index{Functorial factorisation}
is an ordered pair $(\alpha, \beta)$
of functors $\text{Map} \ccal \to \text{Map} \ccal$
such that
$$d \circ \alpha = d, \quad c \circ \alpha = d \circ \beta, \quad c \circ \beta = c,
\quad f=\beta (f) \circ \alpha (f)$$
for all $f \in \text{Map} \ccal$.
Hence any commutative square
\begin{displaymath}
\xymatrix@!C{
A \ar[d]_u \ar[r]^f & B \ar[d]^v \\
C \ar[r]_g & D }
\end{displaymath}
induces a commutative square
\begin{displaymath}
\xymatrix@!C{
A \ar[d]_u \ar[r]^{\alpha (f)} & (c \circ \alpha )(f)
\ar[r]^{\beta (f)} \ar[d]_{(c \circ \alpha )(u,v) } & B \ar[d]^v \\
C \ar[r]_{\alpha (g)} & (c \circ \alpha )(g) \ar[r]_{\beta (g)}  & D }
\end{displaymath}
noting that $(u,v)$ is a morphism
in $\text{Map} \ccal$ between $f$ and $g$.
\end{enumerate}
\end{definition}

\begin{definition}
Suppose $i \co A \to B$ and $p \co X \to Y$
are maps in a category $\ccal$. Then $i$
has the \textbf{left lifting property with respect to $p$}\index{Left lifting property}
and $p$ has the \textbf{right lifting property with respect to $i$}\index{Right lifting property}
if, for every commutative diagram of the following form
\begin{displaymath}
\xymatrix@!C{
A \ar[d]_i \ar[r]^f & X \ar[d]^p \\
B \ar[r]_g & Y }
\end{displaymath}
there is a lift $h \co B \to X$ such that $hi=f$
and $ph=g$.
\end{definition}

\begin{definition}
A \textbf{model structure}\index{Model structure}
on a category $\ccal$ is three subcategories of $\ccal$
called weak equivalences\index{Weak equivalence},
cofibrations\index{Cofibration} and fibrations\index{Fibration}
and two functorial factorisations $(\alpha, \beta)$
and $(\gamma, \delta)$ satisfying the following properties:
\begin{enumerate}
\item If $f$ and $g$ are morphisms of $\ccal$ such that
$gf$ is defined and two of $f$, $g$ and $gf$ are
weak equivalences, then so is the third.
\item If $f$ and $g$ are morphisms of $\ccal$
such that $f$ is a retract of $g$
and $g$ is a weak equivalence, cofibration
or fibration, then so is $f$.
\item Define a map to be an
\textbf{acyclic cofibration}\index{Acyclic cofibration}
if it is both a cofibration and a weak equivalence.
Similarly, define a map to be an
\textbf{acyclic fibration}\index{Acyclic fibration}
if it is both a fibration and a weak equivalence. Then
acyclic cofibrations have the left lifting property with
respect to fibrations and cofibrations have the left lifting
property with respect to acyclic fibrations.
\item For any morphism $f$, $\alpha(f)$ is a cofibration,
$\beta(f)$ is an acyclic fibration, $\gamma(f)$ is
an acyclic cofibration and $\delta(f)$ is a fibration.
\end{enumerate}
\end{definition}

\begin{definition}
A \textbf{model category}\index{Model category}
is a category $\ccal$ with all small limits and colimits
and a model structure on $\ccal$.
\end{definition}

When we need to emphasise the properties of maps in
diagrams we will use the following shorthand:
$\xymatrix{X \ar[r]^{\sim} & Y}$ for the weak equivalences,
$\xymatrix{X \ar@{>->}[r] & Y}$ for the cofibrations and
$\xymatrix{X \ar@{->>}[r]& Y}$ for the fibrations.
Since a model category has all small limits and colimits
it has an initial object $\emptyset$ and a terminal object
$*$, these are the colimit and limit of the empty diagram
respectively. So for any object $X$ there is a unique pair of maps
$\emptyset \to X \to *$. We call $X$ \textbf{cofibrant}\index{Cofibrant}
if the map $\emptyset \to X$ is a cofibration, similarly
$X$ is called \textbf{fibrant}\index{Fibrant} if the map
$X \to *$ is a fibration.
By using the factorisation axioms one can take the
map $\emptyset \to X$ and factor it into a cofibration
followed by an acyclic fibration:
$\emptyset \to \cofrep X \to X$.
We call this process cofibrant replacement.
The functorial factorisations ensure that
$\cofrep$\index{c@$\cofrep$}
is a functor. We can perform the equivalent construction
for the map $X \to *$ and we obtain
$\fibrep$\index{f@$\fibrep$} the fibrant
replacement functor. We will decorate this
notation where necessary to indicate
which model structure we are considering.
If the canonical map $\emptyset \to *$
is an isomorphism we call the model category
\textbf{pointed}\index{Pointed}.

We now cut to the chase and give a rough and ready
theorem stating the existence of a homotopy category.
The proof of this theorem (that is, the construction
of the homotopy category) is a little involved and we
leave it to the excellent accounts of model categories
that we have already mentioned.

\begin{theorem}
If $\ccal$ is a model category, then there is a
category $\ho \ccal$, called the
\textbf{homotopy category}\index{Homotopy category}
of $\ccal$ with a functor
$\gamma \co \ccal \to \ho \ccal$ such that
$\gamma f$ is an isomorphism if and only if
$f$ is a weak equivalence.
Furthermore if $F \co \ccal \to \dcal$ is
any functor which takes every weak equivalence of $\ccal$
to an isomorphism of $\dcal$ then there is
a unique functor $\ho F \co \ho \ccal \to \dcal$
such that $\ho F \circ \gamma = F$.
\end{theorem}

Now we move on to \cite[Section 1.3]{hov99} and introduce
the langauge necessary to compare model categories and
the notion of equivalent model categories. The
definition below is \cite[Definition 1.3.1]{hov99}
and the following lemma is
\cite[Lemma 1.3.10]{hov99}.

\begin{definition}
If $\ccal$ and $\dcal$ are model categories then a
functor $F \co \ccal \to \dcal$ is a
\textbf{left Quillen functor}\index{Left Quillen functor}
if $F$ preserves cofibrations and acyclic cofibrations.
Similarly $F$ is a
\textbf{right Quillen functor}\index{Right Quillen functor}
if $F$ preserves fibrations and acyclic fibrations.
An adjoint pair
$F : \ccal \overrightarrow{\longleftarrow} \dcal : G$ is a
\textbf{Quillen pair}\index{Quillen pair} if either
the left adjoint is a left Quillen functor, or equivalently
the right adjoint is a right Quillen functor.
\end{definition}

\begin{lemma}
A left Quillen functor $F \co \ccal \to \dcal$ passes to a functor
$LF \co \ho \ccal \to \ho \dcal$, similarly
a right Quillen functor $G \co \dcal \to \ccal$ passes to a functor
$RG \co \ho \dcal \to \ho \ccal$.
A Quillen pair induces an adjunction
$LF : \ho \ccal \overrightarrow{\longleftarrow} \ho \dcal : RG$.
\end{lemma}

We now give
\cite[Definition 1.3.12 and Proposition 1.3.13]{hov99}.
\begin{definition}
A Quillen pair $F : \ccal \overrightarrow{\longleftarrow} \dcal : G$ is a
\textbf{Quillen equivalence}\index{Quillen equivalence} if,
for all cofibrant $X$ in $\ccal$ and all fibrant $Y$ in $\dcal$,
a map $f \co X \to GY$ is a weak equivalence of $\ccal$
if and only if $\widetilde{f} \co FX \to Y$ is a weak equivalence
of $\dcal$.
\end{definition}

\begin{proposition}
A Quillen pair $(F,G)$ is a Quillen equivalence if and only
if the adjoint pair $(LF,RG)$ is an equivalence of categories.
\end{proposition}

This result is why we use model categories,
it allows us to compare homotopy categories by checking a
relatively simple criterion. Now we introduce some more
structure, that of a monoidal product (such as the smash
product of spaces).

\begin{definition}\label{def:Qbifunctor}
An \textbf{adjunction of two variables}\index{Adjunction of two variables}
(\cite[Definition 4.2.12]{hov99})
$\otimes \co \ccal \times \dcal \to \ecal$,
is a functor $\otimes$ with two `right adjoints'
$\hom_r \co \dcal^{op} \times \ecal \to \ccal$
and $\hom_l \co \ccal^{op} \otimes \ecal \to \dcal$.
A \textbf{Quillen bifunctor}\index{Quillen bifunctor}
(\cite[Definition 4.2.1]{hov99})
is an adjunction of two variables $\otimes$
such that for cofibrations $f \co U \to V$
in $\ccal$ and $g \co W \to X$ in $\dcal$
the induced map (called the
\textbf{pushout product}\index{Pushout product})
$$f \square g \co
V \otimes W \coprod_{ U \otimes W} U \otimes X
\longrightarrow
V \otimes X$$
is a cofibration of $\ecal$ which is a weak equivalence
when one of $f$ or $g$ is.
\end{definition}

Now we give \cite[Definition 4.2.6]{hov99}.
\begin{definition}
A \textbf{monoidal model category}\index{Monoidal model category}, $\ccal$,
is a monoidal category $(\ccal, \otimes, \sphspec)$ that is a model category
such that $\otimes \co \ccal \times \ccal \to \ccal$ is a Quillen bifunctor
and for any cofibrant replacement of the unit
$q \co \cofrep S \to S$
the natural map $\cofrep S \otimes X \to S \otimes X$
is a weak equivalence for all cofibrant $X$.
These two requirements are also known as the
\textbf{pushout product axiom}\index{Pushout product axiom}.
\end{definition}

\begin{proposition}
The closed monoidal product on a monoidal model category passes
to a closed monoidal product on the homotopy category.
\end{proposition}
This proposition is part
of \cite[Theorem 4.3.2]{hov99}.
Often we have a monoidal product that is
symmetric, for this we use \cite[Definition 4.1.4]{hoverr} below.

\begin{definition}
A \textbf{symmetric monoidal
model category}\index{Symmetric monoidal
model category} is a monoidal category
$\ccal$ with a natural commutativity
isomorphism $T_{X,Y} \co X \otimes Y \to Y \otimes X$.
We require $T$ to be compatible with the unit isomorphisms
and associativity of $\ccal$
and satisfy $T_{Y,X} \circ T_{X,Y} = \id_{X \otimes Y}$.
\end{definition}

Now we want to know when this extra structure
is preserved by Quillen functors.
The following two definitions are
\cite[Definitions 3.3 and 3.6]{ss03monequiv}.
\begin{definition}
A \textbf{monoidal functor}\index{Monoidal functor} is
a functor between monoidal categories
$F \co \ccal \to \dcal$  with a morphism
$\nu \co S_\dcal \to F(S_\ccal)$ and natural morphisms
$FX \otimes FY \to F(X \otimes Y)$ which are coherently associative
and unital. If these maps are isomorphisms then
$F$ is a \textbf{strong monoidal functor}\index{Strong monoidal functor}.
A \textbf{symmetric monoidal functor}\index{Symmetric
monoidal functor} between symmetric monoidal
categories is a monoidal functor $F$ such that
the following diagram commutes.
$$\xymatrix{
FX \otimes FY \ar[r] \ar[d]_{T_{FX,FY}} &
F(X \otimes Y) \ar[d]_{F(T_{X,Y})} \\
FY \otimes FX \ar[r] & F(Y \otimes X)
}$$
\end{definition}

\begin{definition}
A \textbf{monoidal Quillen pair}\index{Monoidal Quillen pair} is
a Quillen pair $L : \ccal \overrightarrow{\leftarrow} \dcal : R$
between monoidal model categories with $R$ a monoidal
functor such that the following conditions hold.
\begin{enumerate}
\item For all cofibrant objects $A$ and $B$ in $\ccal$
the induced map (see below) $L(A \otimes_\ccal B) \to LA \otimes_\dcal LB$
is a weak equivalence in $\dcal$.
\item For some (hence any) cofibrant replacement of the unit in $\ccal$,
$\cofrep S_\ccal \to S_\ccal$, the composite
$L\cofrep S_\ccal \to LS_\ccal \to S_\dcal$
is a weak equivalence in $\dcal$.
\end{enumerate}

The map $L(A \otimes_\ccal B) \to LA \otimes_\dcal LB$
is the adjoint of the composite:
$A \otimes_\ccal B \to RLA \otimes_\ccal RLB \to R(LA \otimes_\dcal LB)$.

A \textbf{strong monoidal Quillen pair}\index{Strong monoidal Quillen pair}
is a monoidal Quillen pair such that the maps
$L(A \otimes_\ccal B) \to LA \otimes_\dcal LB$
and $LS_\ccal \to S_\dcal$ are isomorphisms, thus
$L$ is a strong monoidal functor.
A monoidal Quillen pair
is a \textbf{symmetric monoidal Quillen pair}\index{Symmetric
monoidal Quillen pair} if the right adjoint is a symmetric monoidal
functor and the following diagram commutes.
$$\xymatrix{
L(X \otimes Y) \ar[r] \ar[d]_{L(T_{X,Y})} &
LX \otimes LY \ar[d]_{T_{LX,LY}}  \\
L(Y \otimes X) \ar[r] & LY \otimes LX
}$$
\end{definition}
Thus a strong monoidal adjunction is precisely the same as
a monoidal Quillen adjunction of \cite[Definition 4.2.16]{hov99}.
The conditions on the left adjoint of a monoidal Quillen pair
ensure that this passes to a strong monoidal functor
on the homotopy categories. We have the result below.
\begin{proposition}
A monoidal Quillen pair gives a strong monoidal adjunction
between the homotopy categories.
\end{proposition}
A monoidal Quillen equivalence between a pair of
monoidal model categories tells us that they have the same
homotopy categories and that we can use this adjunction
to compare the monoidal products.
In particular, there is
\cite[Theorem 3.12]{ss03monequiv}, very roughly this says that
if you add in some extra assumptions
then a monoidal Quillen equivalence
passes to Quillen equivalences
between categories of modules
and monoids.
A strong monoidal Quillen equivalence
behaves better from a categorical point of view.
For example, it will
preserve structures such as
enrichments, tensorings and cotensorings.

We need a technical condition on our model categories
to ensure good behaviour. This is of great importance
when constructing new model categories from existing ones.
The following definition is \cite[Definition 2.1.17]{hov99},
we will explain the new concepts in the definition shortly.
\begin{definition}
A \textbf{cofibrantly generated}\index{Cofibrantly generated}
model category is a model category $\ccal$ with sets of maps $I$
and $J$ such that:
\begin{enumerate}
\item the domains of $I$ are small relative to $I$-cell,
\item the domains of $J$ are small relative to $J$-cell,
\item the class of fibrations is the class of maps
with the right lifting property with respect to every map in $J$,
\item the class of acyclic fibrations is the class of maps
with the right lifting property with respect to every map in $I$.
\end{enumerate}
We call $I$ the set of
\textbf{generating cofibrations}\index{Generating cofibrations}
and $J$ the set of \textbf{generating acyclic
cofibrations}\index{Generating acyclic cofibrations}.
\end{definition}

For $\lambda$ an ordinal, a $\lambda$-sequence in
$\ccal$ (a category with all small colimits) is
a colimit preserving functor $X \co \lambda \to \ccal$.
If $\kappa$ is a cardinal, an ordinal $\lambda$
is $\kappa$-filtered if it is a limit ordinal
and if $A \subseteq \lambda$ with $|A| \leqslant \kappa$,
then $\sup A < \lambda$.
This is from \cite[Definitions 2.1.1 and 2.1.2]{hov99},
the following pair of definitions are
\cite[Definitions 2.1.3 and 2.1.9]{hov99}.
\begin{definition}
Let $\ccal$ be a category with all small colimits, $I$
a collection of maps in $\ccal$ and $\kappa$ a cardinal.
Then an object $A$ of $\ccal$
is \textbf{$\kappa$-small with respect to $I$}
if the map of sets
$$\colim_{\beta < \lambda} \ccal(A, X_\beta) \longrightarrow
\ccal(A, \colim_{\beta < \lambda} X_\beta) $$
is an isomorphism for all
$\kappa$-filtered ordinals $\lambda$
and all $\lambda$-sequences $X$ such that
$X_\beta \to X_{\beta +1}$ is in $\dcal$.
We say that \textbf{$A$ is small with respect
to $I$} if there is an cardinal
$\kappa$ such that $A$ is
$\kappa$-small with respect to $I$.
We say that $A$ is \textbf{small}\index{Small}
if it is small with respect to $\ccal$ itself.
\end{definition}

\begin{definition}
For a set of maps $I$ in a model category
the collection $I$-cell of
\textbf{relative $I$-cell complexes}\index{Relative cell complex}
is the collection of transfinite compositions of pushouts of maps of $I$.
That is a map in $I$-cell has form
$X_0 \to \colim_{\beta < \lambda} X_\beta$
for some ordinal $\lambda$ and some $\lambda$
sequence $X$, such that for each
$\beta + 1 < \lambda$ there is a pushout square as below
such that $g_\beta$ is in $I$.
$$\xymatrix{
C_\beta \ar[r] \ar[d]_{g_\beta} &
X_\beta \ar[d] \\
D_\beta \ar[r] \ar@{}[ur]|\ulcorner &
X_{\beta + 1} \ar@{}[ul]|{\phantom{\cdot} \cdot}}$$
\end{definition}

The definition and development below
are taken from \cite{ss03stabmodcat},
which uses the language of triangulated categories, which
we do not define.
Instead, we note that
the homotopy categories of all of the model categories
that we will use are triangulated categories.
In the language of \cite[Definition 2.1.1]{ss03stabmodcat}
these categories are stable model categories.
In particular each of these has a `shift' functor,
a pair of inverse equivalences of the homotopy category
$(\Sigma, \Omega)$. For spectra this is the suspension
and loop pair, for chain complexes of $R$ modules the
adjoint pair is $(- \otimes_R R[1], \hom_R (R[1], -))$.
For $n \geqslant 0$ we let $\Sigma^n$ be the $n$-fold
composite of $\Sigma$ and $\Sigma^{-n}$ be the
$n$-fold composite of $\Omega$.
The definition below is \cite[Definition 2.1.2]{ss03stabmodcat}.
\begin{definition}
An object $X$ of a triangulated category $\tcal$ is
\textbf{compact}\index{Compact} if
$$[X, \bigvee_i Y_i] \cong \bigoplus_i [ X , Y_i]$$
for any collection of objects $Y_i$.
A full triangulated subcategory of $\tcal$ is called a
\textbf{localising subcategory}\index{Localising subcategory}
if it is closed under coproducts in $\tcal$.
A set $\pcal$, of objects of $\tcal$, is called a set
of \textbf{generators}\index{Generator} if the only
localising subcategory of $\tcal$ that contains
$\pcal$ is $\tcal$ itself.
\end{definition}
We will say that an object $X$ of a stable model
category is compact or a generator
if it is so in the homotopy category.
In the homotopy category of a stable model category
we can define $[X,Y]_*$, the
\textbf{graded set of maps}\index{Graded maps}
in the homotopy category from $X$ to $Y$ to be the
graded set with $[X,Y]_n = [\Sigma^n X,Y]$ for $n\ \in \zz$.

\begin{lemma}
Let $\tcal$ be a triangulated category with infinite coproducts
and let $\pcal$ be a set of compact objects. Then the following are
equivalent:
\begin{enumerate}
\item The set $\pcal$ is a set of generators.
\item An object $X$ of $\tcal$ is acyclic if and only if
$[P,X]_*=0$ for all $P \in \pcal$.
\end{enumerate}
\end{lemma}
\begin{proof}
This is \cite[Lemma 2.2.1]{ss03stabmodcat}.
\end{proof}

The following is \cite[Definition 3.3]{ss00},
this is an important condition which
allows one (see \cite[Theorem 4.1]{ss00})
to make model structures for categories
of modules and algebras over a ring object.
\begin{definition}
Let $\ccal$ be a monoidal model category,
let $P$ be the class of maps of the form
$f \otimes \id_Z \co X \otimes Z \to Y \otimes Z$
where $f$ is an acyclic cofibration and
$Z$ is an object of $\ccal$.
The model category $\ccal$ satisfies the
\textbf{monoid axiom}\index{Monoid axiom}
if every map in $P$-cell is a weak equivalence.
\end{definition}

\newpage

\section{Spaces, Spectra and Equivariance}\label{sec:spacesandspectra}

We will use many different categories, as
is common in algebraic topology. We introduce
some of the more important ones here. All
of our examples are pointed model categories.
We reiterate that throughout $G$ will be a
compact Lie group.

\paragraph*{Spaces}\label{para:spaces} The category of spaces we use
is $\tscr_*$\index{T@$\tscr_*$}, compactly generated
weak Hausdorff based spaces.
Based means that every object of the category has a
distinguished point, called the basepoint.
Maps of based spaces are required to
preserve the basepoints.
This is a symmetric monoidal
category under the smash product. The model structure
is cofibrantly generated with generating cofibrations
the boundary inclusions $S^{n-1}_+ \to D^n_+$ and generating
acyclic cofibrations the inclusions
$D^n_+ \to (D^n \times I)_+$, for $n \geqslant 0$.
The weak equivalences
are the weak homotopy equivalences ($\pi_*$-isomorphisms)
and the fibrations are the Serre fibrations: those maps
with the right lifting property with respect to the
generating acyclic cofibrations. This
material has been taken from \cite[Section 2.4]{hov99}.
In general whenever we mention a topological space
we mean an object of $\tscr_*$.

\paragraph*{$G$-Spaces}\label{para:Gspaces}
We can also consider $G \tscr_*$\index{G T@$G \tscr_*$}, based
$G$-equivariant spaces,
as in \cite[Chapter III, Section 1]{mm02}.
The basepoint of a space is required to be $G$-fixed.
The generating cofibrations are those maps
$(G/H \times S^{n-1})_+ \to (G/H \times D^n)_+$
where $n \geqslant 0$ and $H$ runs through the closed subgroups of $G$.
The generating acyclic cofibrations are the maps
$(G/H \times  D^n)_+ \to (G/H \times D^n \times I)_+$.
A map $f$ of $G$-spaces is a weak equivalence or fibration
if and only if each $f^H$ (the map induced on $H$-fixed points)
is a weak equivalence or fibration of topological spaces. This
category is symmetric monoidal under the smash product
of $G$-spaces, where $G$ acts diagonally.

\paragraph*{Simplicial Sets}\label{para:simpsets}
We take the following
information from \cite[Chapter 3]{hov99}.
Let $\Delta$ be the category with objects
$[n]=\{0,1, \dots, n\}$ for $n \geqslant 0$
and morphisms the set maps such that
$x \geqslant y $ implies $f(x) \geqslant f(y)$.
A \textbf{simplicial set}\index{Simplicial set}
is a contravariant functor $\Delta \to \SET$.
This is the category of simplicial sets
and we will denote the category of based simplicial
sets by $\SSET$.
The generating cofibrations are the maps
$\partial \Delta[n] \to \Delta[n]$ for $n \geqslant 0$.
The generating acyclic cofibrations
are $\Lambda^r[n] \to \Delta[n]$ for $n > 0$
and $0 \leqslant r \leqslant n$. The cofibrations
are the injective maps and the fibrations
are precisely the Kan fibrations. The
weak equivalences are those maps $f$ such that the geometric
realisation $|f|$ is a weak homotopy equivalence of topological
spaces. We must also use $\sqq \leftmod$\label{app:simpQmod},
simplicial $\qq$-modules.
This is the category of
contravariant functors $\Delta \to \qq \leftmod$.
The model structure on this
category has fibrations and weak equivalences
defined in terms of underlying simplicial sets.

\paragraph*{Symmetric Spectra}\label{para:symspec}
We take our definitions of symmetric spectra
from \cite{hss00}. A \textbf{symmetric spectrum}\index{Symmetric spectrum}
is a collection $\{ X_n \}_{n \geqslant 0}$ of based simplicial sets
with a basepoint preserving
left action of $\Sigma_n$ on $X_n$.

This collection must have based maps $S^1 \smashprod X_n \to X_{n+1}$
for each $n \geqslant 0$ such that
the map $S^p \smashprod X_n \to X_{n+p}$ is
$\Sigma_p \times \Sigma_n$-equivariant for $p \geqslant 1$
and $n \geqslant 0$ (with $\Sigma_p$ permuting the
$p$-factors of $S^p=(S^1)^p$).
A map $f \co X \to Y$ of symmetric spectra is then a sequence of pointed
$\Sigma_n$-equivariant maps $f_n \co X_n \to Y_n$
which commute with the structure maps.
There is a tensor product $\otimes$ of
symmetric spectra $X$ and $Y$ defined by
$$(X \otimes Y)_n = \bigvee_{p+q=n}
(\Sigma_n)_+ \smashprod_{\Sigma_p \times \Sigma_q} (X_p \smashprod X_q).$$
Now we consider a symmetric spectrum $S$ with
$S_n=S^n$ (with the permutation action of
$\Sigma_n$ on $S_n$), this is a ring object with
respect to $\otimes$.
A symmetric spectrum is naturally a left $S$-module. The smash product of
symmetric spectra is then given by $X \otimes_S Y$,
the tensor product over $S$.
A map is a cofibration if it has the left lifting property
with respect to those maps $f$ such that each
$f_n$ is an acyclic fibration of simplicial sets.
The fibrations are characterised by the
right lifting property and the fibrant objects
are the $\Omega$-spectra, those spectra such that
the adjoints of the structure maps
($S^1 \smashprod X_n \to X_{n+1}$) are
weak equivalences.
A spectrum $E$ is injective if given
$f \co X \to Y$, a monomorphism
and a level-wise equivalence
and a map $g \co X \to E$
 there is an $h \co Y \to E$
such that $hf=g$.
The weak equivalences are those maps $f$
such that $E^0 f$ is an isomorphism for
every injective $\Omega$-spectrum $E$. The notation
for this category is $Sp^\Sigma(\SSET)$ which we will usually
shorten to $Sp^\Sigma$\index{S@$Sp^\Sigma$}.

There are also versions of symmetric spectra defined in terms
of topological spaces as constructed in \cite{mmss01}.
In that paper they also consider a model structure
$Sp^\Sigma_+$\label{app:possymspec} called the positive stable model structure.
Here the fibrant objects are those spectra
which are $\Omega$-spectra in all levels
greater than 0 and the weak equivalences
are defined in terms of injective positive
$\Omega$-spectra. The paper
\cite{hov01} generalises these results
still further and considers symmetric spectra
in general model categories.
Fix some monoidal model category $\ccal$ and choose some
cofibrant object $K$ to take the place of $S^1$.
The construction of symmetric spectra
in $\ccal$ then follows the same plan as above,
replacing simplicial sets with $\ccal$
and $S^1$ with $K$. This category will
be denoted $Sp^\Sigma( \ccal; K)$ or just
$Sp^\Sigma( \ccal)$\label{app:gensymspec}\index{S @$Sp^\Sigma( \ccal)$}.

\paragraph*{Orthogonal Spectra}\label{para:orthogspec}
We take this from \cite[Chapter II]{mm02}.
Choose some collection of finite dimensional
irreducible real representations of $G$. We let $U$
be the direct sum of countably many copies of each
of these representations, we say that $U$ is a
\textbf{universe}\index{Universe}.
If we take every finite dimensional representation
in our collection then we call $U$ \textbf{complete}.
Now we define an
\textbf{indexing $G$-space in $U$}\index{Indexing space}
to be a finite dimensional sub $G$-inner product space
of $U$.

Define $\iscr_G^U$ to be a category with objects
those real inner $G$-product spaces
isomorphic to an indexing $G$-space in $U$.
The morphisms of this category are
the $G$-spaces of (non-equivariant)
linear isometric isomorphisms, with
$G$ acting by conjugation.
An \textbf{$\iscr_G$-space} is then a functor,
enriched over based $G$-spaces, from
$\iscr_G^U$ to based $G$-spaces.
Let $G \iscr \tscr$ be the category of
$\iscr_G$-spaces with morphisms the
$G$-equivariant maps.
There is then an obvious notion of
an external smash product of two
$\iscr_G$-spaces $X$ and $Y$:
$X \bar{\smashprod} Y \co \iscr_G^U \times \iscr_G^U \to G \tscr_*$.
Using a left Kan extension one can internalise
this and obtain a smash product operation on
$\iscr_G$-spaces.
The category of \textbf{orthogonal $G$-spectra},
$G \iscr \sscr$\index{G IS@$\GIS$} is the category of
left $\sphspec$-modules in $G \iscr \tscr$,
where $\sphspec$ is the $\iscr_G$-space which takes
$V$ to $S^V$, called the sphere spectrum.
The monoidal product on orthogonal
$G$-spectra is the smash product over $\sphspec$,
$X \smashprod_\sphspec Y := \coeq
(X \smashprod \sphspec \smashprod Y \overrightarrow{\to}
X \smashprod Y)$.

\paragraph*{EKMM $S$-modules}\label{para:ekmm}
 We also make use of the category of
\textbf{$G$-equivariant EKMM $S$-modules}. The brief details
that we give are taken from \cite[Chapter IV]{mm02},
which takes the work of \cite{EKMM97} and adapts it to
the equivariant setting. We start with the category of
$G$-prespectra. An object of this category is a collection of
$G$-spaces $X(V)$, for each indexing space $V$ of some $G$-universe $U$
with $G$-equivariant structure maps $S^{W-V} \smashprod X(V) \to X(W)$
that satisfy the obvious transitivity condition for
$V \subset W \subset X$. A map of $G$-prespectra is
a collection of equivariant maps
$f(V) \co X(V) \to Y(V)$ that commute with the structure maps.
A $G$-May spectrum is a prespectrum such the adjoints of the
structure maps, $X(V) \to F( S^{W-V} , X(W))$,
are homeomorphisms. Any $G$-prespectrum can be made into
a $G$-May spectrum, this is called spectrification.

Let $\iscr(U,U)$ be the space of linear
isometries $U \to U$, with $G$ acting by conjugation.
Now recall the notion of the half-twisted smash product from
\cite{EKMM97}. This construction can be applied to the equivariant
case and for a $G$-May spectrum $X$ we have a $G$-May spectrum
$\iscr(U,U) \ltimes X$.
There is a monad $\mathbb{L}= \iscr(U,U) \ltimes (-)$ on $G$-May spectra
with $G$-equivariant structure maps. We then consider
$\mathbb{L}$-spectra: the category of modules over $\mathbb{L}$,
where the action map is required to be a $G$-map.
The category of $\mathbb{L}$-spectra has an
associative and commutative smash product $\smashprod_{\lscr}$.
The sphere May-spectrum $S$ (the spectrification
of the prespectrum which takes value $S^V$ at level $V$)
is an $\mathbb{L}$-spectra.
If $N$ is an $\mathbb{L}$-spectrum
there is a natural map of $\mathbb{L}$-spectra
$\lambda_N \co S \smashprod_{\lscr} N \to N$.
The category of $S$-modules is defined to be
those $\mathbb{L}$-spectra $N$ such that
$\lambda_N$ is an isomorphism.
This is our desired category of
EKMM $S$-modules, written $G \mcal$. Let $V$ be an indexing
space and $H$ a subgroup of $G$, then
a generalised sphere spectrum\index{Generalised
sphere spectrum} is a spectrum of the form
$S \smashprod_{\mathscr{L}} \mathbb{L} \Sigma_V^\infty (G/H_+ \smashprod S^n)$
for $n \geqslant 0$. These are used to define the model structure
on $G \mcal$\index{G M@$G\mcal$} can be thought of the building blocks
of the category.

Sometimes we will use the generic term spectrum
to indicate either of an orthogonal spectrum
or an EKMM $S$-module. We will do so when
either the category is implicitly understood or when we are making
a statement that applies to either of the above categories.

\paragraph*{Differential graded $ R$-Modules}\label{para:dgrmod}
It is appropriate to define the usual
model structure on chain complexes
of $R$-modules, $dg R \leftmod$\index{dg R@$dg R \leftmod$}, in this section.
This structure is known as the
\textbf{projective model structure}\index{Projective model structure}.
We take these definitions and results from
\cite[Section 2.3]{hov99}.
A map of chain complexes is a weak equivalence
if it is a homology isomorphism and
a fibration if it is a surjection.
Cofibrations are level-wise
split monomorphisms with cofibrant
cokernel. For each $n \in \zz$, let $S^{n}R$
be the chain complex concentrated
in degree $n$, where it takes value $R$
and let $D^n R$ be the chain complex
with $R$ in degrees $n$ and $n-1$
and zeroes elsewhere, with
the identity as the differential
from degree $n$ to $n-1$.
This is cofibrantly generated
model category with generating cofibrations
the maps $S^{n-1}R \to D^n R$
and generating acyclic cofibrations
$0 \to D^n R$. We will also need to consider
$dg R \leftmod_+$\label{app:dgrmod+}, the category of
non-negatively graded chain
complexes of $R$-modules. This has a model
structure with weak equivalences the homology isomorphisms
and fibrations the surjections.

\section{The categories $\GIS$ and $G \mcal$}\label{sec:GIS}
Orthogonal spectra and EKMM $S$-modules
are the most important of the categories that
we will consider. We give some basic results
that will be of use later.
Both of these categories have forgetful
functors to the category of $G$-prespectra.
The weak equivalences of
orthogonal spectra and EKMM $S$-modules
are defined in terms of these underlying prespectra.
The following is \cite[Chapter III, Definition 3.2]{mm02}.

\begin{definition}
For $H$ a subgroup of $G$ and $r$ an integer,
the homotopy group $\pi_r^H(X)$\index{pi @$\pi_r^H$} of a
$G$-prespectrum $X$ is
$$\begin{array}{rcll}
\pi_r^H(X) & = & \colim_V \pi_r^H ( \Omega^V X(V))
& \text{if }r \geqslant 0 \\
\pi_{-r}^H(X) & = & \colim_{V \supset \rr^r}
\pi_0^H ( \Omega^{V-\rr^r} X(V))
& \text{if }r > 0 \\
\end{array} $$
where the colimits run over the $G$-indexing spaces in $U$.
\end{definition}

\begin{theorem}\label{thm:orthogmodel}
For $G$ a compact Lie group, the following classes of maps
define a cofibrantly generated,
proper, closed symmetric monoidal
model structure on $\GIS$\index{G IS@$\GIS$}, the category of
$G$-equivariant orthogonal spectra. This model
structure satisfies the monoid axiom.
The weak equivalences
are those maps $f$ such that $\pi_r^H(f)$
is an isomorphism
for all subgroups $H$ of $G$ and all integers $r$.
The cofibrations
are the maps with the left lifting property with respect to maps
which are level-wise acyclic fibrations of $G$-spaces.
The fibrations are those maps
 with the right lifting property with respect to
cofibrations which are also weak equivalences.
\end{theorem}
\begin{proof}
This summary consists of
\cite[Chapter III, Theorem 4.2 and Proposition 7.5]{mm02}.
\end{proof}

As one would expect, we will often call weak equivalences
\textbf{$\pi_*$-isomorphisms}\index{pi-isomorphism@$\pi_*$-isomorphism}.
We will shortly define a model structure
on EKMM spectra and we would like a strong monoidal
Quillen equivalence between $\GIS$ and EKMM spectra.
To obtain such an equivalence we need
a slightly different model structure on orthogonal spectra
called the \textbf{positive model structure}\index{Positive model structure}
(see \cite[Chapter III, Section 5]{mm02}).
The other reason to use the positive model structure
is to create a model structure of commutative ring spectra in
orthogonal spectra. One cannot do this with
the usual model structure on orthogonal spectra
as is commented upon in \cite[Section 14]{mmss01}
which references \cite{lewis91}.

A positive level acyclic fibration is a map of orthogonal spectra that is
an acyclic fibration of $G$-spaces on all levels $V$ with $V^G \neq 0$.
A positive cofibration is a map with the left
lifting property with respect to the positive level
acyclic fibrations.
A positive fibration is a map with the right lifting property
with respect to the positive cofibrations that are also
$\pi_*$-isomorphisms.
\begin{theorem}\label{thm:positiveorthogmodel}
The positive cofibrations, positive fibrations
and weak equivalences
define a cofibrantly generated, proper, closed
symmetric monoidal model structure on
$G$-equivariant orthogonal spectra. This model
category satisfies the monoid axiom and will be denoted
by $\GIS_+$\index{G IS+@$\GIS_+$}.
The identity functor is the left adjoint of a Quillen equivalence
$\id \co \GIS_+ \to \GIS$.
\end{theorem}
\begin{proof}
This is
\cite[Chapter III, Theorem 5.3 and Propositions 5.8 and 7.3]{mm02}.
\end{proof}

We now turn to EKMM $S$-modules, the model structure we are interested in
is the generalised cellular structure. The generating cofibrations
are the maps $E \to I \smashprod E$ and the
generating acyclic cofibrations are the maps
$(I \smashprod E) \to (I \smashprod E) \smashprod I_+$
for a generalised sphere spectrum.
\begin{theorem}\label{thm:ekmmmodel}
For $G$ a compact Lie group, the $\pi_*$-isomorphisms,
generalised cofibrations and restricted $q$-fibrations
form a cofibrantly generated, proper, closed
symmetric monoidal model structure on
$G \mcal$\index{G M@$G\mcal$}.
There is a strong symmetric monoidal Quillen
equivalence
$$\mathbb{N} : \GIS_+ \overrightarrow{\longleftarrow}
G\mcal : \mathbb{N}^{\#}.
\index{$N$@$(\nn,\nn^{\#})$}
$$
\end{theorem}
\begin{proof}
We have taken this from
\cite[Chapter IV, Theorems 1.1, 1.2, and 2.9]{mm02}.
\end{proof}

We remind the reader that the categories
$\GIS$, $\GIS_+$ and $G \mcal$
depend on the choice of universe $U$.
When we wish to specify the universe
we are working with, we shall decorate the notation
for these categories accordingly.
We will always require that our universe is complete.
We will (temporarily) denote the sphere spectrum
in each of these categories by $\sphspec$.

\begin{definition}
For $G$ a compact Lie group the
\textbf{Burnside ring}\index{Burnside ring} $A(G)$ is
defined to be $[\sphspec, \sphspec]^G$, maps in
the homotopy category of $G \mcal$ or $\GIS$.
\end{definition}

See \cite[Chapter V, Definitions 2.1 and 2.9]{lms86} for more details.
The ring multiplication is given by composition, which is commutative
since $[f \circ g] = [f \smashprod g]= [g \smashprod f]$,
for homotopy classes of maps $[f]$ and $[g]$.
It is a well known result that when $G$ is finite,
$[\sphspec, \sphspec]^G$ is isomorphic as a ring to the
Grothendieck group of isomorphism classes of finite $G$-sets.
The following result implies that $\pi_0^G(\sphspec) \cong A(G)$.

\begin{lemma}\label{lem:spherehomotopy}
Let $X$ be an orthogonal spectrum, then for
any subgroup $H$ of $G$ and integers $q \geqslant 0$, $p>0$
$$ \begin{array}{rcl}
{}[\Sigma^q S^0 \smashprod G/H_+, X]^G & \cong & \pi^H_q(X) \\
{}[F_p S^0 \smashprod G/H_+, X]^G & \cong &  \pi^H_{-p}(X)
\end{array} $$
where the left hand side denotes maps in the homotopy category
of $\GIS$.
\end{lemma}
\begin{proof}
We apply \cite[Chapter III, Theorem 4.16]{mm02}
which expresses maps of orthogonal spectra
in terms of their underlying prespectra.
Then \cite[Part III, Proposition 2.8]{adams} relates
maps in the homotopy category of $G$-prespectra to
homotopy groups.
\end{proof}

The same result holds of $G \mcal$, the only change necessary
is to use \cite[Chapter IV, Theorem 2.9]{mm02} to move to prespectra.

\begin{definition}
For a subgroup $H$ of $G$, the inclusion
of $H$ in $G$ will be written $\iota_H$.
The map $G \to \{ e \}$ will be written
$\varepsilon_G$.
\end{definition}

From these maps we have the change of groups functors
$\iota_H^*$ and $\varepsilon_G^*$.
For $X$ a $G$-space, $\iota^*_H(X)$ is $X$ considered as an
$H$-space. A non-equivariant space $Y$ can be thought of as a $G$-space
with trivial action, we call this $G$-space $\varepsilon_G^*(Y)$,
this functor is known as the inflation functor.

\begin{lemma}
There is a Quillen pair
$\varepsilon_G^* : \tscr_* \overrightarrow{\longleftarrow}
G \tscr_* : (-)^G$\index{$E$@$(\varepsilon_G^*,(-)^G)$}
and for each subgroup $H$ of $G$ there is a Quillen pair
$G_+ \smashprod_H (-) : H \tscr_* \overrightarrow{\longleftarrow}
G \tscr_* : \iota^*_H.$\index{$G$@$(G_+ \smashprod_H (-),\iota^*_H)$}
Furthermore these functors are related by the natural isomorphism
$\varepsilon_H^* \cong \iota_H^* \circ \varepsilon_G^* $.
\end{lemma}

There are spectrum level versions of these functors for both
$\GIS$ and $G \mcal$. The definitions differ of course, but the
essential idea comes from the space level versions.
More details can be found in \cite[Chapter V, Sections 2 and 3]{mm02}
for $\GIS$ and \cite[Chapter VI, Sections 1 and 3]{mm02}
for $G \mcal$.
Since we are working equivariantly we need to adjust
our notion of compact slightly.
\begin{definition}
An object $X$ of $\GIS$ or $G \mcal$ is
\textbf{$H$-compact}\index{H-compact@$H$-compact} if
$$[\iota_H^* X, \bigvee_i Y_i]^H \cong \bigoplus_i [\iota_H^* X , Y_i]^H $$
for any collection of $H$-spectra $Y_i$.
\end{definition}
That is, $X$ is $H$-compact if $\iota_H^* X$
is compact in the category of $H$-spectra.

\begin{lemma}\label{lem:compactgenerators}
The suspension spectra $\Sigma^\infty G/H_+$
for $H$ a closed subgroup of $G$
are a set of $G$-compact generators for the
category of $G$-spectra. Hence $\sphspec$ is $H$-compact
for each subgroup $H$ of $G$.
\end{lemma}
\begin{proof}
This is a well known fact, see
\cite[Chapter I, Definition 4.4 and Lemma 5.3]{lms86}
or \cite[Theorem 9.4.3]{hps97}.
We prove this result here, though we will need
some definitions and results from later in the work.
That the $G/H_+$ are generators follows from
Lemma \ref{lem:spherehomotopy}.
That these objects are $G$-compact follows from
Proposition \ref{prop:HomotopyofHoColim}.
Take some collection $\{ Y_i \}_{i \in I}$,
we describe $\bigvee_{i \in I} Y_i$
in terms of a filtered colimit.
Consider the diagram $\mathbb{P}(I)$ with object
set the collection of subsets of $I$ and morphisms
the inclusions. This is obviously a filtered diagram
and if we define a functor $X \co \mathbb{P}(I) \to \GIS$
by $X(J) = \bigvee_{j \in J} Y_j$
with morphisms the obvious inclusions
we see that $\colim_{J} X(J) \cong \bigvee_{i \in I} Y_i$.
Furthermore the maps $X(J) \to X(J')$ for $J \subseteq J'$
are all $h$-cofibrations (Definition \ref{def:hcof}), so that
$\colim_{J} X(J)$ is weakly equivalent to
$\hocolim_{J} X(J)$ (Definition \ref{def:hocolim}).
Hence we have the following isomorphisms
which complete the proof.
$$\pi_*^H(\bigvee_{i \in I} Y_i)
\cong \pi_*^H(\colim_{J} X(J))
\cong \colim_J \pi_*^H(X(J))
\cong \bigoplus_{i \in I} \pi_*^H(Y_i)$$
\end{proof}

For the result below we need a definition: the $G$-space
$\fcal G$\index{F G@$\fcal G$} is the collection of
subgroups of $G$ with finite index in their normaliser
with topology given by the Hausdorff topology.
\begin{lemma}[tom Dieck]
For $G$ a compact Lie group, there is an isomorphism of rings
\index{Dieck isomorphism@tom Dieck's isomorphism}
$A(G) \otimes \qq \cong C(\fcal G / G , \qq).$
\end{lemma}
\begin{proof}
We have taken this result from \cite[Chapter V, Lemma 2.10]{lms86},
which references \cite[Lemma 6]{tdieck77}.
\end{proof}

In \cite{gre98b} a space $\scal_f G$\index{s fg@$\scal_f G$} is
constructed, as a set it consists of
the closed subgroups of $G$, but it does not have the
topology induced from the Hausdorff metric.
We can relate $\scal_f G$ to the space $\fcal G$ via the equivalence relation
$\sim$ on $\scal_f G$ which is generated by the following: two subgroups
$H \trianglelefteqslant H'$
are related by $\sim$ if the quotient $H'/H$ is a torus. Then
$(\scal_f G/ \sim) \cong \fcal G$ and we use
tom Dieck's isomorphism
to see that an idempotent of $A(G) \otimes \qq$ corresponds to an
open and closed $G$-invariant subspace of $\scal_f G$ that is a union of $\sim$ classes.

\begin{definition}
Let $a \in A(G) \otimes \qq$,
then the \textbf{support}\index{Support}
of $a$ is the set of $H \leqslant G$ such that
$a(H) \in \qq$ is non-zero (considering $a$ as a continuous map
$(\scal_f G/ \sim) /G \to \qq$).
\end{definition}

\begin{definition}\label{def:hcof}
Working in either of $\GIS$ or $G \mcal$, a map $f \co X \to Y$ is called an
$h$-\textbf{cofibration}\index{h-cofibration@$h$-cofibration}
if it satisfies the \textbf{$G$-homotopy extension property}
\index{G-homotopy extension property@$G$-homotopy extension property}
defined as follows:
whenever there is a pair of maps of $G$-spectra $F \co X \smashprod I_+ \to Z$ and
$g \co Y \to Z$ such that $F \circ i_0 = g \circ f$
there exists a map
$F' \co Y \smashprod I_+ \to Z $ making the diagram below commute.
\begin{displaymath}
\xymatrix{
X \ar[r]^(0.4){i_0} \ar[d]_f & X \smashprod I_+ \ar[d]^F \ar@(r,u)[ddr]^{f \smashprod \id} & \\
Y \ar[r]^(0.4)g \ar@(d,l)[drr]^{i_0} & Z & \\
& & Y \smashprod I_+ \ar[ul]_{F'} }
\end{displaymath}
\end{definition}
We have taken this definition from \cite[Section 5]{mmss01},
as stated in that section there is a universal test case.
Let $Z = Mf$ (the mapping cylinder, see Definition \ref{def:cone/cofibre}),
and let $g$ and $F$ be the
evident maps. If a suitable $F'$ exists in this case,
then $f$ is an $h$-cofibration.
The shorthand notation
is $\xymatrix{X \ar@{{+}{-}{>}}[r] & Y}$.

\begin{lemma}\label{lem:hcofLLP}
A map $f \co X \to Y$ is an $h$-cofibration if and only if
it has the left lifting property with respect to the class of maps
$ ev_0 \co F(I_+,B) \to B$. That is, $f \co X \to Y$
is an $h$-cofibration if and only if
every commutative diagram of the form below has a lift.
\begin{displaymath}
\xymatrix@!C{
X \ar[d]_f \ar[r]^g
& F(I_+,B) \ar[d]^{ev_0} \\
Y \ar@{.>}[ur] \ar[r]_h & B }
\end{displaymath}
\end{lemma}
\begin{proof}
It is easy to show that this is an equivalent condition
to the definition, perhaps the only point worth noting is that if
$h \co X \smashprod I_+ \to B$ is the adjoint map to
$\tilde{h} \co X \to F(I_+,B)$ then
$h \circ i_0 = ev_0 \circ \tilde{h}$.
\end{proof}

From this description it is clear that pushouts,
colimits, retracts and compositions of
$h$-cofibrations are also $h$-cofibrations.
There are some standard maps which
are always used in homotopy theory, we have
$i_0 \co S^0 \to I_+$, it takes the non-basepoint
point of $S^0$ to $0 \in I_+$. Equally there is
$i_1 \co S^0 \to I_+$ which sends that point to $1$.
We also need $j \co S^0 \to I$ which includes $S^0$
as the endpoints into $I$.

\begin{lemma}\label{lem:basiccofspaces}
The maps $i_0, i_1 \co S^0 \to I_+$
and $j \co S^0  \to I$ are cofibrations
of $G$-spaces.
\end{lemma}
\begin{proof}
The map $i_0$ is a generating acyclic cofibration of
the model category of $G$-spaces, hence $i_1$ is
a cofibration.
Label the points of $S^0_+$ as 0, 1 and $+$ (the basepoint)
and similarly for the
endpoints and basepoint of $I_+$.
We can express $S^0$ and $I$ as quotients of the spaces
$S^0_+$ and $I_+$ by the pushout diagrams below.
\begin{displaymath}
\xymatrix@!C{
{\{0 \} }_+ \ar[d] \ar[r] & S^0_+ \ar[d]^a
& \qquad & { \{0 \} }_+ \ar[d] \ar[r] & I_+ \ar[d]^b \\
{*} \ar@{}[ur]|\ulcorner \ar[r] & S^0 \ar@{}[ul]|{\phantom{\cdot} \cdot}
& \qquad & {*} \ar@{}[ur]|\ulcorner \ar[r] & I \ar@{}[ul]|{\phantom{\cdot} \cdot} }
\end{displaymath}
Now we take a test diagram, with $f$ an acyclic fibration of $G$-spaces
\begin{displaymath}
\xymatrix@!C{
S^0_+ \ar[d]^i \ar[r]^a & S^0 \ar[r]^\alpha \ar[d]^j & A \ar@{->>}[d]^f_{\sim}  \\
I_+  \ar[r]_b & I \ar[r]_\beta & B}
\end{displaymath}
this diagram gives a lift $g \co I_+ \to A$
and since $g(+)=g(0)$ this passes to a map
$h \co I \to A$ and this provides the requisite lift to show that
$j$ is a cofibration of $G$-spaces.
\end{proof}

\begin{lemma}\label{lem:hcofsGtopological}
If $f \co X \to Y$ is a cofibration of $G$-spaces and
$A$ is a $G$-spectrum (in $\GIS$ or $G \mcal$), then
$\id \smashprod f \co A \smashprod X \to A \smashprod Y$
is an $h$-cofibration. Smashing with $A$ preserves
$h$-cofibrations of $G$-spectra.
\end{lemma}
\begin{proof}
The first statement follows from standard adjunctions
relating smashing spectra with $G$-spaces
to the space of maps between two $G$-spectra.
Looking at the universal test case it is easy to see that
smashing with a spectrum preserves $h$-cofibrations
(for orthogonal spectra this statement is
\cite[Chapter III, Lemma 7.1]{mm02}).
\end{proof}

We now give the analogue of
\cite[Chapter III, Lemma 2.5]{mm02} for $G \mcal$.

\begin{lemma}\label{lem:smodhcof}
A cofibration of $G \mcal$ is an $h$-cofibration
\end{lemma}
\begin{proof}
A generating cofibration has the form $E \to CE=E \smashprod I$
for $E$ a generalised sphere $S$-module.
These are all $h$-cofibrations of $G$-spectra since the maps $S^0 \to I$ and
$S^0 \to I_+$ are cofibrations of $G$-spaces.
Now recall that $h$-cofibrations are preserved by forming
relative cell complexes and retracts to complete the proof.
\end{proof}

\begin{corollary}\label{cor:basiccofs}
For a $G$-spectrum $X$ (in $\GIS$ or $G \mcal$), the maps
$i_0, i_1 \co X  \to X \smashprod I_+$
and $j \co X  \to CX$ are $h$-cofibrations.
In addition, if $X$ is cofibrant, then these maps are cofibrations.
\end{corollary}
\begin{proof}
The first statement follows from Lemma \ref{lem:hcofsGtopological}
and the second from the fact that
$\smashprod \co \GIS \times G \tscr_* \to \GIS$
and
$\smashprod \co G \mcal \times G \tscr_* \to G \mcal$
are Quillen bifunctors.
We note that the operations above are Quillen bifunctors
because $\GIS$ and $G \mcal $ are $G$-topological
(\cite[Chapter III, Definition 1.14]{mm02}),
in fact these conditions are equivalent.
\end{proof}

\begin{lemma}
Every object of $G \mcal$ is fibrant.
\end{lemma}
\begin{proof}
We must check that following diagram has a lift,
for $E$ a generalised sphere spectrum.
\begin{displaymath}
\xymatrix{*+<0.5cm>{CE} \ar[r]^f \ar@{>->}[d]_{i_0}^{\sim} & X \ar@{+->}[d] \\
CE \smashprod I_+ \ar[r] & {*} }
\end{displaymath}
There is a retraction map
$r \co CE \smashprod I_+ \to CE$ such that $r \circ i_0=1$ and
we take the lift to be $f \circ r$.
\end{proof}

We can give the analogues of \cite[Chapter III, Propositions 7.3 and 7.4]{mm02}
for $G \mcal$.
\begin{lemma}\label{lem:sModMonoid}
Let $X$ be a generalised cofibrant spectrum in $G \mcal$,
then the functor $X \smashprod -$ preserves weak equivalences.
For any acyclic cofibration $f$ and any spectrum $Z$,
the map $f \smashprod \id_Z$ is an $h$-cofibration and
a $\pi_*$-isomorphism. Furthermore pushouts and sequential colimits
of such maps are $h$-cofibrations and $\pi_*$-isomorphisms.
\end{lemma}
\begin{proof}
The spectrum $X$ is a retract of a generalised cell complex $Y$. That is:
we have the following diagram $X \to Y \to X$ with composite
map the identity. It suffices to prove this result for generalised cell complexes
by the following argument.
Assume the result holds for $Y$ and
take $f \co A \to B$ a weak equivalence, then we have the diagram:
\begin{displaymath}
\xymatrix@!C{
X \smashprod A \ar[d] \ar[r]
& Y \smashprod A \ar[r]  \ar[d]^{\sim}
& X \smashprod A \ar[d] \\
X \smashprod B \ar[r]
& Y \smashprod B \ar[r]
& X \smashprod B}
\end{displaymath}
hence $X \smashprod A \to X \smashprod B$ is a retract of a weak equivalence
and thus a weak equivalence.
The arguments in
\cite[Chapter III, Proposition 7.3]{mm02}
show that we can reduce this result to proving that
if $C$ is a spectrum with trivial homotopy groups, then
$C \smashprod E$ has trivial homotopy groups for $E$ a generalised sphere $S$-module.
Fix $n$ and a $G$-indexing space $V$ and consider the generalised sphere $S$-module
$E = S \smashprod_{\mathscr{L}} \mathbb{L} \Sigma^\infty_V(G/H_+ \smashprod S^n)$.
Since $E$ is strongly dualisable
(\cite[Chapter III]{lms86} or \cite[Chapter XVI, Section 7]{may96})
we see (by the conditions on $C$)
$$\pi_n^K(C \smashprod E) = [S^n, C \smashprod E]^K
\cong [S^n \smashprod F(E, S), C]^K=0.$$
The rest of the lemma follows by the same proof
as given for \cite[Proposition 12.5]{mmss01}.
For the last statement we use
\cite[Chapter IV, Remark 2.8]{mm02}.
\end{proof}

\begin{rmk}
As stated above $\GIS$ and $G \mcal$
are Quillen equivalent.
In fact \cite{mm02} shows much more, it proves that
every construction in $G \mcal$
is the same (up to homotopy) to the corresponding construction
in $G$-equivariant orthogonal spectra.
Such constructions include fixed and orbit spectra, geometric fixed point spectra,
change of groups and change of universe. So if one were to define an
`equivariant stable homotopy theory', then we would have a theorem - both
$G$ equivariant $S$-modules
and $G$-equivariant orthogonal spectra model
the same `equivariant stable homotopy theory'.  When $G$ is
the trivial group the category of $G$-equivariant orthogonal spectra
is a model for the stable homotopy category.
\end{rmk}

\section{Homotopy Colimits}\label{sec:hocolim}
When one has to construct an object of a model category
as a colimit from homotopy level
information there will often be choices made in the construction.
The correct gadget to organise these choices is a homotopy colimit.
Given diagrams $C$ and $D$ in our model category,
a map of diagrams $f \co C \to D$ will induce
a map $\colim(f) \co \colim(C) \to \colim(D)$.
If $f$ is an object-wise weak equivalence the
induced map will not necessarily be a weak equivalence.
A good source for a modern description of these homotopy colimits is \cite{dhks04}.
We will need to use homotopy pushouts and homotopy sequential colimits.
All of our categories are cofibrantly generated, thus
we can use functorial factorisation in these categories
to give some functoriality in our definition of homotopy colimits.
Most of the constructions and results of this section hold for quite
general model categories. Our main assumptions are that the model
category is proper and topological (see \cite[Definition 4.2.18]{hov99}).
We have taken the following from \cite[section 10]{ds95}.

\begin{definition}
Consider a diagram $B \leftarrow A \to C$,
let $\cofrep A \to A$ be the cofibrant replacement
of $A$. Factor
the maps $\cofrep A \to B$ and $\cofrep A \to C$
into cofibrations followed by
acyclic fibrations.
\begin{displaymath}
\xymatrix{
B' \ar@{->>}[d]^{\sim}
& \ \cofrep A \ \ar@{>->}[l] \ar@{>->}[r] \ar@{->>}[d]^{\sim}
& C' \ar@{->>}[d]^{\sim} \\
B & A \ar[l] \ar[r] & C }
\end{displaymath}
The colimit of $B' \leftarrow \cofrep A \to C'$
is the \textbf{homotopy pushout}\index{Homotopy pushout}
of $B \leftarrow A \to C$, written
$\hocolim (B \leftarrow A \to C)$.
\end{definition}

This construction has the property that a map of pushout diagrams
that is an object-wise weak equivalence induces a
weak equivalence between the homotopy pushouts,
this is verified in \cite[section 10]{ds95}.
The reader can also use the definition of a homotopy pushout in \cite[13.5]{hir03}
which has the desired property since our categories are proper (hence left proper).
In that section they also prove that there is a weak equivalence between
this definition and the construction above.

\begin{definition}\label{def:cone/cofibre}
Take a map $f \co X \to Y$, then $Y/X$ is
the colimit of the diagram $* \leftarrow X \to Y$.
The \textbf{mapping cylinder}\index{Mapping cylinder} of $f$,
$M_f$, is the pushout of the diagram
$X \smashprod I_+ \leftarrow X \to Y$ using the map $i_0 \co X \to X \smashprod I_+$.
The \textbf{mapping cone}\index{Mapping cone|see{Cofibre}}, $Cf$,
or \textbf{cofibre}\index{Cofibre} of $f$ is the colimit of the diagram
$CX \leftarrow X \to Y$ (see the diagram below).
The \textbf{homotopy cofibre}\index{Homotopy cofibre} of $f$ is the homotopy pushout
of the diagram $* \leftarrow X \to Y$.
\end{definition}

\begin{displaymath}
\xymatrix@!C{
& X \ar[r]^f  \ar@{{+}{-}{>}}[d]_{i_0} & Y \ar@{{+}{-}{>}}[d]\\
X \ar@{{+}{-}{>}}[r]^{i_1} \ar[d]
& X \smashprod I_+ \ar@{}[ur]|\ulcorner \ar[r] \ar[d]
& M_f \ar[d] \ar@{}[ul]|{\phantom{\cdot} \cdot} \\
{*} \ar@{{+}{-}{>}}[r] \ar@{}[ur]|\ulcorner
& CX \ar[r] \ar@{}[ul]|{\phantom{\cdot} \cdot} \ar@{}[ur]|\ulcorner
& Cf \ar@{}[ul]|{\phantom{\cdot} \cdot}  }
\end{displaymath}
In the diagram above the maps $i_0$ and $i_1$ are the
inclusions of $X$ into the $0$ and $1$ ends of
$X \smashprod I_+$ and the reader
should recall that the map $X \to CX$ is an $h$-cofibration.
The mapping cone construction has required property
that a commutative square as below,
such that the vertical maps
are weak equivalences, induces a weak equivalence
of the cofibres $Cf \to Cg$.
$$\xymatrix{
X \ar[r]^f \ar[d]^\sim & Y \ar[d]^\sim \\
A \ar[r]_g & B }$$
This is proved by repeated application of
\cite[Chapter III, Theorem 3.5]{mm02}
(or \cite[Chapter IV, Remark 2.8]{mm02} for $G \mcal$)
which state that given a diagram as below,
\begin{displaymath}
\xymatrix@!C{
X \ar[d]^\sim &  Y  \ar@{+->}[l] \ar[r] \ar[d]^\sim & Z \ar[d]^\sim \\
X' &  Y' \ar@{+->}[l] \ar[r] & Z' }
\end{displaymath}
the induced map of pushouts is a weak equivalence.
The dual construction to $Cf$ gives $Ff$, the
\textbf{fibre}\index{Fibre} of a map,
defined in terms of the pullback of $f$ over the
map $F(I , Y) \to Y$.
Note that if $X$ and $Y$ are cofibrant
then $Cf$ is cofibrant by Corollary \ref{cor:basiccofs} and the fact
that cofibrations are preserved by pushouts (using the diagram $CX \leftarrow X \to Y$).
The dual to $Y/X$ is $f^{-1}*$, the pre-image of the basepoint of $Y$.
For the sake of completeness and as a useful exercise in the above definitions we
record the result below which can be thought of as saying that the mapping cone
is a construction of the homotopy cofibre.

\begin{lemma}
The mapping cone of a map $f \co X \to Y$ is weakly equivalent
to the homotopy cofibre of $f$.
\end{lemma}
\begin{proof}
We begin by replacing $X$ by a cofibrant object, $\cofrep X \to X$
then we factorise the maps
$\cofrep X \to *$ and $\cofrep X \to Y$
into cofibrations followed by
acyclic fibrations,
\begin{displaymath}
\xymatrix@!C{
Z' \ar@{->>}[d]^{\sim} &
\ \cofrep X \ \ar@{>->}[l] \ar@{>->}[r] \ar@{->>}[d]^{\sim} &
Y' \ar@{->>}[d]^{\sim} \\
{*} & X \ar[l] \ar[r] & Y }
\end{displaymath}
We now draw a comparison diagram
\begin{displaymath}
\xymatrix{
Z'  \ar@{->>}[d]_0^\sim & \ar@{>->}[l] \ \cofrep X \
\ar@{>->}[r]^{f'} \ar[d]_{i_1}^\sim & Y' \ar[d]_{i_1}^\sim  \\
C \cofrep X  & \ar[l]_q \cofrep X \smashprod I_+ \
\ar@{>->}[r]^{f' \smashprod 1}  & Y \smashprod I_+  \\
C \cofrep X  \ar[u]^{=}_{\sim} & \ar@{+->}[l]  \cofrep X \
\ar@{>->}[r]^{f'} \ar[u]^{i_0}_\sim & Y' \ar[u]^{i_0}_\sim }
\end{displaymath}
The map labelled 0 is the map which sends $Z'$ to the basepoint of $C \cofrep X$.
The maps $f'$ and $f' \smashprod 1$ are cofibrations, hence $h$-cofibrations and
all vertical maps are weak equivalences. Thus, the pushouts of
the horizontal diagrams are weakly equivalent.
The final comparison is
\begin{displaymath}
\xymatrix@!C{
C \cofrep X \ar[d]^\sim & \ \cofrep X \ \ar@{+->}[l] \ar@{>->}[r] \ar[d]^\sim & Y' \ar[d]^\sim \\
C X & \ X \ \ar@{+->}[l] \ar[r] & Y }
\end{displaymath}
and this induces a weak equivalence of the pushouts by the standard argument.
\end{proof}

\begin{lemma}\label{lem:homotopycofibre}
For any two homotopic maps $f_0, f_1 \co X \to Y$ there is a chain of weak equivalences
between $\cofibre(f_0)$ and $\cofibre(f_1)$.
\end{lemma}
\begin{proof}
Use the commutative diagram below.
\begin{displaymath}
\xymatrix@!C{
CX  \ar[d]_{i_0}^\sim   & \ar@{+->}[l] \ X \ar[r]^{f_0} \ar[d]_{i_0}^\sim  & Y \ar[d]_=^\sim \\
CX \smashprod I_+    & \ar@{+->}[l] \ X \smashprod I_+ \ar[r]^F       & Y  \\
CX  \ar[u]^{i_1}_\sim & \ar@{+->}[l] \ X \ar[r]^{f_1} \ar[u]^{i_1}_\sim   & Y \ar[u]^=_\sim }
\end{displaymath}
The maps $i_0$ and $i_1$ are (level equivalences and hence) weak equivalences
and $X \to CX$ is an $h$-cofibration. Hence we
apply \cite[Chapter III, Theorem 3.5]{mm02} and see that
$\cofibre(f_0)$, $\cofibre(F)$ and $\cofibre(f_1)$ are all weakly equivalent.
\end{proof}

\begin{lemma}\label{lem:homcomcof}
A homotopy commuting square
\begin{displaymath}
\xymatrix@!C{
A \ar[d]_p \ar[r]^f & B \ar[d]^q \\
C \ar[r]_g & D }
\end{displaymath}
induces a zig-zag of maps between $\cofibre(f)$ and $\cofibre(g)$.
Furthermore if $p$ and $q$ are weak equivalences
then the zig-zag consists of weak equivalences.
\end{lemma}
\begin{proof}
We know that $q \circ f \simeq g \circ p$ and thus by
Lemma \ref{lem:homotopycofibre} we know that these maps have weakly equivalent
cofibres with a zig-zag of comparisons
$\cofibre(q \circ f) \to \cofibre(F) \leftarrow \cofibre(g \circ p)$.
All we need now is the pair of diagrams
\begin{displaymath}
\xymatrix@!C{
A \ar[d]_= \ar[r]^f & B \ar[d]^q    & \qquad & A \ar[d]_p \ar[r]^{g \circ p} & D \ar[d]_= \\
A \ar[r]_{q \circ f} & D & \qquad & C \ar[r]_g & D }
\end{displaymath}
which will give us the comparison maps $\cofibre(f) \to \cofibre(q \circ f)$
and $\cofibre(g \circ p) \to \cofibre(g)$. These are weak equivalences
when $p$ and $q$ are.
\end{proof}

\begin{definition}\label{def:hocolim}
For a sequence of maps
$$X_0 \overset{f_0}{\longrightarrow} X_1 \overset{f_1}{\longrightarrow}
X_2 \overset{f_2}{\longrightarrow} X_3 \overset{f_3}{\longrightarrow} X_4 \longrightarrow \dots$$
we define the \textbf{homotopy sequential colimit}\index{Homotopy sequential colimit}
or \textbf{telescope}\index{Telescope}
to be the colimit of the following diagram.
\begin{displaymath}
\xymatrix{
*+<0.3cm>{X_0}  \ar@{>->}[r]^{g_0} \ar@{->>}[d]^{h_0}_= &
*+<0.3cm>{X'_1} \ar@{>->}[r]^{g_1} \ar@{->>}[d]^{h_1}_{\sim} &
*+<0.3cm>{X'_2} \ar@{>->}[r]^{g_2} \ar@{->>}[d]^{h_2}_{\sim} &
*+<0.3cm>{X'_3} \ar@{>->}[r]^{g_3} \ar@{->>}[d]^{h_3}_{\sim} &
*+<0.3cm>{X'_4} \ar@{>->}[r]^{g_4} \ar@{->>}[d]^{h_4}_{\sim}  &  \dots \\
X_0 \ar[r]_{f_0} &
X_1 \ar[r]_{f_1} &
X_2 \ar[r]_{f_2} &
X_3 \ar[r]_{f_4} &
X_4 \ar[r]_{f_5} & \dots }
\end{displaymath}
Where $h_{i+1} \circ g_i$ is a factorisation of $f_i \circ h_i$ into a
cofibration followed by a acyclic fibration.
We denote this construction by $\hocolim_i X_i$.
\end{definition}

\begin{lemma}
For a map of colimit diagrams as below,
there is an induced map of homotopy colimits
$\hocolim_i X_i \to \hocolim_i Y_i$.
\begin{displaymath}
\xymatrix{
X_0 \ar[r]^{g_0} \ar[d]^{h_0} &
X_1 \ar[r]^{g_1} \ar[d]^{h_1} &
X_2 \ar[r]^{g_2} \ar[d]^{h_2} &
X_3 \ar[r]^{g_3} \ar[d]^{h_3} &
X_4 \ar[r]^{g_4} \ar[d]^{h_4} &  \dots \\
Y_0 \ar[r]_{f_0} &
Y_1 \ar[r]_{f_1} &
Y_2 \ar[r]_{f_2} &
Y_3 \ar[r]_{f_4} &
Y_4 \ar[r]_{f_5} & \dots }
\end{displaymath}
\end{lemma}
\begin{proof}
This follows from the fact that we have assumed functorial
factorisation in our definition of a model category.
\end{proof}

This lemma with the following proposition
will imply that if the above map of colimit
diagrams is an object-wise weak equivalence, then the
induced map of homotopy colimits is a weak equivalence.

\begin{proposition}
For a map of colimit diagrams that is an object-wise weak equivalence
as below, the induced map of colimits is a weak equivalence.
\begin{displaymath}
\xymatrix{
X_0 \ar@{+->}[r]^{g_0} \ar[d]^{h_0}_{\sim} &
X_1 \ar@{+->}[r]^{g_1} \ar[d]^{h_1}_{\sim} &
X_2 \ar@{+->}[r]^{g_2} \ar[d]^{h_2}_{\sim} &
X_3 \ar@{+->}[r]^{g_3} \ar[d]^{h_3}_{\sim} &
X_4 \ar@{+->}[r]^{g_4} \ar[d]^{h_4}_{\sim} &  \dots \\
Y_0 \ar@{+->}[r]_{f_0} &
Y_1 \ar@{+->}[r]_{f_1} &
Y_2 \ar@{+->}[r]_{f_2} &
Y_3 \ar@{+->}[r]_{f_4} &
Y_4 \ar@{+->}[r]_{f_5} & \dots }
\end{displaymath}
\end{proposition}
\begin{proof}
We begin by extending the diagram to include the cofibres of the
$h_i$, we have $l_i \co Y_i  \to C{h_i}$,
which is an $h$-cofibration and $k_i \co C{h_i}  \to C{h_{i+1}}$.
The map $k_i$ is an $h$-cofibration by inspection
and since each $C{h_i}$ is acyclic (consider the
long exact sequence of a cofibration)
each $k_i$ is a weak equivalence.

We now apply \cite[Chapter III, Theorem 3.5 or Chapter IV, Remark 2.8]{mm02}
to see that $\colim_i(C{h_i} )$
is acyclic. Now we apply the standard yoga of colimits
and left adjoints (recall that $CX_i = X_i \smashprod I$)
to see that $\colim_i(C{h_i} ) \cong C{\colim_i (h_i)}$,
Hence, both are acyclic and we can conclude
that $\colim_i (h_i)$ is a weak equivalence.
\end{proof}

Now we give a rather specific lemma,
analogous to Lemma \ref{lem:homcomcof}
which allows us to take homotopy level
information and use it to create a zig-zag of maps
between homotopy colimits. The constructions in this result
are standard. We use this lemma in the proof of
Theorem \ref{thm:generalsplitting}.

\begin{lemma}\label{lem:homcomtelescope}
Consider two sequential colimit diagrams
$g_i \co X_i \to X_{i+1}$ and
$f_i \co Y_i \to Y_{i+1}$, with each $f_i$
an $h$-cofibration ($i \geqslant 0$).
Assume that there is a collection of weak equivalences
$h_i \co X_i \to Y_i$, with homotopies
$F_i \co X_i \smashprod I \to Y_i$
such that $F_i \circ i_0 = h_{i+1} \circ g_i$
and $F_i \circ i_1 = f_i \circ h_i$.
Then $\hocolim_i X_i$ and $\colim_i Y_i$
are weakly equivalent.
\end{lemma}
\begin{proof}
We start by drawing the information above as a
diagram. We have only assumed that
$g_i \circ h_i$ and $h_{i+1} \circ f_i$
are homotopic, so we obtain a
homotopy commuting diagram as below.
\begin{displaymath}
\xymatrix{
X_0 \ar[r]^{g_0} \ar[d]^{h_0} &
X_1 \ar[r]^{g_1} \ar[d]^{h_1} &
X_2 \ar[r]^{g_2} \ar[d]^{h_2} &
X_3 \ar[r]^{g_3} \ar[d]^{h_3} &
X_4 \ar[r]^{g_4} \ar[d]^{h_4} &  \dots \\
Y_0 \ar@{+->}[r]_{f_0} \ar@{}[ur]|{F_0}&
Y_1 \ar@{+->}[r]_{f_1} \ar@{}[ur]|{F_1}&
Y_2 \ar@{+->}[r]_{f_2} \ar@{}[ur]|{F_2}&
Y_3 \ar@{+->}[r]_{f_4} \ar@{}[ur]|{F_3}&
Y_4 \ar@{+->}[r]_{f_5} \ar@{}[ur]|{F_4}& \dots }
\end{displaymath}
We now perform the classical construction of
the homotopy colimit to replace
$g_i \co X_i \to X_{i+1}$ with a sequence
of $h$-cofibrations $k_i \co Z_i \to Z_{i+1}$
with weak equivalences $a_i \co X_i \to Z_i$
and $r_i \co Z_i \to X_i$ such that
$r_i \circ a_i = \id_{X_i}$ and $r_{i+1} \circ k_i = g_i \circ r_i$.
Let $Z_0 =X_0$ and $a_0=r_0= \id_{X_0}$, assume inductively we have created
stage $i$, we construct stage $i+1$ of $Z_i$
as a pushout below.
$$
\xymatrix{
X_i \ar[r]^{g_i} \ar@{{+}{-}{>}}[d]_{i_0 \circ a_i}^{\sim} &
X_{i+1} \ar@{{+}{-}{>}}[d]^{a_{i+1}}_\sim \\
Z_i \smashprod I_+ \ar@{}[ur]|\ulcorner \ar[r]_{b_i} &
Z_{i+1} \ar@{}[ul]|{\phantom{\cdot} \cdot} \\
} $$
The map $a_{i+1} \co X_{i+1} \to Z_{i+1}$,
is an $h$-cofibration and a weak equivalence by
\cite[Chapter IV, Theorem 3.5]{mm02}.
We define $k_i \co Z_i \to Z_{i+1}$ as the composite
$b_i \circ i_1$, which is an $h$-cofibration.
The map $r_{i+1}$ is induced by
$\id_{X_{i+1}}$ and the map
$Z_i \smashprod I_+ \to X_{i+1}$
given by retracting to $Z_i$,
then applying $k_i \circ r_i$.
Since $a_{i+1}$ is a weak equivalence so is
$r_i$. The collection $r_i \co Z_i \to X_{i+1}$
gives a map of colimit diagrams.
Thus we have induced weak equivalences
$\hocolim_i Z_i \to \colim_i Z_i$
and $\hocolim_i Z_i \to \hocolim_i X_i$.

Now we compare the $Z_i$ to the $Y_i$, this is slightly
more complicated.
We will construct weak equivalences $m_i \co Z_i \to Y_i$
and homotopies $G_i \co Z_i \smashprod I_+ \to Y_{i+1}$
such that $m_i \circ a_i =h_i$,
and $m_{i+1} \circ b_i = G_i$.
Since $Z_0 =X_0$, we can take $m_0 =h_0$.
We construct $G_0$ using the fact that
$a_0$ is an $h$-cofibration.
\begin{displaymath}
\xymatrix{
X_0
\ar[r]^(0.4){i_1} \ar[d]_{a_0} &
X_0 \smashprod I_+
\ar[d]^{F_0} \ar@(r,u)[ddr]^{a_0 \smashprod \id} & \\
Z_0
\ar[r]^(0.4){f_0 \circ m_0} \ar@(d,l)[drr]^{i_1} &
Y_1 & \\
& &
Z_0 \smashprod I_+ \ar[ul]_{G_0} }
\end{displaymath}
From the definition of $Z_1$ the maps
$h_1 \co X_1 \to Y_1$ and $G_0 \co Z_0 \smashprod I_+ \to Y_1$
induce $m_1 \co Z_1 \to Y_1$.
We now inductively make $m_{i+1}$ from $G_i$
and $G_{i+1}$ from $m_{i+1}$.
The $m_i$ induce a map of colimits
$\colim_i Z_i \to \colim_i Y_i$. Since
$m_i \circ a_i = h_i$, $m_i$ is a weak
equivalence for all $i \geqslant 1$.
Hence the induced map of colimits
is a weak equivalence.
\end{proof}

\begin{proposition}\label{prop:hocolimproperties}
If $X_0$ is cofibrant then $\hocolim_i X_i$ is cofibrant.
If either of $X_0$ or $Y$ is cofibrant then there is a weak equivalence
$$\hocolim_i (Y\smashprod X_i ) \to Y \smashprod \hocolim_i (X_i).$$
\end{proposition}
\begin{proof}
This follows by standard manipulations of the definitions and the fact
that smashing with a cofibrant object preserves weak equivalences.
\end{proof}

\section{Rational Sphere Spectra}\label{sec:rationalsphere}
We construct a `rational sphere spectrum' which
we will use in Section \ref{sec:GISQ} to construct
model categories of rational spectra.
We will construct $\qq$ as a group and translate this into spectra.
We will do this for both $G \mcal$ and $\GIS$.
Later we will specialise to $G \mcal$ and make a rational
sphere spectrum that is a commutative ring spectrum.
In order to make the following clear
we use $\sphspec$ for the unit of the smash product of
$\GIS$ and $S$ the unit for $G \mcal$, we work in $\GIS$ to begin with.

Take a free resolution of $\qq$ as an abelian group, $0 \to R
\overset{f}{\to} F \to \qq \to 0$. Since a free abelian group is
a direct sum of copies of $\zz$ we can rewrite this short exact
sequence as $0 \to \bigoplus_i \zz \overset{f}{\to} \bigoplus_j
\zz \to \qq \to 0$. Since $\qq$ is flat, the sequence $0 \to
\bigoplus_i M \overset{f}{\to} \bigoplus_j M \to \qq \otimes M \to
0$ is exact for any abelian group $M$.
Hence for each subgroup $H$ of $G$, we have an injective  map
(which we also denote as $f$)
$\bigoplus_i A(H) \overset{f}{\to} \bigoplus_j A(H)$
and $\bigoplus_j A(H) / \bigoplus_i A(H) \cong A(H) \otimes \qq$.

\begin{lemma}
For $H$, a subgroup of $G$,
$$[\bigvee_i \sphspec, \bigvee_j \sphspec]^H \cong
\prod_i \bigoplus_{j} [\sphspec, \sphspec]^H \cong
\prod_i \bigoplus_{j} A(H)
\cong \hom_{A(H)} \Big( \bigoplus_{i} A(H),\bigoplus_{j} A(H) \Big).$$
\end{lemma}
\begin{proof}
Since maps out a coproduct is a product of maps
we have an isomorphism
$$[\bigvee_i \sphspec, \bigvee_j \sphspec]^H \cong
\prod_i [\sphspec, \bigvee_j  \sphspec]^H .$$
Now we apply the fact that $\sphspec$ is $H$-compact
and known isomorphisms to obtain the result.
\end{proof}

Thus we can choose
$g \co \fibrep \bigvee_i \sphspec \to \fibrep \bigvee_j \sphspec$,
a representative for the homotopy class corresponding to $f$.

\begin{definition}
For the map $g$ as constructed above, the cofibre of $g$ is the
\textbf{rational sphere spectrum}\index{Rational
sphere spectrum}\index{S rational@${S^0{\qq}}$}
and we have a cofibre sequence
$$\fibrep \bigvee_i  \sphspec \overset{g}{\longrightarrow}
\fibrep \bigvee_j \sphspec \longrightarrow S^0{\qq}. $$
\end{definition}
A different choice of representative for the homotopy class $[g]$
will induce a weak equivalence between the cofibres,
and hence (up to weak equivalence) $S^0{\qq} $ is independent of
this choice of representative.
Returning to $G \mcal$ we can perform the analogous construction:
$\bigvee_i \cofrep S \to \bigvee_j \cofrep S
\to S^0_\mcal{\qq}$\index{S rational QM@$S^0_\mcal{\qq}$},
where we need to take cofibrant replacements (since $S$ is not cofibrant
in $G \mcal$). For $\GIS_+$ we construct
$\fibrep \bigvee_i \cofrep \sphspec \to
\fibrep \bigvee_j \cofrep \sphspec
\to S^0{\qq}_+$\index{S rational positive @$S^0{\qq}_+$}.
Note that there is an inclusion $\alpha \co \sphspec \to \fibrep \bigvee_j \sphspec$
which sends $\sphspec$ to the term of $\bigvee_j \sphspec$
corresponding to $1 \in \qq$. We have similar maps for $G \mcal$
and $\GIS_+$ using $\cofrep S \to \bigvee_j \cofrep S$ and
$\cofrep \sphspec \to \fibrep \bigvee_j \cofrep \sphspec$.

\begin{proposition}
The rational sphere spectra $S^0{\qq}$, $S^0_\mcal{\qq}$
and  $S^0{\qq}_+$ are cofibrant.
\end{proposition}
\begin{proof}
We give the proof for $\GIS$, the other cases follow by the same
argument.
The diagram below gives the definition of the cofibre of $g$.
\begin{displaymath}
\xymatrix@!C{
*+<0.5cm>{\fibrep {\bigvee}_i \sphspec} \ar[r]^g \ar@{>->}[d]_{j} &
*+<0.5cm>{ \fibrep {\bigvee}_j \sphspec} \ar@{>->}[d] \\
C {\fibrep \bigvee}_i \sphspec \ar[r] \ar@{}[ur]|\ulcorner
& Cg \ar@{}[ul]|{\phantom{\cdot} \cdot} }
\end{displaymath}
Since a wedge of cofibrant objects is again cofibrant,
$\bigvee_i \sphspec$ is cofibrant, hence so is
$\fibrep \bigvee_i \sphspec$.
Corollary \ref{cor:basiccofs} tells us that $j$ is a cofibration
and cofibrations are preserved by pushouts, so
$\bigvee_j \sphspec \to Cg$ is a cofibration from a cofibrant object to
$Cg$.
\end{proof}

\begin{proposition}\label{prop:rathomgps}
Let $X$ be an orthogonal $G$-spectrum, then
for any subgroup $H$ of $G$ there is a natural isomorphism
$\pi_*^H(X \smashprod S^0{\qq}) \cong \pi_*^H(X)\otimes \qq$.
Furthermore, the map
$(\id_X \smashprod \alpha)_* \co \pi_*^H(X) \to \pi_*^H( S^0{\qq} \smashprod X)$
acts as $x \to x \otimes 1$. The corresponding statement
holds in $G \mcal$: there is an isomorphism
$\pi_*^H(X \smashprod S^0{\qq}) \cong \pi_*^H(X)\otimes \qq$.
The map $\pi_*^H(\cofrep S \smashprod X) \to \pi_*^H( S^0{\qq} \smashprod X)$
acts as $x \to x \otimes 1$.
\end{proposition}
\begin{proof}
Using the cofibre sequence which defines
$S^0{\qq}$ we have the following collection of isomorphic
long exact sequences of homotopy groups
$$\begin{array}{ccccccc}
\dots \longrightarrow &
\pi_n^H(X \smashprod \fibrep \bigvee_i \sphspec) & \overset{(\id \smashprod g)_*}{\longrightarrow} &
\pi_n^H(X \smashprod \fibrep \bigvee_j \sphspec) & \longrightarrow &
\pi_n^H(X \smashprod S^0{\qq})    & \longrightarrow
\dots \\
\dots \longrightarrow &
\pi_n^H(\bigvee_i X) & \overset{(\id \smashprod g)_*}{\longrightarrow} &
\pi_n^H(\bigvee_j X) & \longrightarrow &
\pi_n^H(X \smashprod S^0{\qq}) & \longrightarrow
\dots \\
\dots \longrightarrow &
\bigoplus_i \pi_n^H( X) & \overset{g \otimes \id}{\longrightarrow} &
\bigoplus_j \pi_n^H( X) & \longrightarrow &
\pi_n^H(X \smashprod S^0{\qq}) & \longrightarrow
\dots \\
\dots \longrightarrow &
\bigoplus_i \zz \bigotimes \pi_n^H( X)   & \overset{g \otimes \id}{\longrightarrow} &
\bigoplus_j \zz \bigotimes \pi_n^H( X)   & \longrightarrow &
\pi_n^H(X \smashprod S^0{\qq}) & \longrightarrow
\dots
\end{array} $$
Since the map $g \otimes \id \co (\bigoplus_i \zz) \otimes \pi_n^H( X) \to
(\bigoplus_j \zz) \otimes \pi_n^H( X)$ is injective for all $n$,
this long exact sequence splits into short exact sequences
and we conclude that
$\pi_*^H(X \smashprod S^0{\qq}) \cong \pi_*^H(X)\otimes \qq$.
The calculation of $(\id_X \smashprod \alpha)_*$ follows
immediately, as do the statements regarding $G \mcal$.
\end{proof}

\begin{lemma}\label{lem:equivratspheres}
The positive rational sphere spectrum $S^0{\qq}_+$
is weakly equivalent to the rational
sphere spectrum, $S^0{\qq}$.
The rational sphere spectra $\mathbb{N} S^0{\qq}_+$ and $S^0_\mcal{\qq}$ are weakly equivalent.
\end{lemma}
\begin{proof}
The first statement follows from the fact that
the positive cofibrant replacement of $\sphspec$ is
weakly equivalent to $\sphspec$.
For the second statement, let $c$ and $f$ denote fibrant replacement in $\GIS_+$
and $\cofrep$ be cofibrant replacement in $G \mcal$.
Since $\mathbb{N}$ is a strong monoidal
left adjoint, we know that the map
$\mathbb{N} c \sphspec \to \mathbb{N} \sphspec \cong S$
is a weak equivalence. Hence we can choose
(via lifting properties) a weak equivalence
$\mathbb{N} c \sphspec \to \cofrep S$.
We can then make the following diagram, which commutes
up to homotopy. Each of the horizontal maps
is some choice of a representative for $g$.
$$\xymatrix{
\nn f {\bigvee}_j  c \sphspec \ar[r]
& \nn f {\bigvee}_i  c \sphspec  \\
*+<0.5cm>{\nn  {\bigvee}_j  c \sphspec}
\ar[r] \ar@{>->}[u]_{\sim} \ar[d]^{\sim}
& *+<0.5cm>{\nn  {\bigvee}_i  c \sphspec}
\ar@{>->}[u]_{\sim} \ar[d]^{\sim} \\
{\bigvee}_j \cofrep S \ar[r]
& {\bigvee}_i \cofrep S
}$$
The vertical maps are weak equivalences,
hence the cofibres of the horizontal
maps are weakly equivalent.
\end{proof}

Later it will be important to have a commutative ring spectrum
that we can call a rational sphere, we give a method to
construct such a object.
We work exclusively in $G \mcal$ in the following.
We note here that the classical concept of
cell spectra can be reworded to:
$X$ is a cell spectrum if the map
$* \to X$ is in $I$-cell.
We may as well have made $S^0_\mcal{\qq}$
in $G \mcal$ from copies of a cellular replacement
of the sphere spectrum
(see \cite[Chapter III, Theorem 2.10]{EKMM97}).
If we do so, then we have the following lemma.

\begin{lemma}\label{lem:S0qqiscell}
The rational sphere spectrum $S^0_\mcal{\qq}$ is a cell $S$-module.
\end{lemma}
\begin{proof}
Recall the following from \cite[Chapter I, Lemma 5.7]{lms86}:
a wedge of cell $S$-modules is a cell $S$-module,
for a cell $S$-module $X$ the canonical map $X \to CX$
is an inclusion of a cell subcomplex
and the pushout of a cellular map
of cell $S$-modules along an inclusion of
a cell subcomplex is a cell $S$-module.
Now we use \cite[Chapter I, Theorem 5.8]{lms86} (Cellular Approximation)
to ensure that the map
$f \co \bigvee_i \Sigma^\infty S^0 \to \bigvee_j \Sigma^\infty S^0$
is cellular and the result follows since
$\Sigma^\infty S^0$ is a cell $S$-module.
\end{proof}

To obtain our commutative ring spectrum
we use \cite[Chapter VIII, Theorem 2.2]{EKMM97},
we give the statement that we will need
below. Here we assume that $E$ is
a cell spectrum (hence cofibrant).
We need to use the language of $E$-equivalences
and $E$-localisation from \ref{sec:localGIS}.

\begin{theorem}\label{thm:algebralocalise}
For a cell commutative $R$-algebra $A$, the localisation
$\lambda \co A \to A_E$ can be constructed as the inclusion of
a subcomplex in a cell commutative $R$-algebra $A_E$.
In particular $A \to A_E$ is an $E$-equivalence
and a cofibration of commutative ring spectra
for any cell commutative $R$-algebra $A$.
\end{theorem}
\begin{proof}
This result goes through with the same proof as in
\cite{EKMM97}.
\end{proof}

\begin{definition}\label{def:S_qq}
Let $S_\qq$\index{Sa @$S_\qq$} be the commutative ring spectrum
constructed as the $S^0_\mcal{\qq}$-localisation of $S$ in
$G \mcal$ with the generalised cellular model structure.
\end{definition}
It follows immediately that the unit
$\eta \co S \to S_\qq$ is an $S^0_\mcal \qq$-equivalence.

\begin{lemma}\label{lem:compareSQtoRQ}
There is a weak equivalence
$S^0_\mcal{\qq} \to S_\qq$.
\end{lemma}
\begin{proof}
Begin with the map $\cofrep S \to \bigvee_j \cofrep S$
which maps into the $\cofrep S$ factor corresponding to $1 \in \qq$.
This is clearly a cofibration, hence we obtain a
cofibration $\cofrep S \to S^0_\mcal{\qq}$
that is an isomorphism of rational homotopy groups
and hence by Lemma \ref{lem:ratequivs}
is an $S^0_\mcal \qq$-equivalence.
We then have the composite map
$\cofrep S \to S \to S_\qq$ which is a $S^0_\mcal \qq$-equivalence
into an $S^0_\mcal{\qq}$-local object.
Thus we obtain a $S^0_\mcal \qq$-equivalence $S^0_\mcal{\qq} \to S_\qq$
(via the lifting properties
in $L_{S^0_\mcal \qq} G \mcal$).
By Theorem \ref{thm:rathomotopymaps},
$S_\qq$ has rational homotopy groups
thus our lift must be a $\pi_*$-isomorphism.
\end{proof}

\chapter{Localisations of $G$-Spectra}\label{chp:locals}

We will define our category of rational spectra
in terms of a Bousfield localisation of
$\GIS$ or $G \mcal$.
We will need other localisations later,
so we begin by looking at the general case
as considered in \cite[Chapter IV, Section 6]{mm02},
our contribution
will be in proving that most of the good model category
properties of $G$-spectra are preserved by localisation
(left properness, the pushout product axiom and the monoid axiom).
We then construct our category of rational $G$-spectra
in Section \ref{sec:GISQ} and prove that it is
independent of the choices in our construction
(see Lemma \ref{lem:ratequivs} and
Propositions \ref{prop:GIS+QequivGISQ}
and \ref{prop:GMequivGIS+Q}).
We give a different construction in terms of
modules over a ring spectrum in Section \ref{sec:SQmod}
and show that this construction is Quillen equivalent to the previous one
(Theorem \ref{thm:localisedtomodules}).

\section{The Bousfield Localisations of $\GIS$ and $G \mcal$}\label{sec:localGIS}

The results below apply to $\GIS$, $\GIS_+$ and $G \mcal$
so we work with notation appropriate to $\GIS$
and will note when changes are needed
for $G \mcal$ or $\GIS_+$.
Aside from the theorem below which allows us to actually
construct these localisations, the major results
are Theorem \ref{thm:locfuncs} and Proposition \ref{prop:rmodlocal}.
The first allows us to compare localised categories
and the second shows that we can localise categories of modules
over a ring spectrum.
The definition below is
\cite[Chapter IV, Definition 6.2]{mm02}.

\begin{definition}\label{def:Estuff}
Let $E$ be a cofibrant spectrum or a cofibrant based $G$-space and let
$X$, $Y$ and $Z$ be orthogonal spectra.
\begin{enumerate}
\item \vskip -0.3cm A map $f \co X \to Y$ is an
$E$-\textbf{equivalence}\index{E-equivalence@$E$-equivalence}
 if $\id_E \smashprod f \co E \smashprod X \to E \smashprod Y$
is a weak equivalence.
\item $Z$ is $E$-\textbf{local}\index{E-local@$E$-local}
if $f^* \co [Y,Z]^G \to [X,Z]^G$ is an isomorphism for all $E$-equivalences
$f \co X \to Y$.
\item An $E$-\textbf{localisation}\index{E-localisation@$E$-localisation}
of $X$ is an $E$-equivalence $\lambda \co X \to Y$ from $X$ to an $E$-local object $Y$.
\item $A$ is $E$-\textbf{acyclic}\index{E-acyclic@$E$-acyclic} if the map $* \to A$ is an $E$-equivalence.
\end{enumerate}
\end{definition}

The following is a standard result, see
\cite[Theorems 3.2.13 and 3.2.14]{hir03}.

\begin{lemma}\label{lem:EequivElocal}
An $E$-equivalence between $E$-local objects is a weak equivalence.
\end{lemma}
\begin{proof}
Take $f \co A \to B$, an $E$-equivalence between $E$-local objects.
The square below commutes and the vertical maps are isomorphisms
by the definitions above.
$$\xymatrix{
[B,A]^G \ar[r]^{f_*} \ar[d]^\cong_{f^*} &
[B,B]^G \ar[d]^\cong_{f^*} \\
[A,A]^G \ar[r]^{f_*} &
[A,B]^G  }$$
There is a unique $[g] \in [B,A]^G$ such that $[g][f]=f^*[g]=[\id_A]$
and it follows that $[f][g]=f_*[g]=[\id_B]$.
Thus $[f]$ is an isomorphism in the homotopy category
and we can conclude that $f$
is a weak equivalence.
\end{proof}

\begin{theorem}\label{thm:GSlocal}
Let $E$ be a cofibrant spectrum or a cofibrant based $G$-space. Then $\GIS$
has an
$E$-model structures\index{E-model structure@$E$-model structure}
whose weak equivalences are the $E$-equivalences and whose $E$-cofibrations are the
cofibrations of $\GIS$.
The $E$-fibrant objects are precisely the fibrant $E$-local
objects and $E$-fibrant approximation
constructs a Bousfield localisation $f_X \co X \to \fibrep_E X$ of $X$ at $E$.
The notation for $E$-model structure on the underlying category of $\GIS$ is
$L_E \GIS$ or $\GIS_E$\index{G IS@$\GIS_E$}\index{G MM@$G \mcal_E$}.
\end{theorem}
\begin{proof}
This result is \cite[Chapter IV, Theorem 6.3]{mm02},
the proof of which is an adaptation
of the material in \cite[chapter VIII]{EKMM97}.
\end{proof}

These are cofibrantly generated model categories,
this is seen by carefully reading the proof of
\cite[Chapter VIII, Theorem 1.1]{EKMM97}.
Let $c$ be a fixed infinite cardinal that
is at least the cardinality of
$E^*(\sphspec)$.
Then define $\tscr$, a test set for $E$-fibrations,
to consist of all inclusions of cell complexes
$X \to Y$ such that the cardinality of the set of cells
of $Y$ is less than or equal to $c$.
Hence the domains of these maps are $\kappa$-small
where $\kappa$ is the least cardinal greater than $c$.

\begin{proposition}\label{prop:localquillen}
The identity map of $\GIS$ is the left adjoint of
a Quillen pair from $\GIS \to L_E \GIS$.
\end{proposition}
\begin{proof}
The cofibrations are unchanged and a $\pi_*$-isomorphism is an
$E$-equivalence since weak equivalences are preserved
by smashing with a cofibrant object.
\end{proof}

\begin{lemma}\label{lem:localcond}
A spectrum $X$ is $E$-local if and only if
$[A,X]^G=0$ for all $E$-acyclic spectra $A$.
\end{lemma}
\begin{proof}
We tackle the only if part first.
Assume $X$ is $E$-local,
take an $E$-acyclic
spectrum $A$, then since the map $* \to A$
is an $E$-equivalence, $[A,X]^G \cong [*, X]^G=0$.

For the converse we assume that
$[A,X]^G=0$ for all $E$-acyclic spectra $A$.
Since this is a homotopy level condition
we may as well assume that $X$ is fibrant.
Take an $E$-equivalence $f \co Y \to Z$
then we must prove that
$[Z, X]^G \cong [Y, X]^G$.
Hence we can assume that $X$ and $Y$ are cofibrant,
thus $Cf$ is cofibrant. By the long exact sequence of
homotopy groups of a cofibre,
it follows that $Cf \smashprod E \cong C(f \smashprod \id_E)$
is acyclic, hence $Cf$ is $E$-acyclic.
We have a fibre sequence $F(Cf,X) \to F(Z,X) \to F(Y,X)$,
Looking at the long exact sequence of
homotopy groups of a fibre, we see that
$F(Cf,X)$ is acyclic if and only if $F(f , \id_X)$ is
a weak equivalence. We have an isomorphism
$\pi_*^H(F(Cf,X)) \cong [G/H_+ \smashprod Cf, X]^G_*$
and since $Cf$ (and hence $G/H_+ \smashprod Cf$)
is $E$-acyclic and $X$ is $E$-local
we see that this is zero.
The weak equivalence $F(Z,X) \to F(Y,X)$
gives an isomorphism $[Z, X]^G \cong [Y, X]^G$
as desired.
\end{proof}

Now we turn our attention to proving some useful results
on the $E$-local model structures. We write $\fibrep_E$ for
fibrant replacement in the $E$-local model categories and
maps in the $E$-local homotopy category between spectra
$X$ and $Y$ will be written $[X,Y]_E$.

\begin{proposition}\label{prop:pushaxiom}
For two cofibrations, $f \co U \to V$ and $g \co W \to X$, the induced map
$$f \square g \co V \smashprod  W \bigvee_{U \smashprod W} U \smashprod X  \to V \smashprod X$$
is a cofibration which is an $E$-acyclic cofibration if either $f$ or $g$ is.
If $X$ is a cofibrant spectrum then the map $\cofrep \sphspec \smashprod X \to X$
is a weak equivalence.
\end{proposition}
\begin{proof}
Since the cofibrations are unchanged by localisation,
we only need to check that the above map is an
$E$-equivalence when $f$ is.
The map $f \smashprod \id_E \co U \smashprod E
\to V \smashprod E$ is a $\pi_*$-isomorphism
and a cofibration. Thus since $(-) \smashprod E$ commutes with
pushouts the map
$$(V \smashprod  W \bigvee_{U \smashprod W} U \smashprod X)
\smashprod E  \to (V \smashprod X) \smashprod E$$
is also a $\pi_*$-isomorphism and a cofibration.
The unit condition is unaffected by localisation, so it holds
in the $E$-local model structure.
\end{proof}

Thus $\GIS_E$ is a monoidal model category.
The above proof is our standard
method for moving results to the $E$-local model
structures for cofibrant $E$, we use it to prove the following.
\begin{proposition}\label{prop:Ehomproperty}
\begin{enumerate}
\item A map is an $E$-equivalence
if and only if its suspension is an $E$-equivalence.

\item A wedge of $E$-equivalences is an $E$-equivalence.

\item If $i \co A \to X$ is an $h$-cofibration and an $E$-equivalence
and $f \co A \to Y $ is any map, then the cobase change
$j \co Y \to X \vee_A Y$ is an $E$-equivalence.

\item If $i$ and $i'$ are $h$-cofibrations and the vertical arrows are $E$-equivalences
in the diagram below, then the induced map of pushouts is an $E$-equivalence.
$$ \xymatrix@!C{
X \ar[d]^{\sim E}
& A \ar[r] \ar@{+->}[l]_{i} \ar[d]^{\sim E}
& Y \ar[d]^{\sim E} \\
X'
& A' \ar[r] \ar@{+->}[l]^{i'}
& Y' } $$

\item If $X$ is the colimit of a sequence of $h$-cofibrations $X_n \to X_{n+1}$, each of
which is an $E$-equivalence, then the map from the initial term $X_0$ into $X$
is an $E$-equivalence.

\end{enumerate}
\end{proposition}
\begin{proof}
This follows from \cite[Chapter III, Theorem 3.5]{mm02} for orthogonal spectra
and from \cite[Chapter IV, Remark 2.8]{mm02} for $G \mcal$.
\end{proof}

\begin{proposition}\label{prop:Esmashisom}
For a cofibrant orthogonal $G$-spectrum $X$, the functor $X \smashprod (-)$
preserves $E$-equivalences.
\end{proposition}
\begin{proof}
The functor $X \smashprod (-)$ preserves $\pi_*$-isomorphisms
hence the result follows from the associativity of the smash product.
\end{proof}

\begin{proposition}\label{prop:Emonoid}
For $i \co A \to X$ an acyclic $E$-cofibration and any spectrum $Y$,
the map $i \smashprod \id_Y \co A \smashprod Y \to X \smashprod Y$
is an $E$-equivalence and an $h$-cofibration.
Moreover, cobase changes and sequential colimits of such maps are
$E$-equivalences and $h$-cofibrations.
\end{proposition}
\begin{proof}
Use \cite[Chapter III, Lemma 7.1]{mm02}
to see that $i \smashprod \id_Y$
is an $h$-cofibration.
By \cite[Chapter III, Proposition 7.4]{mm02}
$i \smashprod \id_{E} \smashprod \id_{Y}$ is an
$h$-cofibration and a $\pi_*$-isomorphism,
hence $i \smashprod \id_{Y}$ is an $E$-equivalence.

We have proved the first statement of the proposition, the second follows from
Proposition \ref{prop:Ehomproperty} and the fact that $h$-cofibrations
are closed under pushouts and sequential colimits.
\end{proof}

For $G \mcal$ we follow the same proof using
Lemma \ref{lem:sModMonoid}.

\begin{proposition}\index{Left proper}
The model category $\GIS_E$ is
left proper.
\end{proposition}
\begin{proof}
We show a stronger result: that in the pushout diagram below,
with $\alpha$ an $h$-cofibration
the map labelled $l$ is an $E$-equivalence.
Since a cofibration is an $h$-cofibration this implies
left properness.
$$ \xymatrix@!C{
A \ar@{}[dr]|\ulcorner|(0.51){\phantom{\cdot} \cdot}
\ar@{{+}{-}{>}}[r]^{\alpha} \ar[d]_{\sim_E} & B \ar[d]^{l} \\
C  \ar[r]_\gamma & D }$$
The functor $(-) \smashprod E$
preserves $h$-cofibrations,
pushouts and takes $E$-equivalences ($\sim_E$) to
$\pi_*$-isomorphisms.
We apply this functor to the pushout diagram above
and use left properness of $\GIS$,
\cite[Chapter III, Lemma 4.13]{mm02},
(or \cite[Chapter IV, Theorem 2.9]{mm02} for $G \mcal$)
to see that $l \smashprod \id_E$ is a $\pi_*$-isomorphism.
Thus $l$ is an $E$-equivalence as desired.
\end{proof}

See Lemmas \ref{lem:rightproperrational} and \ref{lem:rightproperfamily}
for right properness in the cases of most interest to us.

\begin{rmk}
The spectra $G/H_+$ where $H$ runs over all subgroups $H$ of $G$
are generators for $\GIS_E$. This well known fact follows
from Lemma \ref{lem:spherehomotopy}.
\end{rmk}

\begin{theorem}\label{thm:locfuncs}
Take a Quillen adjunction between monoidal model categories
with a strong monoidal left adjoint
$F : \cscr \overrightarrow{\longleftarrow} \dscr : G.$
Let $E$ be cofibrant in $\cscr$ and assume that
the model categories $\cscr_E$ and $\dscr_{FE}$ exist.
Then $(F,G)$ passes to a strong monoidal Quillen pair
$F : \cscr_E \overrightarrow{\longleftarrow} \dscr_{FE} : G.$
Furthermore, if $(F,G)$ form a Quillen equivalence,
then they pass to a Quillen equivalence of the localised categories.
\end{theorem}
\begin{proof}
Since the cofibrations in $\cscr_E$ and $\dscr_{FE}$ are unchanged
$F$ preserves cofibrations. Now take an acyclic cofibration in $\cscr$ of the form
$f \smashprod \id_E \co X \smashprod E \to Y \smashprod E$, applying $F$ and
using the strong monoidal condition we have a weak equivalence in $\dscr$,
$Ff \smashprod \id_{FE} \co FX \smashprod FE \to FY \smashprod FE$.
Hence $F$ takes $E$-acyclic cofibrations to $FE$-acyclic cofibrations and we have a Quillen pair.

To prove the second statement we
show that $F$ reflects $E$-equivalences
between cofibrant objects and that $F \cofrep GX \to X$ is an $E$-equivalence
for all $X$ fibrant in $\dscr_{FE}$. These conditions are an equivalent
definition of Quillen equivalence by \cite[Corollary 1.3.16(b)]{hov99}.
The first follows since strong monoidality
allows us to identify $F(f \smashprod \id_E)$
and $Ff \smashprod \id_{FE}$ for a map $f$ in $\cscr$
and $F$ reflects weak equivalences between cofibrant objects.
The second condition is equally simple:
we know that an $E$-fibrant object is fibrant,
and that cofibrant replacement is unaffected
by Bousfield localisation. Hence
$F \cofrep GX \to X$ is a weak equivalence
and thus an $E$-equivalence.
\end{proof}

\begin{corollary}\label{cor:Nadjunctlocal}
For $E$, a positive cofibrant orthogonal $G$-spectrum,
the Quillen equivalence
$\mathbb{N} : \GIS_+ \overrightarrow{\longleftarrow} G \mcal : \mathbb{N}^{\#}$
passes to a Quillen equivalence
$$\mathbb{N} : L_{E} \GIS_+ \overrightarrow{\longleftarrow} L_{\nn E}G \mcal : \mathbb{N}^{\#}.$$
\end{corollary}

\begin{rmk}\label{rmk:changemodellocal}
If $E$ is positive cofibrant in $\GIS_+$
then the identity functor is the left adjoint of a Quillen equivalence
from $(\GIS_+)_E$ to $\GIS_E$.
Of course, if $E'$ is a cofibrant spectrum in $\GIS$
and $\cofrep_+ E'$ is its positive cofibrant replacement
then $\GIS_{E'} = \GIS_{\cofrep_+ E'}$ hence
$\GIS_{E'}$ is Quillen equivalent by $(\GIS_+)_{\cofrep_+ E'}$.
\end{rmk}

Theorem \ref{thm:locfuncs} implies
that the forgetful functor $\iota^*_H$
and the inflation functor $\varepsilon^*_G$
(and their right adjoints $F_H(G_+,-)$ and $(-)^G$)
pass to Quillen functors on the $E$-local categories.

\begin{lemma}
The pair $(G_+ \smashprod_H (-), \iota^*_H )$
pass to a Quillen pair between the categories
$L_{\iota_H^*E} H \iscr \sscr $ and $L_E \GIS$.
\end{lemma}
\begin{proof}
This does not follow from the above theorem.
As always $G_+ \smashprod_H (-)$ preserves cofibrations
in the $E$-local model structure.
The isomorphism (\cite[Chapter V, Proposition 2.3]{mm02})
$(G_+ \smashprod_H X) \smashprod Y \cong G_+ \smashprod_H (X \smashprod \iota_H^*Y)$,
which is natural in $H$-spectra $X$ and $G$-spectra $Y$,
applied to the case $Y=E$ proves that the left adjoint preserves the acyclic
cofibrations of the localised category.
\end{proof}

It is clear that one could attempt to repeat the process of localisation and
localise $\GIS_E$ at $F$ (either a cofibrant spectrum or a cofibrant $G$-space).
The cofibrations would be unchanged, the weak equivalences
would be those maps $f \co X \to Y$ such that
$f \smashprod \id_F \co X \smashprod F \to Y \smashprod F$ is an $E$-equivalence
and the fibrations would be given by a lifting property.
Initially it appears that one would have to check that this does give a model category,
essentially reproving Theorem \ref{thm:GSlocal} for $L_E\GIS$, but this is not the case.
One simply has to notice that these new weak equivalences are $(F \smashprod E)$-equivalences
and that this proposed new model structure on $L_E\GIS$
coincides with the $(F \smashprod E)$-model structure on $\GIS$.
We have proved the theorem below.

\begin{theorem}\label{thm:LEGSlocal}
Let $F$ be a cofibrant object of $\GIS_E$ (or, equally, of $\GIS$)
or a cofibrant based $G$-space. Then $L_E \GIS$ has a
Bousfield $F$-model structure
with weak equivalences the $(F \smashprod E)$-equivalences and cofibrations the
cofibrations of $\GIS$. The fibrant objects are precisely the fibrant $(F \smashprod E)$-local
objects of $\GIS$ and $(F \smashprod E)$-fibrant approximation
constructs a Bousfield localisation $f_X \co X \to \fibrep_{F \smashprod E} X$ of $X$
at $F \smashprod E$.
The notation for the $F$-model structure on the underlying category of $\GIS$ is
$L_F L_E \GIS$.
Furthermore we have the following identifications of model categories
$$L_F L_E \GIS = L_{F \smashprod E} \GIS = L_{E \smashprod F} \GIS=L_E L_F \GIS.$$
\end{theorem}

Combining the above theorem with the various results of this section
we have the following summary.
\begin{corollary}\label{cor:localsummary}
The category $L_F L_E \GIS$ is a cofibrantly generated, left proper, monoidal
model category satisfying the monoid axiom,
the spectra $\sphspec \smashprod G/H_+$ for $H$ a subgroup of $G$ form
a countable set of generators.
\end{corollary}

\begin{proposition}\label{prop:rmodlocal}
Let $R$ be a commutative ring spectrum and $E$ be a cofibrant $G$-spectrum
or cofibrant $G$-space, then there is a model structure $L_E (R \leftmod)$
on $R$-modules with weak equivalences and fibrations the $E$-equivalences
and $E$-fibrations of underlying spectra and cofibrations as for $R \leftmod$.
\end{proposition}
\begin{proof}
The $E$-local model structure on $G$-spectra gives
rise to the model category as defined in the theorem
since it satisfies the monoid axiom.
\end{proof}

Note that $L_E (R \leftmod)$ is precisely the model structure
of $R \leftmod$ localised at $E \smashprod R$.
It is easily seen that the weak equivalences and
cofibrations are the same.

\section{The categories $\GIS_\qq$ and $G \mcal_\qq$}\label{sec:GISQ}
Take $E=S^0{\qq}$, we will call the
$E$-local model structure the rational
model structure and write
$L_{S^0{\qq}} \GIS$ or $\GIS_\qq$\index{G ISQ@${\GIS_\qq}$}.
For $G \mcal$ we take $S^0_\mcal{\qq}$ and
write the localised category as
$G \mcal_\qq$\index{G MQ@${G\mcal_\qq}$}.
We will write $\GIS_\qq^+$ for the rationalisation of the
positive model structure $\GIS_+$ (using $S^0 {\qq}_+$).
We call $E$-equivalences \textbf{rational equivalences}\index{Rational equivalence},
or \textbf{$\pi_*^\qq$-isomorphisms}\index{pi-isomorphism@$\pi_*^\qq$-isomorphism}.
Equally, $E$-fibrations will be called
\textbf{rational fibrations}\index{Rational fibration}
and acyclic $E$-cofibrations will be called \textbf{acyclic
rational cofibrations}\index{Acyclic rational cofibration}
The set of rational homotopy classes of maps $X$ to $Y$ will be written $[X,Y]^G_\qq$
and we will write $\fibrep_\qq$ for fibrant replacement in the localised category.
Note that since the cofibrations agree, the rational acyclic fibrations
are the acyclic fibrations of $\GIS$.
Hence factorising a map into a cofibration followed by a
rational acyclic fibration is the same operation in
both $\GIS$ and $\GIS_\qq$.
We will prove that our rationalised categories
$\GIS_\qq$, $\GIS_\qq^+$
and $G \mcal_\qq$ are Quillen equivalent
(Propositions \ref{prop:GIS+QequivGISQ}
and \ref{prop:GMequivGIS+Q}),
so that we can switch between these at will.
The lemma below shows that our model structure
is independent of our choice of rational
sphere spectrum.
\begin{lemma}\label{lem:ratequivs}
For a map $g \co X \to Y$ the following are equivalent
(where $\fibrep$ denotes fibrant replacement in $\GIS$,
which is unnecessary for $G \mcal$):
\begin{enumerate}
\item \vskip -0.0cm $(\fibrep g)_*^H \co \h_*((\fibrep X)^H; \qq)
\to \h_*((\fibrep Y)^H; \qq)$ is an isomorphism for all $H$.
\item $(\fibrep g)_*^H \co \pi_*((\fibrep X)^H) \otimes \qq \to
\pi_*((\fibrep Y)^H) \otimes \qq$ is an isomorphism for all $H$.
\item $g \smashprod \id \co X \smashprod S^0{\qq} \to Y \smashprod S^0{\qq}$ is a $\pi_*$-isomorphism.
\item $g \co X \to Y $ is a $\pi^\qq_*$-isomorphism.
\end{enumerate}
\end{lemma}
\begin{proof}
We have shown in Proposition \ref{prop:rathomgps} that the last
two conditions are equivalent.
Statements $(ii)$ and $(iv)$ are equivalent by
\cite[Chapter V, Proposition 3.2]{mm02}
(this is obvious for $G \mcal$).
The first two are well known to be equivalent,
we are simply noting that our construction of $S^0{\qq}$
for $G$ the trivial group gives $\h \qq$, an
Eilenberg-Mac Lane spectrum.
\end{proof}

\begin{lemma}\label{lem:ratLES}
For any map $f \co X \to Y$ of $G$-prespectra and any $H \subset G$, there are natural
long exact sequences
$$\xymatrix@C-0.42cm@R-0.6cm{
\dots \ar[r] &
\pi_q^H(Ff)     \otimes \qq \ar[r] &
\pi_q^H(X)      \otimes \qq \ar[r] &
\pi_q^H(Y)      \otimes \qq \ar[r] &
\pi_{q-1}^H(Ff) \otimes \qq \ar[r] &
\dots, \\
\dots \ar[r] &
\pi_q^H(X)      \otimes \qq \ar[r] &
\pi_q^H(Y)      \otimes \qq \ar[r] &
\pi_q^H(Cf)     \otimes \qq \ar[r] &
\pi_{q-1}^H(X)  \otimes \qq \ar[r] &
\dots }$$
and the natural map $\nu \co Ff \to \Omega Cf$ is a $\pi_*$-isomorphism.
\end{lemma}
\begin{proof}
This follows from the fact that $\qq$ is flat, (tensoring with $\qq$
preserves exact sequences) and
\cite[Chapter III, Theorem 3.5]{mm02} for orthogonal spectra
(\cite[Chapter IV, Remark 2.8]{mm02} for $G \mcal$).
\end{proof}

\begin{lemma}\label{lem:rightproperrational}
The categories $\GIS_\qq$ and $G \mcal_\qq$
are right proper.
\end{lemma}
\begin{proof}
We follow the proof of \cite[9.10]{mmss01}
and show a stronger statement:
in a pullback diagram as below, if $\beta$
is a level fibration then $r$ is a $\pi_*^\qq$-isomorphism.
$$\xymatrix@!C{
W \ar[r]^{\delta} \ar[d]_{r} & X \ar[d]^{\sim_\qq} \\
Y  \ar[r]_\beta & Z
\ar@{}[ul]|\lrcorner|(0.52){\cdot \hskip 2.5pt } }$$
Let $\beta^{-1}*$ be the pullback of $\beta$ over a point (the dual construction to
$Z/Y$; it is the pre-image of the basepoint of $Z$). This pullback is constructed
levelwise and for each level $V$,
$\delta(V)^{-1}* \cong \beta(V)^{-1}*$ (by simply writing down the
definitions of these spaces), thus the map of
spectra $\delta^{-1}* \to \beta^{-1}*$ is an isomorphism.

Now we use \cite[4.65]{hatAT} to see that the fibre of a fibration of spaces
is homotopy equivalent to the pre-image of the basepoint of the codomain.
The map of spectra $\beta$ is a level $G$-fibration, that is
for each level $V$ and subgroup $H$ of $G$, $ \beta(V)^H$ is a fibration.
The fixed point functor $(-)^H$ is a right adjoint and
$F(I, B^H) \cong F(I,B)^H$ for any spectrum $B$
(\cite[Chapter III, Lemma 1.6]{mm02}).
So it follows that the fibre of
$\beta(V)$ is $G$-homotopy equivalent to
$\beta(V)^{-1}* =(\beta^{-1}*)(V)$
and thus we have a level $G$-equivalence
between $F \beta$, the fibre of $\beta$, and $\beta^{-1}*$.
Similarly we have a level $G$-equivalence
between $F \delta$, the fibre of $\delta$, and $\delta^{-1}*$.
Hence, we have a level $G$-equivalence (and thus a $\pi_*$-isomorphism)
between $F \delta$ and $F \beta$.

We now apply the long exact sequence of rational homotopy groups
of a fibration and the five-lemma
to conclude that $r$ is an $E$-equivalence.
\end{proof}

\begin{theorem}\label{thm:rathomotopymaps}
For any $X$ and $Y$, $[X,Y]^G_\qq$ is a rational vector space.
If $Z$ is an $S^0{\qq}$-local object of $\GIS$ then
$Z$ has rational homotopy groups.
There is a natural isomorphism
$[X,Y]^G_\qq \cong [X \smashprod S^0{\qq},Y \smashprod S^0{\qq}]^G.$
\end{theorem}
\begin{proof}
The argument is the same for both model categories,
so we use notation appropriate to $\GIS$.
For each integer $n$ we have a self-map of $\fibrep \sphspec$
which represents multiplication by $n$ at the model category level,
applying $(-) \smashprod X$ we obtain a self-map of
$\fibrep \sphspec \smashprod X$. Since this map is an isomorphism
of rational homotopy groups it is a weak equivalence of $\GIS_\qq$
and so in the homotopy category of $\GIS_\qq$ we have an isomorphism
$n \co X \to X$ hence $[X,Y]^G_\qq$ is a rational vector space.

Let $\fibrep_\qq Z$ be the fibrant replacement
of $Z$ in $\GIS_\qq$. The map $Z \to \fibrep_\qq Z$
is a rational equivalence between
$S^0{\qq}$-local objects and hence is a
$\pi_*$-isomorphism.
We can describe the homotopy groups of
$Z$ in terms of $[\Sigma^p G/H_+, Z]^G$
and $[F_q G/H_+, Z]^G$ for $p \geqslant 0$
and $q > 0 $ by Lemma \ref{lem:spherehomotopy}.
The result then follows by the isomorphisms
$$[A,Z]^G \cong [A, \fibrep_\qq  Z]^G \cong [A, Z]^G_\qq$$
which hold for any cofibrant $G$-spectrum $A$ by
Proposition \ref{prop:localquillen}.

Now we turn to the final part of this theorem.
The map $Y \smashprod S^0{\qq} \to \fibrep_\qq (Y \smashprod S^0{\qq})$
is a $\pi_*^\qq$-isomorphism between objects with
rational homotopy groups, hence it is a
$\pi_*$-isomorphism. For any $G$-spectrum $X$,
$X \smashprod S^0{\qq}$ is rationally equivalent to $X$.
Combining these with Proposition \ref{prop:localquillen}
we obtain isomorphisms as below.
$$ \begin{array}{rcll}
[X,Y]^G_\qq
& \cong & [X \smashprod S^0{\qq},Y \smashprod S^0{\qq}]^G_\qq \\
& \cong & [X \smashprod S^0{\qq},\fibrep_\qq (Y \smashprod S^0{\qq})]^G \\
& \cong & [X \smashprod S^0{\qq},Y \smashprod S^0{\qq}]^G
\end{array} $$
\end{proof}

\begin{corollary}
If the spectrum $X$ is $H$-compact in $\GIS$  then
it is $H$-compact in $\GIS_\qq$. In particular, the generators of
$\GIS_\qq$ are $G$-compact, or equivalently,
$\sphspec$ is $H$-compact in $\GIS_\qq$ for all $H$.
\end{corollary}
\begin{proof}
Take a collection $\{ Y_i \}_{i \in I}$
in $\GIS$ and let $\fibrep_\qq$ be fibrant replacement
in $\GIS_\qq$. Then $[G/H_+, \bigvee_i Y_i]_\qq^G$
is isomorphic to
$$[G/H_+ \smashprod S^0{\qq}, \bigvee_i Y_i \smashprod S^0{\qq}]^G
\cong [G/H_+ \smashprod S^0{\qq},
\fibrep_\qq \big( \bigvee_i Y_i \smashprod S^0{\qq} \big)]^G.$$
Since $G/H_+ \to G/H_+\smashprod S^0{\qq}$ is a rational
equivalence the above is isomorphic to
$$[G/H_+ , \fibrep_\qq \big( \bigvee_i Y_i \smashprod S^0{\qq} \big)]^G
\cong
[G/H_+ , (\bigvee_i Y_i) \smashprod S^0{\qq}]^G .$$
Thus the result follows from the compactness of $G/H+$
in $\GIS$.
\end{proof}

\begin{proposition}\label{prop:GIS+QequivGISQ}
There is a Quillen equivalence
$$\id : \GIS_\qq^+ \overrightarrow{\longleftarrow} \GIS_\qq : \id.$$
\end{proposition}
\begin{proof}
By Lemma \ref{lem:equivratspheres} we see that
$\GIS_\qq$ can be constructed by localising at $S^0{\qq}_+$,
hence Theorem \ref{thm:locfuncs} gives us the result, see
Remark \ref{rmk:changemodellocal}.
\end{proof}

\begin{proposition}\label{prop:GMequivGIS+Q}
The model structures $G\mcal_\qq$
and $L_{\mathbb{N} S^0\qq_+} G \mcal$ are equal.
The adjoint pair $(\mathbb{N}, \mathbb{N}^{\#})$
are a strong monoidal Quillen equivalence
$$\mathbb{N} : \GIS_\qq^+ \overrightarrow{\longleftarrow}
L_{\mathbb{N} S^0\qq_+} G\mcal : \mathbb{N}^{\#}$$
\end{proposition}
\begin{proof}
We defined $G \mcal_\qq$ as the localisation
$L_{S^0_\mcal{\qq}} G\mcal$.
By Lemma \ref{lem:equivratspheres}
the cofibrant objects $\mathbb{N} S^0\qq_+$
and $S^0_\mcal {\qq}$ are weakly equivalent.
It follows therefore that a map is an
$\mathbb{N} S^0\qq_+$-equivalence
if and only if it is an
$S^0_\mcal {\qq}$-equivalence.
Thus $L_{S^0_\mcal{\qq}} G\mcal$
has the same weak equivalences and cofibrations
as $L_{\mathbb{N} S^0\qq_+} G \mcal$.
The second statement follows from Theorem \ref{thm:locfuncs},
since $\GIS_\qq^+$ is the localisation of $\GIS_+$
with respect to the cofibrant object $S^0\qq_+$.
\end{proof}

\section{$S_\qq$-Modules}\label{sec:SQmod}
We give an alternative, but Quillen equivalent method of constructing a
category of rational $G$-spectra. We do not consider orthogonal
spectra in this section. The advantage to doing so is that
every object in this new category will be fibrant. This
is one of the technical requirements necessary to apply the results
of \cite{greshi}, see Remark \ref{rmk:whyfibrant}.

\begin{definition}
We have a model category of modules
over $S_\qq$ (see Definition \ref{def:S_qq}) from
\cite[Chapter IV, Theorem 2.11]{mm02}
this will be written $S_\qq \leftmod$. This is a proper
closed symmetric monoidal model category.
\end{definition}

\begin{lemma}\label{lem:SQiscRQ}
A map is an $S^0_\mcal{\qq}$-equivalence if and only if it is a
$\cofrep S_\qq$-equivalence.
\end{lemma}
\begin{proof}
The result follows from the
zig-zag of weak equivalences
(which exists for any spectrum $X$):
$X \smashprod S^0_\mcal{\qq}
\leftarrow \cofrep X \smashprod S^0_\mcal{\qq}
\rightarrow \cofrep X \smashprod S_\qq
\leftarrow \cofrep X \smashprod \cofrep S_\qq
\rightarrow X \smashprod \cofrep S_\qq$.
\end{proof}

\begin{lemma}\label{lem:RQmodareSQlocal}
All $S_\qq$-modules are $S^0_\mcal{\qq}$-local
and thus all $S_\qq$-modules have rational homotopy groups.
\end{lemma}
\begin{proof}
Since fibrations of $S_\qq$-modules are defined
in terms of their underlying $S$-modules
and all $S$-modules are fibrant, we must show that
$[A,M]^G=0$ for any $S^0_\mcal{\qq}$-acyclic $A$ and any $S_\qq$-module $M$.
We adapt the following argument from \cite[13.1]{adams}.
Take $[f] \in [A,M]$ then $[f]$ is zero if and only if
$f^* \co [M,M] \to [A,M]$ is the zero map.
$$\xymatrix@!C{
[M,M]^G \ar[r]^{f^*} \ar[d]_{\nu^*} & [A,M]^G \\
[M \smashprod^L S_\qq,M]^G \ar[r]_{(f \smashprod \id)^*}
& [A \smashprod^L S_\qq, M]^G \ar[u]^{(\id \smashprod \eta )^* }  }$$
Since $A \smashprod S^0_\mcal{\qq}$ is acyclic, it follows
by Lemma \ref{lem:SQiscRQ} that
$A \smashprod^L S_\qq = \cofrep A \smashprod \cofrep S_\qq$
is acyclic, hence $[A \smashprod^L S_\qq, M]^G=0$.
From this it is clear that $f^*=0$ and $M$ is
an $S^0{\qq}$-local object.
\end{proof}

\begin{theorem}\label{thm:localisedtomodules}
There is a strong symmetric monoidal Quillen equivalence:
$$S_\qq \smashprod (-) :  G \mcal_\qq \overrightarrow{\longleftarrow} S_\qq \leftmod : U.$$
\end{theorem}
\begin{proof}
The above functors form a strong monoidal Quillen pair
(with the usual structure on $G \mcal $),
this is standard and comes from the construction of a
monoidal model structure on $S_\qq \leftmod$.
Now consider the localised case, since the cofibrations are unaffected
by localisation, $S_\qq \smashprod (-) \co G \mcal_\qq \to S_\qq \leftmod$ preserves
cofibrations. Consider an acyclic rational cofibration $X \to Y$,
we know that $S_\qq \smashprod (-)$ applied
to this gives a cofibration, we must check that
it is also a $\pi_*$-isomorphism. We see that
$X \smashprod S^0_\mcal{\qq} \to Y \smashprod S^0_\mcal{\qq}$ is a
cofibration and a $\pi_*$-isomorphism,
so in turn
$X \smashprod S^0_\mcal{\qq} \smashprod S_\qq \to
Y \smashprod S^0_\mcal{\qq} \smashprod S_\qq$ is an $h$-cofibration and
a $\pi_*$-isomorphism.
This proves that
$X \smashprod S_\qq \to Y \smashprod S_\qq$
is a $\pi^\qq_*$-isomorphism between
$S_\qq$-modules, which we know have rational homotopy groups and thus
this map is a $\pi_*$-isomorphism.

We prove that this is a Quillen equivalence using the
characterisation of Quillen equivalences from
Theorem \ref{thm:generalsplitting}.
The right adjoint preserves and detects all weak equivalences, since
a $\pi_*^\qq$-isomorphism between objects with rational homotopy groups is a
$\pi_*$-isomorphism. So to finish this proof we must show that the map
$X \to S_\qq \smashprod X$ is a rational equivalence for all cofibrant $S$-modules $X$.
This follows since smashing with a cofibrant
object will preserve the $\pi_*^\qq$-isomorphism $S \to S_\qq$.
\end{proof}

\begin{rmk}
This model category is proper by \cite[Chapter IV 2.11]{mm02}.
\end{rmk}

\chapter{Splitting Rational $G$-Spectra}\label{chp:splitting}
Idempotents of the Burnside ring
split the homotopy category of $G$-spectra,
we provide a model category level version of this
in the second section (Theorem \ref{thm:generalsplitting}).
To prove this theorem we must use some classical
results on $G$-spectra which we
give in Section \ref{sec:eqstabhom}. We show that
this splitting applies to all our categories
of rational $G$-spectra in Section \ref{sec:compsplit}.
By considering families of subgroups in
the final section of this chapter we obtain
a particularly nice form of the splitting
(Theorem \ref{thm:familysplitting}).
That is, we use \cite[Chapter IV, Section 6]{mm02}
to understand the split pieces of the category
of $G$-spectra in this case.
The arguments in this chapter will apply equally
well to $\GIS_\qq$, $\GIS_\qq^+$ and $G\mcal_\qq$. For definiteness
we work with $\GIS$ in Section \ref{sec:eqstabhom},
Section \ref{sec:split} and Section \ref{sec:idemsplit}.
We will only consider other categories of equivariant spectra in
Section \ref{sec:compsplit}.

\section{Equivariant Stable Homotopy Theory}\label{sec:eqstabhom}
We prove a medley of basic results about
spectra and equivariant spectra. These are
the tools we will use to prove the splitting theorem
in the following section. It is not easy to find the
proofs of these results, which is why they are included here.
The most important of these results are
Proposition \ref{prop:HomotopyofHoColim}
and Proposition \ref{prop:geomfixedpoints}.

\begin{lemma}\label{lem:twoforthree}
In $\GIS$, let $f \co E \to F$ be a map between cofibrant objects, then
a map $g \co X \to Y$ which is an $E$-equivalence and a $Cf$-equivalence
is also an $F$-equivalence.
\end{lemma}
\begin{proof}
Consider the map of cofibre sequences
\begin{displaymath}
\xymatrix{
E \smashprod X \ar[r] \ar[d]^{\id_E \smashprod g}&
F \smashprod X \ar[r] \ar[d]^{\id_F \smashprod g} &
Cf \smashprod X \ar[d]^{\id_{Cf} \smashprod g} \\
E \smashprod Y \ar[r] &
F \smashprod Y \ar[r] &
Cf \smashprod Y }
\end{displaymath}
this gives a map of long exact sequences of homotopy groups
\cite[Chapter III, Theorem 3.5]{mm02} and we apply the five lemma.
\end{proof}

\begin{rmk}\label{rmk:twoforthree}
In fact, the above proof shows that for
a map $f \co E \to F$ of cofibrant objects,
one only has to check that
two out of three of the maps
$\id_E \smashprod g$,
$\id_F \smashprod g$ and
$\id_{Cf} \smashprod g$
are weak equivalences to conclude that all three are weak equivalences.
Also note that if $E$ and $F$ are cofibrant, then so is $Cf$.
A wedge $E \vee F$ of cofibrant objects comes
with a cofibre sequence $E \to E \vee  F \to (E \smashprod I) \vee F$.
Since $(E \smashprod I) \vee F$ is weakly equivalent to
$F$, a map is an $E \vee F $-equivalence if and
only if it is an $E$-equivalence
and an $F$-equivalence.
\end{rmk}

Since $A(G) := [\sphspec, \sphspec]^G$
we know that $A(G) \otimes \qq = [\sphspec, \sphspec]^G_\qq$.
So $[a]$, an element of the
Burnside ring, can be represented by a self map $a$ of
$\fibrep_\qq \sphspec$ and hence for any spectrum
$X$, we have $\id_X \smashprod a$, a self-map of $X \smashprod \fibrep_\qq \sphspec$.
\begin{definition}
Let $e$ be a self-map of $\fibrep_\qq \sphspec$ such that
$[e] \in A(G) \otimes \qq$ is an idempotent (thus
$e \circ e \simeq e$). We define $eX$ to be the homotopy colimit
of
$$X \smashprod \fibrep_\qq \sphspec
\overset{\id \smashprod e}{\longrightarrow}
X \smashprod \fibrep_\qq \sphspec
\overset{\id \smashprod e}{\longrightarrow}
X \smashprod \fibrep_\qq \sphspec
\overset{\id \smashprod e}{\longrightarrow}
X \smashprod \fibrep_\qq \sphspec
\overset{\id \smashprod e}{\longrightarrow} \dots;$$
fibrant replacement $\sphspec \to \fibrep_\qq \sphspec$ provides
a map $X \to eX$.
\end{definition}

\begin{rmk}\label{rmk:restrictedidempotents}
An idempotent $e \in A(G) \otimes \qq$ has a support $S \subseteq S_f G$.
Given a subgroup $H$ of $G$ (with inclusion $\iota_H$), there is an idempotent
$\iota_H^*(e) \in A(H) \otimes \qq$.
This idempotent is supported on the set
$\iota_H^* (S) : = \{K \leqslant H | K \in S \}$.
So considering $i_H^*(e)$ as a continuous map
$\fcal H / H \to \qq$, $\iota_H^*(e) (K)_H \neq 0$
exactly when $K \in S$.
That is, $\iota_H^*(e)$ is non-zero on the $H$-conjugacy class
of a subgroup $K$ of $H$ if and only if $e$ is non-zero
on the $G$-conjugacy class of $K$.
\end{rmk}

The following well-known result can be described as proving that
the homogenous spaces $G/H_+$ are small
in the homotopy category of spectra
with respect to the $h$-cofibrations.
We give some comments on this after the statement and proof of the result.
This type of result is proven in greater generality in
\cite[Section 4]{hov08}.
\begin{proposition}\label{prop:HomotopyofHoColim}
For a sequential colimit diagram
$f_i \co X_i \to X_{i+1}$
there is an isomorphism of groups
$\pi_*^H (\hocolim_i X_i) \cong \colim_i \pi_*^H(X_i)$.
\end{proposition}
\begin{proof}
The homotopy colimit is formed
by replacing the diagram $f_i \co X_i \to X_{i+1}$
by a sequence $g_i \co Y_i \to Y_{i+1}$
of $h$-cofibrations,
we can now consider this as a sequence of
$h$-cofibrations of prespectra.
In order to take $H$-fixed points we must
first apply $\iota^*_H$.
Recall that $\iota^*_H$ is both a right and left adjoint
and hence commutes with all limits
and colimits, it preserves $h$-cofibrations and we have the relation
$\iota^*_H F(X,Y) \cong  F(\iota^*_H X,\iota^*_H Y)$.
Thus we can suppress the notation for $\iota^*_H$
and this will cause no difficulty.
Fix some subgroup $H$, then we have the homotopy group
$\colim_V \pi_q^H( \Omega^V (\colim_i Y_i)(V))$.
Since colimits of prespectra are
created levelwise, this is equal to
$\colim_V \pi_q( \Omega^V \colim_i (Y_i(V))^H)$.
Now we are working in $H$-spaces and $S^V$
(or for the negative case $S^{V-\rr^q}$) is a finite $H$-CW
complex. So there are subgroups $K_\alpha$ of $H$,
where $\alpha$ runs over some finite set such that
$ S^V \cong \colim_\alpha G/K_\alpha \smashprod S^{n_\alpha}$.
Thus the standard adjunctions give
$$\begin{array}{rcl}
F( S^V ,\colim_i Y_i(V))^H
& \cong & \lim_\alpha F(S^{n_\alpha},\colim_i Y_i(V))^{K_\alpha} \\
& \cong & \lim_\alpha F(S^{n_\alpha},\colim_i Y_i(V)^{K_\alpha}).
\end{array}$$
Hence we can write $\pi_q( \Omega^V \colim_i (Y_i(V))^H) \cong
\lim_\alpha \pi_{q+n_\alpha}( \colim_i (Y_i(V))^{K_\alpha})$.
Now we use the fact that $h$-cofibrations are levelwise monomorphisms
so that our colimit commutes with taking fixed points.
Furthermore $(-)^{K_\alpha}$ takes $h$-cofibrations of $H$-spaces
to cofibrations of spaces, thus since $S^{q+n_\alpha}$ is
a compact space
$$\pi_{q+n_\alpha}( \colim_i (Y_i(V)^{K_\alpha})) \cong
\colim_i \pi_{q+n_\alpha}(  Y_i(V)^{K_\alpha}). $$
Since sequential colimits commute with finite limits
and colimits we can repack all of the above to obtain
our stated result.
\end{proof}

This result also holds for $G \mcal$ with almost the same proof,
with one extra point of justification. A colimit
of $h$-cofibrations of inclusion prespectra is
an inclusion prespectrum. Spectrification for inclusion
prespectra is given by $LX(V)= \colim_{W \supset V} \Omega^{W-V}X(W)$.
Hence, by the arguments in the above proof, colimits of
$h$-cofibrations of spectra are given by the levelwise colimit.

\begin{rmk}
Our proof above actually holds for filtered colimits, rather than just
sequential colimits.
In the proof of Lemma \ref{lem:compactgenerators}
we described a coproduct in terms of a filtered colimit.
One can make the converse construction and describe a
filtered colimit as a cofibre of coproducts.
Take a filtered colimit diagram $\{ Y_i \}_{i \in I}$
such that the maps $Y_i \to Y_j$ are all $h$-cofibrations.
For $i \in I$ let $I(i)$ be the subset of those $j \in I$
such that there is a map $Y_i \to Y_j$.
We define a map $f \co \bigvee_{i \in I} Y_i \to
\bigvee_{i \in I} Y_i$ by
sending $Y_i \to \bigvee_{j \in I(i)} Y_j$
and including this in $\bigvee_{i \in I} Y_i$.
Since the homotopy category is additive we can take
a representative $g \co X \to Z$
for the homotopy class of maps
$[1-f] \co \cofrep \fibrep \bigvee_{i \in I} Y_i \to
\cofrep \fibrep  \bigvee_{i \in I} Y_i$.
The cofibre of $g$ is weakly
equivalent to $\colim_i Y_i$ by looking at the long exact
sequence of a cofibration and noting that $(1-f)_*$
is injective on homotopy groups.
So in the homotopy category of $G$-spectra,
the notion of small and compact are
the same, though one must
be careful about constructing
filtered colimits.
Replacing colimits by a cofibre of coproducts
is a standard construction, see
\cite[Definition 2.23]{hps97}.
\end{rmk}

\begin{corollary}\label{cor:comrel}
Let $e$ be an idempotent of the rational Burnside ring and
let $X$ be any orthogonal spectrum. Then
for any subgroup $H$ of $G$ we have the
inclusion $\iota_H \co H \to G$ and isomorphisms
$\pi^H_*(e\sphspec \smashprod X) \otimes \qq \cong
\pi^H_*(eX) \otimes \qq \cong \iota^*_H(e) \pi^H_*(X) \otimes \qq$.
\end{corollary}
\begin{proof}
The first isomorphism follows from Proposition \ref{prop:hocolimproperties}
the rest will follow from the previous result.
Note that the sequential colimit of an idempotent
is isomorphic to the image of the idempotent.
Since we have suppressed $\iota^*_H$ in our notation
for homotopy groups, we must account for its action
on the idempotent hence the term $\iota^*_H(e)$.
\end{proof}

\begin{proposition}\label{prop:idemsplit}
The map $X \to X \prod X \to eX \prod (1-e) X$
is a rational equivalence for any $X$ and any idempotent $e$
of the Burnside ring.
\end{proposition}
\begin{proof}
They certainly have isomorphic rational homotopy groups
and examination of the maps involved shows that the result is true.
\end{proof}

\begin{proposition}
Let $X$ be a cofibrant orthogonal $G$-spectrum. Let $[e]$ be an
idempotent of the rational Burnside ring of $G$ with support $S \subseteq S_f G$.
Then for any $L \in S$, $\Phi^L(eX)$ is rationally equivalent to
$\Phi^L(X)$ as non-equivariant spectra,
otherwise $\Phi^L(eX)$ is non-equivariantly rationally acyclic.
\end{proposition}
\begin{proof}
We perform much of the proof in the category of
rational orthogonal $L$-spectra,
that is, we apply the forgetful functor
$\iota_L^*$ to all the spectra involved, ($\iota_L$ is the inclusion of $L$ in $G$).
This functor $\iota_L^*$ preserves fibrations, cofibrations, weak equivalences
and more (\cite[Chapter V, Lemma 2.2]{mm02}) so this
presents us with no difficulty. In general when we write
$\Phi^H (X) $ for a $G$-spectrum $X$ and
$H$ a subgroup of $G$, we mean $\Phi^H (\iota_H^* X) $.
Also note that $\iota_L^*(eX)$ is weakly equivalent to
$\iota_L^*(e) \iota_L^*(X)$ as rational $L$-spectra.

We now apply Proposition \ref{prop:hocolimproperties} to obtain
a weak equivalence (of $L$-spectra) $eX \to e \sphspec \smashprod X$ between cofibrant
objects and hence (by Ken Brown's lemma) we have a weak equivalence
of non-equivariant spectra $\Phi^L(eX) \to \Phi^L(e \sphspec \smashprod X)$.
Using \cite[Chapter V, Section 4]{mm02} we have a further weak equivalence
$$ \Phi^L(e \sphspec) \smashprod \Phi^L(X) \to \Phi^L(e \sphspec \smashprod X) $$
and we know that  $\Phi^L$ commutes with colimits of $h$-cofibrations
so that $\Phi^L(e \sphspec) \simeq \Phi^L(e) \Phi^L \sphspec$.
Since $\sphspec$ is the suspension spectrum of $S^0$ which is
$G$-fixed, $\Phi^L \sphspec \cong \sphspec$, now we only have to understand the map
$\Phi^L(e)$. The construction of tom Dieck's isomorphism
$$A(G) \otimes \qq \cong C(\fcal G / G , \qq)$$
takes a map of spectra $f$ to the map
$H \mapsto deg \Phi^H(f)$, where degree has
the usual algebraic topology definition in terms of homology,
which coincides with homotopy (since we are working rationally and stably).
Thus
$$\pi_*(\Phi^L(e) \sphspec)  \cong \Phi^L(e)_* \pi_*( \sphspec) = deg \Phi^L(e) \pi_*( \sphspec ) $$
which is either zero or $\pi_*( \sphspec )$ according to whether or not $L \in S$.
When $L \notin S$ we are smashing $\Phi^L(X)$ with an acyclic object, hence
(since $\Phi^L$ preserves cofibrations and $X$ is cofibrant)  $\Phi^L(eX)$ is an acyclic
non-equivariant orthogonal spectrum.
For the other case we have a weak equivalence of non-equivariant spectra
$\sphspec \to \Phi^L( e) \sphspec$ (since $\Phi^L(e)$ is a $\pi_*$-isomorphism)
and so we have our result.
\end{proof}

The following theorem is given in \cite[Chapter XVI, Theorem 6.4]{may96}
we provide a `folk-proof' of this `folk-theorem'
in the language of orthogonal $G$-spectra.

\begin{proposition}\label{prop:geomfixedpoints}
Let $f \co X \to Y$ be a map of cofibrant $G$-equivariant orthogonal spectra,
then $\Phi^H f$ is a weak equivalence of non-equivariant orthogonal spectra
for all $H \leqslant G$ if and only if
$f$ is a weak equivalence of $G$-equivariant orthogonal spectra.
\end{proposition}
\begin{proof}
We must use $\iota^*_H \co \GIS \to H \sscr $
in order to apply $\Phi^H$ to a $G$-equivariant spectrum.
As with the previous proposition, we omit
$\iota_H^*$-notation in the following proof.
The `if' part is immediate by Ken Brown's lemma applied to $\Phi^H$
(since it preserves (acyclic) cofibrations).
The converse begins by noting that $\Phi^H$ preserves cofibre
sequences, so it suffices to prove (for cofibrant $Z$) that
if $\Phi^H(Z)$ is acyclic for all $H$ then $Z$ is acyclic as a $G$-spectrum.
We use \cite[Chapter V, Proposition 4.17]{mm02} to
replace $\Phi^H(Z)$ by the weakly equivalent object
$\big( \fibrep (Z \smashprod E \widetilde{\fscr_H}) \big)^H$.
The family $\fscr_H$ is the collection of subgroups of
$H$ which do not contain $H$: it is the family of
proper subgroups of $H$.

We begin an inductive argument at $\{ e \}$,
where $E \widetilde{\fscr_{\{ e\}} } = S^0$.
Since $\Phi^{\{e\}}Z$ is acyclic,
$\pi_*((\fibrep Z)^{\{e\}})=\pi_*^{\{e\}}(\fibrep Z) = 0$.
Hence $Z$ is contractible as an $\{e \}$-spectrum.
Now take a subgroup $H$ and assume inductively
that $Z$ is contractible as a $K$-spectrum for all strict
subgroups $K$ of $H$. We will show that
$Z \smashprod E \widetilde{\fscr_H}$ and
$Z \smashprod E {\fscr_H}_+$ are contractible $H$-spectra
and apply Lemma \ref{lem:twoforthree}
to see that $Z$ will be $H$ contractible.

Consider $\big( \fibrep (Z \smashprod E \widetilde{\fscr_H})\big)^K$,
for a strict subgroup $K$ of $H$. By inductive assumption
$Z$ is $K$-acyclic, hence so is $Z \smashprod E \widetilde{\fscr_H}$,
thus $\pi_*^K( Z \smashprod E \widetilde{\fscr_H})=0$.
We started by assuming that $\Phi^H(Z)$ is acyclic,
hence  $\pi_*^H( Z \smashprod E \widetilde{\fscr_H})=0$.
We have completed the first half of our inductive step.
We will now prove that $[Z \smashprod E {\fscr_H}_+, M]^H =0$
for any $H$-spectrum $M$ and thus $Z \smashprod E {\fscr_H}_+$
will be acyclic as an $H$-spectrum.
The space $E {\fscr_H}_+$ is made from cells of type
$H/K$ for $K$ a strict subgroup of $H$,
$E {\fscr_H}_+ = \colim_\alpha H/K^\alpha_+ \smashprod S^{n_\alpha}$.
It follows by standard manipulations that
$$[Z \smashprod E {\fscr_H}_+, M]^H \cong
\lim_\alpha [Z \smashprod S^{n_\alpha}, M]^{K^\alpha}$$
using the fact that $H/K^\alpha_+ \smashprod Z \cong H_+ \smashprod_{K^\alpha} Z$
as $H$-spectra. This last term is zero since $Z$ is acyclic as a $K^\alpha$-spectrum.
\end{proof}

Note that in the above we use
induction on the poset of closed subgroups
of a compact Lie group. Such
inductive arguments are valid since there are no
infinite descending chains of subgroups
in a compact Lie group.

\begin{corollary}
The above proposition also holds in the rationalised case:
$f \co X \to Y$ is a rational equivalence of $G$-spectra if and only if
$\Phi^H f$ is a rational equivalence of non-equivariant spectra
for all subgroups $H$ of $G$.
\end{corollary}
\begin{proof}
We temporarily let the rational sphere spectrum of
$G$-spectra as $S^0\qq_G$, then we see that
$\Phi^H (S^0{\qq}_G) = S^0{\qq}_{ \{e \} }$
for all $H$. The result then follows by the usual
properties of $\Phi^H$ (see \cite[V section 4]{mm02}).
\end{proof}

\section{The Splitting}\label{sec:split}
We now prove the most important result of this
chapter, indeed, the most important
result of Part \ref{part:Gspec}: Theorem \ref{thm:generalsplitting}.

\begin{definition}
Recall the definition of the
\textbf{product model category}\index{Product model category}
from \cite[Example 1.1.6]{hov99}.
Given model categories $M_1$ and $M_2$ we can put a model category structure on
$M_1 \times M_2$. A map $(f_1,f_2)$ is a cofibration, weak equivalence or fibration
if and only if $f_1$ is so in $M_1$ and $f_2$ is so in $M_2$. Similarly
a finite product of model categories has a model structure where a map
is a cofibration, weak equivalence or fibration
if and only if each of its factors is so.
\end{definition}

\begin{rmk}
If $M_1$ and $M_2$ both satisfy any of the following:
left properness, right properness, the pushout product axiom, the monoid axiom or
cofibrant generation, then so does $M_1 \times M_2$.
\end{rmk}

\begin{proposition}\label{prop:adjunct}
If $E$ and $F$ are cofibrant orthogonal
$G$-spectra or cofibrant $G$-spaces,
then there is a strong monoidal Quillen adjunction
$$\Delta :  \GIS \rightleftarrows L_E \GIS \times L_F \GIS : \prod$$
(we follow the usual convention of placing the left adjoint on top).
\end{proposition}
\begin{proof}
Take a map $f \co X \to Y$ in $\GIS$, then
$\Delta (f) = (f ,f) \co (X,X) \to (Y,Y)$.
For a map $(a,b) \co (A,B) \to (C,D)$,
$\prod (a,b)$ is given by
$a \prod b \co A \prod B \to C \prod D$.
That these are an adjoint pair is easy to see
since all of the isomorphisms and equalities below are natural.
$$\begin{array}{rcl}
L_E \GIS \times L_F \GIS \big( (X,X) , (A,B) \big)  & = &
L_E \GIS \big(X , A \big) \times L_F \GIS \big(X , B \big) \\
& = & \GIS \big(X , A \big) \times \GIS \big(X , B \big) \\
& \cong & \GIS \big( X, A \prod B \big)
\end{array}$$
If $f$ is a (acyclic) cofibration of $\GIS$, then
it is clear that $(f,f)$ is a (acyclic) cofibration of $L_E \GIS \times L_F \GIS$,
thus $\Delta$ is a left Quillen functor.
It is easy to see that this is a strong monoidal adjunction.
\end{proof}

It is clear that this result can be extended to any finite product of
localisations. Now we turn to a case when this adjunction is a
Quillen equivalence. We state the following theorem
rationally, but one can easily see that
a similar result will hold in any localised
case, as well as for $\GIS$ itself.

\begin{theorem}\label{thm:generalsplitting}
Let $\{ E_i \}_{i \in I}$ be a finite collection of cofibrant
orthogonal $G$-spectra or $G$-spaces. If
$E_i \smashprod E_j$ is rationally acyclic for $i \neq j$
and $\bigvee_{i \in I} E_i$ is rationally equivalent to $\sphspec$
then we have a strong monoidal Quillen equivalence
$$\Delta : \GIS_\qq
\overrightarrow{\longleftarrow}
\prod_{i \in I} L_{E_i} \GIS_\qq : \prod .$$
\end{theorem}
\begin{proof}
Proposition \ref{prop:adjunct} implies that this
is a strong monoidal Quillen pair, so we must show
that this is a Quillen equivalence.
We use the characterisation of Quillen equivalences
of \cite[Lemma 4.1.7]{hss00}, we must show that:
\begin{enumerate}
\item $\prod$ detects and preserves weak equivalences between fibrant objects,
\item $X \to \prod \Delta X \to \prod \fibrep \Delta X$ is a $\pi^\qq_*$-isomorphism for all
cofibrant $X$ in $\GIS_\qq$
\end{enumerate}
where $\fibrep$ denotes fibrant replacement in $\prod_{i \in I} L_{E_i} \GIS_\qq$.
The first statement follows easily, take a map $f \co A \to B$ between fibrant objects in
$\prod_{i \in I} L_{E_i} \GIS_\qq$.
The map $f$ is a weak equivalence exactly when each factor
$f_i$ is a rational $E_i$-equivalence.
Recall that a rational $E_i$-equivalence
between $E_i \smashprod S^0{\qq}$-local objects
is a $\pi_*$-isomorphism,
hence each $f_i$ is a $\pi_*$-isomorphism.
Thus $f$ is a weak equivalence if and only if
$\prod_{i} f_{i}$ is a $\pi_*$-isomorphism,
(since homotopy groups commute with products).

Now we must show that $X \to \prod \fibrep \Delta X$
is a weak equivalence for all cofibrant $X$.
This is equivalent to proving that
$\prod \fibrep \Delta X \to \fibrep_{E_i} X$ is a rational
equivalence for each $i$, where
$\fibrep_{E_i}$ denotes fibrant replacement in
$L_{E_i} \GIS_\qq$.
Since this is a finite product, it is weakly equivalent
to $\bigvee_{j \in I} \fibrep_{E_j} X$.
So we will prove that
$\bigvee_{j \in I} \fibrep_{E_j} X \to \fibrep_{E_i} X$
is a rational equivalence.
Since $\bigvee_{i \in I} E_i$ is rationally
equivalent to $\sphspec$ we can use Remark \ref{rmk:twoforthree}
repeatedly to see that a map $f$ in $\GIS_\qq$ is a rational equivalence
if and only if $f \smashprod \id_{E_i}$ is
a rational   equivalence for all $i \in I$.
This amounts to showing that
$\fibrep_{E_j} X \smashprod E_i$ is rationally acyclic
when $i \neq j$.

Since $\bigvee_{i \in I} E_i$ and $\sphspec$
are rationally equivalent,
they are isomorphic in $\ho \GIS_\qq$.
Thus the idempotent map
$p_i \co \bigvee_{i \in I} E_i \to E_i \to \bigvee_{i \in I} E_i$
gives an idempotent map of
$\sphspec$ in $\ho \GIS_\qq$.
Choosing a representative for
this map $e_i \co \fibrep_\qq \sphspec \to \fibrep_\qq \sphspec$
we obtain $e_i \fibrep_\qq \sphspec$.
Since $p_i$ is a cofibration,
$p_i \bigvee_{i \in I} E_i$ is
rationally equivalent to $E_i$.
Hence so is $(\fibrep_\qq p_i) (\fibrep_\qq \bigvee_{i \in I} E_i)$.
We can construct a homotopy commuting square
$$\xymatrix{
\fibrep_\qq \sphspec \ar[r]^{e_i} \ar[d]_\sim &
\fibrep_\qq \sphspec              \ar[d]_\sim \\
\fibrep_\qq \bigvee_{i \in I} E_i \ar[r]_{\fibrep_\qq p_i} &
\fibrep_\qq \bigvee_{i \in I} E_i
}$$
such that the vertical maps are weak equivalences (since
$\fibrep_\qq \sphspec$ and $\bigvee_{i \in I} E_i$
are rationally equivalent).
Now we apply Lemma \ref{lem:homcomtelescope} to
conclude that $e_i \fibrep_\qq \sphspec$
is rationally equivalent to $E_i$.
We now see that a rational $E_i$-equivalence
is the same as a rational $e_i \fibrep_\qq \sphspec$-equivalence.
Thus, by Corollary \ref{cor:comrel},  $f \co X \to Y$ is a rational $E_i$-equivalence
if and only $e_i^* \pi_*^H (f) \otimes \qq$
is an isomorphism for all $H$.
So we have isomorphisms $e_i^* \pi_*^H (\fibrep_{E_i} X) \otimes \qq \to
e_i^* \pi_*^H (X) \otimes \qq$ and
$\pi_*^H (E_i \smashprod X) \otimes \qq \cong
e_i^* \pi_*^H (X) \otimes \qq$.

The projection $\bigvee_{j \in I} E_j \to E_i$ is a rational $E_i$
equivalence (since $E_i \smashprod E_j$ is rationally acyclic
for $i \neq j$).
Now we use the fact that $\fibrep_{E_i} X$ is $E_i$-local to obtain isomorphisms
$$\pi_n^H (\fibrep_{E_i} X) \otimes \qq \cong
[\Sigma^n \sphspec, \iota_H^* \fibrep_{E_i} X]^H \otimes \qq \cong
[\Sigma^n E_i , \iota_H^* \fibrep_{E_i} X]^H \otimes \qq \cong
[e_i \Sigma^n \sphspec, \iota_H^* \fibrep_{E_i} X]^H \otimes \qq$$
Now we apply \cite[Chapter III, Theorems 2.4 and 2.7]{mm02}, since
$e_{i} \Sigma^n \sphspec$ is a colimit of cofibrant objects
it is non-degenerately based, thus we obtain the $\lim^1$
exact sequence of pointed sets below.
$$* \to {\lim}^1 [\Sigma^{n+1} \sphspec, \iota_H^* \fibrep_{E_i} X ]_l^H
\to [e_{\fscr} \Sigma^n \sphspec, \iota_H^* \fibrep_{E_i} X ]_l^H
\to \lim [\Sigma^{n} \sphspec, \iota_H^* \fibrep_{E_i} X ]_l^H \to *$$
Since $e_{i}^*$ (that is, $e_{i}$ acting on the first factor)
is an idempotent on sets of maps in the rational homotopy
category and $[\Sigma^{n+1} \sphspec, \iota_H^* \fibrep_{E_i} X ]_l^H
\cong [\Sigma^{n+1} \sphspec, \iota_H^* \fibrep_{E_i} X ]^H_\qq$
the tower it creates satisfies the Mittag-Leffler condition
(\cite[Definition 3.5.6]{weib}) hence the $\lim^1$ term is zero.
So we see that
$$[e_{i} \Sigma^n \sphspec, \iota_H^* \fibrep_{E_i} X ]_l^H
\cong e_{i}^*[\Sigma^{n} \sphspec, \iota_H^* \fibrep_{E_i} X ]_l^H $$
(a limit of idempotent maps is equivalent to taking the image).
Now consider the action of idempotents on maps in the homotopy category:
$e_{i}^*[f] = [f] \smashprod [e_{i}]$ and
$(e_{j})_* [f] = [f] \smashprod [e_{j}]$,
hence $(e_{j})_* (e_{i})^*=(e_i)_*e_{j}^*=0$.
We can now finish our argument as follows:
$$\pi_*^H (\fibrep_{E_i} X \smashprod E_j) \otimes \qq \cong
(e_j)_*(e_i)^* \pi_*^H (\fibrep_{E_i} X) \otimes \qq =0.$$
\end{proof}

\begin{corollary}\label{cor:possiblesplittings}
Finite orthogonal idempotent decompositions
of the unit of $A(G)$
correspond to finite splittings of $\GIS$.
The same statement also holds for
$A(G) \otimes \qq$ and $\GIS_\qq$.
\end{corollary}
\begin{proof}
Let $1 \in A(G)$ be the sum of (a finite collection of)
idempotents $e_i$, such that $e_i e_j =0$ whenever $i \neq j$,
then $\GIS$ splits as the product of the
localised categories $L_{e_i \sphspec} \GIS$.
Conversely, if $\GIS$ splits as the product of
localised categories $L_{E_i} \GIS$
then $A(G) \cong \bigoplus_i [\sphspec, \sphspec]_{E_i}^G$.
Hence, for each $i$, we have an idempotent element $e_i$
which is the unit map in factor $i$ and the
trivial map elsewhere.
\end{proof}

\begin{rmk}\label{rmk:splittingidempotents}
Let $E_i$ be a collection of spectra satisfying the assumptions
of the splitting theorem with
corresponding idempotents $e_i$.
Let $X$ and $Y$ be spectra,
then $[X,Y]^G$ is an $A(G)$-module.
If $Y$ is $E_i$-local then $[X,Y]^G$
is isomorphic to $e_i [X,Y]^G$
(and is also isomorphic to maps in
the homotopy category of $L_{E_i} \GIS$).
If $M$ is an $A(G)$-module then
$M$ with $e_i$ inverted is
given by $e_i M$ (since $e_i$ is an idempotent).
Thus we can say that localisation at $E_i$
inverts $e_i \in A(G)$.
Equally, rationalisation inverts the primes
in $A(G)$. This explains why our two
kinds of localisation (rationalisation
and splitting) behave the same: in each
case we are simply inverting elements
of $A(G)$.
\end{rmk}

\section{Comparisons}\label{sec:compsplit}
We show that the splitting theorem for $\GIS_\qq$
implies the corresponding splitting for
$\GIS_\qq^+$, $G \mcal_\qq$ and $S_\qq \leftmod$. We let $\cofrep_+$
denote cofibrant replacement in $\GIS_+$
\begin{theorem}\label{thm:positivesplitting}
Let $\{ E_i \}_{i \in I}$ be a finite collection of cofibrant
orthogonal $G$-spectra or $G$-spaces. If
$E_i \smashprod E_j$ is rationally acyclic for $i \neq j$
and $\bigvee_{i \in I} E_i$ is rationally equivalent to $\sphspec$
then we have a strong monoidal Quillen equivalence
$$\Delta : \GIS_\qq^+
\overrightarrow{\longleftarrow}
\prod_{i \in I} L_{\cofrep_+ E_i} \GIS_\qq^+ : \prod .$$
\end{theorem}
\begin{proof}
We can replace the collection $\{ E_i \}$
by their positive cofibrant replacements
$\{ \cofrep_+ E_i \}$. It is clear that
$\bigvee_{i \in I} \cofrep_+ E_i$ is rationally
equivalent to $\sphspec$ and
$\cofrep_+ E_i \smashprod \cofrep_+ E_j$ is
rationally acyclic for $i \neq j$.
Now we can compare the statement above
to that for $\GIS$
by using Remark \ref{rmk:changemodellocal}.
Thus we see that the category
$L_{\cofrep_+ E_i} \GIS^+_\qq$
is Quillen equivalent to
$L_{E_i} \GIS_\qq$,
hence the splitting applies to the
positive stable case.
\end{proof}

\begin{theorem}\label{thm:comparisons}
Let $\{ E_i \}_{i \in I}$ be a finite collection of positive cofibrant
orthogonal $G$-spectra or $G$-spaces. Assume that
$E_i \smashprod E_j$ is rationally acyclic for $i \neq j$
and $\bigvee_{i \in I} E_i$ is rationally equivalent to $\sphspec$.
We have the following pair of commutative (see proof) diagrams
of strong symmetric monoidal Quillen equivalences.
\begin{displaymath}
\xymatrix@R+0.4cm@C+0.1cm{
G \mcal_\qq \ar@<+1ex>[r]^(0.4)\Delta
\ar@<+1ex>[d]^{\nn^{\#}}
& \ar@<+1ex>[l]^(0.6)\prod
\prod_{i \in I} L_{\nn E_i} G \mcal_\qq
\ar@<+1ex>[d]^{\prod_{i \in I} \nn^{\#}} \\
\GIS_\qq^+ \ar@<+1ex>[u]^{\nn}
\ar@<+1ex>[r]^(0.4)\Delta
& \ar@<+1ex>[l]^(0.6)\prod
\prod_{i \in I} L_{E_i} \GIS_\qq^+
\ar@<+1ex>[u]^{\prod_{i \in I} \nn }}
\qquad \qquad
\xymatrix@R+0.4cm@C+0.1cm{
S_\qq \leftmod \ar@<+1ex>[r]^(0.4)\Delta
\ar@<+1ex>[d]^{U}
& \ar@<+1ex>[l]^(0.6)\prod
\prod_{i \in I} L_{\nn E_i} S_\qq \leftmod
\ar@<+1ex>[d]^{\prod_{i \in I} U} \\
G \mcal_\qq \ar@<+1ex>[u]^{S_\qq \smashprod (-)}
\ar@<+1ex>[r]^(0.4)\Delta
& \ \ar@<+1ex>[l]^(0.6)\prod
\prod_{i \in I} L_{\nn E_i} G \mcal_\qq
\ar@<+1ex>[u]^{\prod_{i \in I} S_\qq \smashprod (-) } }
\end{displaymath}
\end{theorem}
\begin{proof}
The left hand diagram consists of Quillen equivalences by
Corollary \ref{cor:Nadjunctlocal} and Theorem \ref{thm:positivesplitting}.
The righthand diagram consists of Quillen equivalences by
Theorem \ref{thm:locfuncs} and Proposition \ref{prop:rmodlocal}.
It is easy to see that the two squares of left adjoints above
commute and hence so do the squares of right adjoints.
\end{proof}

\begin{rmk}
There is a Quillen pair
$ S_\qq \leftmod \overrightarrow{\longleftarrow} L_{\nn E_i} S_\qq \leftmod$
arising from the change of model structure functors
$ \id : G \mcal_\qq \overrightarrow{\longleftarrow} L_{\nn E_i} G \mcal_\qq : \id$.
\end{rmk}

\section{Idempotent Families Induce Splittings}\label{sec:idemsplit}
Equivariant considerations give us a class
of examples where we can apply our splitting result:
Theorem \ref{thm:familysplitting}.
In particular our splitting for the $O(2)$ case will
be of this form.

\begin{definition}
A collection of subgroups of $G$, $\fscr$, is called a \textbf{family}\index{Family}
if it is closed under conjugation and taking subgroups. The complement of
this set in the set of all subgroups of $G$ is a \textbf{cofamily}\index{Cofamily},
the cofamily associated to the family $\fscr$ will be denoted $\widetilde{\fscr}$.

We have the \textbf{universal $\fscr$-space}\index{Universal $\fscr$-space}
$E\fscr$\index{E@$E\fscr$}.
This is a $G$-CW complex constructed from cells of orbit type $G/H$
with $H \in \fscr$.
This space has the universal property:
$E\fscr^H$ is a contractible space for $H \in \fscr$ and
$E\fscr^H= \emptyset$ for $H \notin \fscr$.
Define a map $\varepsilon \co E\fscr_+ \to S^0$ by using the
projection $E\fscr=E\fscr^{e} \to *$ and then adding a disjoint point to both.
The cofibre of $\varepsilon$, $C\varepsilon$, will be called the
\textbf{universal $\widetilde{\fscr}$-space}\index{Universal $\widetilde{\fscr}$-space}
and will be written $E \widetilde{\fscr}$\index{E@$E \widetilde{\fscr}$}.
 Applying the basic fact:
$C{f^H} \cong (Cf)^H$, we can then  see by a simple calculation that
$E\widetilde{\fscr}^H$ is a contractible space for $H \in \fscr$ and
$E\widetilde{\fscr}^H= S^0$ for $H \notin \fscr$.
\end{definition}

\begin{definition}\label{def:familymodelcat}
We set $\FS = L_{E\fscr_+} \GIS$\index{F IS@$\FS$}
and $\FFS = L_{E\widetilde{\fscr}}\GIS$\index{F ISF@$\FFS$},
these are known as the
\textbf{Bousfield $\fscr$-model structure}\index{Bousfield
$\fscr$-model structure} and
\textbf{Bousfield $\widetilde{\fscr}$-model structure}\index{Bousfield
$\widetilde{\fscr}$-model structure} on $\GIS$
(\cite[Chapter IV, Section 6]{mm02}).
We let $[X,Y]^\fscr$ denote the set of
maps between spectra $X$ and $Y$ in the homotopy category of
$\FS$  and similarly we use $[X,Y]^{\widetilde{\fscr}}$ for $\FFS$.
In turn we have rationalised categories
$\FS_\qq$ and $\FFS_\qq$
which are $\FS$ and $\FFS$ localised at $S^0{\qq}$.
\end{definition}
These are cofibrantly generated, left proper, symmetric monoidal model categories
satisfying the monoid axiom by Corollary \ref{cor:localsummary}.
We have versions for $G \mcal$,
where we set
$\fscr \mcal = L_{E\fscr_+} G \mcal$\index{F M@$\fscr \mcal$}
and
$\widetilde{\fscr} \mcal = L_{E\widetilde{\fscr}}
G \mcal$\index{F M@$\widetilde{\fscr} \mcal$},
we also have the rationalised versions
$\fscr \mcal_\qq$ and $\widetilde{\fscr} \mcal_\qq$.

\begin{proposition}\label{prop:familyequiv}
The following conditions on a map $f \co X \to Y$ are equivalent.
\begin{enumerate}
\item $f$ is an $E \fscr_+$-equivalence.
\item $f_* \co \pi_*^H(X) \to \pi_*^H(Y)$ is an isomorphism for all $H \in \fscr$.
\end{enumerate}
\end{proposition}
\begin{proof}
This result is \cite[Chapter IV, Proposition 6.7]{mm02}.
\end{proof}

A map satisfying the second condition is called an
\textbf{$\fscr$-equivalence}\index{F-equivalence@$\fscr$-equivalence}.
We will only need the next few results for $G \mcal$,
so we state them in that notation.
Recall that the generating cofibrations
and acyclic cofibrations of $G \mcal$ are defined in terms
of the objects $\Sigma^\infty_V (G/H_+ \smashprod S^n)$ for $H$ a subgroup
of $G$ and $V$ an indexing space.
If we restrict these sets to only use those $H$ in
some family $\fscr$, then we obtain the notions of
$\fscr$-cofibrations, $\fscr$-fibrations
and $\fscr$-equivalences.
These collections of maps form a model category
by the following result.

\begin{theorem}\label{thm:Fmodelstructure}
The category $G \mcal$ has an
\textbf{$\fscr$-model structure}\index{f-model structure@$\fscr$-model structure}
with weak equivalences the
$\fscr$-equivalences, cofibrations the $\fscr$-cofibrations and fibrations as defined by the
lifting property. This is a compactly generated proper model structure and
the identity functor gives the left adjoint of a Quillen equivalence
from the $\fscr$-model structure on $G \mcal$ to $\fscr \mcal$.
\end{theorem}
\begin{proof}
This is \cite[Chapter IV Theorems 6.5 and 6.9]{mm02}.
\end{proof}

Let $N$ be a normal subgroup of $G$,
then the subgroups of $N$ form a family $\fscr(N)$
of subgroups of $G$.
We denote the $\fscr(N)$-model structure on $G \mcal$
by $G \mcal(N)$\index{G MNq@$G \mcal(N)$}\label{app:GMNmodel}.
This model structure coincides with the model structure
on $G \mcal$ lifted (see Lemma \ref{lem:lift}) over the right adjoint
$\iota^*_N \co G \mcal_U \to N \mcal_{\iota^*_N U}$
(where $U$ is a $G$-universe, so $\iota^*_N U$ is an $N$-universe).
Hence there is a Quillen pair
$G_+ \smashprod_N (-) : N \mcal \overrightarrow{\longleftarrow}
G \mcal(N) : \iota_N^*$.

\begin{proposition}\label{prop:NonGmodelmonoidal}
For $N$ a normal subgroup of $G$,
$G \mcal(N)$ is a monoidal model category that
satisfies the monoid axiom.
\end{proposition}
\begin{proof}
The identity map $G \mcal(N) \to G \mcal$
is a left Quillen functor, since the generating
cofibrations and acyclic cofibrations
of $G \mcal(N)$ are a subset of those for $G \mcal$.
Hence the pushout product and monoid axioms follow
from those for $G \mcal$.
\end{proof}

\begin{theorem}\label{thm:familylocalisations}
For any orthogonal spectra $X$ and $Y$ we have natural isomorphisms
$$\begin{array}{lcl}
{[X,Y]^\fscr} & \cong & [X \smashprod E \fscr_+, Y \smashprod E \fscr_+]^G \\
{[X,Y]^{\widetilde{\fscr}}} & \cong & [X \smashprod E \widetilde{\fscr},
Y \smashprod E \widetilde{\fscr}]^G.
\end{array}$$
Thus we have an equivalence of categories between $\ho \FFS$ and the
full subcategory of objects $X \smashprod E \fscr_+$ in $\ho \GIS$.
Equally there is an equivalence of categories between $\ho \FFS$ and the
full subcategory of objects $X \smashprod E \widetilde{\fscr}$ in $\ho \GIS$.
The map $\rho \co X \to F( E \fscr_+ , X)$ induced by
$E \fscr_+ \to S^0$ gives an $E \fscr_+ $-equivalence from $X$
to an object that is $E \fscr_+$-local.
\end{theorem}
\begin{proof}
This is \cite[Chapter IV, Theorems 6.11 and 6.14]{mm02}
and the last sentence is \cite[Chapter IV, Proposition 6.4]{mm02}.
\end{proof}

\begin{lemma}\label{lem:familysmash}
For any family $\fscr$, $\pi_*^H(E \fscr_+ \smashprod E \widetilde{\fscr}) =0.$
\end{lemma}
\begin{proof}
We apply the functors $\Phi^H$ and the result follows from
Proposition \ref{prop:geomfixedpoints}.
\end{proof}

\begin{corollary}\label{cor:equivs}
Take a family of subgroups $\fscr$, then a map $f \co X \to Y$
is a $\pi_*$-isomorphism if and only if $f$ is an $E \fscr_+$-equivalence
and an $E \widetilde{\fscr}$-equivalence.
\end{corollary}
\begin{proof}
We have the cofibre sequence $E \fscr_+ \to \sphspec \to E \widetilde{\fscr}$
and an $\sphspec$-equivalence is a $\pi_*$-isomorphism.
\end{proof}

\begin{definition}\label{def:idemfamily}
Let $\fscr$ be a family of subgroups of $G$ such that $\fscr$
is an open and closed $G$-invariant subspace
of $\scal_f G$ that is a union of $\sim$ classes. Then we call such a
collection an \textbf{idempotent family}\index{Idempotent family}.
\end{definition}
An idempotent family $\fscr$ of $G$ corresponds via
tom Dieck's isomorphism to an idempotent
$e_\fscr$ of the rational Burnside ring of $G$ and the associated cofamily
corresponds to the complement $1-e_\fscr$, which we will also denote by
$e_{\widetilde{\fscr}}$.

\begin{lemma}\label{lem:Geidemfamily}
Consider a short exact sequence of compact Lie groups
$1 \to G_e \to G \to F \to 1$.
Then $\fscr$, the set of subgroups of $G_e$
(the identity component of $G$), is an idempotent family.
\end{lemma}
\begin{proof}
The identity component of $G$ is normal,
hence the set of subgroups of $G_e$
is closed under conjugation and taking subgroups,
thus $\fscr$ is a family.
Recall that since $G$ is compact,
$F = G /G_e$ is finite.
Let $\fscr'$
be the set of conjugacy classes of groups in $\fscr$ with
finite index in their normaliser.
Take $H \in \fscr'$, by \cite[Chapter II, Corollary 5.6]{bred} we know that
if $K \in \fcal G$ is in some sufficiently small neighbourhood of
$H$ in the space $\fcal G$, then $K$ is subconjugate
to $H$ and so $K$ is a subgroup of $G_e$. It follows that
$\fscr'$ is open in $\fcal G/G$.
Now take $(K)$ to be in $(\fcal G/G) \setminus \fscr'$,
so there is a $g \in G \setminus G_e$ such that $K \cap gG_e$ is non-empty.
Then any $L \in \fcal G$ that is sufficiently close to $K$
also has a non-trivial intersection with $gG_e$
so $L$ is not a subgroup of $G_e$,
it follows that $\fscr'$ is also closed.
Hence $e_\fscr$,
the characteristic function of $\fscr'$,
is a continuous map $\fcal G /G \to \qq$.
Thus $e_\fscr$ is an idempotent, since $e_\fscr(H)=1$
if $(H) \in \fscr$ and zero otherwise.
It follows that the support of $e_\fscr$
(a subset of $S_f G$) is $\fscr$.
\end{proof}

\begin{lemma}\label{lem:familiestoidempotents}
Let $\fscr$ be an idempotent family of subgroups of $G$,
with corresponding idempotent $e_\fscr$.
Then the composite $E \fscr_+ \to S^0 \to e_\fscr S^0$
is a rational $\pi_*$-isomorphism of orthogonal $G$-spectra.
Equally $E \widetilde{\fscr}$ and
$e_{\widetilde{\fscr}} S^0$ are rationally equivalent.
Furthermore a map $f$ is an $E \fscr_+$-equivalence if and only if
$\iota_H^*(e_\fscr) \pi_*^H (f) \otimes \qq$ is an isomorphism.
A map $f$ is an $E \widetilde{\fscr}$-equivalence if and only if
$\iota_H^*(e_{\widetilde{\fscr}}) \pi_*^H (f) \otimes \qq$ is an isomorphism.
\end{lemma}
\begin{proof}
Look at the geometric fixed points of
$E \fscr_+ \to S^0 \to e_\fscr S^0$ to see that
this is a rational equivalence.
Similarly we have a zig-zag
of rational equivalences
$e_{\widetilde{\fscr}} S^0 \to e_{\widetilde{\fscr}} E \widetilde{\fscr}
\leftarrow E \widetilde{\fscr}$.
The second statement then follows by Corollary \ref{cor:comrel}.
\end{proof}

\begin{corollary}
For an idempotent family $\fscr$ and idempotents as above
we have the following collection of rational equivalences for any orthogonal
$G$-spectrum $X$.
$$ \begin{array}{rcccl}
e_\fscr X & \simeq & e_\fscr \sphspec \smashprod X & \simeq & E \fscr_+ \smashprod X \\
e_{\widetilde{\fscr}} X & \simeq & e_{\widetilde{\fscr}} \sphspec \smashprod X & \simeq &
E \widetilde{\fscr} \smashprod X
\end{array} $$
\end{corollary}

\begin{lemma}\label{lem:rightproperfamily}
For $\fscr$ an idempotent family, the categories
$\FS_\qq$, $\FFS_\qq$ (and their $S$-module counterparts
$\fscr \mcal_\qq$ and $\widetilde{\fscr} \mcal_\qq$) are right proper.
\end{lemma}
\begin{proof}
Let $e \in A(G) \otimes \qq$ be an idempotent, then
for any exact sequence of $A(G) \otimes \qq$-modules
$ \dots \to M_i \to M_{i-1} \to \dots $ the sequence
$ \dots \to e M_i \to e M_{i-1} \to \dots $ is exact.
Right properness follows from the proof of
Lemma \ref{lem:rightproperrational} by
applying $e_\fscr$ (or $e_{\widetilde{\fscr}}$)
to the long exact sequence of rational homotopy groups
of a fibration.
\end{proof}

\begin{theorem}\label{thm:familysplitting}
For $G$ a compact Lie group and $\fscr$ an idempotent family of subgroups
of $G$, we have a strong monoidal Quillen equivalence of cofibrantly generated,
proper, monoidal model categories satisfying the monoid axiom
$$\Delta : \GIS_\qq \rightleftarrows \FS_\qq \times \FFS_\qq : \prod.$$
In particular we have the following natural isomorphism for any $G$-spectra $X$ and $Y$
$$[X,Y]^G_\qq \cong
[X \smashprod E \fscr_+, Y \smashprod E \fscr_+]^G_\qq
\oplus
[X \smashprod E \widetilde{\fscr}, Y \smashprod E \widetilde{\fscr}]^G_\qq.$$
\end{theorem}
\begin{proof}
By Lemma \ref{lem:familysmash} and
Proposition \ref{prop:idemsplit}
we see that $E\fscr_+$ and $E \widetilde{\fscr}$
satisfy the assumptions of Theorem \ref{thm:generalsplitting}.
The description of $[X,Y]^G_\qq$ follows from
Theorem \ref{thm:familylocalisations}
and Theorem \ref{thm:rathomotopymaps}.
\end{proof}

Thus in terms of rational cohomology theories we have a natural isomorphism
$E^* \cong E_\fscr^* \oplus E_{\widetilde{\fscr}}^*$
for $E_\fscr$ an $E \fscr_+$-localisation of $E$
and $E_{\widetilde{\fscr}}$ an $E\widetilde{\fscr}$-localisation of $E$.
The proof shows that $E_\fscr \vee E_{\widetilde{\fscr}}$
is rationally weakly equivalent to $E$ and hence in terms of rational homology theories
$E_* \cong (E_\fscr)_* \oplus (E_{\widetilde{\fscr}})_*$.

\begin{rmk}
Consider $\GIS_\qq$ and fix a family $\fscr$, we still have
$E \fscr_+ \to S^0 \to E \widetilde{\fscr}$ and the
smash product of these two spectra is trivial.
When will $E \fscr_+ \vee E \widetilde{\fscr}$ and $\sphspec$
be rationally equivalent? As described in
Corollary \ref{cor:possiblesplittings},
if there is a rational equivalence then
there must be an idempotent $e \in A(G) \otimes \qq$
such that $E \fscr_+$ is rationally equivalent to $e \sphspec$,
whence $\fscr$ will correspond to an idempotent of $A(G) \otimes \qq$.
Conversely if $\fscr$ corresponds to an idempotent of
$A(G) \otimes \qq$ then $E \fscr_+ \vee E \widetilde{\fscr}$ and $\sphspec$
will be rationally equivalent. So we only obtain these splittings
when $\fscr$ is an idempotent family.
\end{rmk}

\part{Finite Groups}\label{part:finite}
\chapter{Rational $G$-Spectra for Finite $G$}\label{chp:fingroups}

We reprove the result of \cite[Example 5.1.2]{ss03stabmodcat}
using the methods of \cite{greshi}
to classify rational $G$-spectra in terms of an algebraic model
for finite $G$ (see Corollary \ref{cor:finiteclassification}).
We describe the algebraic model in the first section
and apply our splitting result to the category of rational $G$-spectra
in the second. In the third we compare each
split piece to the relevant part of the algebraic model.
Our input to this new proof consists of three pieces of work:
the splitting of the category of $G$-spectra,
showing that the results of \cite{greshi} can be applied
to this setting and proving Proposition \ref{prop:homotopycalc}
and Theorem \ref{thm:finiteintrinsicformality}
to complete the classification.
Section \ref{sec:finitecomp} is the first time we
will need to consider right modules
over an enriched category and use the Morita equivalence of
Theorem \ref{thm:monoidalmorita}. Thus, we also include
a chapter on enriched categories in this part of the thesis.

\section{The Algebraic Category}\label{sec:finitealg}
We use the description of rational $G$-cohomology theories as
implied by \cite[Appendix A]{gremay95} to obtain the following definition.
Since the algebraic model is a product of categories,
we work piecewise and replace $dg\qq W_G H \leftmod$
by $\rightmod \ecal_a^H$
in Proposition \ref{prop:FiniteAlgMorita}.

\begin{definition}
The algebraic model for rational $G$-spectra for finite $G$ is
$$dg\mathcal{A}(G)=
\prod_{(H) \leqslant G} dg\qq W_G H \leftmod.
$$
Here we are using the projective model structure on the categories of modules.
\end{definition}

\begin{rmk}\label{rmk:hopf&Gdiagonal}
The rational group ring of $G$
is a Hopf algebra with co-commutative coproduct.
Since $\qq G \otimes_\qq \qq G \cong \qq(G \times G)$
the product is induced by the group multiplication
$G \times G \to G$, the coproduct by the diagonal
$G \to G \times G$ and the antipode by the inversion
map $G \to G$. Because $\qq G$ is a Hopf-algebra
there is a monoidal product on $dg \qq G$-modules.
\end{rmk}

\begin{definition}\label{def:QGtensorproduct}
For $M$ and $N$ in $dg \qq G \leftmod$
define their tensor product to be
$M \otimes_\qq N$ with the diagonal
$G$-action. There is an internal function object, defined as
$\underhom_\qq(M, N)$ with $G$-action by conjugation.
\end{definition}

\begin{definition}
For any $X \in dg \qq G \leftmod$,
there is an $dg \qq G$-map
$Av_G \co X \to X^G$
defined by $Av_G(x) =  |G|^{-1} \Sigma_{g \in G} gx$.
\end{definition}

\begin{lemma}
The tensor product above gives
$dg \qq G \leftmod$
the structure of a closed symmetric monoidal
$dg \qq$-model category that satisfies the monoid axiom.
\end{lemma}
\begin{proof}
Let $\Delta$ be the coproduct of $\qq G $,
$T$ be the interchange of factors map and
$\nu_M$ and $\nu_N$ be the
$\qq G$-action maps of $dg \qq G$-modules $M$ and $N$.
The action of $\qq G$ on $M \otimes_\qq N$ is then
defined as the composite:
$$
\begin{array}{rcl}
\qq G \otimes_\qq M \otimes_\qq N
& \overset{\Delta}{\to} &
\qq G \otimes_\qq \qq G \otimes_\qq M \otimes_\qq N \\
& \overset{\id \otimes T \otimes \id}{\to} &
\qq G \otimes_\qq M \otimes_\qq \qq G \otimes_\qq N \\
& \overset{\nu_M \otimes \nu_N}{\to} & M \otimes_\qq N.
\end{array}
$$
It is clear that this is a commutative, associative monoidal product
with unit $\qq$ (with trivial $G$-action).
To prove that the pushout product axiom holds,
it suffices (by \cite[Lemma 3.5(1)]{ss00}) to check
the following pair of conditions.
\begin{enumerate}
\item If $f$ and $g$
are generating cofibrations then the pushout product,
$f \square g$, is a cofibration.
\item If $f$ is a generating cofibration and $g$
is a generating acyclic cofibration then
$f \square g$ is a weak equivalence.
\end{enumerate}
Let $f$ and $g$ be generating cofibrations,
then $f \square g$ is an inclusion with
cokernel $\qq(G \otimes G )$ (in some degree).
We claim that this cokernel is cofibrant: by
Remark \ref{rmk:hopf&Gdiagonal} this is isomorphic
(as a $\qq G $-module) to $\bigoplus_{g \in G} \qq G$
and the claim follows from the fact that
$\qq G $ is cofibrant as a $\qq G $-module.
Taking a generating cofibration $f$ and a generating acyclic
cofibration $g$, it follows that
$f \square g$ is a weak equivalence since
both the domain and codomain are acyclic.

There is a strong symmetric monoidal adjoint pair
$\varepsilon^* : dg \qq \overrightarrow{\longleftarrow}
dg \qq G : (-)^{G} $,
where $\varepsilon^*(X)$ is $X$ with trivial action and the right adjoint
is the fixed point functor.
We show that this is a Quillen pair by
proving that the right adjoint
preserves fibrations and weak equivalences.
Take $f \co X \to Y$ a surjection
and let $y \in Y^G$, then there is an $x$ such that
$f(x) = y$. Since $Av_G(x) \in X^G$
and $f(Av_G(x)) =Av_G(f(x))= Av_G(y) =y$, it follows
that $f^G$ is surjective.
That $(-)^G$ preserves homology isomorphisms
follows immediately from the isomorphism
$H_*(X^G) \cong (H_*X)^G$.
Thus $dg \qq G$ is a $dg \qq$-model category.

For a commutative ring $R$, there is a tensor
product on $dg R \leftmod$, $- \otimes_R -$.
This category with the projective model structure
satisfies the monoid axiom, as proven in
\cite[Proposition 3.1]{shiHZ}.
This result in the case $R = \qq$
will imply that $dg \qq G \leftmod$
satisfies the monoid axiom. Thus we
write out the proof of
\cite[Proposition 3.1]{shiHZ},
adapted to the notation of $\qq G$.
By \cite[Lemma 3.5(2)]{ss00} it suffices to show that
transfinite composition and pushouts of
maps of the form $j \otimes \id_Z \co A \otimes Z \to B \otimes Z$
are weak equivalences, for $j$ a generating acyclic cofibration.
The generating acyclic cofibrations for $dg \qq G$
are maps $0 \to D^n(\qq G)$ for some integer $n$.
Take any $Z \in dg \qq G \leftmod$, then it is easy to check that
$D^n(\qq G) \otimes_\qq Z$ is also acyclic.
Then we note that $0 \to D^n(\qq G) \otimes_\qq Z$
is an injection and a homology isomorphism.
Such maps are closed under pushouts and
transfinite compositions (they are acyclic cofibrations
in the injective model structure on
$dg \qq G$-modules), hence the monoid axiom holds.
\end{proof}

Note that this proves shows that $\otimes_i \qq G$
is cofibrant as a $dg \qq G$-module.
\begin{lemma}
The model category of $dg \qq G \leftmod$
is generated by $\qq G$.
\end{lemma}

\begin{definition}
Let $\gcal_{a,G} = \{ \qq, \qq G, \qq (G \times G ), \qq(G \times G \times G), \dots \}$
and let $\ecal_{a,G}$ be the $dg \qq$-category
with object set $\gcal_{a,G}$
and $dg \qq$-mapping object given by
$\ecal_{a,G} (X, Y) = \underhom_\qq(X,Y)^{G }$.
Now we define $\gcal_{a,G}^H = \gcal_{a,W_G H}$
and $\ecal_{a,G}^H = \ecal_{a,W_G H}$. We will
usually suppress the $G$
and reduce this notation to
$\gcal_a^H$ and $\ecal_a^H$.
\end{definition}

\begin{proposition}\label{prop:FiniteAlgMorita}
There is a strong symmetric monoidal Quillen equivalence
$$ (-) \otimes_{\ecal_a^H} \gcal_a^H :
\rightmod \ecal_a^H \overrightarrow{\longleftarrow}
dg \qq W_G H \leftmod : \underhom(\gcal_a^H, -) $$
and furthermore this is an adjunction of closed
symmetric monoidal $dg \qq$-model categories.
\end{proposition}
\begin{proof}
This is an application of \cite[Proposition 3.6]{greshi},
see Theorem \ref{thm:monoidalmorita}.
\end{proof}

\begin{lemma}\label{lem:GmapsareQG}
There is an isomorphism of rings
$\underhom_\qq( \qq G, \qq G)^G \cong \qq G$.
\end{lemma}
\begin{proof}
This is a standard result.
A $G$-map $\qq G \to \qq G$ is defined by the image of $1$,
let $\tilde{g}$ represent the $G$-map which sends
$1 \to g$, for $g \in G$. These are a set of generators
for $\underhom_\qq( \qq G, \qq G)^G$.
We define the above ring isomorphism to be that
map which sends $\tilde{g}$ to
$g^{-1} \in G$.
\end{proof}

\section{Topological Models and Splitting}\label{sec:finitetop}
We take tom-Dieck's isomorphism
in the case of a finite group,
$A(G) \otimes \qq \cong \prod_{(K)\leqslant G} \qq$,
and see how it applies
to the model category of rational $G$-spectra.
For each conjugacy class of subgroups, $(H) \leqslant G$,
there is an idempotent $e_H$ given by projection onto factor $(H)$.
We use this in Theorem \ref{thm:finitesplitting}
to split the category of rational $G$-spectra
into a collection of model categories each generated
by a single object. We also provide a version
of this splitting in terms of modules
over a ring spectrum (Proposition \ref{prop:GspecHtoSHmod}).
The advantage of this second description is that every
object of the category is fibrant,
which is needed for Theorem \ref{thm:finitemoritaequiv}
(see Remark \ref{rmk:whyfibrant}).

\begin{definition}
For a group $G$, with subgroups $H$ and $K$,
we say that $K$ is \textbf{subconjugate}\index{Subconjugate}
to $H$ if the $G$-conjugacy class of $K$ contains
a subgroup of $H$, we write $K \leqslant_G H$. In turn $K$ is
\textbf{strictly subconjugate}\index{Strictly subconjugate}
to $H$ if the $G$-conjugacy class of $K$ contains
a strict subgroup of $H$, the notation for this is $K <_G H$.
\end{definition}

\begin{definition}
Take $H$ a subgroup of $G$, then we have a pair of families
of subgroups of $G$:
$[\leqslant_G H]$ -- the family of all subgroups of $G$
which are subconjugate to $H$
and $[<_G H]$ -- the family of all subgroups of $G$
which are strictly subconjugate to $H$.
We can then form $G$-CW complexes
$E[\leqslant_G H]_+$\index{E [ H]@$E[\leqslant_G H]_+$}
and $E[<_G H]_+$\index{E [ H]@$E[<_G H]_+$}.
There is a map $E[<_G H]_+ \to E[\leqslant_G H]_+$, we call the cofibre
of this map $E\langle H\rangle $\index{E \lange H \rangle @$E \langle H \rangle$}.
\end{definition}

Note that since $E[<_G H]_+$ and $E[\leqslant_G H]_+$
are cofibrant as $G$-spaces, the space $E\langle H\rangle$
is also cofibrant as a $G$-space. We can also describe $E\langle H\rangle$
as $E[\leqslant_G H]_+ \smashprod E\widetilde{[<_G H]}$.
Since geometric fixed point functors preserve cofibre sequences,
$\Phi^K (E\langle H\rangle)$ is contractible unless
$(K)=(H)$, whence it is non-equivariantly
equivalent to $\sphspec$.

\begin{lemma}
The families $[\leqslant_G H]$ and
$[<_G H]$ are idempotent families (see Definition \ref{def:idemfamily}),
with corresponding idempotents
$e_{[\leqslant_G H]} = \Sigma_{(K) \leqslant H} e_K$ and
$e_{[<_G H]} = \Sigma_{(K) < H} e_K$.
\end{lemma}
\begin{proof}
Since $G$ is finite, any collection of conjugacy classes of
subgroups of $G$ defines a unique idempotent of the Burnside ring.
\end{proof}

\begin{theorem}\label{thm:finitesplitting}
There is a strong symmetric monoidal Quillen equivalence between the category
of rational $G$-spectra and the product of the categories
$L_{E \langle H \rangle} G \mcal_\qq$, as $H$ runs over the set of conjugacy
classes of subgroups of $G$.
$$\Delta : G \mcal_\qq
\overrightarrow{\longleftarrow}
\prod_{(H) \leqslant G} L_{E \langle H \rangle} G \mcal_\qq
: \prod$$
The left adjoint takes a rational $G$-spectrum $X$,
to the constant collection of $X$ in every factor.
The right adjoint takes the collection
$Y_H$ to the $G$-spectrum $\prod_{(H)} Y_H $.
\end{theorem}
\begin{proof}
Using Theorems \ref{thm:generalsplitting} and
\ref{thm:comparisons} all we have to show is that
$E \langle H \rangle \smashprod E \langle K \rangle$
is rationally acyclic whenever $(H) \neq (K)$ and
$\bigvee_{(H) \leqslant G} E \langle H \rangle$ is rationally
equivalent to $\sphspec$ (working temporarily in
orthogonal spectra).
We claim that $E \langle H \rangle$ is rationally
equivalent to $e_H \sphspec$, from which
both conditions above follow immediately. Since
$E \langle H \rangle$ can be constructed as
$E[\leqslant_G H]_+ \smashprod E\widetilde{[<_G H]}$
it is rationally equivalent to
$e_{[\leqslant_G H]} \sphspec \smashprod (1-e_{[<_G H]}) \sphspec$
by Lemma \ref{lem:familiestoidempotents}
and hence rationally equivalent to $e_{[\leqslant_G H]}(1-e_{[<_G H]}) \sphspec$.
We then have a zig-zag of $\pi_*^\qq$-isomorphisms
$$e_H \sphspec \to
e_H (e_{[\leqslant_G H]}(1-e_{[<_G H]}) \sphspec)
\leftarrow e_{[\leqslant_G H]} \sphspec \smashprod (1-e_{[<_G H]}) \sphspec$$
which proves our claim.
Since $\sphspec \to \bigvee_{(H) \leqslant G} e_H \sphspec$
is a rational equivalence we have our result.
\end{proof}

\begin{corollary} \label{cor:EHequivs}
A map $f \co X \to Y$ in $G \mcal_\qq$  is a rational $E \langle H \rangle$-equivalence
if and only if $e_H f \co e_H X \to e_H Y$ is a rational equivalence.
\end{corollary}
\begin{proof}
One checks this by looking at geometric fixed points.
\end{proof}

\begin{corollary}
The Quillen equivalence
$$\Delta : G \mcal_\qq
\overrightarrow{\longleftarrow}
\prod_{(H) \leqslant G} L_{E \langle H \rangle} G \mcal_\qq : \prod$$
is a symmetric monoidal adjunction of
closed $Sp^\Sigma_+$-algebras.
\end{corollary}

We now fix some subgroup $H$
and study $L_{E \langle H \rangle} G \mcal_\qq$.

\begin{lemma}\label{lem:modelequality}
There is an equality of model structures:
$$L_{E \langle H \rangle} G \mcal_\qq =
L_{E \langle H \rangle} L_{E [ \leqslant_G H]_+ } G \mcal_\qq $$
that is to say, the weak equivalences, cofibrations and fibrations
agree.
\end{lemma}
\begin{proof}
We claim that we have the equality:
$L_{E \langle H \rangle} G \mcal =
L_{E \langle H \rangle} L_{E [ \leqslant_G H]_+ } G \mcal $,
where the cofibrations agree by definition.
The map $E \widetilde{ [ \leqslant_G H]} \to *$ is an
$E \langle H \rangle$ equivalence
(look at the homotopy groups and the idempotents).
Hence, considering the cofibre sequence
which defines $E \widetilde{ [ \leqslant_G H]}$
we have a weak equivalence.
$$E [ \leqslant_G H]_+ \smashprod E \langle H \rangle \to E \langle H \rangle$$
It then follows that an
$E [ \leqslant_G H]_+ \smashprod E \langle H \rangle$-equivalence
is an $E \langle H \rangle$-equivalence. So the weak equivalences
of $L_{E \langle H \rangle} G \mcal$ and
$L_{E \langle H \rangle} L_{E [ \leqslant_G H]_+ } G \mcal $
agree and thus we have proved our claim.
The result then follows immediately.
\end{proof}

\begin{rmk}
The weak equivalences of
$L_{E [ \leqslant_G H]_+ } G \mcal_\qq $
are those maps $f$ such that $\pi_*^K(f) \otimes \qq$
is an isomorphism for all $K \leqslant_G H$.
This is \cite[IV, Proposition 6.7]{mm02}.
\end{rmk}

\begin{lemma}
A map $f$ in $L_{E \langle H \rangle} G \mcal_\qq $
is a weak equivalence if and only if the induced map of homotopy groups
$i_H^*(e_H) \pi_*^H(f) \otimes \qq$ is a isomorphism.
\end{lemma}
\begin{proof}
Lemma \ref{lem:modelequality} and
Corollary \ref{cor:EHequivs} show that $f$
is a weak equivalence if and only if
$i_K^*(e_H) \pi_*^K(f) \otimes \qq$ is an isomorphism for
all $K \leqslant_G H$.
Now note that $K$ is a strict subset of $H$ then
$i_K^*(e_H)=0$, hence for any map $f$,
$i_K^*(e_H) \pi_*^K(f) \otimes \qq$ will be an isomorphism.
\end{proof}

Let $H \leqslant G$ then there is an idempotent
$e^G_H \in A(G) \otimes \qq$ and
$\iota_H^*(e^G_H) = e^H_H \in A(H) \otimes \qq$.
Thus we reword the lemma above as:
a map $f$ in $L_{E \langle H \rangle} G \mcal_\qq $
is a weak equivalence if and only if
$e_H^H \pi_*^H(f)$ is a isomorphism.

Now we must obtain a version of this splitting
with every object of the split categories
fibrant (see Remark \ref{rmk:whyfibrant}).
To achieve this we will apply the same
approach as in Section \ref{sec:rationalsphere}
and construct a suitable ring spectrum via
Theorem \ref{thm:algebralocalise}.

\begin{lemma}\label{lem:familycofibre}
Given $\fscr \subset \fscr'$, families of subgroups of $G$,
the cofibre of the induced map of classifying spaces
$E\fscr_+ \to E \fscr'_+$, is a cell complex.
\end{lemma}
\begin{proof}
One can either prove this directly or note that
it follows by the same proof as for
Lemma \ref{lem:S0qqiscell}.
\end{proof}

Recall from \cite[Chapter III, Proposition 2.6]{EKMM97}
that the smash product of a pair of cell complexes
is also a cell complex.

\begin{lemma}\label{lem:SHobject}
There is an $S^0_\mcal \qq \smashprod E \langle H \rangle$-local
commutative cell $S$-algebra
$S_H$\index{SH@$S_H$} whose
unit map is a rational $E \langle H \rangle$-equivalence
and an inclusion of cell complexes.
Furthermore every $S_H$-module is
$S^0_\mcal \qq \smashprod E \langle H \rangle$-local.
\end{lemma}
\begin{proof}
This result is an application of Theorem \ref{thm:algebralocalise}
using the cell object $S^0_\mcal \qq \smashprod E \langle H \rangle$
to create a commutative ring $S_H$ which is the
$S^0_\mcal \qq \smashprod E \langle H \rangle$-localisation of $S$.
By construction the unit map $S \to S_H$ is a
rational $E \langle H \rangle$-equivalence, hence
$S^0_\mcal \qq \smashprod E \langle H \rangle$ is $\pi_*$-isomorphic
to $S_H \smashprod S^0_\mcal \qq \smashprod E \langle H \rangle$.
Since $S_H$ is $S^0_\mcal \qq$-local, it has rational homotopy groups,
thus there is a
zig-zag of weak equivalences
$S^0_\mcal \qq \smashprod S_H \leftarrow \cofrep S \smashprod S_H \to S_H$.
Equally $S_H$ is weakly equivalent to
$S_H \smashprod \bigvee_{(K)} E \langle K \rangle$.
Since $S_H$ is $E \langle H \rangle$-local,
$S_H \smashprod E \langle K \rangle$ is acyclic
whenever $(H) \neq (K)$ (this is part of the proof of
Theorem \ref{thm:generalsplitting}).
It follows that $S_H \smashprod \bigvee_{(K)} E \langle K \rangle$
is weakly equivalent to $S_H \smashprod E \langle H \rangle$.
Thus $S_H$ is $\pi_*$-isomorphic to $S^0_\mcal \qq \smashprod E \langle H \rangle$.
Now we can use the proofs of
Lemmas \ref{lem:SQiscRQ} and \ref{lem:RQmodareSQlocal}
to show the last statement.
\end{proof}

\begin{proposition}\label{prop:GspecHtoSHmod}
The adjoint pair of the free $S_H$-module functor
and the forgetful functor
$$S_H \smashprod (-) : L_{E \langle H \rangle} G \mcal_\qq
\overrightarrow{\longleftarrow}
S_H \leftmod : U $$
is a strong symmetric monoidal Quillen equivalence.
\end{proposition}
\begin{proof}
The proof of Theorem \ref{thm:localisedtomodules}
can be applied in this case. The two points to note are: an
$S^0 \qq \smashprod E \langle H \rangle$-equivalence
between $S^0 \qq \smashprod E \langle H \rangle$-local objects
is a $\pi_*$-isomorphism and the unit map is
an $S^0 \qq \smashprod E \langle H \rangle$-equivalence.
\end{proof}

\begin{lemma}
The object $\cofrep(G/H_+) \smashprod S_H$
is a $G$-compact, cofibrant and fibrant generator
of $S_H \leftmod$.
\end{lemma}
\begin{proof}
Every object of $S_H \leftmod$ is fibrant and since
$\cofrep(G/H_+)$ is a cofibrant spectrum, so is
$\cofrep(G/H_+) \smashprod S_H$.
This object is $G$-compact since the right adjoint $U$ commutes with
filtered colimits and $\cofrep(G/H_+)$ is a $G$-compact $G$-spectrum.
Since the weak equivalences of $L_{E \langle H \rangle} G \mcal_\qq $
are the $i_H^*(e_H) \pi_*^H(f) \otimes \qq$-isomorphisms
it follows that $\cofrep(G/H_+)$ generates this model category.
Hence $\cofrep(G/H_+) \smashprod S_H$ generates $S_H \leftmod$.
\end{proof}

\section{Comparing Ringoids}\label{sec:finitecomp}
We use the results of \cite{greshi} to
replace $S_H \leftmod$ by $\rightmod \ecal_{top}^H$
(Theorem \ref{thm:finitemoritaequiv}).
This category is Quillen equivalent to the
category $\rightmod \ecal_t^H$
(Theorem \ref{thm:finiteEtopisEt}).
We show that the homology of $\ecal_t^H$
is given by $\ecal_a^H$ in
Proposition \ref{prop:homotopycalc}.
Then we use Theorem \ref{thm:finiteintrinsicformality}
to prove that $\ecal_t^H$ and $\ecal_a^H$
are quasi-isomorphic.
This will complete our classification
of rational $G$-spectra for finite $G$ and we summarise
this classification in Corollary \ref{cor:finiteclassification}.

\begin{definition}\label{def:finitegcalH}
Let $\gcal_{top}^H$ be the set of all smash products of
$\cofrep(G/H_+) \smashprod S_H$ (including the identity as the
zero-fold smash). Let $\ecal_{top}^H$ be the spectral category on
the objects of $\gcal_{top}^H$, so
by the proof of Theorem \ref{thm:spectralalgebras},
$$\ecal_{top}^H ( X ,Y) = \sing \mathbb{U} (i^* \nn^\# U F_{S_H}(X,Y))^{G}.$$
\end{definition}

With the exception of the unit, all objects of $\gcal_{top}^H$
are cofibrant and all objects are fibrant.
We use the results of \cite{greshi}
to replace this category of $S_H$-modules
by a category of modules over an endomorphism ringoid $\ecal_{top}^H$.

\begin{theorem}\label{thm:finitemoritaequiv}
The adjoint pair
$$(-) \smashprod_{\ecal_{top}^H} \gcal_{top}^H :
\rightmod \ecal_{top}^H
\overrightarrow{\longleftarrow}
S_H \leftmod : \underhom(\gcal_{top}^H,-)$$
is a Quillen equivalence
and an strong symmetric monoidal adjunction
of closed symmetric
monoidal spectral model categories.
\end{theorem}
\begin{proof}
This follows from Theorem \ref{thm:monoidalmorita}
with the adjustments as made in the proof of
Theorem \ref{thm:so2morita}.
\end{proof}

\begin{theorem}\label{thm:finiteEtopisEt}
There is a zig-zag of Quillen equivalences between
$\rightmod \ecal_{top}^H$ (enriched over
$Sp^\Sigma_+$) and a category
$\rightmod \ecal_{t}^H$
(enriched over $dg \qq \leftmod$).
These equivalences
are Quillen modules over the appropriate
enrichments. This zig-zag induces
an isomorphism of graded $\qq$-categories:
$\pi_*(\ecal_{top}^H) \cong \h_* \ecal_t^H$.
\end{theorem}
\begin{proof}
This is contained in the proof of \cite[Theorem 4.1]{greshi}
which is based on \cite[Corollary 2.16]{shiHZ}
and we go through this in some detail in
Section \ref{sec:movingaccross the machine}.
\end{proof}

\begin{rmk}\label{rmk:monoidalnoneqcase}
We consider the above theorem in the case of the trivial
group where our work reduces to that of
\cite{shiHZ}.
Here $\gcal_{top}$ has just one object
and $\rightmod \ecal_{top}$ is equivalent to
$S_\qq \leftmod$. Moving from
$\rightmod \ecal_{top}$ to $\rightmod \ecal_t$
is then just applying the functors of
\cite{shiHZ} to the spectrum $S_\qq$.
The resulting chain complex is then weakly equivalent
to $\qq$, as the comparison between
$\rightmod \ecal_t$ and $\rightmod \ecal_a$ below
will prove.
With reference to Remark \ref{rmk:monoidalissue}
this classification can be made symmetric monoidal.
This is shown by using the four
step comparison of \cite{HZcorrection},
where we use the fibrant replacement
functor of commutative rings in $Sp^\Sigma(dg \qq \leftmod)$
as constructed in \cite[Proposition 3]{HZcorrection}.
This fibrant replacement functor comes
from a model structure where weak equivalences
and fibrations are defined in terms of the
underlying category $Sp^\Sigma(dg \qq \leftmod)$.
\end{rmk}

\begin{proposition}\label{prop:finitecalc}
There is an isomorphism of rings
$$ \pi_* (F_{S_H} (G/H_+ \smashprod S_H,
G/H_+ \smashprod S_H )^G )
\cong \qq W_G H.$$
\end{proposition}
\begin{proof}
We can make the following identifications:
$$ \begin{array}{rcl}
F_{S_H} (G/H_+ \smashprod S_H, G/H_+ \smashprod S_H )^G
& \cong &
F(G/H_+, G/H_+ \smashprod S_H )^G  \\
& \cong & (G/H_+ \smashprod S_H)^H.
\end{array} $$
Thus we must calculate $\pi_*( (G/H_+ \smashprod S_H)^H)$,
as a rational vector space.
This is isomorphic to
$\pi_*((G/H_+ \smashprod E \langle H \rangle)^H) \otimes \qq$.
Now $\iota^*_H E[ \leqslant_G H]$
is $H$-equivariantly weakly equivalent to $S$,
so $E \langle H \rangle$ is
$H$-equivariantly weakly equivalent to
$E \widetilde{\fscr_H}$ (see the proof of
Proposition \ref{prop:geomfixedpoints}).
Thus we have an isomorphism
$$\pi_*((G/H_+ \smashprod E \langle H \rangle)^H) \otimes \qq
\cong \pi_*(\Phi^H G/H_+) \otimes \qq.
$$
The following is standard:
$\Phi^H \Sigma^\infty G/H_+ \simeq \Sigma^\infty (G/H^H)
= \Sigma^\infty W_G H$, the suspension spectrum of a finite set.
Thus $\pi_*(\Phi^H G/H_+) \otimes \qq
\cong \pi_*(W_G H_+) \otimes \qq \cong \qq W_G H$,
hence there is an isomorphism of rational vector spaces
$ \pi_* (F_{S_H} (G/H_+ \smashprod S_H,
G/H_+ \smashprod S_H )^G )
\cong \qq W_G H$.

Now we prove that we have an isomorphism of rings.
For each $gH \in W_G H$ there is a $G$-map $\widetilde{gH} \co G/H_+ \to G/H_+$
which takes $kH \to kgH$. For $g_1 H$ and $g_2 H$ in $W_G H$,
$\widetilde{g_2 H} \circ \widetilde{g_1 H} = \widetilde{g_1 g_2 H}$.
The set of $\widetilde{gH }$ for $gH \in W_G H$ generate the ring
$\pi_* (F_{S_H} ( G/H_+ \smashprod S_H, G/H_+ \smashprod S_H )^G )$.
We send $\widetilde{gH}$ to $g^{-1}H$ to obtain a
ring isomorphism as desired.
\end{proof}

\begin{lemma}
For an integer $i \geqslant 1$,
the $i$-fold product of $G/H$ contains
$|W_G H|^{i-1}$ disjoint copies of $G/H$.
More precisely, in the Burnside ring $A(G)$
$$G/H^{\times_{i}} = |W_G H|^{i-1} \cdot G/H + R$$
where the remainder $R$ consists of coset spaces
$G/K$ with $(K) \neq (H)$.
Equally, there is an isomorphism of
$\qq W_G H$-modules,
$\qq(W_G H^{\times_i}) \cong \bigoplus_{|W_G H|^{i-1}} \qq W_G H$.
\end{lemma}
\begin{proof}
The $G$-set $G/H^{\times_{i}}$ can only consist
of homogenous spaces $G/K$ for $K$ subconjugate
to $H$. Thus, to find the number of copies of
$G/H$ in $G/H^{\times_{i}}$ it suffices to calculate
the size of the $H$-fixed point set:
$|(G/H^{\times_{i}})^H| = |(G/H)^H|^i$. Since
$(G/H)^H = W_G H$, $|(G/H^{\times_{i}})^H|= |W_G H|^{i-1} |(G/H)^H|$
and the result follows. The statement
about $\qq W_G H$-modules is obvious.
\end{proof}

\begin{proposition}
For integers $i,j \geqslant 1$,
$$
\pi_* (F_{S_H} (G/H_+^{\smashprod_i} \smashprod S_H,
G/H_+^{\smashprod_j} \smashprod S_H )^G )
\cong \hom_{\qq W_G H} (\qq (W_G H^{\times_i}), \qq (W_G H^{\times_j})).$$
\end{proposition}
\begin{proof}
Using our understanding of the Burnside ring, we can write the
above term as
$$
\pi_* (F (R_+, G/H_+^{\smashprod_j} \smashprod S_H)^G )
\oplus
\left( \bigoplus_{|W_G H|^{i-1}}
\pi_* (F (G/H_+, G/H_+^{\smashprod_j} \smashprod S_H)^G ) \right)
$$
where $R$ is some wedge of spaces of form $G/K_+$ for
$(K) \neq (H)$.
We deal with the $R$-part first.
Consider $\pi_* (F (G/K_+, G/H_+^{\smashprod_j} \smashprod S_H)^G )$
for $(K) \neq (H)$, by arguments in the proof of
Proposition \ref{prop:finitecalc}
this is isomorphic to
$i_K^*(e_H) \pi_*^K( G/H_+^{\smashprod_j}) \otimes \qq$. Now
$i_K^*(e_H)=0$ whenever $(K) \neq (H)$, so this is zero.
Since all the terms of the $R$-part have this form,
$\pi_* (F (R_+, G/H_+^{\smashprod_j} \smashprod S_H)^G )=0$.
It remains to calculate
$e_H \pi_*^H( G/H_+^{\smashprod_j}) \otimes \qq$, which
by arguments in Proposition \ref{prop:finitecalc}
is isomorphic to
$\pi_*(\Phi^H G/H_+^{\smashprod_j}) \otimes \qq$
which is of course
$\pi_*( W_G H_+^{\smashprod_j}) \otimes \qq$.
In turn this is isomorphic to
$\qq ( W_G H^{\times_j})$ and the result follows immediately.
\end{proof}

\begin{proposition}\label{prop:homotopycalc}
There is an isomorphism of graded $\qq$-categories
$$\ecal_a^H \cong \h_*\ecal_t^H.$$
\end{proposition}
\begin{proof}
By Theorem \ref{thm:finiteEtopisEt} it suffices
to show that there is an isomorphism
$\ecal_a^H \cong \pi_*(\ecal_{top}^H)$
and we begin by proving
that the object sets of these categories are isomorphic.
In each case there
is an object $\sigma_1$ and a unit $\sigma_0$,
such that every non-unit object is
a product of copies of $\sigma_1$.
For $\ecal_a^H$ the unit is $\qq$
and $\sigma_1$ is $\qq W_G H$.
For $\ecal_{top}^H$ the unit is $S_H$
and $\sigma_1 = \cofrep(G/H_+) \smashprod S_H$.
Thus we define an isomorphism $\ob \ecal_a^H \to \ob \h_*\ecal_t^H$
by taking the $i$-fold product of $\qq W_G H$
(written $\sigma_i$) to the $i$-fold
product of $\cofrep(G/H_+) \smashprod S_H$.
We can consider these
graded $\qq$-categories to have the
object set: $\{ \sigma_i | i \geqslant 0 \}$.
The previous result implies that
$\pi_*(\ecal_{top}^H)(\sigma_i, \sigma_j)
\cong \ecal_{a}^H(\sigma_i, \sigma_j)$
as $\qq$-modules.
We must now show that this isomorphism is compatible
with the composition operation in these graded-$\qq$-categories.

We have the isomorphism
$\pi_*(\ecal_{top}^H)(\sigma_i, \sigma_j)
\cong [\sigma_i, \sigma_j]^G_*$
where the right hand side means graded maps in the homotopy category of
$S_H$-modules. This isomorphism specifies the composition rule of
the enriched category $\pi_*(\ecal_{top}^H)$.
Our calculations above allow us to write
$[\sigma_i, \sigma_j]^G_*$ as
$$\Big[ \bigvee_{|W_G H|^{i-1}} G/H_+ \smashprod S_H ,
\bigvee_{|W_G H|^{j-1}} G/H \smashprod S_H \Big]^G_*.$$
Then we define $(y,x,\widetilde{gH})$,
to be that map which takes the
$x$-factor of $\bigvee_{|W_G H|^{i-1}} G/H_+$
to the $y$-factor of $\bigvee_{|W_G H|^{j-1}} G/H_+$
by the rule $H \mapsto gH$.
This is a rational basis for $[\sigma_i, \sigma_j]^G_*$.
It is easy to check that composition behaves as follows:
$(z,y,\widetilde{g_2H}) \circ (y,x, \widetilde{g_1 H} ) = (z, x, \widetilde{g_1 g_2H})$.
Now we note that
$\hom_{\qq W_G H} (\qq (W_G H^{\times_i}), \qq (W_G H^{\times_j}))$
is isomorphic to
$$\hom_{\qq W_G H} \Big( \bigoplus_{|W_G H|^{i-1}} \qq W_G H,
\bigoplus_{|W_G H|^{j-1}} \qq W_G H \Big)$$
and write $(y,x,gH)$ for the map which takes
the $x$-factor of $\bigoplus_{|W_G H|^{i-1}} \qq W_G H$
to the $y$-factor of $\bigoplus_{|W_G H|^{j-1}} \qq W_G H$
by $H \mapsto gH$.
The isomorphism of the theorem is then just:
$(y,x,\widetilde{gH}) \mapsto (y,x,g H)$.
\end{proof}

\begin{theorem}\label{thm:finiteintrinsicformality}
If $\ecal$ is a $dg \qq$-category with
$\h_* \ecal$ concentrated in degree zero,
then $\ecal$ is quasi-isomorphic to $\h_* \ecal$
as $dg \qq$-categories.
\end{theorem}
\begin{proof}
We will create a $dg \qq$-category $C_0 \ecal$
and a zig-zag of quasi-isomorphisms:
$\ecal \overset{\sim}{\longleftarrow} C_0 \ecal
\overset{\sim}{\longrightarrow} \h_0 \ecal = \h_* \ecal.$
Let $C_0$\index{C@$C_0$} be the $(-1)$-connective cover functor,
which is right adjoint to the inclusion of $dg \qq\leftmod_+$
into $dg \qq \leftmod$.
If $X$ is a $dg \qq$-module, then $(C_0 X)_n =X_n$
for $n >0$ and $(C_0 X)_0 = ker(\partial_0)$.
We have a counit $C_0 X \to X$ in $dg \qq \leftmod$ and this is a monoidal
natural transformation. Hence, given $X \smashprod Y \to Z$,
we have a map $C_0 X \smashprod C_0 Y \to C_0 Z$
and a commuting diagram
$$\xymatrix{
C_0 X \smashprod C_0 Y \ar[r] \ar[d] & C_0 Z \ar[d] \\
X \smashprod Y \ar[r] & Z.
}$$
Thus we have a $dg \qq$-category
$C_0 \ecal$ with a map of ringoids
$C_0 \ecal \to \ecal$. Since $\ecal$ has homology concentrated
in degree zero this is a quasi-isomorphism.

For $X$ a $dg \qq$-module we have a map $C_0 X \to \h_0 X$ which sends
$X_i$ to zero for $i > 0$ and sends
$\ker(\partial_0) \to \h_0 X$ by the quotient.
We can consider $\h_0$ as a functor
$dg \qq \leftmod_+ \to \qq \leftmod$,
this has a right adjoint
which includes $\qq \leftmod$ into $dg \qq \leftmod_+$ by taking a
$\qq$-module $M$ to the chain complex
with $M$ in degree zero and zeroes elsewhere.
The map $C_0 X \to \h_0 X$ is induced by the
unit of this adjunction.
The functor $\h_0$ is monoidal, as is the inclusion of
$\qq \leftmod$ into $dg \qq \leftmod_+$, thus we obtain
a $dg \qq \leftmod_+$-category $\h_0 \ecal$.
Furthermore, the map $C_0 X \to \h_0 X$ is induced
by the unit of the adjunction and is a monoidal
natural transformation.
Thus we obtain $C_0 \ecal \to \h_0 \ecal= \h_* \ecal$,
which is a quasi-isomorphism.
\end{proof}

\begin{corollary}\label{cor:finiteETtoEA}
There is a zig-zag of quasi-isomorphisms of
$dg \qq$-categories.
$$\ecal_t^H \overset{\sim}{\longleftarrow} C_0 \ecal_t^H
\overset{\sim}{\longrightarrow} \h_* \ecal_t^H \cong \ecal_a^H$$
hence there is a zig-zag of Quillen equivalences
of $dg \qq \leftmod$-model categories.
$$\rightmod \ecal_t^H
\overleftarrow{\longrightarrow} \rightmod C_0 \ecal_t^H
\overrightarrow{\longleftarrow} \rightmod \h_* \ecal_t^H
\cong \rightmod \ecal_a^H.$$
\end{corollary}
\begin{proof}
This follows from Proposition \ref{prop:spectralmaps}.
\end{proof}

\begin{rmk}
With reference to Remark \ref{rmk:monoidalissue},
we note that if $\rightmod \ecal_t^H$ was a monoidal
category, then $\h_* \ecal_t^H$ can be shown to have the same
monoidal structure as $\ecal_a^H$.
It would follow that the comparison between
$\ecal_t^H$ and $\ecal_a^H$ would preserve the monoidal
product on these categories. Hence, the zig-zag of the above result
would be a zig-zag of strong monoidal equivalences
by Proposition \ref{prop:monoidalspectralfunctor}.
We would then be able to conclude
that the zig-zag between $S_H \leftmod$
and $dg \qq W_G H \leftmod$
would consist of symmetric monoidal equivalences.
\end{rmk}

\begin{corollary}\label{cor:finiteclassification}
If $G$ is a finite group, then
the model category of rational $G$-spectra
is Quillen equivalent to the algebraic model for rational
$G$-spectra:
$$dg \mathcal{A}(G) =
\prod_{(H) \leqslant G} dg\qq W_G H \leftmod .$$
\end{corollary}
\begin{proof}
We begin with Theorem \ref{thm:finitesplitting},
which splits rational $G$-spectra into the product
$\prod_{(H) \leqslant G} L_{E \langle H \rangle} G \mcal_\qq$.
Applying Proposition \ref{prop:GspecHtoSHmod}
to each factor of this category allows us to move to
$\prod_{(H) \leqslant G} S_H \leftmod$.
Next we use Theorem \ref{thm:finitemoritaequiv}
to move to modules over a spectral category,
$\prod_{(H) \leqslant G} \rightmod \ecal_{top}^H$.
We move to algebra ($\prod_{(H) \leqslant G} \rightmod \ecal_{t}^H$)
with Theorem \ref{thm:finiteEtopisEt},
and then use Corollary \ref{cor:finiteETtoEA}
to get to the category
$\prod_{(H) \leqslant G} \rightmod \ecal_a^H$.
Finally we use Proposition \ref{prop:FiniteAlgMorita}
to complete the result.
\end{proof}

\begin{rmk}
We can relate our work to \cite[Appendix A]{gremay95}
(specifically A.15 and Theorem A.16) as follows:
$$[X,Y]_\qq^{G|H} \cong [X,\fibrep_H Y]_\qq^{G} \cong [X,e_H Y]_\qq^{G}$$
where the first entry is maps in $\ho L_{E \langle H \rangle} G \mcal_\qq$
and $\fibrep_H$ is fibrant replacement in this model category.
The $G$-spectra $\fibrep_H Y$ and $e_H Y$ are rationally equivalent which
gives us the second isomorphism in the above.
Hence we can use the results of that paper to write the following,
where $[i^* \Phi^H X,i^* \Phi^H  Y]_\qq$ is the collection
of rational homotopy classes of non-equivariant maps of naive $W_G H$-spectra,
which is a $\qq W_G H$-module.
$$[X,Y]_\qq^{G} \cong \bigoplus_{(H)}[X,e_H Y]_\qq^{G}
\cong \bigoplus_{(H)} \{ [i^* \Phi^H X,i^* \Phi^H  Y] \}^{W_G H} $$
\end{rmk}

\chapter{Enriched categories}\label{chp:enrichcat}

The methods of \cite{greshi} rely heavily on
highly structured model categories.
We will need to use enrichments, tensorings,
cotensorings and algebra structures, so we
go through the definitions and basic constructions here.
This is largely a service chapter where we
introduce the language needed for the comparisons
of Section \ref{sec:finitecomp} and
Chapter \ref{chp:WskewSQmod}.
The first three sections are mainly providing definitions
and results from \cite{kell05}, \cite[Chapter 4]{hov99}
and \cite{ss03stabmodcat} respectively.
The final section is the Morita equivalence:
that one can replace a category by modules
over an endomorphism ringoid.
In \cite{greshi} they state that one can do so
in a monoidal fashion, we have given full details
of this result in Theorems \ref{thm:monoidalmorita}
and \ref{thm:monmodules}. Thus our work
in this chapter is mostly in relating the various
definitions and giving more details on monoidal
considerations.
As well as the above-named references
we will also make use of \cite{bor94}, from which
we take some more technical results on enriched categories.
For the model category considerations we use
\cite{dugshi} and \cite{dug06}, which overlap somewhat.

\section{$\nu$-Categories}\label{sec:nucats}
Throughout $\nu$ will be a monoidal category (symmetric when necessary),
the unit of $\nu$ will be $\mathcal{I}$ and the product will be $\otimes$.
When needed, we will let $\nu$ have an internal function object $\underhom$,
so that $\nu$ will be a closed monoidal category (see \cite[Section 1.5]{kell05}).

\begin{definition}
A \textbf{$\nu$-category}\index{nu-category@$\nu$-category} $\mathcal{A}$ is a class of objects with a
\textbf{Hom-object}\index{Hom-object} $\mathcal{A}(A,B)$ in $\nu$ for each pair of objects.
For each triple of objects there is a \textbf{composition law}\index{Composition law}:
$M \co \mathcal{A}(B,C) \otimes \mathcal{A}(A,B) \to \mathcal{A}(A,C)$, a morphism in $\nu$.
For each object there is an \textbf{identity element}\index{Identity element}:
$j \co \mathcal{I} \to \mathcal{A}(A,A)$, a morphism in $\nu$.
These must satisfy the usual five-fold associativity diagram and
a pair of triangles describing the identity elements.
\end{definition}

\begin{definition}
A \textbf{$\nu$-functor}\index{nu-functor@$\nu$-functor} $T \co \mathcal{A} \to \mathcal{B}$ of
$\nu$-categories is a functor $T \co \ob \mathcal{A} \to \ob \mathcal{B}$
with a map in $\nu$:
$T=T_{A,B} \co \mathcal{A}(A,B) \to \mathcal{B}(TA,TB)$
for each pair of objects of $\mathcal{A}$.
These maps must satisfy the relations $TM = M (T\otimes T)$
and $Tj =j$.
\end{definition}

When $\nu$ is the category of abelian groups, a $\nu$-category
is also known as a \textbf{ring with many objects}\index{Ring with many objects}.
A $\nu$-category $\mathcal{A}$ with one object $a$, is just a ring:
composition gives multiplication on the abelian group $\mathcal{A}(a,a)$.
Hence we will also refer to enriched categories as
\textbf{ringoids}\index{Ringoid} and a $\nu$-functor is then a
map of ringoids.

\begin{definition}
For $\nu$-functors $T,S \co \mathcal{A} \to \mathcal{B}$,
a \textbf{$\nu$-natural transformation}\index{nu-natural
transformation@$\nu$-natural transformation} $\alpha \co T \to S$
is an $\ob \mathcal{A}$-indexed family of \textbf{components}\index{Components}
$\alpha_A \co \mathcal{I} \to \mathcal{B}(TA,SA)$ satisfying
the naturality condition (in $\nu$)
$M(\alpha_B \otimes T)l^{-1} = M(S \otimes \alpha_A )r^{-1}$
where $l$ and $r$ are the left and right unit isomorphisms of
$\nu$.
\end{definition}

The composite of $\beta$ and $\alpha$ has components
$M(\beta_A \otimes \alpha_A)$. The composite $Q \alpha$
has components $Q \alpha_A$, and $\alpha P$
has components $(\alpha P)_D= \alpha_{PD}$.

\begin{definition}
A \textbf{$\nu$-adjunction}\index{nu-adjunction@$\nu$-adjunction}
of $\nu$-categories is an adjunction
$(F,G)$ consisting of $\nu$-functors
together with an isomorphism in $\nu$:
$\mathcal{B}(FA, B) \cong \mathcal{A}(A,GB)$.
\end{definition}

The above material was taken from
\cite[Section 1.2]{kell05}.

\begin{definition}
If $\nu$ is a symmetric monoidal category and
$\mathcal{A}$ is a $\nu$-category then
$\mathcal{A} \times \mathcal{A}$ is a $\nu$-category
with
$$\mathcal{A} \times \mathcal{A}\big( (a,b),(c,d) \big)
:= \mathcal{A}(a,c) \otimes \mathcal{A}(b,d)$$
and composition defined using the
symmetry of $\nu$ as follows
$$\xymatrix{
{\big({\mathcal{A}(b,c)} \otimes {\mathcal{A}(y,z)}\big)} {\otimes}
{\big({\mathcal{A}(a,b)} \otimes {\mathcal{A}(x,y)} \big)}
\ar[d]^{\id \otimes T \otimes \id} \\
{\big({\mathcal{A}(b,c)} \otimes {\mathcal{A}(a,b)}\big)} {\otimes}
{\big({\mathcal{A}(y,z)} \otimes {\mathcal{A}(x,y)} \big)}
\ar[d] \\
{\mathcal{A}(a,c)} \otimes {\mathcal{A}(x,z)} .}
$$
\end{definition}

The following is taken from \cite[Page 2]{day70}.
\begin{definition}
Let $\nu$ be a symmetric monoidal category.
A \textbf{monoidal $\nu$-category}\index{Monoidal nu-category@Monoidal $\nu$-category}
is a $\nu$-category $\mathcal{A}$ with a $\nu$-functor (the monoidal product)
$\smashprod \co \mathcal{A} \times \mathcal{A} \to \mathcal{A}$,
with a unit object $\sphspec \in \ob \mathcal{A}$ and
$\nu$-natural isomorphisms for associativity and the unit.
If there is a $\nu$-natural isomorphism
between the $\nu$-functors $\smashprod$ and $\smashprod \circ T$
where $T$ interchanges the factors of $\mathcal{A} \times \mathcal{A}$
then $\mathcal{A}$ is called \textbf{symmetric}.
A \textbf{monoidal $\nu$-functor}\index{Monoidal $\nu$-functor}
is a $\nu$-functor that preserves the monoidal structure.
That is, $\Psi \co \mathcal{A} \to \mathcal{B}$,
a $\nu$-functor, is
\textbf{monoidal}\index{Monoidal nu-functor@Monoidal $\nu$-functor}  if
$\smashprod_\mathcal{B} \circ (\Psi \times \Psi) \cong
\Psi \circ \smashprod_\mathcal{B}$ as $\nu$-functors.
If $\mathcal{A}$ and $\mathcal{B}$ are symmetric then we can
also require that $\Psi$ respects the symmetry isomorphisms.
\end{definition}
We spell out part of this definition, for any $a$, $b$, $c$ and $d$ in $\mathcal{A}$
we have a map
$$(\mathcal{A} \times \mathcal{A}) \big( (a,b),(c,d) \big)
:= \mathcal{A}(a,c) \otimes \mathcal{A}(b,d)
\longrightarrow \mathcal{A}(a \smashprod b, c \smashprod d  )$$
that satisfies various unital and associativity diagrams.
The symmetry of $\nu$ is used to relate the composition in
$\mathcal{A} \times \mathcal{A}$ to the composition in
$\mathcal{A}$.
We can now use the construction of \cite{day70},
as taken from \cite[Definition 3.5 and A.2]{greshi}.
\begin{definition}
Assume that $\mathcal{A}$ is a symmetric monoidal $\nu$-category.
Then we can define a \textbf{box product}\index{Box product} of
$\nu$-functors $T,S \co \mathcal{A} \to \nu$.
$$T \Box_\mathcal{A} S (a) =
\int^{x,y} \mathcal{A}(x \smashprod y, a ) \otimes T(x) \otimes S(y)$$
\end{definition}
The relevant coequaliser defining the above coend
uses the following maps, first we have the `action' on $T$ and $S$.
$$\mathcal{A}(u \smashprod v, a ) \otimes
\mathcal{A}(x , u ) \otimes
\mathcal{A}(y , v ) \otimes
T(x) \otimes S(y)
\to
\mathcal{A}(u \smashprod v, a ) \otimes T(u) \otimes S(v). $$
The second uses the monoidal structure on $\mathcal{A}$:
$\mathcal{A}(x , u ) \otimes \mathcal{A}(y , v ) \to
\mathcal{A}(x \smashprod y, u \smashprod v )$,
followed by composition.
There is an external product on the category
of $\nu$-functors $\mathcal{A} \to \nu$.
Take two such functors $T$ and $S$, then
$T \bar{\smashprod} S \co \mathcal{A} \times \mathcal{A} \to \nu$
is a $\nu$-functor with
$(T \bar{\smashprod} S)(a,b) = T(a) \otimes S(b)$.
We can describe the box product as
the left Kan extension of $\otimes$
along $\smashprod$, the monoidal product of $\mathcal{A}$.

\begin{proposition}\label{prop:MonFunctorsEnrichCat}
Let $F \co \nu \to \mu$ be a symmetric monoidal
functor and $\ccal$ be a symmetric monoidal
$\nu$-category. Then $F \ccal$ is a symmetric monoidal
$\mu$-category.
\end{proposition}
\begin{proof}
We define $F \ccal$ to have the same class of objects as
$\ccal$. On maps we define $(F \ccal)(a,b)= F( \ccal(a,b))$.
The monoidality of $F$ gives us the composition rule
and identity elements. The monoidal structure
is created from the monoidal structure on $\ccal$,
but to prove that the monoidal product is a
$\mu$-functor we must use the symmetry of $F$.
The required unital and associativity
$\mu$-natural isomorphisms of the monoidal product then follow from those
in $\ccal$.
This result is an application of \cite[Proposition A.3(b)]{dugshi}
with monoidal structures considered.
\end{proof}

The need for symmetry in the above will cause us difficulty
in Chapter \ref{chp:WskewSQmod} when we create an enriched category
$D (\phi^* N \tilde{\qq} \ecal_{top})$ by applying the functor
$D$ to the enriched category $\phi^* N \tilde{\qq} \ecal_{top}$.

We now give a few results which give us some rules 
we can use when working with coends or ends.
Note that the Yoneda lemma states that the end
$\int_c \underhom(\ccal(x,c) , F(c)) $
exists, whereas
Lemma \ref{lem:endandcoend} and Corollary \ref{cor:endandcoendeval}
assume that the relevant coend exists. Of course, we will be
working in the context of model categories, whence 
the ends and coends we consider below will always 
exist.

\begin{lemma}[Enriched Yoneda]\label{lem:enrichyoneda}
Take $\nu$ to be a symmetric closed monoidal category and
let $\ccal$ be a small $\nu$-category. Then for any $x \in \ccal$
and any $\nu$-functor $F \co \ccal \to \nu$
the $\nu$-object $\nu \nat(\ccal(x,-), F):= \int_c \underhom(\ccal(x,c) , F(c)) $
exists and is naturally isomorphic to $F(x)$.
\end{lemma}
\begin{proof}
This lemma is \cite[Proposition 6.3.5]{bor94}.
\end{proof}

\begin{lemma}\label{lem:endandcoend}
Consider a $\nu$-functor $F \co \ccal \to \nu$ between
$\nu$-categories (with $\ccal$ small) and let $x \in \ccal$, then whenever
the following coend exists,
there is an isomorphism (natural in $F$ and $x$) in $\nu$:
$$\int^{c \in \ccal} F(c) \otimes \ccal(c,x) \overset{\cong}{\longrightarrow} F(x) .$$
\end{lemma}
\begin{proof}
We write the coend as a coequaliser
of coproducts
$$\coprod_{c,d} F(c) \otimes \ccal(c,d) \otimes \ccal(d,x)
\overrightarrow{\longrightarrow}
\coprod_{e} F(e) \otimes \ccal(e,x).$$
One arrow corresponds to composition
and the other to the action of $\ccal(c,d)$ on $F(c)$
$$F(c) \otimes \ccal(c,d) \to F(c) \otimes \nu (F(c), F(d)) \to F(d).$$
The action map then induces a map in $\nu$:
$\coprod_{c} F(c) \otimes \ccal(c,x) \to F(x)$, that this coequalises
the two maps follows from the definition of a $\nu$-functor.
Given any test object $T$ for this coequaliser
we can use factor $x$ of the given map
$\alpha \co \coprod_{c} F(c) \otimes \ccal(c,x) \to T$
and the composite
$F(x) \cong F(x) \otimes \mathcal{I} \to F(x) \otimes \ccal(x,x)$
to give a map $F(x) \to T$.
It follows easily that this map is unique and satisfies the
required commutativity condition. Hence $F(x)$ is
isomorphic to the coend.
\end{proof}

\begin{corollary}\label{cor:endandcoendeval}
Consider a $\nu$-functor $F \co \ccal^{op} \to \nu$ between
$\nu$-categories (with $\ccal$ small) and let $x \in \ccal$, then whenever
the following coend exists,
there is an isomorphism (natural in $F$ and $x$) in $\nu$:
$$\int^{c \in \ccal} F(c) \otimes \ccal(x,c)
\overset{\cong}{\longrightarrow} F(x). $$
\end{corollary}

\section{$\ccal$-Modules}\label{sec:Cmodules}
We now wish to move from enrichments
to tensorings of categories (and then to tensorings, cotensorings
and enrichments all at the same time).
To avoid confusion with the previous information
and to match the notation of \cite[Chapter 4]{hov99}
we let $\ccal$ be a symmetric monoidal category
(\cite[Definition 4.1.4]{hov99})
with product $\otimes$ and unit $1_\ccal$.
\begin{definition}
A \textbf{$\ccal$-module}\index{c-module@$\ccal$-module} is a category $\dcal$ with a
functor $\ccal \times \dcal \to \dcal$.
The symmetry of $\ccal$ allows us to
use left-modules and right-modules interchangeably.
A \textbf{$\ccal$-module functor}\index{c-module functor@$\ccal$-module functor}
between $\ccal$-modules
$\dcal$ and $\ecal$ is a functor $F \co \dcal \to \ecal$
with a natural isomorphism $c \otimes Fd \cong F(c \otimes d)$
satisfying the following pair of coherence conditions.
Let $c$ and $c'$ be objects of $\ccal$ and $d$ an object of $\dcal$.
The first condition is that the two ways of
getting from $L( c \otimes (c' \otimes d))$ to
$c \otimes (c' \otimes Ld)$ must be equal.
The second is that the two ways to get from
$L(d \otimes 1_\ccal)$ to $Ld$ must agree.

A \textbf{(symmetric) $\ccal$-algebra}\index{c-algebra@$\ccal$-algebra}
is a category $\dcal$ with a
(symmetric) monoidal structure and a strong (symmetric) monoidal functor
$i_\dcal \co \ccal \to \dcal$.
A \textbf{(symmetric) $\ccal$-algebra functor}\index{c-algebra
functor@$\ccal$-algebra functor} between $\ccal$-algebras
$\dcal$ and $\ecal$ is a strong (symmetric) monoidal functor
$F \co \dcal \to \ecal$
with a monoidal natural isomorphism $F \circ i_\dcal \to i_\ecal$.
\end{definition}

This definition is an abbreviation of
\cite[Definitions 4.1.6 -- 4.1.9]{hov99}.
We now want to add still more structure: to consider categories
enriched, tensored and cotensored over another category.
We let $\ccal$ be a closed (symmetric) monoidal category \cite[4.1.13]{hov99}
with product $\otimes$ and unit $1_\ccal$, the internal function object of $\ccal$ will be
$\F_\ccal$ (recall that the adjective closed means that we have an adjunction
of two variables rather than just a bifunctor).

\begin{definition}
A \textbf{closed $\ccal$-module}\index{closed c-module@Closed $\ccal$-module}
is a category $\dcal$ with an
adjunction of two variables
$\ccal \times \dcal \to \dcal$.
Thus we have three bifunctors
\begin{displaymath}
\begin{array}{crcl}
\otimes    : & \ccal \times \dcal & \longrightarrow & \dcal \\
\underhom : & \ccal^{op} \times \dcal & \longrightarrow & \dcal \\
\underhom_\ccal : & \dcal^{op} \times \dcal & \longrightarrow & \ccal. \\
\end{array}
\end{displaymath}
In \cite{hov99} the notation $\hom_r$ is used for $\underhom_\ccal$
and $\hom_l$ for $\underhom$. We also have the adjunctions
$\ccal (c, \underhom_\ccal(d,d')) \cong
\dcal(c \otimes d, d') \cong
\dcal(d, \underhom(c,d'))$.
We call $\otimes$ the \textbf{tensor}\index{Tensor}
operation, $\underhom$ the \textbf{cotensor}\index{Cotensor}
and $\underhom_\ccal$ the \textbf{enrichment}\index{Enrichment}.
\end{definition}
Thus a closed $\ccal$-module is a $\ccal$-category,
with the enrichment given by $\underhom_\ccal$.
\begin{lemma}
Let $\dcal$ be a closed $\ccal$-module, then for any $d,d'$ in $\dcal$ and $c$ in $\ccal$,
there are canonical isomorphisms in $\ccal$
$$\F_\ccal(c, \underhom_\ccal(d,d')) \cong
\underhom_\ccal(d \otimes c, d') \cong
\underhom_\ccal(d, \underhom(c,d')) $$
which, after applying $\ccal(1_\ccal, -)$, reduce to the isomorphisms
of the definition above.
\end{lemma}
\begin{proof}
This is \cite[Lemma A.2]{dug06}.
\end{proof}

\begin{lemma}
Take an adjunction $L : \mathcal{D} \leftrightarrows \mathcal{E} : R$
between closed $\ccal$-modules, then the following
statements are equivalent:
\begin{enumerate}
\item There are natural isomorphisms in $\ecal$,
$c \otimes Ld \cong L(c \otimes d)$,
which reduce to the canonical isomorphisms for $c =1_\ccal$.
\item There are natural isomorphisms in $\dcal$,
$\underhom(c, Re) \cong R \underhom(c,e)$
which reduce to the canonical isomorphisms for $c =1_\ccal$.
\item There are natural isomorphisms in $\ccal$,
$\underhom_\mathcal{C}(d,Re) \cong \underhom_\mathcal{C}(Ld,e)$
which after applying $\ccal(1_\ccal,-)$
reduce to the adjunction between $L$ and $R$.
\end{enumerate}
\end{lemma}
\begin{proof}
This is \cite[Lemma A.7]{dug06} and we give a proof
of this simple result.
Take any $c,d,e$ in $\ccal,\dcal,\ecal$
respectively. Then
\begin{displaymath}
\begin{array}{rcccl}
\ecal(c \otimes Ld,e) & \cong & \ecal(Ld,\underhom(c,e))
 & \cong & \dcal(d,R\underhom(c,e)) \\
\ecal(L(c \otimes d),e) & \cong & \dcal(c \otimes d, Re)
 & \cong & \dcal(d,\underhom(c,Re)) \\
\end{array}
\end{displaymath}
so since an object in a category is determined (up to natural isomorphism)
by the maps out of (or the maps into) that object,
we can see that the first two conditions are equivalent.
The next two collections of isomorphisms show (respectively) that the first
and last conditions are equivalent and that the second and
third are equivalent.
\begin{displaymath}
\begin{array}{rcccl}
\ccal (c,\underhom_\ccal (d,Re)) & \cong & \dcal (c \otimes d, Re)
 & \cong & \ecal(L (c \otimes d),e) \\
\ccal (c,\underhom_\ccal (Ld,e)) & \cong & \ecal (c \otimes Ld, e) \\
\\
\ccal (c,\underhom_\ccal (d,Re)) & \cong & \dcal (c \otimes d, Re)
 & \cong & \dcal (d,\underhom (c,Re)) \\
\ccal (c,\underhom_\ccal (Ld,e)) & \cong & \ecal (Ld, \underhom (c, e))
& \cong & \dcal ( d , R\underhom (c, e) )
\end{array}
\end{displaymath}
The statement about reductions are similarly routine.
\end{proof}

\begin{definition}
An \textbf{adjunction of closed $\ccal$-modules}\index{Adjunction
of closed c-module@Adjunction of closed $\ccal$-modules} is an adjoint
pair such that the left adjoint is a $\ccal$-module functor.
A \textbf{closed (symmetric) $\ccal$-algebra}\index{closed c-algebra@Closed
$\ccal$-algebra} is a category $\dcal$ with a
closed (symmetric) monoidal structure and a strong (symmetric) monoidal adjunction
$i_\dcal : \ccal \leftrightarrows \dcal : j_\ccal$ \cite[Definition 4.1.14]{hov99}.
An \textbf{adjunction of closed (symmetric) $\ccal$-algebras}\index{adjunction
of closed c-algebra@Adjunction of $\ccal$-algebras} between $\ccal$-algebras
$\dcal$ and $\ecal$ is a strong (symmetric) monoidal adjunction
$L : \dcal \leftrightarrows \ecal : R$
with a monoidal natural isomorphism $L \circ i_\dcal \to i_\ecal$
(so $L$ is a $\ccal$-algebra functor).
\end{definition}

\begin{lemma}
Take a strong monoidal adjunction
$L : \ccal \leftrightarrows \dcal : R$.
Then a $\dcal$-module category $M$ can be given the structure
of a $\ccal$-module category by setting
$$X \otimes c = X \otimes Lc,
\quad
F_\ccal(c,Y)=F(Lc,Y)
\quad and \quad
\underhom_\ccal (X,Y) = R \underhom_\dcal (X,Y)$$
note that the direction of the adjunction is essential to this lemma.
\end{lemma}
\begin{proof}
This is \cite[Lemma A.5]{dug06}.
\end{proof}

\begin{corollary}
If there is a strong (symmetric) monoidal adjunction
$L : \ccal \overrightarrow{\longleftarrow} \dcal : R$
between closed (symmetric) monoidal categories
then $\ccal$ and $\dcal$ are closed (symmetric) $\ccal$-algebras
and $(L,R)$ is an adjunction of closed (symmetric) $\ccal$-algebras.
\end{corollary}

We now link these structures back to the definitions of \cite{kell05}.
\begin{proposition}\label{prop:adjunctenrich}
If $L : \dcal \overrightarrow{\longleftarrow} \ecal : R$
is an adjunction of closed $\ccal$-modules then both
$L$ and $R$ are $\ccal$-functors.
\end{proposition}
\begin{proof}
Showing this for the right adjoint takes a rather large diagram chase,
see \cite[Proposition A.9]{dug06} for a proof of the statement.
That the left adjoint is a $\ccal$-functor follows from the statement for
the right adjoint and \cite[Proposition 6.7.2]{bor94}.
For $d \in \dcal$ there is a unit map $d \to RLd$ in $\dcal$
and for $e \in \ecal$ there is a counit map $LRe \to e$ in $\ecal$.
The $\ccal$-structures on $L$ and $R$ arise from these maps as below.
$$\underhom_\ccal(d',d) \to \underhom_\ccal(d',RLd) \cong \underhom_\ccal(Ld',Ld)$$
$$\underhom_\ccal(e,e') \to \underhom_\ccal(LRe,e') \cong \underhom_\ccal(Re,Re')$$
\end{proof}

\begin{rmk}
Let $\ccal$ be a (symmetric) monoidal category,
then we can give $\ccal$
the structure of a closed $\ccal$-module or
a closed (symmetric) $\ccal$-algebra.
If $\dcal$ is a closed (symmetric) $\ccal$-algebra, then it is
certainly a closed $\ccal$-module.
So now we can ask: when is a closed $\ccal$-module $\dcal$
a closed (symmetric) $\ccal$-algebra?
Obviously the primary requirement is that $\dcal$
should be a (symmetric) monoidal category,
but we also need an associativity relation:
$c \otimes (d \smashprod d') \cong (c \otimes d) \smashprod d'$
satisfying the appropriate coherence diagrams.
Then we have a strong monoidal (symmetric) adjunction
$$(-) \otimes 1_\dcal : \ccal \overrightarrow{\longleftarrow}
\dcal : \underhom_\ccal(1_\dcal, -)$$
making $\dcal$ into a closed (symmetric)
$\ccal$-algebra.
This is precisely analogous to the case of modules and
algebras over a ring. When considering model structures
the only additional condition is to require that
the adjunction is a Quillen pair, which is automatic when
$1_\dcal$ is cofibrant.
\end{rmk}

\section{Modules over an Enriched Category}\label{sec:Emodules}
This section gives the language needed to
compare our categories of $G$-spectra to
algebra, that is, it allows us to make
use of the results of \cite{shiHZ}.
We show how to move from spectra to algebra in
Section \ref{sec:movingaccross the machine}.
To help the notation we fix our `base category' to
be the category of symmetric spectra, thus we
state all definitions and results in terms
enrichments over symmetric spectra.
It should be clear how to replace the category of symmetric spectra
with any other closed symmetric monoidal model category
(which satisfies the monoid axiom) in the following.
In particular we will later use $Sp^\Sigma_+$,
$Sp^\Sigma(\sqq \leftmod)$, $Sp^\Sigma(dg \qq \leftmod_+)$
and $dg \qq \leftmod$ in the place of symmetric spectra.
Note that to prove Theorem \ref{thm:monmodules}
(the main result of this section)
we must either assume the unit is cofibrant or assume that
smashing with a cofibrant object
preserves weak equivalences.

We now give \cite[Definition 3.5.1]{ss03stabmodcat},
which takes the notion of
a closed $Sp^\Sigma$-module and adds model structure conditions.

\begin{definition}
A \textbf{spectral model category}\index{Spectral model category}
is a model category $\ccal$
which is a closed $Sp^\Sigma$-module and the action map
$\otimes \co Sp^\Sigma \times \ccal \to \ccal$ is a Quillen bifunctor
such that $\cofrep \sphspec \otimes X \to \sphspec \otimes X$
is a weak equivalence for all cofibrant $X \in M$.
A \textbf{spectral left Quillen functor}\index{Spectral left Quillen functor}
is a left Quillen functor that is also a $Sp^\Sigma$-module
functor.
A \textbf{spectral right Quillen functor}\index{Spectral right Quillen functor}
is a right Quillen functor $R \co \ccal \to \dcal$
between spectral model categories that preserves
(up to natural isomorphism in $\dcal$)
the cotensor operation:
$\underhom (K, RX) \cong R \underhom (K,X)$
and this isomorphism must be coherent in the sense
that the two maps from
$\underhom (L, \underhom (K, RX))$ to
$R\underhom (L, \underhom (K, X))$
must be equal.
\end{definition}

A spectral model category is a $Sp^\Sigma$-model category
in the sense of \cite[Definition 4.2.18]{hov99}.
As (briefly) mentioned after \cite[Definition 4.1.12]{hov99}
the correct notion of a `map of spectral model categories'
is a spectral adjunction of categories as defined in
\cite[Definition 3.9.2]{ss03stabmodcat}
which we give below.

\begin{definition}
Let $L : \ccal \overrightarrow{\longleftarrow} \dcal : R$
be an adjoint pair between spectral model categories $\ccal$ and $\dcal$.
A \textbf{spectral adjunction}\index{Spectral adjunction} is an adjunction of
closed $Sp^\Sigma$-modules.
We call such a pair a \textbf{spectral Quillen pair}\index{Spectral Quillen pair}
if the functors $(L,R)$ are also a Quillen pair.
A \textbf{spectral Quillen equivalence}\index{Spectral Quillen equivalence}
is a spectral Quillen pair that is also 
a Quillen equivalence.

A \textbf{closed (symmetric) monoidal spectral model
category} is a closed (symmetric) $Sp^\Sigma$-algebra
such that the adjunction
$Sp^\Sigma \ \overrightarrow{\longleftarrow} \ \dcal$
is a Quillen pair.
An \textbf{adjunction of closed (symmetric) monoidal spectral model
categories} is an adjunction of
closed (symmetric) $Sp^\Sigma$-algebras
that is a Quillen pair.
\end{definition}
Thus, a spectral adjunction is an adjunction with a
natural isomorphism of symmetric spectra
$\underhom_\mathcal{C}(A,RX) \cong \underhom_\mathcal{D}(LA,X)$.
Furthermore, $L$ is a spectral left Quillen functor and $R$
is spectral right Quillen functor.
We introduce some language for categories enriched over symmetric spectra
which don't necessarily have model structures or aren't
tensored or cotensored over symmetric spectra.

\begin{definition}\label{def:ocalmod}
A \textbf{spectral category}\index{Spectral category}
(\cite[Definition 3.3.1]{ss03stabmodcat})
is a category $\mathcal{O}$
enriched over the category of symmetric spectra $Sp^\Sigma$,
i.e. $\mathcal{O}$ is an $Sp^\Sigma$-category.
A \textbf{spectral functor}\index{Spectral functor} is an $Sp^\Sigma$-functor and
a \textbf{spectral adjunction}\index{Spectral adjunction} is an $Sp^\Sigma$-adjunction.
A (right) \textbf{$\mathcal{O}$-module}\index{O-module@$\mathcal{O}$-module}
is a contravariant spectral functor
$\mathcal{O} \to Sp^\Sigma$, these modules form a category
denoted $\rightmod \mathcal{O}$.
A \textbf{morphism of $\mathcal{O}$-modules}\index{Morphism of
O-module@Morphism of $\mathcal{O}$-modules} is
an $Sp^\Sigma$-natural transformation.
\end{definition}

We spell out the above requirements for an $\mathcal{O}$-module $M$.
For every object of $\mathcal{O}$ we have a symmetric spectrum $M(o)$ and we must
have coherently unital and associative maps of symmetric spectra
$M(o) \smashprod \mathcal{O} (o',o) \to M(o')$
for pairs of objects $o,o'$ in $\mathcal{O}$. A morphism of such modules
is a collection of maps of symmetric spectra $M(o) \to N(o)$ which are
strictly compatible with the above action.

\begin{definition}
For each $o \in \ocal$ we have a \textbf{free module}\index{Free module}
(\cite[Definition 3.3.1]{ss03stabmodcat}),
$F_o$. This is an $\ocal$-module defined by
$F_o(o')=\mathcal{O}(o',o)$ with $\ocal$-action given by composition.
\end{definition}

\begin{rmk}
The definitions above fit neatly into the framework of \cite[Section 6.2]{bor94},
If $\mathcal{O}$ is an $Sp^\Sigma$-category then so is
$\mathcal{O}^{op}$, \cite[Proposition 6.2.2]{bor94}.
Then we have the following definitions from \cite[Proposition 6.3.1]{bor94}:
$$\begin{array}{rccl}
\text{Category of left $\mathcal{O}$-modules}, &
\mathcal{O} \leftmod  &=& Sp^\Sigma[\mathcal{O}, Sp^\Sigma] \\
\text{Category of right $\mathcal{O}$-modules}, &
\rightmod \mathcal{O} &=& Sp^\Sigma[\mathcal{O}^{op}, Sp^\Sigma]
\end{array}$$
the categories of covariant and contravariant
spectral functors from $\mathcal{O}$ to $Sp^\Sigma$.
Morphisms in these categories are the $Sp^\Sigma$-natural
transformations.
\end{rmk}

\begin{theorem}\label{thm:modOmodelcat}
Let $\mathcal{O}$ be a spectral category, then the category of
$\mathcal{O}$-modules with object-wise weak equivalences
of $Sp^\Sigma$, object-wise fibrations of $Sp^\Sigma$
and cofibrations as necessary gives a
cofibrantly generated spectral model structure.
The collection of free modules give a set of compact generators.
\end{theorem}
\begin{proof}
This is part of \cite[Theorem A.1.1]{ss03stabmodcat}.
\end{proof}

We take the definition below from
\cite[Section A.1]{ss03stabmodcat}.
\begin{definition}\label{def:indres modules}
For a spectral functor of spectral
categories $\Psi \co \mathcal{O} \longrightarrow \mathcal{R}$, we have
a \textbf{restriction of scalars}\index{Restriction of scalars} functor
$$\Psi^* \co \rightmod \mathcal{R} \longrightarrow
\rightmod \mathcal{O}, \quad M \mapsto M \circ \Psi.$$
This has a left adjoint $(-) \smashprod_\mathcal{O} \mathcal{R}$, the
\textbf{extension of scalars}\index{Extension of scalars} functor. It is defined as an enriched coend
and is similar to many other constructions in this work,
$$ N \smashprod_\mathcal{O} \mathcal{R} =\int^o N(o) \smashprod F_{\Psi(o)} $$
we give a construction below as a
coequaliser of a pair of $\mathcal{R}$-module homomorphisms.
$$\bigvee_{o,p \in \mathcal{O}}
N(p) \smashprod \mathcal{O}(o,p) \smashprod F_{\Psi(o)}
\overrightarrow{\longrightarrow}
\bigvee_{q \in \mathcal{O}}
N(q) \smashprod F_{\Psi(q)}$$
We call $\Psi$ a \textbf{stable equivalence}\index{Stable equivalence}
if each $\Psi_{o,p} \co \mathcal{O}(o,p) \to \mathcal{R}(o,p)$
is a weak equivalence of symmetric spectra.
\end{definition}

\begin{proposition}\label{prop:spectralmaps}
For a spectral functor of spectral
categories $\Psi \co \mathcal{O} \longrightarrow \mathcal{R}$
there is a Quillen pair
$$(-) \smashprod_\mathcal{O} \mathcal{R} :
\rightmod \ocal \overrightarrow{\longleftarrow}
\rightmod \mathcal{R} : \Psi^* .$$
Furthermore if $\Psi$ is a stable equivalence
this adjoint pair is a Quillen equivalence.
\end{proposition}
\begin{proof}
This is part of \cite[Theorem A.1.1]{ss03stabmodcat}.
\end{proof}

We now add an extra condition on $\ocal$ to ensure that
$\rightmod \ocal$ is in fact a closed monoidal model category.
Note that we will use the fact that
in $Sp^\Sigma$ (or $Sp^\Sigma_+$) smashing with a
cofibrant object preserves weak equivalences.
\begin{theorem}\label{thm:monmodules}
If the category $\mathcal{O}$ is a symmetric monoidal $Sp^\Sigma$-category
then the category $\rightmod \mathcal{O}$ is a closed symmetric monoidal
model category satisfying the monoid axiom.
\end{theorem}
\begin{proof}
This is \cite[Proposition 3.7]{greshi},
adjusted to the setting of an $Sp^\Sigma$-category,
with full details given.
The monoidal product is the box product
$$M \Box N (o) =
\int^{p,q} \mathcal{O}(o, p \smashprod q,) \otimes M(p) \otimes N(q).$$
We let $\smashprod$ be the monoidal product of $\ocal$
and $\sphspec$ the unit.
The unit of $\rightmod \ocal$ is the module 
$\underhom_{Sp^\Sigma}(-, \sphspec)$, which is easy to prove
using Corollary \ref{cor:endandcoendeval}.
It is then a matter of formality that the internal function object
exists, it is defined as
$$\underhom_\Box (M,N)(o) =
Sp^\Sigma \nat \big(M,N(o \smashprod -)\big)
= \int_p \underhom \big(M(p), N(o \smashprod p)\big).$$

Now we show that the pushout product and monoid axiom hold,
to do so we must understand the generating cofibrations.
Thus we write out part of the proof
of \cite[Theorem 6.1]{ss03monequiv}.
This proves the model structure exists by identifying
$\rightmod \ocal$ with a category of algebras over a triple.

Let $\mathbb{I}_\ocal$ be the $Sp^\Sigma$-category with the same objects
as $\ocal$ and maps given by
$\mathbb{I}_\ocal(o,o) = \sphspec$ and a point otherwise.
There is a canonical map of $Sp^\Sigma$-categories
$u \co \mathbb{I}_\ocal \to \ocal$,
which is given by the unit map
$\mathbb{I}_\ocal(o,o) \to \ocal(o,o)$.
The adjoint pair
of restriction and extension of scalars (Definition \ref{def:indres modules}),
gives the required triple $T$ (\cite[Chapter VI]{mac}) on $\rightmod \mathbb{I}_\ocal$
and the algebras
over this triple are $\ocal$-modules.
Now $\rightmod \mathbb{I}_\ocal$ is simply an $\ocal$-indexed product of
copies of $Sp^\Sigma$ hence we can give it the product model structure
\cite[1.1.6]{hov99}. The generating (acyclic) cofibrations for
$\rightmod \mathbb{I}_\ocal$
are given by maps $f \co A \to B$ concentrated in factor $o$
where $f$ is a generating (acyclic) cofibration for $Sp^\Sigma$.
Hence a generating (acyclic) cofibration for $\rightmod \ocal$
has form $A \smashprod F_o \to B \smashprod F_o$ where
$f$ is a generating (acyclic) cofibration for $Sp^\Sigma$.
In fact this argument shows that smashing with a free module
takes (acyclic) cofibrations of $Sp^\Sigma$ to
(acyclic) cofibrations of $\rightmod \ocal$.

The monoid axiom follows from the monoid axiom for
$Sp^\Sigma$ (for the positive case see \cite[Theorem 14.2]{mmss01}).
Take a generating acyclic cofibration
$A \smashprod F_o \to B \smashprod F_o$ and a module $M$,
for $p \in \ocal$
we have $((A \smashprod F_o) \Box M)(p)= A \smashprod (F_o \Box M)(p)$
hence $(A \smashprod F_o) \Box M \to  (B \smashprod F_o) \Box M$
is a weak equivalence in $\rightmod \ocal$. Since pushouts
and transfinite compositions of $\ocal$-modules are
constructed object-wise, the rest of the monoid axiom follows.

The pushout product axiom for $\rightmod \ocal$ uses the isomorphism
$(A \smashprod F_o) \Box (B \smashprod F_p) \cong
(A \smashprod B) \smashprod F_{o \smashprod p}$,
and the pushout product axiom for symmetric spectra.
We must also prove that for any cofibrant module $M$ the map
$$\cofrep \ocal(-, \sphspec) \Box M \to M$$ is a
weak equivalence. We actually prove this result holds without
the assumption that $M$ is cofibrant.
The unit of $Sp^\Sigma$
is cofibrant, hence so is $\ocal(-, \sphspec)$
and there is nothing to check.
In a category where the unit is not cofibrant,
such as $Sp^\Sigma_+$,
an alternative proof exists whenever
smashing with cofibrant objects
preserves weak equivalence.
The important step is in identifying
$\cofrep \ocal(-, \sphspec)$ with
$\cofrep \sphspec \smashprod \ocal(-, \sphspec )$, the latter
is cofibrant as an $\ocal$-module since $\cofrep \sphspec$
is cofibrant in $Sp^\Sigma_+$. The map
$$\cofrep \sphspec \smashprod \ocal(-, \sphspec )
\to \ocal(-, \sphspec)$$
is a weak equivalence since
$\cofrep \sphspec \smashprod X \to X$
is a weak equivalence for any symmetric spectrum $X$.
So we have reduced the problem to
proving that $\cofrep \sphspec \smashprod M \to M$
is a weak equivalence, but this is obvious from the above statements.
\end{proof}

\begin{proposition}\label{prop:monoidalspectralfunctor}
Let $\Psi \co \mathcal{O} \longrightarrow \mathcal{R}$
be a monoidal functor of monoidal spectral categories.
Then $\big((-) \smashprod_\mathcal{O} \mathcal{R},\Psi^* \big)$
is a strong monoidal Quillen adjunction. If the categories and $\Psi$
are also symmetric then the adjunction will be a strong symmetric monoidal
adjunction.
\end{proposition}
\begin{proof}
That the left adjoint is strong monoidal is an easy
exercise in manipulating coends. Thus the right adjoint
has a weak monoidal structure and the left adjoint
has the required model structure properties.
When the unit is cofibrant this is all that is required.
In the $Sp^\Sigma_+$ case, the result follows from
the identification of
$\cofrep \ocal(-, \sphspec)$ with
$\cofrep \sphspec \smashprod \ocal(-, \sphspec )$
and the isomorphism
$\ocal(-, \sphspec ) \smashprod_\mathcal{O} \mathcal{R} \cong
\mathcal{R} (-, \sphspec )$.
\end{proof}

\begin{definition}
Given a pair of spectral functors $F,G \co \ccal \to \dcal$
between spectral categories
there is a symmetric spectrum of spectral natural transformations.
For brevity we call this
the \textbf{spectrum of natural transformations}\index{Spectrum of natural transformations}
and it is defined as
(\cite[Proposition 6.3.1]{bor94})
$$Sp^\Sigma \nat(F,G) = \int_c \underhom(Fc,Gc).$$
As an equaliser of products we have
$$Sp^\Sigma \nat(F,G) \longrightarrow \prod_{c \in \ccal} \underhom(Fc,Gc)
\overrightarrow{\longrightarrow} \prod_{c, c' \in \ccal}
\underhom (\ccal(c,c') \otimes Fc, Gc')$$
where $\underhom$ is the internal function object
for symmetric spectra.
\end{definition}

\begin{rmk}
The above is the general definition, but when working with right modules over
$\mathcal{O}$ one must remember to account for the change in variance.
So for right $\mathcal{O}$-modules $M$ and $N$ the \textbf{spectrum of
morphisms}\index{Spectrum of morphisms} of $\mathcal{O}$-modules is given below.
$$Sp^\Sigma \nat(M,N) \longrightarrow \prod_{c \in \ccal} \underhom(Mc,Nc)
\overrightarrow{\longrightarrow} \prod_{c', c \in \ccal}
\underhom (\ocal(c,c') \otimes Mc, Nc')$$
\end{rmk}

Hence $\rightmod \ocal$ is an $Sp^\Sigma$-category.

\begin{proposition}\label{prop:tencotenmodule}
The category of $\mathcal{O}$-modules is tensored and
cotensored over symmetric spectra.
\end{proposition}
\begin{proof}
The tensor and cotensor are given object-wise; so for a symmetric spectrum
$A$ and a (right) $\mathcal{O}$-module $M$
$$(A \otimes M)(o) = A \smashprod M(o) \qquad \underhom(A,M)(o)=\underhom(A,M(o)).$$
Now we show that these are compatible with the enrichment, so that:
$$Sp^\Sigma \nat(M, \underhom(A,N)) \cong
Sp^\Sigma \nat(A \otimes M, N) \cong
\underhom(A, Sp^\Sigma \nat(M,N)).$$
Writing this out in terms of ends we require a pair of relations
\begin{eqnarray*}
\int_o \underhom(A \smashprod M(o),N(o)) & \cong &
\int_o \underhom(M(o), \underhom(A,N(o))) \\
\int_o \underhom(A, \underhom(M(o),N(o))) & \cong &
\underhom(A, \int_o \underhom(M(o),N(o)))
\end{eqnarray*}
the first of which is certainly true and
the second follows from \cite[Proposition 6.7.3]{bor94}:
the associativity of $Sp^\Sigma$
ensures that $(A \smashprod (-), \underhom(A,-))$
is a $Sp^\Sigma$-adjunction.
\end{proof}

In truth the above enrichment, tensoring and
cotensoring over symmetric spectra actually arises
from giving $\rightmod \ocal$ the structure of a
$Sp^\Sigma$-algebra.
The result as stated below holds in greater generality,
one can replace $Sp^\Sigma$ by any other
symmetric monoidal model category provided that
either the unit is cofibrant or that
smashing with a cofibrant object
preserves weak equivalences.
\begin{theorem}
If $\ocal$ is a symmetric monoidal $Sp^\Sigma$-category,
then there is a strong symmetric monoidal Quillen pair
$$(-) \otimes F_{\sphspec} :
 Sp^\Sigma \overrightarrow{\longleftarrow} \rightmod \ocal
: \underhom (F_{\sphspec}, -)$$
where $F_{\sphspec}$ is the free module on the unit of $\ocal$.
Thus $\rightmod \ocal$ is an $Sp^\Sigma$-algebra.
This structure is compatible with the enrichment, tensoring and cotensoring
above.
\end{theorem}
\begin{proof}
We clearly have a symmetric strong monoidal pair
and we have already shown that  $(-) \otimes F_{\sphspec}$
preserves (acyclic) cofibrations in
Theorem \ref{thm:monmodules}.
To recover the tensoring one simply needs to use the basic properties
of coends (such as commutation with left adjoints and interchange of factors)
and Corollary \ref{cor:endandcoendeval} to see
$A \otimes M \cong (A \otimes F_{\sphspec}) \Box M$.
For the cotensoring, the same kind of argument suffices to show
$\underhom (A, M) \cong \underhom_\Box (A \otimes F_{\sphspec}, M)$,
but we need the end version of Corollary \ref{cor:endandcoendeval}
since we are now in a `right-handed' case.
Finally, the enrichment can be identified as
$\underhom(M,N) \cong \underhom ( F_{\sphspec} , \underhom_\Box (M,N))$.
\end{proof}

Note that by the enriched Yoneda lemma
(Lemma \ref{lem:enrichyoneda}),
$\underhom (F_{\sphspec}, M) \cong M(\sphspec)$ for any
$\ocal$-module $M$.

\begin{rmk}\label{rmk:freemodules}
It is instructive to look at this
result in a little more generality.
For each $o \in \ocal$ there is a Quillen pair
$(-) \smashprod F_{o} :
 Sp^\Sigma \overrightarrow{\longleftarrow} \rightmod \ocal
: \ev_o$
where $\ev_o (M)=M(o)$, the evaluation functor.
To show that this is an adjoint pair
we use the enriched Yoneda lemma (Lemma \ref{lem:enrichyoneda}),
$$\rightmod \ocal(A \smashprod F_{o}, M)
\cong Sp^\Sigma(A, \underhom(F_{o}, M))
\cong Sp^\Sigma(A, M(o))$$
we can now use Theorem \ref{thm:monmodules} to see that this is a Quillen pair.
Alternatively, we can show this more directly:
since fibrations and weak equivalences are defined object-wise,
each $\ev_o$ is a right Quillen functor.
\end{rmk}


\section{Morita Equivalences}\label{sec:moritaequiv}
The theorem below is vital to our work,
it allows us to replace $G$-spectra by modules
over an $Sp^\Sigma$-category $\ecal(\gcal)$.

\begin{definition}
Consider a set of objects $\mathcal{G}$
in a spectral model category $\mathcal{D}$.
We define $\ecal(\gcal)$, the
\textbf{endomorphism ringoid of $\gcal$}\index{Endomorphism ringoid},
to be the spectral category with object set $\gcal$
and $\ecal(\gcal)(g,g') = \underhom_{Sp^\Sigma}(g,g')$.
\end{definition}

\begin{definition}\label{def:moritafunctors}
For a spectral model category $\mathcal{D}$ and a full subcategory
$\ecal(\mathcal{G})$ defined by an object set $\mathcal{G}$,
there is a functor (sometimes called the tautological functor)
$$ \underhom(\mathcal{G},-) \co \mathcal{D} \longrightarrow \rightmod \ecal(\mathcal{G}) $$
defined by $\underhom(\mathcal{G},d)(g)=\underhom_{Sp^\Sigma}(g,d)$
(this is \cite[Definition 3.9.1]{ss03stabmodcat}).
This has a left adjoint
$- \smashprod_{\ecal(\mathcal{G})} \mathcal{G}$
and for a module $M$ this
is given by a coend
$ \int^{g \in \gcal} M(g) \smashprod g$.
This can also be written in terms of a
coequaliser of coproducts:
$$\bigvee_{g,h \in \mathcal{G}}
M(h) \smashprod \ecal(\mathcal{G})(g,h) \smashprod g
\overrightarrow{\longrightarrow}
\bigvee_{g \in \mathcal{G}}
M(g) \smashprod g.$$
\end{definition}

\begin{theorem}\label{thm:monoidalmorita}
When $\dcal$ is a closed symmetric monoidal spectral category and
the object set $\mathcal{G}$ consists of cofibrant and fibrant objects
the adjunction
$$(-) \smashprod_{\ecal(\mathcal{G})} \mathcal{G} :
\rightmod \ecal(\mathcal{G})\overrightarrow{\longleftarrow}
\mathcal{D}
: \underhom(\mathcal{G},-)
\index{$E$@$((-) \smashprod_{\ecal(\mathcal{G})} \mathcal{G},
\underhom(\mathcal{G},-))$}
$$
is a spectral Quillen pair.
If $\mathcal{G}$ is a set of compact generators for $\mathcal{D}$
then this Quillen pair is a spectral Quillen equivalence.
If $\gcal$ is closed under the monoidal product, then this pair
is an adjunction of closed strong symmetric monoidal spectral model categories.
\end{theorem}
\begin{proof}
The first two statements are \cite[Theorem 3.9.3]{ss03stabmodcat}.
So we must consider the case where
$\gcal$ is closed under the monoidal product and show that this
adjunction is strong monoidal, this follows by \cite[Proposition 3.6]{greshi}
and we give an explicit proof below.
The left adjoint is strong symmetric monoidal by the following series of isomorphisms,
where we use the assumption that the collection of generators of $\dcal$
form a symmetric monoidal $Sp^\Sigma$-category.

\begin{eqnarray*}
(M \Box N) \smashprod_{\ecal(\mathcal{G})} \mathcal{G}
& = & \int^{g \in \mathcal{G}} \int^{x,y \in \mathcal{G}}
\big({\ecal(\mathcal{G})}(g, x \smashprod y)
\smashprod M(x) \smashprod N(y) \big) \smashprod g \\
& = & \int^{x,y \in \mathcal{G}} M(x) \smashprod N(y) \smashprod
\left( \int^{g \in \mathcal{G}}
{\ecal(\mathcal{G})}(g, x \smashprod y) \smashprod g \right) \\
& \cong & \int^{x,y \in \mathcal{G} } M(x) \smashprod N(y) \smashprod (x \smashprod y) \\
& \cong & \int^{x,y\in \mathcal{G} } (M(x) \smashprod x) \smashprod (N(y) \smashprod y) \\
& \cong & \int^{x \in \mathcal{G}} (M(x) \smashprod x)
\smashprod \int^{y \in \mathcal{G}} (N(y) \smashprod y) \\
& \cong & \left( \int^{x \in \mathcal{G}} M(x) \smashprod x \right)
\smashprod \left( \int^{y \in \mathcal{G}} N(y) \smashprod y \right) \\
& = & M \smashprod_{\ecal(\mathcal{G})} \mathcal{G}
\smashprod N \smashprod_{\ecal(\mathcal{G})} \mathcal{G}
\end{eqnarray*}
It is clear that the left adjoint preserves the
$Sp^\Sigma$-algebra structure, that is, for a symmetric spectrum $X$,
$$ (X \smashprod \ecal(\mathcal{G})(-, \sphspec)) \smashprod_{\ecal(\mathcal{G})} \mathcal{G}
\cong i_\dcal (X)$$ where $i_\dcal \co Sp^\Sigma \to \dcal$ is the left adjoint
of the Quillen pair giving $\dcal$ a
closed symmetric monoidal spectral category structure.
We also need the following to be a weak equivalence in $Sp^\Sigma$.
$$\cofrep \ecal(\mathcal{G})(-, \sphspec) \smashprod_{\ecal(\mathcal{G})} \mathcal{G}
\to \ecal(\mathcal{G})(-, \sphspec) \smashprod_{\ecal(\mathcal{G})} \mathcal{G} \cong \sphspec$$
We use the characterisation
$\cofrep \ecal(\mathcal{G})(-, \sphspec) =
\cofrep \sphspec \smashprod \ecal(\mathcal{G})(-, \sphspec)$
and the result follows immediately.
\end{proof}

In Chapter \ref{chp:WskewSQmod} we will need a slight alteration of this result
to a context the unit object in $\gcal$ is not cofibrant. We shall go through
the changes needed in Theorem \ref{thm:so2morita}.
It should be clear that the correct notion of compact for the above result
when working in a $G$-equivariant setting is $G$-compactness.

\begin{rmk}\label{rmk:whyfibrant}
The reader should note that the requirement that every object
of $\gcal$ is fibrant is essential
to know that $\ecal(\mathcal{G})(g,g')$ has the correct homotopy type.
It can be quite difficult to meet this requirement
\emph{and} have $\gcal$ closed under the
monoidal product since the smash product of a pair of
fibrant objects in a model category
is not usually fibrant.
Equally, there is no reason to expect
the monoidal product and fibrant replacement functor to be compatible.
That is, the fibrant replacement functor is not usually a monoidal functor.
The solution of \cite{greshi} and of our work
is to work in a category with every object fibrant.
This is the reason why we have constructed rational spectra
as $S_\qq$-modules in EKMM spectra.
Similarly for the finite case we considered
$S_H \leftmod$ for each conjugacy class of subgroups $(H)$.
Looking ahead to Part \ref{part:O(2)}, where we specialise
to the group $O(2)$, we take our model
for rational $SO(2)$-spectra to be $\iota^* S_\qq \leftmod$.
Every object of this category is fibrant, hence so is
every object of the associated skewed category.
\end{rmk}

\part{The Continuous Dihedral Group}\label{part:O(2)}
\chapter{Rational $O(2)$-Equivariant Spectra}\label{chp:O2spec}
We now concentrate on the group $O(2)$. We apply our splitting result
and state some of the homotopy level calculations
on rational $O(2)$-spectra from \cite{gre98a}.

\section{Basics}
The closed subgroups of $O(2)$ are $O(2)$ itself, $SO(2)$, the finite dihedral
groups and the finite cyclic groups.
We will use the notation $D^h_{2n}$ to represent the dihedral subgroup of order
$2n$ containing $h$, an element of $O(2) \setminus SO(2)$. The conjugacy class
of such subgroups for a fixed $n \geqslant 0$ will be written $D_{2n}$.
We write $C_n$ for the cyclic group of order $n$.
For the rest of Part \ref{part:O(2)}, we let $W=O(2) / SO(2)$, the group of order 2.
We define a \textbf{cyclic group} to be any subgroup of $SO(2)$,
and a \textbf{dihedral group} to be any group of form
$D^h_{2n}$ or $O(2)$ itself.
We take $\cscr$\index{C@$\cscr$} to be the family consisting of all cyclic subgroups
of $O(2)$. The cofamily associated to $\cscr$ consists of the dihedral groups,
we call this set $\dscr$\index{D@$\dscr$} and we shall write
$E\dscr$ for $E \widetilde{\cscr}$.

\begin{lemma}
The family of cyclic subgroups, $\cscr$, is an idempotent family.
\end{lemma}
\begin{proof}
This follows from Lemma \ref{lem:Geidemfamily} since $SO(2)$
is the identity component of $O(2)$.
We illustrate this result with Figure \ref{phig}
below where we draw $\fcal O(2) / O(2)$.
\end{proof}

\begin{figure}[hbt]
\begin{center}
\setlength{\unitlength}{1cm}
\framebox[0.8\textwidth]{
\begin{picture}(6,3.5)(2,1)

\put(0,2){\circle*{0.1}}

\put(0,4){\circle*{0.1}}

\put(10,2){\circle*{0.1}}
\put(5,2){\circle*{0.1}}
\put(2.5,2){\circle*{0.1}}
\put(1.25,2){\circle*{0.1}}
\put(0.4,1.95){$\cdots$}

\put(-0.2,1.5){$O(2)$}
\put(-0.2,3.5){$SO(2)$}

\put(9.8,1.5){$D_{2}$}
\put(4.8,1.5){$D_{4}$}
\put(2.3,1.5){$D_{6}$}
\put(1.05,1.5){$D_{8}$}

\end{picture}}
\end{center}
\caption{\label{phig} $\fcal O(2) / O(2)$.}
\end{figure}

\begin{definition}\label{def:CandDspectra}
The model category of \textbf{cyclic spectra}\index{Cyclic spectra}
is $\cscr \mcal_\qq$ and the model category of
\textbf{dihedral spectra}\index{Dihedral spectra}
is $\widetilde{\cscr} \mcal_\qq$, which we write as $\dscr \mcal_\qq$
(see Definition \ref{def:familymodelcat}).
In most cases we will no longer explicitly mention
the fact that all categories are rationalised.
\end{definition}

\begin{theorem}\label{thm:O2splitting}
There is a strong monoidal Quillen equivalence
$$\Delta : O(2) \mcal _\qq \overrightarrow{\longleftarrow}
\cscr \mcal_\qq \times \dscr \mcal_\qq : \prod.$$
In particular, we have the following natural isomorphism for any $G$-spectra $X$ and $Y$
$$[X,Y]^{O(2)}_\qq \cong
[X \smashprod E \cscr_+, Y \smashprod E \cscr_+]^{O(2)}_\qq
\oplus
[X \smashprod E \dscr, Y \smashprod E \dscr]^{O(2)}_\qq.$$
Furthermore, we have a Quillen equivalence:
$$S_\qq \leftmod \overrightarrow{\longleftarrow}
L_{E \cscr_+} S_\qq \leftmod
\times
L_{E \dscr} S_\qq \leftmod.$$
\end{theorem}
\begin{proof}
This is an application of Theorem \ref{thm:familysplitting}
and Theorem \ref{thm:comparisons}.
\end{proof}

Denote the determinant representation of $O(2)$
by $\delta$. This is a one dimensional
real representation of $O(2)$. For
$n > 0$, $n \delta$ is an $n$-dimensional real
representation with $(n \delta)^H$ equal
to $0$ for $H \in \dscr$ and $\rr^n$ for
$H \in \cscr$.

\begin{lemma}
The universal space for the family $\cscr$ is given by the
universal space for the group $W$, that is, $E \cscr = E W$.
A construction of universal space for the cofamily $\dscr$, $E\dscr$,
is $S^{\infty \delta}$.
\end{lemma}
\begin{proof}
Consider $EW$ as an $O(2)$-space by letting
$O(2)$ act via the quotient homomorphism $O(2) \to O(2) / SO(2) =W$.
Then we note that $H \in \cscr$ acts trivially and
$H \in \dscr$ acts through $W$, so that $EW$ has the required universal property.

Using the inclusion of $n \delta$ into $(n+1) \delta$
we have $S^{n \delta} \to S^{(n+1) \delta}$ and the colimit
of these maps is $S^{\infty \delta}$. Since each map is
an inclusion, $(S^{\infty \delta})^H = \colim_n (S^{n \delta})^H$
and hence is equal to $S^0$ for $H \in \dscr$
and $S^{\infty}$ for $H \in \cscr$.
Since $S^{\infty}$ is
the infinite sphere it is
weakly equivalent to a point.
Hence $S^{\infty \delta}$ has the required
universal property.
\end{proof}

We compare our splitting above to \cite[Proposition 3.1]{gre98a} and hence
relate maps in $\ho \cscr \mcal_\qq$ to maps
of $\ho SO(2) \mcal_\qq$ with a $W$-action.
In Chapter \ref{chp:cyclicspectra} we improve on this
and obtain a model category version of this result.
Note that since $SO(2)$ is a normal subgroup of $O(2)$,
$\cscr \mcal$ is Quillen equivalent to
$O(2) \mcal (SO(2))$, the $\fscr(SO(2))$-model structure
on $O(2) \mcal$, see Theorem \ref{thm:Fmodelstructure}.

\begin{proposition}\label{prop:relatetoO2cohom}
The forgetful map induces a natural isomorphism
$[X,Y]^{O(2)}_{\cscr,\qq} \cong \big\{ [X,Y]^{SO(2)}_\qq \big\}^W$,
where $\big\{ [X,Y]^{SO(2)}_\qq \big\}^W $ is the set of maps in the homotopy category of
rational $SO(2)$-spectra with a homotopy action of $W$.
Furthermore, $\ho \cscr \mcal_\qq$ is equivalent to
$\ho (SO(2) \mcal_\qq)^W$, the homotopy category of rational
$SO(2)$-spectra with a homotopy action of $W$.
\end{proposition}
\begin{proof}
By Theorem \ref{thm:O2splitting}
(see also Remark \ref{rmk:splittingidempotents})
it is clear that
$[X,Y]^{O(2)}_{\cscr,\qq}$ is isomorphic to
$e_\cscr [X,Y]^{O(2)}_{\qq}$.

The proof of
Theorem \ref{thm:generalsplitting} implies that
this is isomorphic to $[E \cscr_+ \smashprod X,Y]^{O(2)}_{\qq}$.
Since $E \cscr_+ = EW_+$ we can apply
\cite[Proposition 3.1]{gre98a} to obtain the result.
The final statement follows from
\cite[Corollary 3.2]{gre98a}.
\end{proof}

\begin{proposition}\label{prop:cylicSQmodarecyclicspectra}
Let $S_\qq \leftmod(\cscr)$ denote the category of
$S_\qq$-modules in $O(2) \mcal$ with model structure
created from the underlying category $O(2) \mcal(SO(2))$.
Then the identity functor $S_\qq \leftmod(\cscr) \to
L_{E \cscr_+} S_\qq \leftmod$ is the left adjoint of a Quillen
equivalence.
\end{proposition}
\begin{proof}
As stated in \cite[Chapter IV, Theorem 6.9]{mm02}
the identity functor from $O(2) \mcal (SO(2))$ to
$\cscr \mcal$ is the left adjoint of a Quillen
equivalence. This gives a
Quillen pair between the module categories.
By inspection these categories have the same
weak equivalences, hence we have the Quillen
equivalence as claimed.
\end{proof}

We give the construction of the derived category
of dihedral spectra from \cite[Section 4]{gre98a}.
For each $n \geqslant 1$ let $V_n$ be a graded $\qq W$-module
and let $V_\infty= \lim_i \prod_{n \geqslant i} V_n$,
a graded $\qq W$-module.
Let $V_0$ be a graded $\qq$-module and
$\sigma \co V_0 \to V_\infty$ be a $W$-map.
Call such information a graded dihedral Mackey functor.
A map of graded dihedral Mackey functors
is a collection of $W$-maps
$f_k \co V_k \to V_k'$ for $k \geqslant 0$ such that
the obvious square relating $\sigma \co V_0 \to V_\infty$
to $\sigma' \co V_0' \to V_\infty'$ commutes.

\begin{proposition}
The homotopy category of dihedral spectra is equivalent
to the category of graded dihedral Mackey functors.
\end{proposition}
\begin{proof}
This is a combination of \cite[Summary 4.1 and Corollary 5.5]{gre98a}.
\end{proof}

We have reduced our study of rational $O(2)$-spectra
to looking at cyclic spectra and dihedral spectra.
For cyclic spectra we can use any of the model categories
$\cscr \mcal_\qq$, $L_{E \cscr_+} S_\qq \leftmod$
or $S_\qq \leftmod(\cscr)$, since these are all Quillen
equivalent. For dihedral spectra we can use
$\dscr \mcal_\qq$ or $L_{E \dscr_+} S_\qq \leftmod$.
The homotopy structure of
dihedral spectra as described above is
quite simple.
I plan (in future work) to give a classification of
dihedral spectra based on the results of
Chapter \ref{chp:fingroups}.

The remainder of the thesis concentrates on the more subtle
case of cyclic spectra. We have a description of the homotopy
category of cyclic spectra, one can think of this description as saying
that cyclic spectra are $SO(2)$-spectra with some extra structure
(a homotopy action of $W$).
We wish to make this notion precise at the model category level
so that we will have a better understanding of cyclic spectra.
We first investigate the general idea in Chapter \ref{chp:catwithinv}
and then apply this to the specific case of cyclic spectra in
Chapter \ref{chp:cyclicspectra}.

\chapter{Categories With Involution}\label{chp:catwithinv}

In order to study cyclic $O(2)$-spectra we
consider the relation between $O(2)$-spaces
and $SO(2)$-spaces (see Example \ref{ex:o2spaces}).
In the first section we have abstracted this relation to the notion of
a category with involution (the analogue of $SO(2)$-spaces)
and its associated skewed category (the analogue of $O(2)$-spaces).
In the second section we have considered monoidal structures
on these categories. In Section \ref{sec:invmodcat} we have given
conditions for a model structure on the
category with involution to pass to the skewed category.
We have given some examples in Section \ref{sec:involuteexamples}
and we recommend that the impatient reader reads up to the
definition of a skewed category and
moves straight to these examples.

\section{Categories with Involution}\label{sec:catinvolute}
We give the basic definitions and constructions
of categories with involution and their skewed categories.
We then investigate the conditions necessary for a
functor or an adjoint pair between categories with involution
to pass to the skewed categories.

\begin{definition}\label{def:catwithinv}
A \textbf{category with involution}\index{Category with involution}
$(\ccal,\sigma)$ is a category $\ccal$ with a functor
$\sigma \co \ccal \to \ccal$ such that $\sigma^2=\id_\ccal$.
We call such a functor $\sigma$
an \textbf{involution}\index{Involution}.
\end{definition}
It follows, of course, that $\sigma$ is both a left and a right adjoint.
We could relax this definition by requiring
that $\sigma^2$ is naturally isomorphic to $\id_\ccal$.
In the following work we would then need to replace any use
of the equality $\sigma^2=\id_\ccal$ by the (specified and fixed)
natural isomorphism $\sigma^2 \to \id_\ccal$.
Note that our involution is a covariant self-functor
of $\ccal$, this differs from some of the literature where
an involution means a functor
$\ccal^{op} \to \ccal$ which is self-inverse.

\begin{definition}
In a category with involution $(\ccal,\sigma)$, a
\textbf{map of order two}\index{Map of order two}
is a map $f \co A \to \sigma A$
such that $\sigma f = f^{-1}$.
\end{definition}

\begin{definition}\label{def:skewcat}
Given a category with involution $(\ccal,\sigma)$, we define the associated
\textbf{skewed category}\index{Skewed category}
$\sigma \# \ccal$ to be the category with
objects $w \co A \to \sigma A$ (also denoted $(A,w)$),
such that $w \circ \sigma w = \id_{\sigma A}$ (so $w$ is a map of order two).
A morphism of such objects is a commutative square:
$$\xymatrix@C+0.2cm{
A \ar[d]^w \ar[r]^f & B \ar[d]^{w'} \\
\sigma A \ar[r]^{\sigma f} & \sigma B }$$
and we denote such a morphism by $f$.
\end{definition}

\begin{proposition}
Let $(\ccal, \tau)$ be a category with involution
and assume that $\ccal$ has coproducts, denoted $\vee$.
Then there is an adjoint pair of functors
$$\mathbb{D} : \ccal \overrightarrow{\longleftarrow} \sigma \# \ccal : \mathbb{P}.$$
The left adjoint is the free functor, and it acts on objects as
$\mathbb{D} X = X \vee \sigma X \to \sigma X \vee X$ with twist $w$
the interchange of factors map. The right adjoint is projection onto the first factor
$\mathbb{P} (w \co A \to \sigma A) = A$.
These functors act on morphisms in the obvious way.
\end{proposition}
\begin{proof}
See Lemma \ref{lem:limcolim} for a more detailed construction of
$\mathbb{D} X$.
First let us understand what information is contained in the map
$$\xymatrix@C+0.4cm{
X \vee \sigma X \ar[d]^w \ar[r]^{(f,g)} & A \ar[d]^{w'} \\
\sigma X \vee X \ar[r]^{(\sigma f, \sigma g)} & \sigma A }$$
Thus we have the requirement $w' \circ (f,g) = (\sigma f, \sigma g) \circ w$,
but $(\sigma f, \sigma g) \circ w = (\sigma g, \sigma f)$,
hence $\sigma g = w' f$, so $g=\sigma (w' f)$.
Note that $w'g= w' \circ \sigma w' \circ \sigma f = \sigma f$ since $w' \sigma w' =\id$.
It follows that a map as above determines and is determined by a
map $f \co X \to A$ in $\ccal$, hence we have our adjunction.
\end{proof}

\begin{lemma}
Let $(\ccal, \tau)$ be a category with involution
and assume that $\ccal$ has products, denoted $\prod$.
Then $\mathbb{P}$ has a right adjoint.
$$\mathbb{P} : \sigma \# \ccal \overrightarrow{\longleftarrow} \ccal : \mathbb{D}'$$
The right adjoint acts on objects as
$\mathbb{D}' X = X \prod \sigma X \to \sigma X \prod X$, with twist $w$
the interchange of factors map, it acts
on morphisms in the obvious way.
\end{lemma}

\begin{lemma}
Consider a general category $\ccal$, if
$\ccal$ has equalisers and coequalisers, then
there we have a triple of functors $(\orb,\varepsilon,\fix)$
arranged into adjoint pairs as below.
$$
\varepsilon :  \ccal
\ \overrightarrow{\longleftarrow} \
\id \# \ccal  : \fix
\quad \quad
\orb :  \id \# \ccal
\ \overrightarrow{\longleftarrow} \
 \ccal : \varepsilon
$$
\end{lemma}
\begin{proof}
For an object $A$ of $\ccal$, or $(B,u) \in \id \# \ccal$ we
have the following definitions
$$ \begin{array}{rcccl}
& & \varepsilon A & = &  \id \co A \to A \in \id \# \ccal \\
\fix(B,u) & = &  B^u & =& \textrm{Eq} \big(B
\underset{\id}{\overset{u}{\overrightarrow{\longrightarrow}}} B  \big)
\in \ccal\\
\orb(B,u) & = & B/u & =& \textrm{Coeq} \big(B
\underset{\id}{\overset{u}{\overrightarrow{\longrightarrow}}} B  \big)
\in \ccal
\end{array}$$
it is easy to see that these are adjoint pairs.
\end{proof}

\begin{rmk}\label{rmk:eqfunc}
The category $\id \# \ccal$ is the category of $C_2$-objects
and $C_2$-equivariant maps in $\ccal$. We have recovered the usual pair of triples:
the forgetful functor $\id \# \ccal \to \ccal$
with its left and right adjoints and
the trivial action functor
$\ccal \to \id \# \ccal$
with its left and right adjoints.
\end{rmk}

\begin{definition}
An \textbf{involutary functor}\index{Involutary functor}
$(F, \alpha) \co (\ccal, \sigma) \to (\dcal, \tau)$ consists of a functor
$F \co \ccal \to \dcal$  and a natural transformation
$\alpha \co F \sigma \to \tau F$
such that $\tau \alpha \circ \alpha \sigma =\id_F$.
\end{definition}

\begin{lemma}
An involutary functor
$(F,\alpha) \co (\ccal, \sigma) \longrightarrow (\dcal, \tau)$
passes to a functor
$(\sigma, \tau) \# F \co  \sigma \# \ccal \longrightarrow \tau \# \dcal$
which we call the \textbf{skewed functor}\index{Skewed functor},
we will often just call this functor $F$.
\end{lemma}
\begin{proof}
We define $(\sigma, \tau) \# F$ on an object $w \co A \to \sigma A$
to be the composite map $\alpha_A \circ Fw \co FA \to \tau FA$.
To see that this is an object of the category we draw the following
commutative diagram.
$$\xymatrix@R+0.4cm@C+0.4cm{
& FA \ar[d]^{Fw} \ar[dr] \ar[ddl]_= \\
& F \sigma A \ar[r]^{\alpha_A} \ar[d]^{F \sigma w} & \tau F A \ar[d]^{\tau F w} \ar[dr] \\
FA \ar[r]^= & F \sigma^2 A \ar[r]^{\alpha_{\sigma A}} &
\tau F \sigma A \ar[r]^{\tau \alpha_A} & \tau^2 F  A \ar[r]^= & FA }$$
For a map $f \co (w \co A \to \sigma A) \to
(w' \co A' \to \sigma A')$ in $\sigma \# \ccal$,
we make the definition: $((\sigma, \tau) \# F) f = Ff$
(labelling maps by their first factor).
This is a morphism of $\sigma \# \ccal$
by the naturality of $\alpha$, as the
diagram below demonstrates.
$$\xymatrix{
FA  \ar[d]^{Fw} \ar[r]^{Ff} &
FA' \ar[d]^{Fw'} \\
F \sigma A  \ar[d]^{\alpha_A} \ar[r]^{F \sigma f} &
F \sigma A' \ar[d]^{\alpha_{A'}} \\
\tau FA  \ar[r]^{ \tau Ff} &
\tau FA' }$$
\end{proof}

\begin{definition}
For a pair of involutary functors
$(F,\alpha) \co (\ccal, \sigma) \longrightarrow (\dcal, \tau)$
and
$(G,\beta) \co (\dcal, \tau) \longrightarrow (\ecal,\rho)$
we define the \textbf{composite involutary functor}\index{Composite involutary functor} to be
$$(G \circ F, \beta F \circ G \alpha) \co (\ccal, \sigma) \longrightarrow  (\ecal,\rho).$$
\end{definition}

\begin{lemma}
The composite involutary functor of a composable pair of involutary functors is an
involutary functor.
\end{lemma}
\begin{proof}
By drawing a larger version of the diagram
defining the skewed functor one can see that
$$\rho (\beta F \circ G \alpha) \circ (\beta F \circ G \alpha) \sigma
= \id_{G \circ F}.$$
We draw the diagram below to explain the definition.
Let $f \co (A,w) \to (A',w')$ be a map in $\sigma \# \ccal$.

$$\xymatrix@R+0.2cm@C+0.4cm{
GFA  \ar[d]^{GFw} \ar[r]^{GFf} &
GFA' \ar[d]^{GFw'} \\
GF \sigma A  \ar[d]^{G \alpha_A} \ar[r]^{GF \sigma f} &
GF \sigma A' \ar[d]^{G \alpha_{A'}} \\
G \tau FA  \ar[d]^{\beta_{FA}} \ar[r]^{G \tau Ff} &
G \tau FA' \ar[d]^{\beta_{FA'}} \\
\rho GFA \ar[r]^{\rho GFf} &
\rho GFA' }$$
\end{proof}

\begin{definition}
An \textbf{involutary natural transformation}\index{Involutary natural transformation}
$\eta \co (F, \alpha) \longrightarrow (G, \beta)$
between involutary functors
$(F, \alpha), \ (G, \beta) \co (\ccal, \sigma) \longrightarrow (\dcal, \tau)$
is a natural transformation $\eta \co F \to G$
such that $\tau \eta \circ \alpha = \beta \circ \eta \sigma$.
\end{definition}

\begin{lemma}
An involutary natural transformation
$\eta \co (F, \alpha) \longrightarrow (G, \beta)$
passes to a natural transformation
between the skewed functors
$$\eta = (\sigma, \tau) \# \eta \co (\sigma, \tau) \# F
\longrightarrow (\sigma, \tau) \# G$$
called the \textbf{skewed natural transformation}\index{Skewed
natural transformation}.
\end{lemma}
\begin{proof}
All we need do is draw the diagram below, which will give the
definition of $(\sigma, \tau) \# \eta$, explain the
requirement $\tau \eta \circ \alpha = \beta \circ \eta \sigma$
and make it clear (through the naturality of each square)
that this will be a natural transformation
between the skewed functors.
$$\xymatrix@R+0.2cm@C+0.4cm{
FA \ar[d]^{Fw} \ar[r]^{\eta_A} &
GA \ar[d]^{Gw} \\
F \sigma A \ar[d]^{\alpha_A} \ar[r]^{\eta_{\sigma A}} &
G \sigma A \ar[d]^{\beta_A} \\
\tau F A \ar[r]^{\tau \eta_A} &
\tau G A }$$
\end{proof}

\begin{definition}
An \textbf{involutary adjunction}\index{Involutary adjunction}
$\big( (F, \alpha), (G, \beta), \eta, \varepsilon \big)$
between two categories with involution
$(\ccal, \sigma)$ and $(\dcal, \tau)$ is
an adjunction
$( F, G, \eta, \varepsilon )$
consisting of involutary functors
$(F, \alpha) \co (\ccal, \sigma) \longrightarrow (\dcal, \tau)$
and
$(G, \beta) \co (\dcal, \tau) \longrightarrow (\ccal, \sigma)$
and involutary natural transformations
$\eta \co \id_{\ccal} \longrightarrow (GF, \beta F \circ G \alpha)$
and
$\varepsilon \co (FG, \alpha G \circ F \beta) \longrightarrow \id_{\dcal}$.
\end{definition}

\begin{lemma}\label{lem:skewadjunct}
An involutary adjunction $\big( (F, \alpha), (G, \beta), \eta, \varepsilon \big)$
passes to an adjunction of the skewed categories.
$\big( (\sigma, \tau) \# F, (\tau, \sigma) \# G,
(\sigma, \tau) \# \eta, (\tau, \sigma) \# \varepsilon \big)$.
Furthermore, this gives the commutative square
(see below) of adjoint functors below.
$$\xymatrix@R+1cm@C+2cm{
\sigma \# \ccal
\ar@<+0.7ex>[r]^{(\sigma, \tau) \# F} \ar@<+0ex>[d]|{\mathbb{P}} &
\tau \# \dcal
\ar@<+0.7ex>[l]^{(\tau, \sigma) \# G} \ar@<+0ex>[d]|{\mathbb{P}} \\
\ccal
\ar@<+0.7ex>[r]^{F} \ar@<+2ex>[u]^{\mathbb{D}} \ar@<-2ex>[u]_{\mathbb{D}'} &
\dcal
\ar@<+0.7ex>[l]^{G} \ar@<+2ex>[u]^{\mathbb{D}} \ar@<-2ex>[u]_{\mathbb{D}'}
}$$
\end{lemma}
\begin{proof}
Since $F$ and $G$ are involutary they
pass to the skewed categories.
We have a unit and counit for the skewed categories
by the assumption that $\eta$ and $\varepsilon$
are involutary.
We must check the following pair of equations of
natural transformations:
$$\big( (\tau, \sigma) \# G \ (\tau, \sigma) \# \varepsilon \big)
\circ \big((\sigma, \tau) \# \eta \ (\tau, \sigma) \# G \big)
= \id_{(\tau, \sigma) \# G}$$
$$\big((\sigma, \tau) \# F \ (\sigma, \tau) \# \eta \big)
\circ  \big( (\tau, \sigma) \# \varepsilon \ (\sigma, \tau) \# F \big)
= \id_{(\sigma, \tau) \#F}$$
but this is immediate from our definitions and the fact that we started with an
adjunction.
By the word commutative, we are claiming that
there are four natural isomorphisms as below.
$$\xymatrix{
{\mathbb{D} \circ F \cong  (\sigma, \tau) \# F \circ \mathbb{D}  }&
{\mathbb{P} \circ (\tau, \sigma) \# G \cong G \circ \mathbb{P}   } \\
{\mathbb{P} \circ (\sigma, \tau) \# F \cong F \circ \mathbb{P}   }&
{(\tau, \sigma) \# G \circ \mathbb{D}' \cong \mathbb{D}' \circ G }
}$$
It is easy to see that the isomorphisms containing
$\mathbb{P}$ exist. This is all that we need to check.
\end{proof}

\begin{rmk}
Consider a diagram of involutary adjunctions
that commutes up to natural isomorphism.
Then, in order to obtain a commuting diagram of
skewed categories, one must require that the
natural isomorphisms giving
the commutativity are involutary.
\end{rmk}

The following lemma shows that in the case of an adjunction,
one only has to check that one of the functors involved
is involutary to deduce that the other functor is involutary.
\begin{lemma}\label{lem:onlyneedhalf}
Consider an adjunction $(F,G, \eta, \varepsilon)$
between two categories with involution
$(\ccal, \sigma)$ and $(\dcal, \tau)$.
Then there is a natural transformation
$\alpha \co F \sigma \to \tau F$
such that $\tau \alpha \circ \alpha \sigma =\id_F$
if and only if there is a natural transformation
$\beta \co G \tau \to \sigma G$
such that $\sigma \beta \circ \beta \tau =\id_G$.
\end{lemma}
\begin{proof}
Assume that $\alpha$ exists, this implies that
$F \sigma$ and $\tau F$ are naturally isomorphic
functors. Hence their right adjoints
$\sigma G$ and $G \tau$ are isomorphic, so we have
a natural transformation
$\beta \co G \tau \to \sigma G$.
The following diagram must commute
$$ \xymatrix{
\dcal(\tau^2 LX,Y) \ar[r]^= \ar[d]^{\tau \alpha \circ \alpha \sigma} &
\dcal(LX, \tau^2 Y) \ar[r]^\phi &
\ccal(X, R \tau^2 Y) \ar[d]^{\sigma \beta \circ \beta \tau} \\
\dcal(L \sigma^2 X,Y) \ar[r]^\phi &
\ccal( \sigma^2 X,RY) \ar[r]^= &
\ccal(X,\sigma^2 RY)
}$$
hence if $\tau \alpha \circ \alpha \sigma =\id_F$
then $\sigma \beta \circ \beta \tau =\id_G$.
It is clear that the converse is also true.
\end{proof}

\section{Involutary Monoidal Categories}\label{sec:invmonoidalcat}
We consider the conditions necessary for a monoidal
product on a category with involution to pass
to a monoidal product on the skewed category.

\begin{definition}
An \textbf{involutary monoidal category}\index{Involutary monoidal category}
is a category with a closed monoidal product
$(\otimes, \hom, \phi)$,
a functor $\sigma$ such that $(\ccal, \sigma)$
is a category with involution and
\begin{enumerate}
\item a natural transformation
$m \co \sigma (-) \otimes \sigma (-) \to \sigma (- \otimes -)$
such that $\sigma m \circ m(\sigma \otimes \sigma) = \id_{- \otimes -}$,
\item an isomorphism $i \co \mathbb{I} \to \sigma \mathbb{I}$
such that $\sigma i \circ i =\id_{\mathbb{I}}$ ($\mathbb{I}$ the unit of $\otimes$).
\end{enumerate}
That is, we require $\sigma$ to be a strong monoidal functor -- $(\sigma, m , i)$,
such that $m$ behaves in a similar way to an involutary natural
transformation and $i$ is a map of order two. If $\ccal$
is a symmetric monoidal category and $\sigma$ is a symmetric monoidal
functor then $(\ccal, \sigma)$ is an
\textbf{involutary symmetric monoidal category}.
\end{definition}

\begin{theorem}
The skewed category of an involutary symmetric monoidal category
is a symmetric monoidal category.
\end{theorem}
\begin{proof}
Let $(\ccal, \sigma)$ be a the category with involution,
with closed monoidal structure $( \otimes, \hom, \phi, m, i)$
such that $\ccal$ is an involuntary monoidal category.
We will use $( \otimes, \hom, \phi)$ to denote the monoidal
product on the skewed category as decorating these
symbols further would be horrific.
We begin by defining the action of the functor
$\otimes$ on a pair of objects
$u \co A \to \sigma A$ and $v \co B \to \sigma B$ to be
$$m_{A,B} \circ u \otimes v \co A \otimes B \to \sigma A \otimes \sigma B \to \sigma (A \otimes B).$$
This is an object of the skewed category
since $m$ is an involutary natural transformation.
The action on maps is obvious from this definition.
The unit is the object $i \co \mathbb{I} \to \sigma \mathbb{I}$,
since $i$ is a map of order two, this is also an object of the category.

Now one must check that $i \co \mathbb{I} \to \sigma \mathbb{I}$ defines a unit,
that the product as above is associative and that certain coherence diagrams
are satisfied see either \cite[VII]{mac} or \cite[Chapter 4]{hov99}.
This is all routine and follows from our assumption that $\sigma$ is strong monoidal,
most coherence diagrams are easy to check, but one can use the `coherence theorem'
of \cite[VII]{mac} to note that all the diagrams that we are checking commute.
Thus we have a monoidal product on $\sigma \# \ccal$ and it is clear that
this product is symmetric provided $\otimes$ is symmetric on $\ccal$ and $\sigma$
is a strong symmetric monoidal functor.

Now we move to showing that this monoidal structure is closed,
so we construct an internal function object.
Let $B$ be an object of $\ccal$.
The functors $- \otimes \sigma B $ and $\sigma(\sigma (-) \otimes B)$
are naturally isomorphic, hence (by Lemma \ref{lem:onlyneedhalf})
there is a canonical natural isomorphism of order two
$\hom(\sigma B, -) \to \sigma \hom (B, \sigma -)$.
A small amount of adjustment will give
$f_{A,B} \co \hom (\sigma A, \sigma B) \to \sigma \hom (A,B)$
and it follows that $\sigma f \circ f (\sigma, \sigma) =\id$,
that is, $f$ is an involutary
natural transformation in the same way that $m$ is.

Alternatively one can construct this natural transformation via
the following diagram (note that since $\sigma^2=\id$, a general object
of $\ccal$ can be written as $\sigma A$ for suitable $A$).
Take $A= \sigma \hom ( \sigma B, \sigma C)$ and follow the identity map round the diagram
to obtain the natural transformation $f \co \hom (\sigma -, \sigma -) \to \sigma \hom (-,-)$.
$$\xymatrix@R+0.2cm@C+0.5cm{
\hom(\sigma A, \hom ( \sigma B, \sigma C))
\ar[r]_(0.55){\phi_{\sigma A, \sigma B, \sigma C}^{-1}} &
\hom(\sigma A \otimes \sigma B, \sigma C)
\ar[r]_{(\sigma m_{\sigma A, \sigma B})^*} &
\hom(\sigma (A \otimes B), \sigma C)
\ar[d]_(0.45)\sigma \\
\hom(\sigma A, \sigma \hom ( B, C)) &
\hom( A , \hom ( B, C))
\ar[l]^\sigma &
\hom( A \otimes B, C)
\ar[l]^(0.45){\phi_{A, B, C}} }$$
Let $(A,u)$, $(B,v)$ and $(C,w)$ be objects of the skewed category,
we define the object $\hom ((A,u), (B,v))$ to be
$$f_{A,B} \circ \hom(\sigma u, v) \co \hom(A,B) \to
\hom(\sigma A, \sigma B) \to \sigma \hom(A,B). $$
We must prove that maps
$(A,u) \otimes (B,v) \to (C,w)$
are in natural bijection with maps
$(A,u) \to \hom \big( (B,v), (C,w) \big)$.
Consider the following pair of triangles, it follows from the construction of
the natural transformation $f$ above
that the left hand triangle commutes if and only if the right hand triangle commutes.
$$\xymatrix@R+0.2cm@C+0.8cm{
\sigma A \otimes \sigma B \ar[r]^g \ar[d]_{m_{A,B}} & \sigma C
& \sigma A \ar[r]^{\phi g} \ar[dr]_{\sigma \phi k} & \hom (\sigma B, \sigma C) \ar[d]^{f_{B,C}}\\
\sigma ( A \otimes B) \ar[ur]_{\sigma k} & & &
\sigma \hom (B,C) }$$
We draw another picture showing three squares. If any one of these
commutes, so do the other two.
$$\xymatrix{
A \otimes B \ar[dd]_{u \otimes v} \ar[r]^k & C \ar[dd]^w
& A \ar[r]^{\phi k} \ar[dd]_u & \hom (B,C) \ar[d]^{\hom (\id,w)}
& A \ar[r]^{\phi k} \ar[dd]_u & \hom (B,C) \ar[dd]^{\hom (\sigma v,w)} \\
& & & \hom (B, \sigma C) \\
\sigma A \otimes \sigma B \ar[r]_g & \sigma C
& \sigma A \ar[r]_(0.4){\phi g} & \hom (\sigma B, \sigma C) \ar[u]_{\hom (v,\id)}
& \sigma A \ar[r]_(0.4){\phi g} & \hom (\sigma B, \sigma C) }$$
Fix $k$, then by the above
$g= w \circ k \circ (\sigma u \otimes \sigma v)$.
We combine these two collections of diagrams
(and remove $g$ from the result) to obtain
the pair of diagrams below.
The left hand diagram commutes if and only if the
right hand diagram commutes.
Thus we have our adjunction on the skewed category.
$$\xymatrix{
A \otimes B \ar[d]_{u \otimes v} \ar[r]^k & C \ar[dd]^w
& A \ar[r]^{\phi k} \ar[dd]_u & \hom (B,C) \ar[d]^{\hom (\sigma v,w)} \\
\sigma A \otimes \sigma B \ar[d]_{m_{A,B}} & & & \hom (\sigma B, \sigma C) \ar[d]^{f_{B,C}}\\
\sigma ( A \otimes B) \ar[r]_{\sigma k} & \sigma C
& \sigma A \ar[r]_{\sigma \phi k} & \sigma \hom (B,C) }$$
\end{proof}

One can remove the assumption that $\otimes$ is symmetric.
The above proof would suffice to show that one obtains
a monoidal skewed category, but one would have to take care
over the fact that there are now two (possibly different)
right adjoints to $\otimes$ (these are $\hom_l$ and $\hom_r$
from Definition \ref{def:Qbifunctor}).

\begin{definition}
An \textbf{involutary monoidal adjunction}\index{Involutary monoidal adjunction}
is an involutary adjunction
$(L,\alpha) : (\ccal, \sigma) \overrightarrow{\longleftarrow} (\dcal, \tau) :(R,\beta)$
such that $L$ is op-monoidal and $R$ is monoidal
and $\beta$ is a monoidal natural transformation $R \tau \to \sigma R$.
An \textbf{involutary symmetric monoidal adjunction}
between involutary symmetric monoidal categories
is an involutary monoidal adjunction
that is a symmetric monoidal adjunction.
\end{definition}

\begin{lemma}\label{lem:invmonadjskewed}
An involutary monoidal adjunction passes
to a monoidal adjunction on the skewed categories.
If this adjunction is strong monoidal or symmetric monoidal
then so is the adjunction on the skewed category.
\end{lemma}
\begin{proof}
Take an involutary monoidal adjunction
$(L,\alpha) : (\ccal, \sigma) \overrightarrow{\longleftarrow} (\dcal, \tau) :(R,\beta).$
The assumptions on $\beta$ are the same as requiring that the
diagrams below commute.
$$ \xymatrix{
{\mathbb{I}}_\ccal \ar[r]^\nu \ar[dd]_{i_\ccal}  &
R{\mathbb{I}}_\ccal \ar[d]^{R i_\dcal} & &
R\tau d \otimes R\tau d'
\ar[d]_{\beta_d \otimes \beta_{d'}}
\ar[r] &
R\tau (d \otimes d')
\ar[d]^{\beta_{d \otimes d'}} \\
&  R \tau {\mathbb{I}}_\ccal \ar[d]^{\beta} &&
\sigma Rd \otimes \sigma Rd'
\ar[r] &
\sigma R( d \otimes d') \\
\sigma {\mathbb{I}}_\ccal \ar[r]^{\sigma \nu}  &
\sigma R {\mathbb{I}}_\ccal } $$
Thus $R$ passes to a monoidal functor on the skewed category,
hence the skewed adjunction is monoidal.
The statement regarding strong monoidal adjunctions
is obvious. For the symmetric monoidal statement
it is easy to check that if $L$ and $R$
are symmetric then so are $\sigma \# L$
and $\tau \# R$.
\end{proof}

\begin{definition}
Let $(\ccal, \sigma)$ be an involutary (symmetric) monoidal model category.
Then $(\ccal, \sigma)$ is an
\textbf{involutary (symmetric) closed $\nu$-algebra}\index{Involutary
closed nu-algebra@involutary closed $\nu$-algebra}
if there is a strong (symmetric) monoidal involutary adjunction
$i_\ccal : \nu \overrightarrow{\longleftarrow} \ccal : j_\ccal$
with respect to the involution  $\id_\nu$ on $\nu$.
\end{definition}

\begin{lemma}
The skewed category of an involutary (symmetric) closed $\nu$-algebra
is a closed (symmetric) $\id \# \nu$-algebra. Furthermore it is
a closed $\nu$-algebra.
\end{lemma}
\begin{proof}
The first statement of the lemma is simply that an involutary adjunction
passes to an adjunction of the skewed categories.
The second statement follows from composing $(i,j)$
with the strong monoidal adjunction $(\varepsilon, \fix)$.
\end{proof}

\section{Involutary Model Categories}\label{sec:invmodcat}
The next logical step in developing
categories with involution is to consider model structures.
We provide criteria for model structures
and Quillen functors on involutary model categories
to pass to the skewed categories.

\begin{definition}
An \textbf{involutary model category}\index{Involutary model category}
$(M,\sigma)$ is a cofibrantly generated model category
$M$ with left Quillen functor $\sigma \co M \to M$ such that $\sigma^2=\id_M$.
That is, $(M, \sigma)$ is a category with involution.
\end{definition}
It follows, of course, that $\sigma$ is also a right Quillen functor
and that $\sigma$ preserves all weak equivalences (since a weak equivalence is a composite
of an acyclic cofibration followed by an acyclic fibration).

\begin{definition}
An \textbf{involutary Quillen functor}\index{Involutary Quillen functor}
is an involutary functor that is a Quillen functor.
\end{definition}

\begin{lemma}\label{lem:limcolim}
The category $\sigma \# M$ has all small limits and colimits.
\end{lemma}
\begin{proof}
Take a diagram $D$ (i.e. a small category) and a functor $F \co D \to \sigma \# M$.
We can form $\colim_D \mathbb{P} F$ (since $M$ is bicomplete)
and since $\sigma$ is a left adjoint we have a canonical isomorphism
$\sigma \colim_D \mathbb{P} F \cong \colim_D \sigma \mathbb{P} F$.
We also have a map $\colim_D w \co \colim_D \mathbb{P} F \to \colim_D \sigma \mathbb{P} F$
we combine these maps to give an object of the skewed category
in the diagram below.
$$\xymatrix@R+0.2cm@C+0.4cm{
& \colim_D \mathbb{P} F
\ar[d]^{\colim_D w}
\ar[dr]
\ar[ddl]_= \\
& \colim_D \sigma \mathbb{P} F
\ar[r]
\ar[d]^{\colim_D \sigma w} &
\sigma \colim_D \mathbb{P} F
\ar[d]^{\sigma \colim_D w}
\ar[dr] \\
\colim_D \mathbb{P} F \ar[r]^=
& \colim_D \sigma^2 \mathbb{P} F
\ar[r] &
\sigma \colim_D \sigma \mathbb{P} F
\ar[r]
& \sigma^2 \colim_D \mathbb{P} F
\ar[d]^= \\
& & & \colim_D \mathbb{P} F}$$
The composite
$\colim_D \mathbb{P} F \to \sigma^2 \colim_D \mathbb{P} F$
is the identity since there is a unique isomorphism between any two
colimits of a diagram.
The case for a limit is identical
since $\sigma$ is also a right adjoint.
\end{proof}

\begin{definition}
Let $(M,\sigma)$ be an involutary model category, we define
a \textbf{weak equivalence} (respectively \textbf{fibration}) of $\sigma \# M$
to be a map $f$ such that $\mathbb{P}f$ is a weak equivalence
(respectively fibration).
\end{definition}

\begin{proposition}
These weak equivalences and fibrations define a model structure
on $\sigma \# M$ and we call this category and model structure the
\textbf{skewed model category}\index{Skewed model category}.
\end{proposition}
\begin{proof}
All that is required is to check that the lifting lemma (below) applies in the case of
$$\mathbb{D} : M \overrightarrow{\longleftarrow} \sigma \# M : \mathbb{P}.$$
We have already shown that $\mathbb{P}$ is a left adjoint and so preserves all colimits.
Let $I$ and $J$ denote the generating cofibrations and acyclic cofibrations of
$M$.
Now we must check that every relative $\mathbb{D}J$-cell complex is a weak equivalence,
Take $k \co A \to B$ a map in $J$, then
$$\mathbb{P} \mathbb{D} k = k \vee \sigma k \co A \vee \sigma A \to B \vee \sigma B $$
and since $\sigma$ is a left Quillen functor (on $M$), it follows that
$\mathbb{P} \mathbb{D} k$ is an acyclic cofibration in $M$.
A relative $\mathbb{D}J$-cell complex in $\sigma \# M$ is
a transfinite composition of pushouts of $\mathbb{D} J$.
Since $\mathbb{P}$ preserves these constructions (it is a left adjoint) and the
set of acyclic cofibrations (of $M$) are closed under these operations
(\cite[proof of 2.2.10]{hov99}), the result follows.
\end{proof}

The original reference for the lifting lemma
is of course \cite[II.4]{quil67}.
The lemma below is a variation on \cite[Theorem 11.3.2]{hir03}
where we assume that the right adjoint preserves
filtered colimits so that the required smallness
conditions hold.
\begin{lemma}[Lifting Lemma]\index{Lifting lemma}\label{lem:lift}
Let $F : M \rightleftarrows N : G$ be an adjoint pair of functors
($F$ is the left adjoint) with $M$ a cofibrantly generated model category.
Let $I$ be the generating cofibrations and
$J$ the generating acyclic cofibrations for $M$.
Assume that $N$ has all small colimits and limits.
Define a map $f$ in $N$ to be a weak equivalence or a fibration
if and only if $Gf$ is so. Define the cofibrations of $N$ to be those maps in $N$
with the correct lifting property. Then this construction
defines a cofibrantly generated model structure
on $N$ provided:
\begin{enumerate}
\item $G$ preserves filtered colimits,
\item every relative $FJ$-cell complex is a weak equivalence.
\end{enumerate}
The sets $FI$ and $FJ$ are the generating cofibrations
and generating acyclic cofibrations for $N$.
\end{lemma}

\begin{lemma}
Let $(M,\sigma)$ be a model category with involution,
then $\mathbb{P} \co \sigma \# M \to M$ preserves cofibrations
and hence is a left Quillen functor.
\end{lemma}
\begin{proof}
The right adjoint $\mathbb{D}'$ preserves fibrations and acyclic fibrations.
\end{proof}

\begin{definition}
An \textbf{involutary Quillen pair}\index{Involutary Quillen pair}
is an involutary adjunction that is also a Quillen pair.
\end{definition}

\begin{lemma}
An involutary Quillen pair between model categories with involution,
passes to a Quillen pair between the skewed model categories.
\end{lemma}
\begin{proof}
Let $(F, \alpha) : (M, \sigma) \rightleftarrows (N, \tau) : (G, \beta)$
be an involutary Quillen pair.
We will show that $(\tau, \sigma) \# G$
is a right Quillen functor.
By Lemma \ref{lem:skewadjunct} we have the equation
$\mathbb{P}_M \circ (\tau, \sigma) \# G = G \circ \mathbb{P}_N$.
Take a map $f$ in $\tau \# N$, if $f$ is a fibration
or acyclic fibration then so is $G (\mathbb{P}_N (f))$.
Hence $(\tau, \sigma) \# G(f)$ is a fibration or
acyclic fibration by the definition of the model structure on
$\sigma \# M$.
\end{proof}

\begin{definition}
An \textbf{involutary Quillen equivalence}\index{Involutary Quillen equivalence}
is an involutary Quillen pair that is also a Quillen equivalence.
\end{definition}

\begin{proposition}
An involutary Quillen equivalence
passes to Quillen equivalence of the skewed model categories.
\end{proposition}
\begin{proof}
Take an involutary Quillen equivalence
$(F, \alpha) : (M, \sigma) \rightleftarrows (N, \tau) : (G, \beta)$
then consider a cofibrant $c \to \sigma c$ in $\sigma \# M$
and a fibrant $d \to \tau d$ in $\tau \# N$.
Then $c$ is cofibrant in $M$ and $d$ is fibrant in $N$
so a map $Fc \to d$ is a weak equivalence if and only if
$c \to Gd$ is a weak equivalence.
But this is precisely the statement that
a map $(Fc \to \tau Fc) \to (d \to \tau d)$
is a weak equivalence if and only if
$(c \to \sigma c) \to (Gd \to \sigma Gd)$ is a weak equivalence.
\end{proof}

\begin{definition}
An \textbf{involutary (symmetric) monoidal model category}\index{Involutary monoidal
model category} is a category with involution $(M, \sigma)$
such that:
\begin{enumerate}
\item $M$ is a (symmetric) monoidal model category,
\item $(M, \sigma)$ is an involutary model category,
\item $(M, \sigma)$ is an involutary (symmetric) monoidal category.
\end{enumerate}
\end{definition}

\begin{lemma}
If $(M, \sigma)$ is an involutary (symmetric) monoidal model category
then $\sigma \# M$ is a (symmetric) monoidal model category. Furthermore,
if $M$ satisfies the monoid axiom, so does $\sigma \# M$.
\end{lemma}
\begin{proof}
By the machinery above we know that $\sigma \# M$ is
a (symmetric) monoidal category. Hence we must show that
this is a monoidal model category. To do so we use the following
alternative form of the pushout product axiom
(see \cite[Lemma 4.2.2]{hov99}).
Take $f \co (C,u) \to (D,v)$ a cofibration in $\sigma \# M$
and $g \co (X,r) \to (Y,s)$ a fibration.
Then we must prove that the induced map
$$\alpha \co \hom((D,v),(X,r)) \to
\hom((D,v),(Y,s)) \prod_{\hom((C,u),(Y,s))} \hom((C,u),(X,r))$$
is a fibration that is acyclic whenever
$f$ or $g$ is.
Since the codomain of $\alpha$ is a pullback
we can apply $\mathbb{P}$ and obtain a
pullback diagram in $M$, we draw this below..
$$\xymatrix{
\hom(D,X)
\ar[dr]^(0.7){\mathbb{P} \alpha}
\ar[drr]
\ar[ddr] \\
& P
\ar[d]
\ar[r]
& \hom(C,X)
\ar[d] \\
& \hom(D,Y)
\ar[r]
& \hom(C,Y) }$$
By the definition of the model structure on
$\sigma \# M$ we must check that
$\mathbb{P} \alpha$ is a fibration that is acyclic whenever
$\mathbb{P} f$ or $\mathbb{P} g$ is.
Since $\mathbb{P} f$ is a cofibration and $\mathbb{P}g$
is a fibration the result follows by the pushout product axiom
for $M$.

We must also prove a result
concerning the cofibrant replacement of $\mathbb{I}$.
Let $\cofrep$ be cofibrant replacement in $\sigma \# M$, the diagram
$$\xymatrix{
*+<0.3cm>{*}
\ar@{>->}[r]
\ar[d]
& \cofrep {\mathbb{I}}
\ar[d]^{\cofrep i}
\ar@{->>}[r]^\sim
& {\mathbb{I}}
\ar[d]^i \\
*+<0.3cm>{*}
\ar@{>->}[r]
& \sigma \cofrep {\mathbb{I}}
\ar@{->>}[r]^\sim
& \sigma {\mathbb{I}} }$$
gives a cofibrant replacement of the unit in $\sigma \# M$.
Thus $* \to \cofrep \mathbb{I} \to \mathbb{I}$
is a cofibrant replacement of the unit in $M$.
Let $(X,r)$ be a cofibrant object of $\sigma \# M$,
then $X = \mathbb{P}(X,r)$ is cofibrant in $M$.
Since $M$ is a monoidal model category,
$\cofrep \mathbb{I} \otimes X \to \mathbb{I} \otimes X$
is a weak equivalence. Thus
$$(\cofrep \mathbb{I}, \cofrep i) \otimes  (X,r)
\to (\mathbb{I}, i) \otimes  (X,r)$$
is a weak equivalence in $\sigma \# M$.
The statement regarding the monoid axiom holds by an
equally straightforward argument using the fact that
$\mathbb{P}$ preserves cofibrations.
\end{proof}

\begin{definition}
An \textbf{involutary (symmetric) monoidal Quillen pair}\index{Involutary monoidal
Quillen pair}
is an involutary (symmetric) monoidal adjunction
that is also a monoidal Quillen pair.
\end{definition}

\begin{proposition}
An involutary monoidal Quillen pair induces
a monoidal Quillen pair on the skewed categories.
If the involutary pair is strong monoidal or symmetric then
so is the skewed adjunction.
\end{proposition}
\begin{proof}
Let $(F, \alpha) : (M, \sigma) \rightleftarrows (N, \tau) : (G, \beta)$
be our involutary monoidal pair.
By Lemma \ref{lem:invmonadjskewed} $(F,G)$
passes to a monoidal adjunction on the skewed category.
This skewed pair will be strong monoidal or symmetric
when $(F,G)$ is. All that remains to check
is that two technical conditions of
a monoidal Quillen pair hold for the
skewed adjunction.  That is, we must check
that if $(X,u)$ and $(Y,v)$
are cofibrant objects of $\sigma \# M$
then the map $$F \big( (X,u) \otimes (Y,v) \big) \to
F(X,u) \otimes F(Y,v)$$ is a weak
equivalence.
Let $(\mathbb{I}_M, i_M)$ be the unit of the
skewed category $\sigma \# M$.
We must also check that the map below is
a weak equivalence for a cofibrant replacement
$(\cofrep \mathbb{I}_M, \cofrep i_M)$ of
$(\mathbb{I}_M, i_M)$.
$$F (\cofrep \mathbb{I}_M, \cofrep i_M) \to F(\mathbb{I}_M, i_M)
\to (\mathbb{I}_N, i_N)$$
Both of these conditions hold since
$(F,G)$ is a monoidal Quillen pair and
the underlying object $X$ of a cofibrant
object $(X,u) \in \sigma \# M$ is cofibrant in $M$.
\end{proof}

\begin{proposition}\label{prop:skewedmodules}
Take $(M, \sigma)$ an involutary symmetric monoidal model
category satisfying the monoid axiom and let
$\alpha \co R \to \sigma R$ a commutative ring object in $\sigma \# M$.
Then $(R \leftmod , \alpha^* \sigma )$ is an involutary symmetric monoidal model
category and is equal to $(R, \alpha) \leftmod$, the category of
$(R, \alpha)$-modules in $\sigma \# M$.
\end{proposition}
\begin{proof}
An object of $(R \leftmod , \alpha^* \sigma )$ is an $R$-module $X$
and an $R$-module map of order two $w \co X \to \alpha^* \sigma X$.
It is important to recall that $\alpha^*$ only affects the
$R$-module structure of an object of $R \leftmod$, the underlying
object of $X$ is unchanged.
Thus we have the diagram
$$ \xymatrix{
R \smashprod X \ar[rr]^\nu \ar[d]^{\id \smashprod w} & & X \ar[dd]^w \\
R \smashprod \alpha^* \sigma X  \ar[d]^{\alpha \smashprod \id} \\
\sigma R \smashprod \sigma x  \ar[r]^{m}
& \sigma (R \smashprod X)  \ar[r]^{\sigma \nu} & \alpha^* \sigma X \\
}$$
where we note that $\sigma \nu \circ m \circ \alpha \smashprod \id$
is the $R$-module action on $\alpha^* \sigma X$.
It is clear that this precisely the requirement that
$X$ is an $(R, \alpha)$-module in $\sigma \# M$.
Equally one can see that the morphisms in these two categories agree.
The (acyclic) fibrations for each of the above categories
come from the model structure on $M$ and hence are the same.
Since $R$ is commutative, $R \leftmod$ is a monoidal model category
satisfying the monoid axiom, hence so is
$(R \leftmod , \alpha^* \sigma )$.
\end{proof}

\begin{rmk}\label{rmk:extendG}
We have constructed all the above for the case of
a functor $\sigma$ such that $\sigma^2= \id$. So
an involutary category $\ccal$ is a category
with an action of $C_2$.
The above can be generalised to
a category with an action of a group $G$.
That is a category $\ccal$, with a group
homomorphism $G \to \textrm{Aut}(\ccal)$ (invertible
functors $\ccal \to \ccal$).
Extending all of the above to general $G$ should
be formal, though the notation would have to be revised.
\end{rmk}

\section{Examples}\label{sec:involuteexamples}
\begin{ex}
Our first example of a category with involution
is a category $\ccal$ with the identity functor.
As described in Remark \ref{rmk:eqfunc} the skewed category
is the category of $C_2$-objects and maps in $\ccal$.
Thus $\id \# \ccal$
is the category of functors $C_2$ to $\ccal$,
where $C_2$ is the one-object category of a group.
Moreover, if $\ccal$
has a monoidal product, then $(\ccal, \id)$
is clearly an involutary monoidal category.
The monoidal product on $\id \# \ccal$ is then
the usual product of $C_2$-objects in $\ccal$.
Similarly any functor $F \co \ccal \to \dcal$
passes to the skewed categories
$F \co \id \# \ccal \to \id \# \dcal$
since $F$ preserves $C_2$-objects and $C_2$-maps.

When one considers model structures
there is a subtlety to consider,
which will reoccur in Example \ref{ex:o2spaces}.
A map $f \co (X,u) \to (Y,v)$ is
a weak equivalence or fibration
if and only if $f \co X \to Y$ is so in $\ccal$.
If $\ccal$ is the category of
topological spaces then
the usual model structure for $C_2$-spaces
has weak equivalences and fibrations those maps
$f \co X \to Y$ such that
$f$ and $f^{C_2} \co X^{C_2} \to Y^{C_2}$
are weak equivalences or fibrations in $\tscr_*$.
The model structure we have constructed on $\id \# \tscr_*$
is the $\fscr$-model structure on $C_2$-spaces
as mentioned on \cite[Page 70]{mm02},
for $\fscr$ the family consisting of the identity subgroup.
This is also known as the model category of
\textbf{free $C_2$-spaces}.
\end{ex}

\begin{ex}\label{ex:o2spaces}
This should be regarded as our motivating example:
it describes $O(2)$-spaces as $SO(2)$-spaces with extra structure.
For $t \in SO(2)$ and $k \in O(2) \setminus SO(2)$
we have the equation $ktkt=1$ in $O(2)$,
($t$ is a rotation and $k$ is a reflection).
Conjugation by $k$
considered as an automorphism of $O(2)$
restricts to the inversion automorphism $j$ of $SO(2)$,
$j(t) = t^{-1}$.
For a general group homomorphism
$f \co G \to H$, we can pull $X$, an $H$-space, back to a $G$-space
by $g * x = f(g) \cdot x$. Since this is a contravariant construction,
we use an upper asterisk and call this $G$-space $f^*X$.
The underlying set of $f^*X$ is the same and
for a $G$-map $g \co X \to Y$,
the underlying set map of $f^*g$
is the same as that for $g$.
Thus the group homomorphism $j$ gives $j^*$,
an involution on $SO(2)$-spaces.
We thus have the skewed category
$j^* \# SO(2) \tscr_*$, this is a
involutary symmetric monoidal model category.
We now see how this category relates to $O(2)$-spaces.

Take $X$ an $O(2)$-space then we can
consider $\iota^*X$, the restriction of $X$ to an $SO(2)$-space
along $\iota \co SO(2) \to O(2)$.
For the rest of this example $h$
will be a fixed reflection.
Since $X$ is an $O(2)$-space
we have a map $h \co \iota^*X \to j^*\iota^*X$.
We show that this is an $SO(2)$ map,
take $t \in SO(2)$ and let $\cdot$
be the $SO(2)$ action on
$\iota^*X$ and $*$ the action
of $SO(2)$ on $j^*\iota^*X$.
We have the equation
$$t*(h \cdot x)=t^{-1} \cdot (h \cdot x)=
(t^{-1} h) \cdot x =  (h t ) \cdot x
= h \cdot (t \cdot x).$$
We can also consider $h$ as a map
$j^*\iota^*X \to \iota^*X$ and it is
clear that $h$ is a map
of order two. Hence for an $O(2)$-space
$X$, $(\iota^*X,h)$ is an object of $j^* \# SO(2) \tscr_*$.
Thus we have a functor $I \co O(2) \tscr_* \to j^* \# SO(2) \tscr_*$.

It should be clear that an object of
$j^* \# SO(2) \tscr_*$ defines an $O(2)$-space.
Let $(X,w)$ be an object of the skewed category,
then we let $O(2)$ act on $X$ by defining
$h \co X \to X$ to be the map $w$.
This defines a functor $C \co
j^* \# SO(2) \tscr_* \to  O(2)\tscr_*$.
It is easy to see that $(C,I)$ are an adjoint
equivalence between $j^* \# SO(2) \tscr_*$
and $O(2) \tscr_*$, furthermore $C$ is a
strong monoidal functor.
The choice of $h$ is unimportant,
as one would expect since any two reflections
are conjugate by a rotation.
Different choices of $h$ will give
different adjoint pairs between
$O(2)$-spaces and $j^* \# SO(2) \tscr_*$.
These can be compared using the change of groups
functor on $O(2)$-spaces by considering
conjugation by a rotation as a group
homomorphism of $O(2)$.
Note that for a $SO(2)$-space $Y$ there won't be usually be an
$SO(2)$-map of order two $Y \to j^*Y$.

Now we turn to model structures,
we can put a model structure
on $O(2)$-spaces where a
map $f$ is a weak equivalence
or fibration if and only if
$\iota^* f$ is so in $SO(2)$-spaces.
This model structure is the $\cscr$-model structure
on $O(2) \tscr_*$ as mentioned on \cite[Page 70]{mm02},
for $\cscr$ the family of subgroups of $SO(2)$.
If we call this model structure $\cscr \tscr_*$,
then we can summarise this example in the statement:
there is a strong symmetric monoidal Quillen equivalence
between $j^* \# SO(2) \tscr_*$ and $\cscr \tscr_*$.
\end{ex}

\begin{ex}
We now consider a well-known algebraic example.
We take the following definition from \cite[5.4]{mccr01}:
for a ring $R$ and a group $G$
with $G$ acting on $R$ by $r \to r^g$,
the \textbf{skew group ring}\index{Skew group ring},
$R \# G$, is the free $R$-module with $G$ as a basis and
multiplication defined as $(hr)(gs)=(hg)(r^gs)$.
Note that this ring is not commutative.

Let $R$ be a commutative ring and $w \co R \to R$
a ring map such that $w^2 = \id_R$ and we obtain
$R \# W$. We also have an involution
$w^*$ (pullback along $w$)
on the category of $R$-modules.
It is easy to see that the category of
$R \# W$-modules and $w^* \# (R \leftmod)$
are isomorphic. One can show that
$R \# W$ is a co-associative, co-commutative co-algebra.
Furthermore the co-product ($\Delta$) is compatible with the ring multiplication
and unit. We use this to define a monoidal product,
for $R \# W$-modules $M$ and $N$ their product is
$M \otimes_R N$ with $R \# W$-module structure given by the composite
$$ \begin{array}{rcl}
R \# W \otimes_R (M \otimes_R N)
& \overset{\Delta}{\to} &
(R \# W \otimes_R R \# W) \otimes_R ( M \otimes_R N) \\
& \overset{T}{\to} &
(R \# W \otimes_R M) \otimes_R (R \# W \otimes_R N) \\
& \to &
M \otimes_R N.
\end{array}$$
This monoidal product on $R \# W \leftmod$ corresponds precisely
to the monoidal product on $w^* \# (R \leftmod)$.
We have model categories of $dg R \# W$-modules
and $dg R$-modules, using the projective model structure.
All of the previous material of this example still applies
and the isomorphism of categories between
$dg R \# W$-modules and $w^* \# (dg R \leftmod)$
is an isomorphism of (monoidal) model structures.
We must mention that \cite[Page 504]{ss00}
briefly considers skew group rings, but doesn't mention
the monoidal structure that we have considered and uses
a different kind of model structure.
This example is the reason why we have chosen the notation
$\sigma \# \ccal$ and the name skewed category.
\end{ex}

\begin{ex}
In \cite{atiyahKR} a cohomology theory $KR$
is defined, it relates to vector bundles of the following
form. Let $X$ be a space with $\zz/2$-action
($u \co X \to X$). Consider
a complex vector bundle $E$ over $X$ with a map of order two
$f \co E \to u^* \bar{E}$, where $\bar{E}$ is
the conjugate bundle of $E$.
Then for suitably nice spaces $X$, $KR(X)$
is the Grothendieck group of isomorphism
classes of vector bundles over $X$
with a map of order two $f \co E \to u^* \bar{E}$.
We have an involution $\tau$ on vector bundles
over $X$, given by $\tau E = u^* \bar{E}$.
The group $KR(X)$ is then the Grothendieck
group of isomorphism classes of objects in the
skewed category.
A $\zz/2$-spectrum $KR$ representing this cohomology theory
is constructed in \cite{duggerKR}.
This construction begins by noting that
one can put a $\zz/2$-action on $\zz \times BU$ using the conjugation
action on $U$. One can then use this action to give a $\zz/2$-spectrum
in a similar way to Corollary \ref{cor:SO2toO2}, where
we give $SO(2)$-spectra with extra information the structure of
$O(2)$-spectra.
We are simply noting that these constructions are
similar to those in this thesis; it could be interesting to study the relation
between the skewed categories of vector bundles and $KR$
in greater detail.
\end{ex}

\begin{rmk}
There is an adjunction
$\fun(\ccal \times \dcal, \ecal) \cong  \fun(\ccal, \fun(\dcal, \ecal))$
for (small) categories $\ccal, \ \dcal, \ \ecal$.
Applying this in the case of the group categories $G$ and $H$
(these each have one object and morphisms correspond to group elements),
we see that the category of $G \times H$-objects in $\ccal$,
$\fun (G \times H, \ccal)$, is isomorphic to the category
of $G$-objects in the category of $H$-objects of $\ccal$
$\fun(G, \fun(H, \ccal))$.
Now consider a semi-direct product $G \ltimes H$
(so $G$ acts on $H$). The skewed category
(or rather, its generalisation to general $G$, see Remark \ref{rmk:extendG})
provides us with an equivalence
between $G \# \fun(H, \ccal)$ and
$\fun (G \ltimes H, \ccal)$.
We can also link this example back to skew group rings,
by \cite[1.5.7]{mccr01} there is an isomorphism or rings
$R (G \ltimes H) \cong RH \# G$.
Thus $R (G \ltimes H) \leftmod$ is equivalent to
$(RH \# G) \leftmod$ and by the example above on skew
group rings this is equivalent to $G \# (RH \leftmod)$.
\end{rmk}

\chapter{Cyclic $O(2)$-Spectra}\label{chp:cyclicspectra}
We study cyclic $O(2)$-spectra and relate
this category to rational $SO(2)$-spectra.
The first section considers $SO(2)$-spectra in general
and proves that this category has an involution
(Lemma \ref{lem:SO(2)involution}).
The second section applies this work to
construct an involution on $\iota^* S_\qq \leftmod$,
a particular model for rational $SO(2)$-spectra.
We have chosen this model as every object is fibrant,
which is necessary for Theorem \ref{thm:so2morita}.
The highlight of this chapter is Theorem \ref{thm:so2involution}
which describes cyclic $O(2)$-spectra in
terms of the skewed category of $\iota^* S_\qq \leftmod$.

\section{An Involution on $SO(2)$-Spectra}\label{sec:SO2specinvolute}
We want to understand cyclic $O(2)$-spectra
by proving an analogue of Example \ref{ex:o2spaces}.
This section makes substantial progress in that direction
with Theorem \ref{thm:extrastructure} and Corollary \ref{cor:SO2toO2}.
We will use these in the next section to prove this chapter's
main result: Theorem \ref{thm:so2involution}.
We have to work harder than our space level example to
relate $SO(2)$-spectra to $O(2)$-spectra
since $G$-equivariant spectra
aren't simply $G$-objects in the category of
non-equivariant spectra.

We now need to specify the universe,
so we write $G \mcal (U)$ for
$G$-equivariant $S$-modules indexed on the $G$-universe $U$.
Fix $U$ to be a complete $O(2)$-universe and write
$\iota^*$ for $\iota^*_{SO(2)}$, this is a
functor $O(2) \mcal (U) \to SO(2) \mcal (\iota^* U)$,
note that $\iota^*U$ is a complete $SO(2)$-universe.
We have the change of groups functor
$j^* \co SO(2) \mcal(\iota^* U) \to SO(2) \mcal(j^* \iota^* U)$
induced from $j$ ($j^*\iota^*U$ is $\iota^*U$ with the opposite $SO(2)$ action).
Since an involution must have the domain equal to the codomain,
we cannot just use $j^*$, we need to change back
to the universe $\iota^* U$.
For $G$-universes $U$ and $U'$,
we have a space $\iscr(U, U')$ of linear isometries $U \to U'$,
with $G$ acting by conjugation.
An object of $G \mcal(U)$ is a $G$-May
spectrum with an action of $\iscr(U, U) \ltimes (-)$
that is also an $S$-module.

\begin{definition}
For universes $U$ and $U'$ we define
$I_U^{U'} = \iscr(U, U') \ltimes_{\iscr(U,U)} (-)$
(see \cite[Definition 2.1]{em97} or \cite[Chapter XXIV, Definition 3.4]{may96}).
\end{definition}
There is a strong symmetric monoidal pair
(\cite[Theroem 1.1]{em97} or \cite[Chapter XXIV, Theorem 3.7]{may96})
$$I_U^{U'} : G \mcal(U) \overrightarrow{\longleftarrow}
G \mcal(U') : I_{U'}^U.
\index{$I$@$(I_U^{U'},I_{U'}^{U})$ }$$
So our first candidate for an involution is $I^{\iota^*U}_{j^* \iota^*U} j^*$.
But this is not good enough for our work, so we introduce a more
classical version of change of universe that is easy to work with.
Since we have used an asterisk to denote change of groups
functors we use a $\dag$ to indicate change of universe
functors. The following is \cite[Chapter I, Proposition 2.5]{lms86}.

\begin{definition}
Let $f \co U \to U'$ be a $G$-linear isometry between
universes. For a $G$-spectrum $X$ indexed on $U'$
we define $f^\dag X$ a $G$-spectrum indexed on $U$
by $(f^\dag X)(V) = X(f(V))$, for $V$ an indexing space of $U$.
For $Y$, a $G$-spectrum indexed on $U$,
we let $(f_\dag Y)(W) = Y(f^{-1}(W))$
($W$ an indexing space of $U'$)
and we have $f_\dag Y$ indexed on $U'$.
There is an adjoint pair
$$f_\dag : G \mcal(U) \overrightarrow{\longleftarrow}
G \mcal(U') : f^\dag
$$\index{$F$@$(f_\dag,f^\dag)$}.
\end{definition}

The following result relates these two definitions,
in our case it tells us that we can replace
the complicated functor $I^{\iota^*U}_{j^* \iota^*U}$ with the more
intuitive change of universe functors $f_\dag$ and $f^\dag$.
We also note \cite[Appendix A, Proposition 5.3]{EKMM97}
which states that $f \ltimes X = f_\dag X$ (considering
$f$ as a point in $\iscr(U,U')$).

\begin{lemma}\label{lem:fdagisI}
Let $f \co U \to U'$ be an isomorphism of $G$-universes. Then there are natural isomorphisms
$f_\dag E \cong I_U^{U'} E$ and $f^\dag E' \cong I^U_{U'} E'$
for $E \in G\sscr U[\mathbb{L}]$ and $E' \in G\sscr U'[\mathbb{L}']$.
\end{lemma}
\begin{proof}
This result is taken from \cite[Lemma 3.5]{may98}.
\end{proof}

Any reflection $h \in O(2)$ gives an $SO(2)$-equivariant isomorphism
of universes $h \co  \iota^*U \to j^* \iota^*U$. As $h$ changes one obtains different
isomorphisms, but since any two reflections are conjugate by a rotation
one can see how these maps will be related.
Thus we have a change of universe functor
$h^\dag \co SO(2) \mcal(j^* \iota^*U) \to SO(2) \mcal(\iota^*U)$
that is isomorphic to $I^{j^* \iota^*U}_{\iota^*U}$.
From here onwards we let $h$ be some fixed
reflection in $O(2)$, Theorem \ref{thm:so2involution} will show
that this choice is not important.

\begin{theorem}\label{thm:extrastructure}
For an $O(2)$-spectrum $X$ indexed on an $O(2)$-universe $U$,
we have an $SO(2)$-map (natural in $O(2)$-maps of $X$)
$h \co \iota^*X \to h^\dag j^*\iota^* X$. The composite
$(h^\dag j^* h) \circ h$ is the identity map of $X$.
\end{theorem}
\begin{proof}
We begin with the definition of the map on the levels of
$\iota^* X$, then show this is a map of spectra, a
map of $\iota^* \mathbb{L}$-spectra and of
$SO(2)$-equivariant $S$-modules in turn.
Take an indexing space $\iota^*V$ of $\iota^*U$,
then we have a map
$$h \co (\iota^*X)(\iota^*V)
= \iota^*(X(V)) \longrightarrow j^*\iota^*(X(V))
= (h^\dag j^*\iota^*X)(\iota^*V).$$
The last equality is simple: $V$ is a sub-inner-$O(2)$-product space
of $U$, hence $hV = V$ as sets, but as usual the action of $SO(2)$ is inverted,
so $hV =j^*V$ as $SO(2)$-inner product spaces.
Now we show that this gives a map of spectra:
take $\iota^*W \subset \iota^*V$ and recall that the structure maps
for $h^\dag j^*\iota^*X$ are given by $j^*\iota^* \sigma \circ h \smashprod \id$
as shown below.
$$
\begin{array}{rcl}
S^{\iota^*W-\iota^*V} \smashprod (h^\dag j^*\iota^*X)(\iota^*V) &
\overset{h \smashprod \id}{\longrightarrow} &
S^{j^*\iota^*W-j^*\iota^*V} \smashprod (h^\dag j^*\iota^*X)(\iota^*V) \\
& = &
j^*\iota^*(S^{W-V} \smashprod X(V)) \\
& \overset{j^*\iota^* \sigma}{\longrightarrow} &
j^*\iota^*(X(W)) \\
\end{array}
$$
The relevant equation to check that $h$ is a map of spectra is
$j^*\iota^* \sigma \circ h \smashprod \id \circ \id \smashprod h = h \iota^* \sigma$.
As maps of the underlying sets, $j^*\iota^* \sigma = \iota^* \sigma = \sigma$,
so the above equation is precisely the requirement that the
structure maps $\sigma$ of $X$ are $O(2)$-equivariant.

Now we show that $h$ is a map of $\iota^* \mathbb{L}$-spectra.
We begin with a diagram which may help to explain the
following work.
We use $\rho$ to denote the action of $\iscr(U, U)$ on $X$.
The idea behind a
half-twisted smash product is that for each
$\theta \in \iscr(\iota^*U, \iota^*U)$
we have a map of non-equivariant spectra
$\rho(\theta) \co X \to \theta^\dag X$.
The diagram below commutes
because $\rho$ is an $O(2)$-equivariant map,
noting that $h^\dag \theta^\dag = (\theta h)^\dag$
which is equal to $(h h \theta h)^\dag = (h \theta h)^\dag h^\dag$.
$$
\xymatrix@!C{
X \ar[r]^{\rho (\theta)} \ar[d]^h
& \theta^\dag X \ar[d]^h \\
h^\dag X \ar[r]^{\rho (h \theta h)}
& (\theta h)^\dag X }
$$
From $\rho$ we have $\iota^* \rho$, the action of
$\iscr(\iota^* U, \iota^* U)$ on $\iota^* X$.
We now define $\rho'$, the action of $\iscr(\iota^*U, \iota^*U)$
on $h^\dag j^*\iota^*X$. We have an isomorphism of functors
$$\big(\iscr(\iota^*U, \iota^*U) \times \{h^{-1}\} \big) \ltimes (-)
\cong \iscr(\iota^*U, \iota^*U) \ltimes h^\dag (-). $$
We can apply conjugation by $h$ (more properly by $h^{-1}$)
to obtain a homeomorphism of $SO(2)$-spaces:
$$C_h \co \iscr(\iota^*U, \iota^*U) \times \{h^{-1}\}
\to \{h^{-1}\} \times \iscr(j^*\iota^*U, j^*\iota^*U).$$
We combine these to give the structure map $\rho'$ as follows, using
$j^* \iota^* \rho$, the action of
$\iscr(j^* \iota^* U, j^* \iota^* U)$ on $j^* \iota^* X$.
$$\iscr(\iota^*U, \iota^*U) \ltimes h^\dag j^*\iota^*X \to
\big( \{h^{-1}\} \times \iscr(j^*\iota^*U, j^*\iota^*U) \big) \ltimes j^*\iota^*X
\to h^\dag j^*\iota^*X$$
The following diagram commutes since
$\rho \co \iscr(U, U) \ltimes X \to X$ is an equivariant map
(note that $h$ acts by conjugation on $\iscr(U, U)$).
$$
\xymatrix@!C{
\iscr(\iota^*U, \iota^*U) \ltimes
\iota^*X \ar[r]^(0.6){\iota^*\rho } \ar[d]^{\id \smashprod h} &
\iota^*X \ar[d]^{h} \\
\iscr(\iota^*U, \iota^*U) \ltimes
h^\dag j^*\iota^*X \ar[r]^(0.61){\rho '} &
h^\dag j^*\iota^*X  }
$$

Thankfully we can now turn to easier considerations and
show that $h$ is a map of ${SO(2)}$-equivariant $S$-modules.
Since the functor $h^\dag j^*$ is isomorphic to
the strong monoidal functor
$I^{\iota^*U}_{j^* \iota^*U} j^*$,
$h^\dag j^*\iota^*X$ is a module
over $h^\dag j^*\iota^*S$.
We have a map of ring spectra $h \co \iota^*S \to h^\dag j^*\iota^*S$
and we can use this map to give $h^\dag j^*\iota^*X$ an $\iota^*S$-module structure.
With this in place we have to check that the following diagram commutes.
$$
\xymatrix@!C{
\iota^*S \smashprod \iota^*X \ar[r]^(0.6){\iota^*\nu} \ar[d]^{\id \smashprod h} &
\iota^*X \ar[d]^{h} \\
\iota^*S \smashprod h^\dag j^*\iota^*X \ar[r]^(0.61){\nu'} &
h^\dag j^*\iota^*X  }
$$
Where $\nu' =\iota^* \nu (h \smashprod \id)$
is the $\iota^*S$-action map on $h^\dag j^*\iota^*X$.
It is clear that this diagram commutes
precisely when $\nu$ is $O(2)$-equivariant.
Thus $h \co \iota^*X \to h^\dag j^*\iota^*X$
is a map of $SO(2)$-spectra, we now prove it
is a map of order two.
Consider the composition
$$(h^\dag j^* h) \circ h \co \iota^*(X(V)) \to j^*\iota^*(X(V)) \to \iota^*(X(V))$$
and we note that (on the level of sets) this is simply a double application of the
automorphism $h$ to (the $O(2)$-space) $X(V)$, hence the composite
$(h^\dag j^* h) \circ h$ is the identity map of $X$.
\end{proof}

\begin{corollary}\label{cor:SO2toO2}
An $SO(2)$ spectrum $Y$ with a map
$f \co Y \to h^\dag j^*Y$ such that
$(h^\dag j^* f) \circ f = \id_Y$ can
be given the structure of an $O(2)$ spectrum.
This construction depends naturally
on $Y$ and $f$.
\end{corollary}
\begin{proof}
This is contained in the proof of the above theorem,
see the proof of Theorem \ref{thm:so2involution}
for some details.
\end{proof}

\begin{lemma}\label{lem:hjisquillen}
The functor $h^\dag j^*$ preserves weak equivalences,
cofibrations and fibrations of $G \mcal$.
\end{lemma}
\begin{proof}
All we need to show is that $h^\dag j^*$
takes generating (acyclic) cofibrations
to (acyclic) cofibrations since
the right adjoint of $h^\dag j^*$
is $h^\dag j^*$ itself.
We let $E$ be a generalised sphere spectrum for $SO(2)$.
Inversion gives an isomorphism $SO(2)/H_+ \to j^* SO(2)/H_+$
and hence we have an isomorphism $E \to h^\dag j^*E$. This gives
an isomorphism of maps between the generating cofibration
$E \to CE$ and its image $h^\dag j^* E \to h^\dag j^* CE$.
Thus the map $h^\dag j^* E \to h^\dag j^* CE$
is a cofibration. We have proven that $h^\dag j^*$ preserves cofibrations,
the same argument suffices to show that
it also preserves acyclic cofibrations.
\end{proof}

\begin{lemma}\label{lem:SO(2)involution}
The functor $h^\dag j^*$ is an involution on
$SO(2) \mcal$, hence the category $(SO(2) \mcal, h^\dag j^*)$
is an involutary model category.
\end{lemma}
\begin{proof}
The functor $h^\dag j^*$ preserves the underlying sets of $SO(2)$-spectra,
hence $h^\dag j^* \circ h^\dag j^* =\id$ and we have an involution on
$SO(2) \mcal$. Lemma \ref{lem:hjisquillen}
proves that $h^\dag j^*$ is a left Quillen functor, so we have a model category
with involution.
\end{proof}

\newpage

\section{The category $\tau \# \iota^*S_\qq \leftmod$}\label{sec:SQinvolute}
We prove that the category of cyclic spectra
is Quillen equivalent to the skewed category of
rational $SO(2)$-spectra (Theorem \ref{thm:so2involution}).
Our model for cyclic spectra
will be modules over $S_\qq \in O(2) \mcal(U)$
(for $U$ a complete $O(2)$-universe)
with model structure lifted from
the $\cscr$-model structure on $O(2) \mcal(U)$.
This is written as $S_\qq \leftmod(\cscr)$
and is Quillen equivalent
to $\cscr \mcal_\qq$ by
Proposition \ref{prop:cylicSQmodarecyclicspectra}.
Our model for rational $SO(2)$-spectra will be modules
over $\iota^* S_\qq \leftmod$\label{app:SQmod}
in $SO(2) \mcal(\iota^* U)$.
There is a ring map of order two
$\alpha \co \iota^* S_\qq \to h^\dag j^*\iota^* S_\qq$,
thus we can consider the functor $\tau = \alpha^* h^\dag j^*$,
this will be our involution on
$\iota^* S_\qq \leftmod$.
This section will prove
the claim that cyclic $O(2)$-spectra are rational
$SO(2)$-spectra with extra structure. In detail, this
structure is a map of order two $X \to \tau X$,
so $\tau \# \iota^*S_\qq \leftmod$ will be Quillen equivalent
to $S_\qq \leftmod(\cscr)$. The proof of this statement
is quite long, so we break down the construction of
$\tau \# \iota^*S_\qq \leftmod$ into several results.

We have made this choice of categories so that every object
is fibrant, thus meeting one technical condition of
Theorem \ref{thm:so2morita}.
It is no more difficult to prove
Theorem \ref{thm:so2involution} with this choice
of categories than with any other.
The difficult part of this section is proving that
$h^\dag j^*$ is a monoidal involution. We would want
this result in any case, to know that our description
of cyclic spectra as a skewed category is a \emph{monoidal}
Quillen equivalence.

\begin{lemma}
The identity map is a natural transformation
$h^\dag j^* \to j^* h^\dag$.
\end{lemma}
\begin{proof}
We draw the following diagram which obviously commutes.
$$\xymatrix{
SO(2) \mcal(\iota^* U) \ar[r]^{h^\dag} \ar[d]^{j^*}
& SO(2) \mcal(j^* \iota^* U) \ar[d]^{j^*} \\
SO(2) \mcal(j^* \iota^* U) \ar[r]^{h^\dag}
& SO(2) \mcal( \iota^* U)
}$$
\end{proof}

In the following we will need the category
$G \iscr(U' ; U)$,
as defined in \cite[Appendix A, Section 2]{EKMM97}, for
$G$-universes $U$ and $U'$.
An object $\mathscr{E} $ of this category is a collection
of spectra $\mathscr{E}_V \in G \mcal(U')$
where $V$ runs over indexing spaces of $U$, with
a transitive system of isomorphisms
$\Sigma^{W-V} \mathscr{E}_W \to \mathscr{E}_V$.
Morphisms are then just a family of morphisms
in $G \mcal(U')$ compatible with the structure maps.
We let $X \bar{\smashprod} Y$ denote the external smash product
of a pair of spectra $X$ and $Y$ in $G \mcal(U)$,
it is a spectrum indexed on a universe $U \oplus U$
defined by $X \bar{\smashprod} Y(V \oplus W) = X(V) \smashprod Y(W)$.

\begin{lemma}
The functor $h^\dag j^*$ is a strong monoidal functor.
\end{lemma}
\begin{proof}
Since $h^\dag$ is naturally isomorphic to $I_{j^* \iota^*U}^{\iota^*U}$,
we know that $h^\dag j^*$ is strong monoidal.
We give a direct proof of this, as we will later need
to show that $h^\dag j^*$ is a monoidal involution.
We require a
natural transformation
$m \co h^\dag j^* (-) \smashprod  h^\dag j^* (-)
\longrightarrow  h^\dag j^* (- \smashprod -)$.
Take $SO(2)$-spectra $X$ and $Y$,
we compare ${h^\dag} j^* X \smashprod {h^\dag} j^* Y$
to ${h^\dag} j^* (X \smashprod Y)$.
The object ${h^\dag} j^* X \smashprod {h^\dag} j^* Y$
is defined to be the first term below
and we have an isomorphism
$$\xymatrix{
\iscr (\iota^* U \oplus \iota^* U,\iota^* U)
\ltimes_{\iscr (\iota^* U,\iota^* U)^2 }
({h^\dag} j^* X \bar{\smashprod} {h^\dag} j^* Y)
\ar[d]^\cong \\
\iscr (\iota^* U \oplus \iota^* U,\iota^* U)
\ltimes_{(\iscr (\iota^* U,\iota^* U) \times \{ h\} )^2}
(j^* X \bar{\smashprod} j^* Y).
}$$
We then apply conjugation by $h$ to obtain
an isomorphism from the above to
$$\iscr (\iota^* U \oplus \iota^* U,\iota^* U)
\ltimes_{( \{ h\} \times \iscr (j^* \iota^* U, j^* \iota^* U)  )^2}
(j^* X \bar{\smashprod} j^* Y).
$$
Because this is a coequaliser this is isomorphic to
$$(\iscr (\iota^* U \oplus \iota^* U,\iota^* U) \times \{ h \oplus h  \}  )
\ltimes_{\iscr (j^* \iota^* U,j^*\iota^* U)  ^2}
(j^* X \bar{\smashprod} j^* Y).
$$
Another application of conjugation by $h$ gives the first term below
and then we have the isomorphism
$$\xymatrix{ ( \{h \} \times \iscr (j^* \iota^* U \oplus j^* \iota^* U, j^*\iota^* U)  )
\ltimes_{\iscr (j^* \iota^* U,j^* \iota^* U)  ^2}
(j^* X \bar{\smashprod} j^* Y)
\ar[d]^\cong \\
h^\dag \iscr (j^* \iota^* U \oplus j^* \iota^* U,j^* \iota^* U)
\ltimes_{\iscr (j^* \iota^* U,j^* \iota^* U)  ^2}
(j^* X \bar{\smashprod} j^* Y).
}$$
We combine these maps to obtain
$m_1 \co {h^\dag} j^* X \smashprod {h^\dag} j^* Y \to
h^\dag (j^* X \smashprod j^* Y)$.
We know that the change of groups functor is strong monoidal
hence there is a natural transformation
$m_2 \co {h^\dag} (j^* X \smashprod  j^* Y) \to
h^\dag j^* (X \smashprod Y)$.
The composite of $m_2$ and $m_1$ gives $m$.
\end{proof}

\begin{proposition}
The category $(SO(2) \mcal, h^\dag j^*)$
is an involutary monoidal model category.
\end{proposition}
\begin{proof}
We have already shown that
$(SO(2) \mcal, h^\dag j^*)$
is an involutary model category,
and $h^\dag j^*$ is a strong monoidal functor.
Thus, all that remains is to prove is the involutary condition
$h^\dag j^* m \circ m(h^\dag j^* \smashprod
h^\dag j^*) = \id$ and to specify a map of order two
$S_{SO(2)} \to h^\dag j^* S_{SO(2)}$.
This second condition is simple:
consider the $O(2)$-equivariant sphere spectrum $S_{O(2)}$.
The unit of $SO(2) \mcal$ is $\iota^* S_{O(2)} = S_{SO(2)}$.
Hence, by Theorem \ref{thm:extrastructure},
we have a ring map of order two
$\alpha \co \iota^* S_{O(2)} \to h^\dag j^* \iota^* S_{O(2)}$.

What remains is a technical proof, one that requires us to look
in great detail at the smash product of $SO(2) \mcal$.
We must check the equation:
$h^\dag j^* m \circ m(h^\dag j^* \smashprod
h^\dag j^*) = \id$.
Our method of proof is as follows. Since $m= h^\dag m_2 \circ m_1$,
we prove a
similar condition for each of the factors in turn.
Let $X$ and $Y$ be in $SO(2) \mcal(\iota^* U)$, then
by the proof that $h^\dag$ is a monoidal functor we see that
$$\id_{X \smashprod Y} = h^\dag m_1 \circ m_1 (h^\dag \smashprod h^\dag ) \co
h^\dag h^\dag X \smashprod h^\dag h^\dag Y \longrightarrow
h^\dag h^\dag (X \smashprod Y).$$
We will prove below a similar condition on $m_2$:
$$\id_{X \smashprod Y} = j^* m_2 \circ m_2 (j^* \smashprod j^*) \co
j^* j^* X \smashprod j^* j^* Y \longrightarrow
j^* j^* (X \smashprod Y).$$
Then assuming that $m_1$ and $m_2$ commute, we have the
following commutative diagram.
$$\xymatrix{
h^\dag j^* h^\dag j^* X \smashprod h^\dag j^* h^\dag j^* Y
\ar[r]^= \ar[d]_{m_1} &
h^\dag h^\dag j^* j^* X \smashprod h^\dag h^\dag j^* j^* Y
\ar[d]_{m_1} \\
h^\dag (j^* h^\dag j^* X \smashprod j^* h^\dag j^* Y)
\ar[r]^= \ar[d]_{h^\dag m_2} &
h^\dag (h^\dag j^* j^* X \smashprod h^\dag j^* j^* Y)
\ar[d]_{h^\dag m_1} \\
h^\dag j^* (h^\dag j^* X \smashprod h^\dag j^* Y)
\ar[d]_{h^\dag j^* m_1}
& h^\dag h^\dag (j^* j^* X \smashprod j^* j^* Y)
\ar[d]_{h^\dag h^\dag m_2} \\
h^\dag j^* h^\dag( j^* X \smashprod j^* Y)
\ar[d]_{h^\dag j^* h^\dag  m_2} \ar[r]^=
& h^\dag h^\dag j^*( j^* X \smashprod   j^* Y)
\ar[d]_{h^\dag  h^\dag j^* m_2} \\
h^\dag j^* h^\dag j^* (X \smashprod Y) \ar[r]^=
& h^\dag h^\dag j^* j^* (X \smashprod Y)
}$$
The left hand vertical composite is
$h^\dag j^* m \circ m(h^\dag j^* \smashprod
h^\dag j^*)$, so once we have shown that
$m_1$ and $m_2$ commute and that $m_2$
satisfies the above condition, we will have our result.

We need a more explicit description of $m_2$,
so we write the half-twisted smash product in terms
of a colimit over pairs of indexing
spaces $V$ and $W$ of the spectra
$\mathscr{E}_{V \oplus W} \smashprod  X(V) \smashprod Y(W)$
for a particular
$\mathscr{E} \in SO(2) \iscr(\iota^*U \oplus \iota^* U ; \iota^*U)$.
That is,
$$X \smashprod Y =  \colim_{V \oplus W} \Big( \mathscr{E}_{V \oplus W}
\smashprod  X(V) \smashprod Y(W) \Big)$$
Now we can describe the action of $m_2$ in more detail,
consider
$$j^* X \smashprod j^* Y =  \colim_{V \oplus W}
\Big( j^* (\mathscr{E}_{V \oplus W})
\smashprod  j^* X(V) \smashprod j^* Y(W) \Big)$$
$m_2$ acts on this colimit termwise, using the obvious isomorphism:
$$n \co j^* (\mathscr{E}_{V \oplus W})
\smashprod  j^* X(V) \smashprod j^* Y(W)
\cong j^* \Big ((\mathscr{E}_{V \oplus W})
\smashprod   X(V) \smashprod  Y(W) \Big).$$
Applying this isomorphism twice we obtain an isomorphism
$$n^2 \co j^*j^* (\mathscr{E}_{V \oplus W})
\smashprod  j^*j^* X(V) \smashprod j^*j^* Y(W)
\cong j^*j^* \Big ((\mathscr{E}_{V \oplus W})
\smashprod   X(V) \smashprod  Y(W) \Big).$$
and this map is the identity. Hence the above condition
for $m_2$ holds.
Now we note that the natural transformations $m_1$
and $m_2$ commute, in the description of the smash product above,
the map $m_1$ only acts on the terms $\mathscr{E}_{V \oplus W}$.
Hence by the naturality properties of a colimit,
the claim follows.
Thus, since $m$ is the composite of the commuting maps $m_1$
and $m_2$ it satisfies the involutary condition as desired.
\end{proof}

\newpage

We define $\tau = \alpha^* h^\dag j^*$, an involution on
$\iota^* S_\qq \leftmod$, using the ring map given in the above proof
$\alpha \co \iota^* S_\qq \to h^\dag j^*\iota^* S_\qq$.
\begin{proposition}\label{prop:SQmodinvolutary}
The category of rational $SO(2)$-spectra, $\iota^* S_\qq \leftmod$,
is an involutary symmetric monoidal model category with involution $\tau$.
\end{proposition}
\begin{proof}
This follows from Proposition \ref{prop:skewedmodules}
since $(SO(2) \mcal, h^\dag j^*)$
is an involutary monoidal model category.
\end{proof}

Thus, we have a model category $\tau \# \iota^*S_\qq \leftmod$.
Using notation introduced below we can also write $I(S_\qq) \leftmod$ for
$\tau \# \iota^* S_\qq \leftmod$.
Recall that $h$ is our fixed reflection in $O(2)$,
$j^*$ is the change of groups functor (from the map
$t \to t^{-1}$ of $SO(2)$), $h^\dag$ the change of
universe functor induced by $h$ and
$\alpha \co \iota^* S_\qq \to h^\dag j^*\iota^* S_\qq$
is a ring map. We use the category
$S_\qq \leftmod(\cscr)$, of $S_\qq$-modules in
$O(2) \mcal$, with weak equivalences those maps
which are $\pi_*^H$-isomorphisms for $H \leqslant SO(2)$.

\begin{theorem}\label{thm:so2involution}
There is a strong symmetric monoidal Quillen equivalence
$$C : \tau \# \iota^* S_\qq \leftmod
\overrightarrow{\longleftarrow}
S_\qq \leftmod(\cscr) : I
\index{$C$@$(C,I)$} $$
where $I X = h \co \iota^*X \to \alpha^* h^\dag j^*\iota^*X$, $If = \iota^*f$
and $C$ is defined in the proof.
\end{theorem}
\begin{proof}
Almost all of the work has been done in Theorem \ref{thm:extrastructure}
and Corollary \ref{cor:SO2toO2}.
We are merely going to repeat the arguments to construct a map
of $O(2)$-spectra from a map in the skewed category and then check
that we have a well-behaved adjunction.
But first, let us recap the construction of an $O(2)$-spectrum from an object
of the skewed category. Take $h \co X \to \alpha^* h^\dag j^* X$
($X$ an $SO(2)$-spectrum and $h$ a map of order two)
then $CX(V)= (X \circ \iota^*)(V)=X(\iota^*V)$ as a topological space. We give $CX(V)$ an
$O(2)$-action by letting $h \in O(2)$ our chosen reflection act as
$h \co X(\iota^*V) \to j^*(X(\iota^*V))$. Now we must add the structure of a prespectrum:
$\sigma \co S^{W-V} \smashprod CX(V) \to CX(W)$
is defined by the corresponding structure map of $X$,
that this is an $O(2)$-map is encoded in the fact that $h$ is a map of prespectra.
It is clear that this gives a spectrum (that is, the adjoints of the structure maps
are equivariant homeomorphisms).

We define the $\mathbb{L}$-spectrum action
using the underlying set of the structure map for $X$
and then check that this
gives an $O(2)$-map. Repeating this process twice more
we see that $CX$ is an $O(2)$-equivariant $S$-module and an
$S_\qq$-module.
Now for the action of $C$ on maps,
with the above construction we can simply define
$Cf_V=f_V$ (as a set map) for $f$ a map as below.
$$\xymatrix@C+0.2cm{
X \ar[d]^{h_X} \ar[r]^f & Y \ar[d]^{h_Y} \\
h^\dag j^* X \ar[r]^{h^\dag j^* f} & h^\dag j^* Y }$$
It is routine to check that $Cf$ defines an $O(2)$-equivariant map of
$S_\qq$-modules.
That these are adjoint functors is immediate,
the pair $(C,I)$ is actually an equivalence of categories.
Now we consider the model structures,
for which we need the diagram below
(left adjoints will be on the top and left).
Recall that $O(2) \mcal (SO(2))$
is the category $O(2) \mcal$ with
the $\cscr$-model structure as defined in
Theorem \ref{thm:Fmodelstructure}.
It should be clear that the right adjoints
of the square and triangle commute, hence
so do the left adjoints (up to natural isomorphism).
The pairs $(S_\qq \smashprod (-),U_1)$
and $(\iota^* S_\qq \smashprod (-),U_2)$ are
the free module and underlying spectrum pairs.
$$\xymatrix@R+0.6cm@C+0.2cm{
& \tau \# \iota^* S_\qq \leftmod
\ar@<-0.5ex>[dl]_C
\ar@<-0.5ex>[dr]_{\mathbb{P}} \\
S_\qq \leftmod (\cscr)
\ar@<-0.5ex>[ur]_I
\ar@<-0.5ex>[rr]_{\iota^*}
\ar@<+0.5ex>[d]^{U_1} &&
\iota^* S_\qq \leftmod
\ar@<-0.5ex>[ul]_{\mathbb{D}}
\ar@<-0.5ex>[ll]_{0(2)_+ \smashprod_{SO(2)}(-)}
\ar@<+0.5ex>[d]^{U_2} \\
O(2) \mcal (SO(2))
\ar@<+0.5ex>[u]^{S_\qq \smashprod (-)}
\ar@<-0.5ex>[rr]_{\iota^*} &&
SO(2) \mcal
\ar@<+0.5ex>[u]^{\iota^* S_\qq \smashprod (-)}
\ar@<-0.5ex>[ll]_{0(2)_+ \smashprod_{SO(2)}(-)} }$$
We show that $I$ is a right Quillen functor. Take a fibration or
acyclic fibration $f$ in
$S_\qq \leftmod(\cscr)$, then $If$ is a fibration or acyclic fibration
exactly when $\mathbb{P}If = \iota^*f$ in $\iota^* S_\qq \leftmod$ is.
Now $\iota^*f$ is a fibration or acyclic fibration in
$\iota^* S_\qq \leftmod$ precisely when it is so in
$SO(2) \mcal$. Since
$U_1 \co S_\qq \leftmod(\cscr) \to O(2) \mcal(SO(2))$ and
$\iota^* \co O(2) \mcal(SO(2)) \to SO(2) \mcal $
are a right Quillen functors, this result follows.
Now we must show that $(C,I)$ is a Quillen equivalence,
the above argument shows that
$I$ detects and preserves all weak equivalences
(since $U_1$ and $\iota^* \co O(2) \mcal(SO(2)) \to SO(2) \mcal $
do so). In fact, $C$ also preserves all weak equivalences,
since $\iota^* \circ C = \mathbb{P}$.
Now we show that the composite
$X \to ICX \to I \fibrep C X$ is a weak equivalence in
$\tau \# \iota^* S_\qq \leftmod$, the first arrow is an isomorphism,
the second is $I$ applied to a weak equivalence,
hence a weak equivalence.

We prove that the left adjoint $C$ is strong symmetric monoidal. Let $(X,u)$
and $(Y,v)$ be objects of the skewed category.
The underlying $SO(2)$-spectra of $C (X,u) \smashprod C (Y,v) $
and $C( (X,u) \smashprod (Y,v))$ are clearly isomorphic
and given by $\iscr(U \oplus U, U) \ltimes_{\iscr(U,U)^2} X \bar{\smashprod} Y$.
The action of $O(2)$ in each case is given by conjugation on
$\iscr(U \oplus U, U)$, by $u$ on $X$ and by $v$ on $Y$.
Thus we see that these objects are isomorphic as $O(2)$-spectra.
The symmetry statement is clear
and $C(\iota^* S_\qq, \alpha) \cong S_\qq$
so $C$ is strong symmetric monoidal.
It remains to check that
if $(\cofrep \iota^* S_\qq, \cofrep \alpha)$ is a cofibrant replacement
of the unit, then
$C (\cofrep \iota^* S_\qq,\cofrep \alpha) \to C (\iota^* S_\qq,\alpha)$
is a weak equivalence. This holds since $C$ preserves all weak equivalences.
\end{proof}

\begin{theorem}\label{thm:spectralalgebras}
The categories $\mcal$, $SO(2) \mcal$, $O(2) \mcal$,  $\iota^* S_\qq \leftmod$ and
$S_\qq \leftmod$
are all closed symmetric $Sp^\Sigma_+$-algebras.
\end{theorem}
\begin{proof}
Each of these categories is a closed symmetric monoidal category,
hence they are algebras over themselves, we can now use the following
series of adjoints to pull this structure back to
$Sp^\Sigma_+$, symmetric spectra of simplicial sets with
the positive model structure.
The geometric realisation and
singular complex adjunction $(|-|, \sing)$\index{$(|-|, \sing)$}
is a strong symmetric monoidal Quillen equivalence between
$Sp^\Sigma_+$ and $Sp^\Sigma_+ (\tscr_*)$\label{app:topsymspec} (\cite[Figure 7.1]{ss03monequiv}).
We have the strong symmetric monoidal Quillen equivalence
$$\mathbb{P} : Sp^\Sigma_+ (\tscr_*) \overrightarrow{\longleftarrow}
\iscr \sscr_+ : \mathbb{U}$$
from \cite{mmss01}.
Combining these gives $\iscr \sscr_+$ the structure of a closed
symmetric $Sp^\Sigma_+$-algebra. Thus we can use the adjunction
$(\nn, \nn^\#)$ to see that $\mcal$ is a closed
symmetric $Sp^\Sigma_+$-algebra.

When working $G$-equivariantly we use the composite functor
$i_* \varepsilon^*_G \co \iscr \sscr_+ \to \GIS_+$
as defined in \cite[Chapter V, Proposition 3.4]{mm02}, which
states that this functor is part of a Quillen pair $(i_* \varepsilon^*_G, (i^* (-))^G)$.
This is a strong symmetric monoidal adjunction, as
noted on \cite[Chapter V, Page 80]{mm02}.
Finally, for $S_\qq \leftmod$ and $\iota^* S_\qq \leftmod$ we use the free
module functor as well. We illustrate for $S_\qq \leftmod$ below.
$$ \xymatrix@C+0.6cm{
Sp^\Sigma_+
\ar@<+0.4ex>[r]^(0.6){\mathbb{P} \circ |-|} &
\iscr \sscr_+
\ar@<+0.4ex>[r]^(0.5){i_* \varepsilon^{*}_{O(2)}}
\ar@<+0.4ex>[l]^(0.4){\sing \circ \mathbb{U}} &
O(2) \iscr \sscr_+
\ar@<+0.4ex>[r]^(0.6){\nn }
\ar@<+0.4ex>[l]^(0.5){(i^* (-))^{O(2)} } &
O(2) \mcal
\ar@<+0.4ex>[r]^(0.5){S_\qq \smashprod (-)}
\ar@<+0.4ex>[l]^(0.4){\nn^\# } &
{S_\qq} \leftmod
\ar@<+0.4ex>[l]^(0.5){U}
 } $$
For a pair of $S_\qq$-modules $X$ and $Y$, the
$Sp^\Sigma_+$-function object
$\underhom (X,Y) \in Sp^\Sigma_+$
is given by $\sing \mathbb{U} (i^* \nn^\# U F_{S_\qq}(X,Y))^{O(2)}$.
\end{proof}

\begin{lemma}\label{lem:inflationandinvolution}
The Quillen equivalence $(C,I)$ is an adjunction of closed
symmetric $\id \# Sp^\Sigma_+$-algebras.
\end{lemma}
\begin{proof}
The composite functor $i_{SO(2)} \co Sp^\Sigma_+ \to \iota^* S_\qq \leftmod$,
as defined above, is an involutary functor.
This can be seen most easily by noting that
$\iota^*_{SO(2)} \circ i_{O(2)} \cong i_{SO(2)}$
(where $i_{O(2)}$ is the functor
$Sp^\Sigma_+$ to $S_\qq \leftmod$ as constructed above)
and applying Theorem \ref{thm:extrastructure}, see Lemma \ref{lem:tauspectral}.
It follows that $\tau \# \iota^* S_\qq \leftmod$ is a
closed symmetric $\id \# Sp^\Sigma_+$-algebra.
The result then follows since $O(2) \mcal (SO(2))$
is a closed symmetric $\tau \# \iota^* S_\qq \leftmod$-algebra.
\end{proof}

\chapter{Understanding $ \tau \# \iota^* S_\qq \leftmod$}\label{chp:WskewSQmod}
This chapter begins the work of proving that
the methods of \cite{greshi} are compatible with
the involution on $SO(2)$-spectra.
We are able to prove that there is a zig-zag of involutary Quillen
equivalences between $\iota^* S_\qq \leftmod$
and a category $\rightmod \ecal_t$. Thus we
have extended \cite[Theorem 4.1]{greshi}
(which is similar to Theorem \ref{thm:finiteEtopisEt})
to the case of cyclic $O(2)$-spectra.
Furthermore, it should be possible to continue this work
and extend the rest of \cite{greshi} to the case of cyclic spectra,
see Remark \ref{rmk:nextsteps}.
We begin by proving that $ \rightmod \ecal_{top} $
is a category with involution.
In Section \ref{sec:invmoritaequiv} we prove that
the Morita equivalence
of Theorem \ref{thm:monoidalmorita} is involutary in
the case of $SO(2)$-spectra. In
Section \ref{sec:movingaccross the machine}
we then prove that the functors of \cite{shiHZ} are
compatible with the involutions. The last section is another Morita
equivalence, which is involutary by the work of the second section.
The conclusion of this chapter is Corollary \ref{cor:cyclicsummary}.

\section{An Involution on $ \rightmod \ecal_{top} $}\label{sec:Etopinv}

The method of \cite{greshi} begins by replacing
rational $SO(2)$-spectra by $\rightmod \ecal_{top}$,
which may be called the category of
topological $SO(2)$-Mackey functors.
We show that this category has an involution in
Proposition \ref{prop:modEtopinvolution}.

Consider the homogenous spaces $\Sigma^\infty SO(2)/H_+$ where $H$ runs over all subgroups of
$SO(2)$, these are a set of generators for $SO(2) \mcal$.
Now we smash these with the rational $SO(2)$-equivariant sphere spectrum,
$\iota^* S_\qq$, to obtain a set of generators for $\iota^* S_\qq \leftmod$.
We can now apply the idempotents and take cofibrant
replacements (in the category of $\iota^*S_\qq$-modules) to
obtain the basic cells $\sigma_H = \cofrep e_H (SO(2) / H_+) \smashprod \iota^* S_\qq$,
since all spectra are fibrant in this model category, we have a set of
cofibrant-fibrant objects which we call $\mathcal{BC}$.
The following result implies that this collection
is a generating set.

\begin{lemma}\label{lem:build}
The homogenous spaces can be obtained from the basic cells.
$$\Sigma^\infty SO(2)/H_+ \smashprod \iota^* S_\qq
 \simeq \bigvee_{K \subseteq H} \sigma_K$$
\end{lemma}
\begin{proof}
See \cite[Lemma 2.1.5]{gre99}.
\end{proof}

We have an alternative construction of the basic cells, which will be of use later.
We will perform some of this work in the categories
of rational $D_{2n}$-spectra and rational $C_n$-spectra.
By similar arguments to the $SO(2)$ case we can
use the inversion map of $C_n$ to create an involution on
rational $C_n$-spectra. We can then construct a
model category of cyclic $D_{2n}$-spectra and see that this
is Quillen equivalent to the skewed category
$\tau \# C_n \mcal_\qq$.
There is an idempotent
$e^{D_{2n}}_{C_n} \in [S,S]^{D_{2n}}_\qq$
corresponding to the subgroup $C_n$ of $D_{2n}$.
This gives an idempotent
in $[S \smashprod EW_+,S \smashprod EW_+]^{D_{2n}}_\qq$.
This group is isomorphic to
self maps of $S$ in the homotopy category
of cyclic $D_{2n}$-spectra.
Thus, we have an idempotent $e_{C_n}' \in
\ho (\tau \# C_n \mcal_\qq)(S,S)$,
maps in the homotopy category of the
skewed category of $C_n$-spectra
from the unit to itself.

We can choose a map
in the skewed category of $C_n$-spectra representing $e_{C_n}'$.
By applying $SO(2) \smashprod_{C_n} (-)$ we obtain
an idempotent map (up to homotopy) of
$SO(2)/C_n \to h^\dag j^* SO(2)/C_n$ in $\tau \# SO(2) \mcal$.
We take homotopy colimits to construct the skewed object
$e_{C_n} SO(2)/C_n \to h^\dag j^* e_{C_n} SO(2)/C_n$.
By smashing with $(\iota^* S_\qq, \alpha)$ and taking a cofibrant
replacement we obtain a basic cell
$w_H \co \sigma_H \to \tau \sigma_H$ in
$\tau \# \iota^* S_\qq \leftmod$. Since a cofibrant
object of the skewed category is cofibrant in the underlying category,
$\sigma_H$ is a construction of a basic cell for
$\iota^* S_\qq \leftmod$.

Now define $\bar{\mathcal{BC}}$ to be the closure of $\mathcal{BC}$
under smash products with the unit included. By monoidality all non-unit objects are
cofibrant and all objects are fibrant. Recall that we are using the smash product of
$\iota^* S_\qq$-modules for this definition.
The full subcategory of $\iota^* S_\qq \leftmod$ with object set
$\gcal_{top}:=\bar{\mathcal{BC}}$ will be denoted $\ecal_{top}$.

\begin{theorem}\label{thm:so2morita}
The Quillen pair
$$(-) \smashprod_{\ecal_{top}} \mathcal{G}_{top} :
\iota^* S_\qq \leftmod \overrightarrow{\longleftarrow}
\rightmod \ecal_{top} : \underhom(\mathcal{G}_{top},-)$$
is a strong symmetric monoidal Quillen equivalence.
\end{theorem}
\begin{proof}
This is part of \cite[Theorem 4.1]{greshi}, we give some details of the proof.
The result is essentially an application of Theorem \ref{thm:monoidalmorita}
but we must adjust the proof slightly since now the unit $\iota^* S_\qq \in \gcal_{top}$
is not cofibrant.
The functor $\underhom_{\iota^* S_\qq} (\iota^* S_\qq, -)$
preserves fibrations and all weak equivalences (since every object of
$\iota^* S_\qq \leftmod$ is fibrant),
hence the above adjunction is a Quillen pair.
It is a Quillen equivalence by the same arguments
of \cite[Theorem 3.9.3]{ss03stabmodcat}, with the following alterations.

The free modules $F_\sigma$ are no longer cofibrant, however, as mentioned in
Theorem \ref{thm:monmodules}
$\cofrep \iota^* S_\qq \smashprod_{\iota^* S_\qq} F_\sigma$
is a cofibrant replacement. The left
derived functor, $(-) \smashprod_{\ecal_{top}}^L \mathcal{G}$,
takes $F_\sigma$ to $\cofrep \iota^* S_\qq \smashprod_{\iota^* S_\qq} \sigma$.
Since $\sigma$ is either $\iota^* S_\qq$ or cofibrant, this
is weakly equivalent to $\sigma$.
We also note that since $\iota^* S_\qq \leftmod$ is a monoidal model category,
the map $\cofrep \iota^* S_\qq \smashprod_{\iota^* S_\qq} M
\to \iota^* S_\qq \smashprod_{\iota^* S_\qq} M$ is a weak equivalence
for any cofibrant module $M$. Thus, by \cite[Lemma 4.2.7]{hov99}
the map $\underhom_{\iota^* S_\qq} (\iota^* S_\qq, M) \to
\underhom_{\iota^* S_\qq} (\cofrep \iota^* S_\qq, M)$
is a weak equivalence for all $\iota^* S_\qq$-modules $M$.
Hence $\underhom_{\iota^* S_\qq} (\iota^* S_\qq, M)$
has the correct homotopy type.
\end{proof}

\begin{lemma}\label{lem:tauspectral}
The functor $\tau= \alpha^* h^\dag j^*$ is a spectral functor,
moreover $(\tau, \tau)$ is an adjunction
of closed symmetric monoidal spectral functors.
\end{lemma}
\begin{proof}
We prove that $\tau$ is an adjunction of closed
symmetric monoidal $Sp^\Sigma_+$-algebras.
Let $K$ be a symmetric spectrum, then
$K \mapsto \iota^* S_\qq \smashprod \iota^* \nn i_* \varepsilon^*_{O(2)} \mathbb{P} |K|$
defines a symmetric monoidal Quillen functor from
$Sp^\Sigma_+$ to $\iota^* S_\qq \leftmod$.
Furthermore, there is a natural isomorphism
of order two
$\iota^* S_\qq \smashprod \iota^* \nn i_* \varepsilon^*_{O(2)} \mathbb{P} |K|
\to
\alpha^* h^\dag j^* (\iota^* S_\qq \smashprod \iota^* \nn i_* \varepsilon^*_{O(2)} \mathbb{P} |K|)$,
this comes from the map $\alpha$ on $\iota^* S_\qq$ and
the $O(2)$ structure of $\nn i_* \varepsilon^*_{O(2)} \mathbb{P} |K|$.
\end{proof}

\begin{corollary}
The functor $\tau$ is a self-inverse map of
ringoid spectra $\tau \co \ecal_{top} \to \tau \ecal_{top}$.
\end{corollary}
\begin{proof}
We define the set $\tau \ecal_{top}$ to be the full subcategory of $\iota^* S^\qq \leftmod$
with object set $\tau \gcal_{top}$. The result is then obvious and we do not introduce
any new notation for the inverse functor $\tau \co \tau \ecal_{top} \to \ecal_{top}$.
\end{proof}

\begin{lemma}
There is an invertible map of ringoid spectra $W \co \ecal_{top} \to \tau \ecal_{top}$.
\end{lemma}
\begin{proof}
This is where we use our new construction of the basic cells
which come with maps of order two: $w_H \co \sigma_H \to \tau \sigma_H$.
On objects, $W$ acts as $\tau$, so $W \sigma = \tau \sigma$. On
the homomorphism spectra $W$ acts as
$$\underhom_{Sp^\Sigma}(\tau w, w')
\co \underhom_{Sp^\Sigma}(\sigma, \sigma')
\to \underhom_{Sp^\Sigma}(\tau \sigma, \tau \sigma')$$
(recall that $\sigma$ is some smash product
of the basic cells and we have defined $w_H$ for
each basic cell $\sigma_H$). It should be obvious that this
defines a ringoid map.
We denote inverse of this map as $W^{-1}$
and this acts $\tau$ on objects and
acts as on homomorphism spectra as $\underhom_{Sp^\Sigma}(w,\tau w')$.
\end{proof}

To simplify our notation we now write $\underhom(X,Y)$ in the place of
$\underhom_{Sp^\Sigma}(X,Y)$, for $X$ and $Y$ in
$\iota^* S_\qq \leftmod$. We also write
$\ecal_{top}(a,b)$ for $\underhom_{Sp^\Sigma}(a,b)$
when $a,b \in \gcal_{top}$

\begin{lemma}\label{lem:TWaremonoidal}
The functors $\tau$ and $W$ as defined above are morphisms
of symmetric monoidal $Sp^\Sigma_+$-categories.
\end{lemma}
\begin{proof}
Because $\tau$ is a strong monoidal functor and
by the definition of $W$ it is obvious that these maps are compatible with
the monoidal structure on $\ecal_{top}$.
Hence we have a commuting diagram for $F=\tau$
and $F=W$.
$$\xymatrix{
\ecal_{top}(\sigma_2, \sigma_3)
\smashprod
\ecal_{top}(\sigma_1, \sigma_2)
\ar[r] \ar[d]^F &
\ecal_{top}(\sigma_1, \sigma_3)
\ar[d]^F \\
(\tau \ecal_{top})(\tau \sigma_2, \tau \sigma_3)
\smashprod
(\tau \ecal_{top})(\tau \sigma_1, \tau \sigma_2)
\ar[r] &
(\tau \ecal_{top})(\tau \sigma_1, \tau \sigma_3)
}$$
\end{proof}

\begin{definition}
We define an involution on $\rightmod \ecal_{top} $
by $\rho=(\tau W)^*$.
That this functor is self-inverse follows immediately
from the relation
$$\tau \circ \ecal_{top}(\tau w, w')=
\ecal_{top}(w, \tau w') \circ \tau.$$
\end{definition}

\begin{proposition}\label{prop:modEtopinvolution}
The pair $(\rightmod \ecal_{top}, \rho)$ give a
monoidal model category with involution that satisfies the monoid axiom.
\end{proposition}
\begin{proof}
Proposition \ref{prop:monoidalspectralfunctor} and
Lemma \ref{lem:TWaremonoidal} shows that $\rho$ is a
strong monoidal functor.
We must prove that $\rho$ is involutary monoidal
and then the rest follows via the machinery
of involutary categories. We must prove that
for $\ecal_{top}$-modules $M$ and $N$,
the map $\rho M \square \rho N \to \rho(M \square N)$
is a map of order two. We draw
$(\rho M \square \rho N)(x)$
as the coequaliser of the diagram below. The left hand vertical map
is induced by the $\ecal_{top}$-action map of $M$ and $N$
and the right hand map by
the monoidal product and composition of $\ecal_{top}$.
$$\xymatrix{ {\bigvee_{a,b,c,d \in \gcal_{top}}}
\rho M(a) \smashprod  \rho N(b) \smashprod
\ecal_{top}(a,c) \smashprod \ecal_{top}(b,d)
\smashprod \ecal_{top}(x, c \smashprod d)
\ar@<+0.8ex>[d] \ar@<-0.8ex>[d] \\
{\bigvee_{e,f \in \gcal_{top}}}
\rho M(e) \smashprod  \rho N(f) \smashprod
\ecal_{top}(x, e \smashprod f)
}$$
Note that $\rho M(e) =M(e)$ as symmetric spectra,
the $\rho$ is to indicate that the
$\ecal_{top}$-action is different.
Our map from this expression to $\rho(M \square N)$
is induced by the maps
$\id_M \smashprod \id_N \smashprod \tau W_{(a,c)}
\smashprod \tau W_{(b,d)} \smashprod \tau W_{(x,c \smashprod d)}$
and $\id_M \smashprod \id_N \smashprod \tau W_{(x,e \smashprod f)}$.
It is clear from the description that our
involution is monoidal.
\end{proof}

\section{An Involutary Morita Equivalence}\label{sec:invmoritaequiv}

In Proposition \ref{prop:skewequivSQtoEt}
we prove that the Morita equivalence
is involutary, so that the category of cyclic $O(2)$-spectra
is Quillen equivalent to the skewed category
of $\rightmod \ecal_{top}$.

\begin{lemma}
The following square commutes up to a natural isomorphism,
$\beta$. The pair $(\underhom(\gcal_{top},-),\beta)$
define an involutary functor.
$$\xymatrix@C+1.2cm{
\iota^* S_\qq \leftmod
\ar[r]^{\underhom(\gcal_{top},-)}
\ar[d]^\tau &
\rightmod \ecal_{top}
\ar[d]^\rho \\
\iota^* S_\qq \leftmod
\ar[r]^{\underhom(\gcal_{top},-)} &
\rightmod \ecal_{top } }$$
\end{lemma}
\begin{proof}
We consider an $\iota^* S_\qq$-module $X$, moving along the top this gives
the module $\rho \underhom(-, X)$, the bottom route
produces $\underhom(-, \tau X)$.
We define a natural isomorphism
$\beta' \co \rho \underhom(-, \tau X) \to \underhom(-, X)$
by $\underhom(w,\id_X) \circ \tau$.
Naturality of $\beta'$ is clear
and since the diagram below obviously commutes,
$\beta'$ is a map of $\ecal_{top}$-modules.
The top horizontal composition is the action
of $\ecal_{top}$ on $\rho \underhom(-, \tau X)$
and the bottom is the action of
$\ecal_{top}$ on $\underhom(-, X)$.
From $\beta'$ we have a natural transformation
$\beta \co \underhom(-, \tau X) \to \rho \underhom(-, X)$.
Since applying $\beta'$ twice gives the
identity map $(\underhom(\gcal_{top},-),\beta)$
is an involutary functor.

$$ \xymatrix{
\underhom(A, \tau X) \smashprod \ecal_{top} (B,A)
\ar[r]^{\id \smashprod \tau \circ W}
\ar[d]^{\tau \smashprod \id} &
\underhom(A, \tau X) \smashprod \ecal_{top} (B,A)
\ar[r] \ar[d]^{\tau \smashprod \tau} &
\underhom(B, \tau X)
\ar[d]^\tau \\
\underhom(\tau A, X) \smashprod \ecal_{top} (B,A)
\ar[d]^{w_A^* \smashprod \id} \ar[r]^{\id \smashprod W} &
\underhom(\tau A,  X) \smashprod (\tau \ecal_{top}) (\tau B, \tau A)
\ar[r]
&
\underhom(\tau B, X)
\ar[d]^{w_B^*} \\
\underhom(A, X) \smashprod \ecal_{top} (B,A)
\ar[rr] &&
\underhom(B,X) }$$
\end{proof}

Now it follows by Lemma \ref{lem:onlyneedhalf}
that the left adjoint to $\underhom(\gcal_{top},-)$ is an involutary functor.
We construct a natural transformation $\alpha$
to prove this directly, since this type of construction
will occur again.
Pick $M$, an object of $\rightmod \ecal_{top}$,
then $M \smashprod_{\ecal_{top}} \gcal_{top}$
is given by the coequaliser in Definition \ref{def:moritafunctors}. The
natural transformation
$$\alpha \co \rho (-) \smashprod_{\ecal_{top}} \gcal_{top}
\longrightarrow \tau((-) \smashprod_{\ecal_{top}} \gcal_{top})$$
is defined by the map of coequalisers given in the diagram below,
where $\mu$ is the action of $\ecal_{top}$ on $M$.
$$ \xymatrix@C+1cm{
{\bigvee}_{g,h \in \mathcal{G}}
M(h) \smashprod \ecal_{top}(g,h) \smashprod g
\ar@<+0.5ex>[r]^(0.6){\mu \circ \tau \circ W_{g,h}}
\ar@<-0.5ex>[r]_(0.6){eval}
\ar[d]_{\id \smashprod \tau \circ W_{g,h} \smashprod w_g} &
{\bigvee}_{g \in \mathcal{G}}
M(g) \smashprod g
\ar[d]^{\id \smashprod w_g} \\
{\bigvee}_{g,h \in \mathcal{G}}
M(h) \smashprod \ecal_{top}(g,h) \smashprod \tau g
\ar@<+0.5ex>[r]^(0.6){\mu}
\ar@<-0.5ex>[r]_(0.6){\tau (eval)\circ \gamma} &
{\bigvee}_{g \in \mathcal{G}} M(g) \smashprod \tau g
}$$
To prove that the above does define a map of coequalisers
it suffices to show that the two diagrams below commute.
$$ \xymatrix@C+1cm{
M(h) \smashprod \ecal_{top}(g,h)
\ar[r]^(0.6){\mu \circ \tau \circ W_{g,h}}
\ar[d]_{\id \smashprod \tau \circ W_{g,h}} &
M(g) \ar[d]^\id &
\ecal_{top}(g,h) \smashprod g
\ar[r]^(0.6){eval}
\ar[d]_{\tau \circ W_{g,h} \smashprod w_g} &
h
\ar[d]_{w_h} \\
M(h) \smashprod \ecal_{top}(g,h)
\ar[r]^(0.6){\mu } &
M(g) &
\ecal_{top}(g,h) \smashprod \tau g
\ar[r]^(0.6){\tau (eval) \circ \gamma}
& \tau h
}$$
The left hand diagram automatically commutes.
To see that the right hand diagram commutes
we need Lemma \ref{lem:evalsquare}. Which we have included
to show how the various maps in the diagram are related.
Note that this natural transformation $\alpha$
clearly satisfies the necessary
condition for it to be an involutary natural transformation.

\newpage

\begin{lemma}\label{lem:evalsquare}
The diagram below commutes for all $g \in \gcal_{top}$ and
is natural in $\iota^* S_\qq$-modules $X$.

$$ \xymatrix@u{
& \tau g \smashprod \underhom(g,X)
\ar@{}[drrrr]|(0.375){(2)}
\ar[dl]^{\gamma}
\ar[dr]^{\id \smashprod \tau}
\ar[rrr]^{\tau w_g \smashprod \id} &&&
g \smashprod \underhom(g,X)
\ar[dl]^{\id \smashprod \tau}
\ar[dr]^{\id \smashprod (\tau w_g)^*} \\
\tau(g \smashprod \underhom(g,X))
\ar@{}[rr]|{(4)}
\ar@{}[drrrr]|(0.625){(3)}
\ar[dr]^{\tau (eval)} &&
\tau g \smashprod \underhom(\tau g,\tau X)
\ar[dl]^{eval}
\ar[r]^{\tau w_g \smashprod \id} &
g \smashprod \underhom(\tau g,\tau X)
\ar@{}[rr]|{(1)}
\ar[dr]^{\id \smashprod (w_g)^*} &&
g \smashprod \underhom(\tau g,X)
\ar[dl]^{\id \smashprod \tau} \\
& \tau X
&&&
g \smashprod \underhom(g,\tau X)
\ar[lll]^{eval} }$$
\end{lemma}
\begin{proof}
The map labelled $\gamma$ is the natural transformation
between $\tau X \otimes A  \to \tau (X \otimes A)$
for a symmetric spectrum $A$ and a rational
$\iota^* S_\qq$-module $X$.
Squares (1) and (2) obviously commute, square (4)
commutes since $\tau$ is a spectral functor
(see \cite[Proposition 6.4.5]{bor94}).
The fact that Square (3) commutes requires a little more consideration
but essentially follows from axiom that evaluation and
composition are compatible.
\end{proof}

\begin{proposition}\label{prop:skewequivSQtoEt}
The adjunction
$((-) \smashprod_{\ecal_{top}} \mathcal{G}_{top},
\underhom(\mathcal{G}_{top},-) )$
is involutary.
\end{proposition}
\begin{proof}
We must prove that the unit $\eta$ and
counit $\varepsilon$ are involutary natural
transformations. For the counit, let $X$ be an
$\iota^* S_\qq$-module. We must prove that the following
diagram commutes.

$$\xymatrix@C+1cm{
{\int^{a \in \gcal_{top}} \underhom(a, \tau X) \smashprod a}
\ar[d]_{\int^{a \in \gcal_{top}} \beta_X(a) \smashprod \id_a} \ar[r]^{\varepsilon_{\tau X}}
& \tau X  \\
{\int^{a \in \gcal_{top}} \rho \underhom(a, X) \smashprod a}
\ar[r]^{\alpha_{\underhom(\gcal_{top}, X)}}
& \tau {\int^{a \in \gcal_{top}} \underhom(a, X) \smashprod a}
\ar[u]^{ \tau(\varepsilon_X) }
}$$
This is routine to check using Lemma \ref{lem:evalsquare}.
Now we consider the unit, let $M$ be an $\ecal_{top}$-module.
We must prove that the diagram below commutes.
$$\xymatrix@C+1cm{
\rho M
\ar[r]^{\eta_{\rho M}}
&
{\underhom( \gcal_{top}, \int^g \rho M(g) \smashprod g)}
\ar[d]_{\underhom( \gcal_{top},\alpha_{M})} \\
{\rho \underhom( \gcal_{top}, \int^g M(g) \smashprod g)}
\ar[u]^{\rho \eta_{M}}
&
{\underhom( \gcal_{top}, \tau \int^g M(g) \smashprod g)}
\ar[l]_{\beta_{\int^g M(g) \smashprod g}}
}$$
Checking that this diagram commutes reduces to proving that
the diagram below commutes.
$$ \xymatrix@C+1cm{
M(k)
\ar[r]^{coeval}
\ar[d]_{coeval} &
\underhom(k, k \smashprod M(k))
\ar[d]^{(w_k \smashprod \id)_*} \\
\underhom(k, k \smashprod M(k))
\ar[r]^{\gamma^{-1} \circ (w_k)^* \circ \tau} &
\underhom(k, \tau k \smashprod M(k)) } $$
This is the coevaluation version of Lemma \ref{lem:evalsquare}
and it commutes since the pair of squares below commute.
$$ \xymatrix@R-0.2cm@C+0.1cm{
M(k)
\ar[r]^{coeval_k}
\ar[d]_{coeval_{\tau k}}
&
\underhom(k, k \smashprod M(k))
\ar[d]^{(w_k \smashprod \id)_*} \\
\underhom(\tau k, \tau k \smashprod M(k))
\ar[r]^{(w_k)^*}
&
\underhom(k, \tau k \smashprod M(k)) \\
M(k)
\ar[r]^{coeval_k}
\ar[d]_{coeval_{\tau k}}
&
\underhom(k, k \smashprod M(k))
\ar[d]^\tau \\
\underhom(\tau k, \tau k \smashprod M(k))
\ar[r]^{\gamma_*}
&
\underhom(\tau k, \tau (k \smashprod M(k))) } $$
\end{proof}

\begin{corollary}\label{cor:invSO2morita}
There is a strong symmetric monoidal Quillen equivalence
between the skewed categories
$$\tau \# (\iota^* S_\qq \leftmod)
\overrightarrow{\longleftarrow}
\rho \# (\rightmod \ecal_{top}).$$
Furthermore this is an adjunction of
closed symmetric $\id \# Sp^\Sigma_+$-algebras.
\end{corollary}
\begin{proof}
The results of this section prove that the Quillen equivalence
of Theorem \ref{thm:so2morita} is an involutary Quillen equivalence.
Furthermore this is a monoidal involutary adjunction,
though we omit the routine proof that
the natural transformation $\beta$ is a monoidal
natural transformation.
It remains to prove that in the diagram below (which gives the algebra
structure on the skewed categories)
the right adjoints (shown on the bottom)
commute up to natural isomorphism.
\begin{displaymath}
\xymatrix@R+0.5cm{
\rho \# (\rightmod \ecal_{top})
\ar@<+0.31ex>[rr]
\ar@<-0.31ex>[dr]_(0.45){\ev_{\iota^* S_\qq}} & &
\tau \# \iota^* S_\qq \leftmod
\ar@<+0.31ex>[ll]
\ar@<0.31ex>[dl]^j \\
& \id \# Sp^\Sigma_+
\ar@<-0.31ex>[ul]_(0.35){F_{\iota^* S_\qq}}
\ar@<0.31ex>[ur]^i }
\end{displaymath}
Since the corresponding statement is true for the underlying categories
all we need check is that the involutions are compatible, in the sense that
for $x \in Sp^\Sigma_+$ and $E \in \iota^* S_\qq \leftmod$
the following diagrams commute.
\begin{displaymath}
\xymatrix@R-0.0cm@C+0.3cm{
{\int^g} (\underhom(g, \iota^* S_\qq) \smashprod x) \smashprod g
\ar[r]
\ar[d] &
\iota^* S_\qq \smashprod x
\ar[dd] \\
{\int^g} \rho(\underhom(g, \iota^* S_\qq) \smashprod x) \smashprod g
\ar[d] \\
\tau {\int^g} (\underhom(g, \iota^* S_\qq) \smashprod x) \smashprod g
\ar[r] &
\tau (\iota^* S_\qq \smashprod x) }
\xymatrix{
\underhom(\iota^* S_\qq, \tau E)
\ar[r]
\ar[d] &
j \tau E
\ar[d] \\
\rho \underhom(\iota^* S_\qq, E)
\ar[r] &
j E }
\end{displaymath}
The left hand diagram commutes by the same arguments we used above
to show that $(-) \smashprod_{\ecal_{top}} \gcal_{top}$
is an involutary functor.
The right hand diagram commutes because
$\tau$ does not alter the underlying symmetric spectrum
of an $\iota^* S_\qq$-module.
\end{proof}

\section{Moving to $ \rightmod \Theta'' \ecal_{top} $ }\label{sec:movingaccross the machine}

Having shown that the equivalence between rational $SO(2)$-spectra and
$\rightmod \ecal_{top}$ is involutary, we now must show that the
zig-zag of equivalences between $\rightmod \ecal_{top}$
and $ \rightmod \Theta'' \ecal_{top} $
(as proved in \cite[Theorem 4.1]{greshi}) is involutary.
There are several steps in this section, but they are
all essentially the same. We begin with the easiest case
where we move from enrichments over positive symmetric spectra to
symmetric spectra. The main result of this section is
Theorem \ref{thm:invmoduleadjunctions}.

\begin{lemma}
The Quillen equivalence between the category of
$\ecal_{top}$-modules over the
model category of positive symmetric spectra
and the category of
$\ecal_{top}$-modules over the
model category of symmetric spectra is
an involutary strong symmetric monoidal Quillen equivalence.
\end{lemma}
\begin{proof}
Recall that there is a Quillen equivalence
$\id : Sp^\Sigma_+ \overrightarrow{\leftarrow} Sp^\Sigma : \id$.
We perform two operations,
first we consider $\ecal_{top}$ as a category enriched over
symmetric spectra and secondly we
consider the category of $Sp^\Sigma$-functors from
$\ecal_{top}^{op}$ to $Sp^\Sigma$.
Thus we have the model category of $\ecal_{top}$-modules
over $Sp^\Sigma_+$.
The involution on this category is as before,
it is given by a map of the ringoid
$\ecal_{top} \to \ecal_{top}$, so we have an involutary adjunction.

Since the only difference between this new model category
and $\rightmod \ecal_{top}$ as before
is in the model structures, there is little to check in order to
prove the result.
The right adjoint preserves fibrations
and the weak equivalences are the same for either
model structure so we have a Quillen equivalence.
\end{proof}

Since the category of $\ecal_{top}$-modules over the
model category of symmetric spectra is
only used briefly, we do not introduce any new notation for it,
but it is essential for the remaining results that we are now using
this model structure on $\ecal_{top}$-modules. In particular,
the unit is now cofibrant.
We now give some of the material in \cite{shiHZ}
since we will need to examine this in some detail to
prove that the equivalence between
$\ecal_{top}$-modules and $\ecal_t$-modules
is involutary.
We give the proposition first and then explain the terms in it.
\begin{proposition}
The following series of adjoint pairs
are Quillen equivalences.
$$
\begin{array}{rcl}
Q :  \h \qq \leftmod & \overrightarrow{\longleftarrow} &
Sp^\Sigma(\sqq \leftmod) : U_1 \\
L : Sp^\Sigma(dg \qq \leftmod_+)  & \overrightarrow{\longleftarrow} &
Sp^\Sigma(\sqq \leftmod) : \phi^*N \\
D :  Sp^\Sigma(dg \qq \leftmod_+) & \overrightarrow{\longleftarrow} &
dg \qq \leftmod : R .
\end{array}
$$
Furthermore, the pair $(Q,U_1)$\index{$Q$@$(Q,U_1)$}
is strong symmetric monoidal,
$(L, \phi^*N)$\index{$L$@$(L, \phi^*N)$}
is symmetric monoidal
and $(D,R)$\index{$D$@$(D,R)$}
is strong monoidal. Each of the right
adjoints preserve all weak equivalences.
\end{proposition}
\begin{proof}
This is \cite[Proposition 2.10]{shiHZ},
where we note that $D$ is not symmetric monoidal,
as explained in \cite{HZcorrection}.
\end{proof}

We begin with the functor
$\tilde{\qq} \co \SSET \to \sqq \leftmod$.
For a simplicial set $X$, we define $(\tilde{\qq} X)_n$
to be the free $\qq$-module on the non-basepoint
simplices of $X_n$, with each $0.s$ identified
with the basepoint for $s \in X_n$.
We fix the object $\tilde{Q} S^1$ and use this
as the suspension object to create
the category $Sp^\Sigma(\sqq \leftmod)$\label{app:simpQsymspec} of
symmetric spectra in simplicial $\qq$-modules.
The functor $\tilde{\qq}$ induces an
adjoint pair
$\tilde{\qq} : Sp^\Sigma \overrightarrow{\longleftarrow}
Sp^\Sigma(\sqq \leftmod) : U$\index{$Q$@$(\tilde{\qq},U)$}.
Let $\h \qq$\index{H Q@${\h \qq}$}
be the symmetric spectrum with level $n$
given by $\tilde{\qq} S^n$. This is a commutative ring spectrum,
hence we have the category of $\h \qq$-modules in symmetric spectra,
$\h \qq \leftmod$\label{app:hqmod}.
There is a  forgetful functor
$U_1 \co Sp^\Sigma(\sqq \leftmod)
\to \h \qq \leftmod$. This has a left adjoint
$Q$, but this is not needed for the work below
since $\ecal_{top} \leftmod$ is enriched over
$Sp^\Sigma$.

The category of symmetric spectra in non-negatively
graded chain complexes, written as
$Sp^\Sigma(dg \qq \leftmod_+)$\label{app:dgQ+symspec},
has suspension object $\qq[1]$
(one copy of $\qq$ in degree 1).
The normalisation functor
$N \co \sqq \leftmod \to dg \qq \leftmod_+$
induces a functor
$\phi^* N \co Sp^\Sigma(\sqq \leftmod)
\to Sp^\Sigma(dg \qq \leftmod_+)$,
with left adjoint $L$.
The functor $R$
takes a chain complex $Y$
to the symmetric spectrum with
$RY_n = C_0 (Y \otimes \qq[m])$ and has a left adjoint $D$.

In the following result we will use the pair
$\tilde{\qq} : Sp^\Sigma \overrightarrow{\longleftarrow}
Sp^\Sigma(\sqq \leftmod) : U$
which are a Quillen pair, but not a Quillen equivalence.
This result is a part of \cite[Theorem 4.1]{greshi}.
\begin{proposition}\label{prop:moduleadjunctions}
For each of the adjoint pairs
$(\tilde{\qq}, U)$, $(L,\phi^*N)$
and $(D,R)$,
the induced adjunction below is a Quillen equivalence.
\begin{eqnarray*}
\tilde{\qq} : \rightmod \ecal_{top}
& \overrightarrow{\longleftarrow} &
\rightmod \tilde{\qq} \ecal_{top} : U' \\
L' : \rightmod\phi^* N \tilde{\qq}  \ecal_{top}
& \overrightarrow{\longleftarrow} &
\rightmod\tilde{\qq} \ecal_{top} : \phi^*N \\
D   : \rightmod \phi^* N \tilde{\qq} \ecal_{top}
& \overrightarrow{\longleftarrow} &
\rightmod D \phi^* N \tilde{\qq} \ecal_{top} : R'
\end{eqnarray*}
\end{proposition}
\begin{proof}
We use the construction
of Proposition \ref{prop:MonFunctorsEnrichCat}
(which is a simplification of \cite[Proposition A.3b]{dugshi})
to obtain the categories $\tilde{\qq} \ecal_{top}$,
$\phi^* N \tilde{\qq} \ecal_{top}$
and $D \phi^* N \tilde{\qq} \ecal_{top}$
which are enriched over symmetric spectra
in simplicial $\qq$-modules,
symmetric spectra in positive chain complexes of
$\qq$-modules and chain complexes
of $\qq$-modules respectively.
The induced adjunctions are defined
in \cite[Section 3]{ss03monequiv} and we give brief details below.
Since $\tilde{\qq}$ and $D$ are strong monoidal
these pass to the categories of modules
as above without change. The right adjoint $\phi^*N$
also passes directly to the module categories
whereas all the other functors must be replaced.
The right adjoints $U'$ and $R'$ are defined via the unit map,
we demonstrate for $U'$. Take a
$\tilde{\qq} \ecal_{top}$-module $M$, we must then give maps
$$\ecal_{top}(\sigma', \sigma) \smashprod U'M (\sigma)
\to U'M(\sigma').$$
We do so by applying the unit map
$\ecal_{top}(\sigma', \sigma) \to U'\tilde{\qq} \ecal_{top}(\sigma', \sigma)$
and then using the monoidality of $U'$ and the action map of $M$.
We will define $L'$ in the proof of Theorem \ref{thm:invmoduleadjunctions}.

The pair $(\tilde{\qq},U')$ induce a Quillen pair
between $\rightmod \ecal_{top}$ and
$\rightmod \tilde{\qq} \ecal_{top}$.
The free modules are a set of generators
for these categories and these
free modules have rational homotopy groups.
It follows that the unit and counit
for the derived adjunctions
are equivalences on these generators,
hence $(\tilde{\qq},U')$ is a Quillen equivalence.
The other two pairs are Quillen equivalences by
\cite[Theorem 6.5]{ss03monequiv}.
\end{proof}

\begin{theorem}\label{thm:invmoduleadjunctions}
For each of the adjoint pairs
$(\tilde{\qq}, U)$, $(L,\phi^*N)$
and $(D,R)$
the induced adjunction
$(\tilde{\qq},U')$, $(L',\phi^*N)$
and $(D,R')$ is an involutary Quillen equivalence
on the categories of modules.
\end{theorem}
\begin{proof}
Because the construction of the ringoids
$\tilde{\qq} \ecal_{top}$,
$\phi^* N \tilde{\qq} \ecal_{top}$ and
$D \phi^* N \tilde{\qq} \ecal_{top}$
is functorial these
ringoids come with self-inverse maps as follows.
$$
\begin{array}{rclcl}
\tilde{\qq} \tau W & : & \tilde{\qq} \ecal_{top}
& \longrightarrow & \tilde{\qq} \ecal_{top} \\
\phi^* N \tilde{\qq} \tau W & : & \phi^* N \tilde{\qq} \ecal_{top}
& \longrightarrow & \phi^* N \tilde{\qq} \ecal_{top} \\
D \phi^* N \tilde{\qq} \tau W & : &  D \phi^* N \tilde{\qq} \ecal_{top}
& \longrightarrow & D \phi^* N \tilde{\qq} \ecal_{top}
\end{array}
$$
We have the following diagram of adjoint pairs for
the pair $(\tilde{\qq}, U')$.
$$ \xymatrix@C+0.5cm@R+0.5cm{
\rightmod \ecal_{top}
\ar@<+0.3ex>[r]^{\tilde{\qq}}
\ar@<+0.3ex>[d]^{\tau W^*} &
\rightmod \tilde{\qq} \ecal_{top}
\ar@<+0.3ex>[l]^{U'}
\ar@<+0.3ex>[d]^{\tilde{\qq}\tau W^*} \\
\rightmod \ecal_{top}
\ar@<+0.3ex>[r]^{\tilde{\qq}}
\ar@<+0.3ex>[u]^{\tau W^*} &
\rightmod \tilde{\qq} \ecal_{top}
\ar@<+0.3ex>[l]^{U'}
\ar@<+0.3ex>[u]^{\tilde{\qq}\tau W^*} &
} $$
The vertical adjunctions are restriction and extension of scalars
in the case where the map of ringoids is an isomorphism.
This diagram commutes in the sense that the square
consisting of left adjoints
(on top and left) commutes and the square of right adjoints
commutes.
We obtain similar diagrams for $(L', \phi^*N)$
and $(D, R')$, we must now show that these pairs are involutary.
In the case of $(\tilde{\qq}, U')$
the above squares of left adjoints and right adjoints
commute precisely, as we now show.
Take an $\ecal_{top}$-module $M$, then
we have the following commutative diagram
$$ \xymatrix@C-1cm{
& \tilde{\qq} M(a) \smashprod \tilde{\qq} \ecal_{top}(b,a)
\ar[dr]^\cong \\
\tilde{\qq} M(a) \smashprod \tilde{\qq} \ecal_{top}(b,a)
\ar[ur]^{\id \smashprod \tilde{\qq} \tau W}
\ar[dr]_\cong &&
\tilde{\qq} (M(a) \smashprod \ecal_{top}(b,a))
\ar[rr]
 & {\phantom{fred}} &
\tilde{\qq} M(b) \\
& \tilde{\qq} (M(a) \smashprod \ecal_{top}(b,a))
\ar[ur]_{\tilde{\qq} (\id \smashprod \tau W)} } $$
with the top path the action map for
$(\tilde{\qq} \tau W)^* \tilde{\qq}M$
and the lower path the action map for
$\tilde{\qq}( (\tau W)^* M)$. So we have shown that
$\tilde{\qq}$ is an involutary functor since it strictly
commutes with the involutions $(\tau W)^*$ and
$(\tilde{\qq} \tau W)^*$.
Now we consider the module-level right adjoint $U'$.
For a $\tilde{\qq} \ecal_{top}$ module $N$, the following composition
defines the action of $\ecal_{top}$ on $U'N$,
from this it is clear that $U'$ also strictly
commutes with the involutions.
$$
U'N(a) \smashprod \ecal_{top}(b,a)
\overset{\id \smashprod \eta}{\rightarrow}
U'N(a) \smashprod U' \tilde{\qq} \ecal_{top}(b,a)
\rightarrow
U'( N(a) \smashprod \tilde{\qq} \ecal_{top}(b,a))
\rightarrow
U'N(b)
$$
It is then obvious that $(\tilde{\qq}, U')$
is an involutary adjunction that is also a
Quillen equivalence. The case
$(D, R')$ is exactly the same, which leaves us with only
the adjunction $(L', \phi^*N)$ to consider.
The right adjoint $\phi^*N$ strictly commutes with the involutions
via the same arguments as for
$\tilde{\qq}$ and $D$. We investigate $L'$
in some detail, take a $\phi^* N \tilde{\qq} \ecal_{top}$-module
$M$, then $L'M(a)$ is defined as the coequaliser of the following diagram
(we describe the maps below).
$$
\bigvee_{b,c} L \big( M(b) \smashprod \phi^* N \tilde{\qq} \ecal_{top}(c,b) \big)
\smashprod \tilde{\qq} \ecal_{top}(a,c)
\overrightarrow{\longrightarrow}
\bigvee_{d} LM(d) \smashprod \tilde{\qq} \ecal_{top}(a,d)
$$
One map is induced by the action of
$\phi^* N \tilde{\qq} \ecal_{top}$ on $M$
and the other is the composite of the op-monoidal structure on $L$,
the counit of $(L, \phi^*N)$ and composition.
We can induce a map of coequalisers
$$L' (\phi^*N \tilde{\qq} \tau W)^*M \longrightarrow
(\tilde{\qq} \tau W)^* L' M$$
by acting as
$L(\id \smashprod \phi^*N \tilde{\qq} \tau W) \smashprod \tilde{\qq} \tau W$
on the first factor and by
$\id \smashprod \tilde{\qq} \tau W$
on the second. This is clearly a morphism of modules and is
a map order two as required. It remains to check that the unit and counit
are involutary natural transformations, these
maps are induced from the unit and counit of $(L, \phi^*N)$
and it is easy to check the required condition.
The proof is very much like that for unit and counit of
the Morita equivalence and we omit it for that reason.
\end{proof}

\begin{lemma}
The categories $\tilde{\qq} \ecal_{top}$ and
$\phi^* N \tilde{\qq} \ecal_{top}$ are monoidal
enriched categories.
Furthermore, the adjunctions $(\tilde{\qq}, U')$ and $(L',\phi^*N)$
are involutary symmetric monoidal Quillen equivalences
and hence induce symmetric monoidal Quillen equivalences on the skewed categories.
\end{lemma}
\begin{proof}
The first statement is an application of
Proposition \ref{prop:MonFunctorsEnrichCat}
since both $\tilde{\qq}$ and $\phi^* N$ are
symmetric monoidal functors.
We prove that $(\tilde{\qq}, U')$ and $(L',\phi^*N)$
are monoidal pairs on the involutary categories.
It is easy to check that the two
right adjoints $U'$ and $\phi^* N$
are monoidal functors.
Since $\tilde{\qq}$ is strong symmetric
monoidal on the base categories
it passes to a strong symmetric monoidal
functor $\tilde{\qq} \co \rightmod \ecal_{top}
\to \rightmod \tilde{\qq} \ecal_{top}$.
Thus $(\tilde{\qq}, U')$ is a strong symmetric monoidal pair.

The left adjoint $L'$ is harder to deal with,
we must show that $\eta \co L' (\phi^* N \tilde{\qq} \ecal_{top}(-, \sphspec))
\to \tilde{\qq} \ecal_{top}(-, \sphspec)$ is a weak equivalence
and that for cofibrant
$\phi^* N \tilde{\qq} \ecal_{top}$-modules $X$ and $Y$,
the map $m \co L' (X \square Y) \to L'X \square L' Y$
is a weak equivalence.
For each $g \in \ecal_{top}$ there is
a $\phi^* N \tilde{\qq} \ecal_{top}$-module
$F_g = \phi^* N \tilde{\qq} \ecal_{top}( -, g)$.
By Remark \ref{rmk:freemodules} these give a Quillen pair
$F_g \smashprod (-) : Sp^\Sigma(dg \qq \leftmod_+)
\overrightarrow{\longleftarrow}
\rightmod \phi^* N \tilde{\qq} \ecal_{top} : \ev_g$,
where $\ev_g(M) =M(g)$ and
the left adjoint takes $X \in Sp^\Sigma(dg \qq \leftmod_+)$ to
$F_g \smashprod X$.
Since the unit of $Sp^\Sigma(dg \qq \leftmod_+)$
is cofibrant, it follows that each $F_g$
is cofibrant as a
$\phi^* N \tilde{\qq} \ecal_{top}$-module.
The unit is also a free module,
it is given by $F_{\iota^* S_\qq}$.
For each $g \in \ecal_{top}$,
there is a square as below. This will allow us to translate results
from the base categories to the module
categories and helps us understand the functor $L'$.
The right adjoints are on the right and
bottom and clearly commute.
$$\xymatrix@C+0.6cm@R+0.6cm{
\rightmod \phi^* N \tilde{\qq} \ecal_{top}
\ar@<+0.4ex>[r]^{L'}
\ar@<+0.4ex>[d]^{\ev_g}
&
\rightmod \tilde{\qq} \ecal_{top}
\ar@<+0.4ex>[l]^{\phi^* N}
\ar@<+0.4ex>[d]^{\ev_g} \\
Sp^\Sigma(dg \qq \leftmod_+)
\ar@<+0.4ex>[r]^{L}
\ar@<+0.4ex>[u]^{F_g \smashprod (-)}
&
Sp^\Sigma(\sqq \leftmod)
\ar@<+0.4ex>[l]^{\phi^* N}
\ar@<+0.4ex>[u]^{F_g \smashprod (-)}
}$$
From this square we can check that $\eta$ is a weak equivalence.
Let $\sym(\qq[1])$ denote the unit of
$Sp^\Sigma(dg \qq \leftmod_+)$ and
$\sym(\tilde{\qq} S^1)$ the unit of $Sp^\Sigma(\sqq \leftmod)$,
these are cofibrant objects.
Then $L' F_{\iota^* S_\qq} \cong F_{\iota^* S_\qq} \smashprod L \sym(\qq[1])$,
so the map $\eta$ is given by
$F_{\iota^* S_\qq} \smashprod L \sym(\qq[1]) \to
F_{\iota^* S_\qq} \smashprod \sym(\tilde{\qq} S^1)$.
It follows that this map is an object-wise weak equivalence
since cofibrant objects in $Sp^\Sigma(\sqq \leftmod)$
preserve all weak equivalences
(this is in the proof of \cite[Corollary 3.4]{shiHZ}).

We now give an argument to prove that
we only need check that $m$ is a weak
equivalence on the free modules.
We are required to prove that
for all cofibrant $X$ and $Y$,
the map $m \co L'( X \square Y) \to L'X \square L'Y$
is a weak equivalence.
This is equivalent to proving
that for all fibrant $Z$ the map
$[ L'X \square L'Y ,Z] \to [L'( X \square Y), Z]$
is a weak equivalence. By the standard adjunctions
(such as the isomorphisms below)
this occurs for all $X$ exactly when
$\phi^* N \hom_\square (L'Y, Z) \to \hom_\square (Y, \phi^* N Z)$
is a weak equivalence.
To prove this, it suffices to show that the composite
$$[F_g, \phi^* N \hom_\square (L'Y, Z)] \cong
[ L'F_g \square L'Y ,Z] \to [L'( F_g \square Y), Z]
\cong [F_g, \hom_\square (Y, \phi^* N Z)] $$
is an isomorphism for the collection of free modules
$F_g = \phi^* N \tilde{\qq} \ecal_{top}( -, g)$
as $g$ runs over the set of objects in $\ecal_{top}$.
Thus, we have shown that we only need prove that
$L'( F_g \square Y) \to L'F_g \square L'Y$
is a weak equivalence for all $g$
and all cofibrant $Y$.
Applying the above argument once more we see
that it suffices to prove that
$L' (F_g \square F_k) \to L'F_g \square L' F_k$
is a weak equivalence for each $g$ and $k$
in $\ecal_{top}$.
We do so now. Using our understanding of $L'$ on free modules
and the isomorphism
$F_g \square F_k \to F_{(g \smashprod k)}$
we must show that
$$F_{(g \smashprod k)} \smashprod L \sym(\qq[1])
\to F_{(g \smashprod k)} \smashprod L \sym(\qq[1])
\smashprod L \sym(\qq[1])$$
is a weak equivalence. This follows
from the corresponding result in
$Sp^\Sigma(\sqq \leftmod)$.
One must also check that
the natural transformations giving
an involutary structure on $\phi^* N$ and
$U'$ are monoidal. This is quite
straightforward due to the nature of the involutions.
\end{proof}

\begin{rmk}\label{rmk:monoidalissue}
Because $D$ is not symmetric, as is stated in
\cite{HZcorrection}, the $dg \qq \leftmod$-category
$D \phi^* N \tilde{\qq} \ecal_{top}$
is not a monoidal $dg \qq \leftmod$-category.
In turn, the category $\rightmod \ecal_t$
cannot be monoidal.
This issue could be resolved by using the
four stage comparison of \cite[Remark 2.11]{shiHZ}
and altering the fibrant replacement functor.
In detail, we replace the pair $(D,R)$
by the functors
$$
\xymatrix@C+0.5cm{
dg \qq \leftmod \ar@<+0.3ex>[r]^{F_0} &
{Sp^\Sigma} (dg \qq \leftmod) \ar@<+0.3ex>[l]^{\ev_0}
\ar@<-0.3ex>[r]_{C_0} &
{Sp^\Sigma} (dg \qq \leftmod_+)
\ar@<-0.3ex>[l]_{i} }
$$
The pair $(F_0, \ev_0)$ are the suspension
and zeroth space adjunction.
The inclusion of positive chain complexes
$i \co dg \qq \leftmod_+ \to dg \qq \leftmod$ has a right adjoint
$C_0$. These are strong symmetric monoidal
Quillen equivalences and $i$ preserves all weak
equivalences.
We can then use these functors to create
a symmetric monoidal enriched category
$\ev_0 i \phi^* N \tilde{\qq} \ecal_{top}$.
Unfortunately, since $\ev_0$ doesn't preserve
all weak equivalences this will not have the correct
homotopy type. One gets round this by inserting
a fibrant replacement functor of
$Sp^\Sigma  (dg \qq \leftmod_+)$-enriched categories
as given by \cite[Proposition 6.3]{ss03monequiv}.
Thus $\rightmod \ecal_{top}$ is Quillen equivalent
to $\rightmod \ev_0 \fibrep i \phi^* N \tilde{\qq} \ecal_{top}$.
But this is not a monoidal category, as $\fibrep$
will not preserve the monoidal product.
It should be possible to alter this
fibrant replacement so that it does preserve
symmetric monoidal structures on enriched categories
and ensure that
$\ev_0  \fibrep i \phi^* N \tilde{\qq} \ecal_{top}$
has the correct homotopy type.
All of the functors relating
$\rightmod \ecal_{top}$ and $\rightmod \ecal_t$
would then be both monoidal and involutary.
An extra step would then be necessary: an adjunction
of extension and restriction of scalars
induced by the quasi-isomorphism
$i \phi^* N \tilde{\qq} \ecal_{top} \to
\fibrep i \phi^* N \tilde{\qq} \ecal_{top}$, but this
would present no difficulty.
\end{rmk}

\begin{lemma}\label{lem:tiwstedadjointmodule}
If $L : M \overrightarrow{\leftarrow} N : R$
is a Quillen module over
$F : \ccal \overrightarrow{\leftarrow} \dcal : G$
and $(L,R)$ is an involutary adjunction, then
$\sigma \# L : \sigma \# M \overrightarrow{\leftarrow}
\tau \# N : \tau \# R$
is a Quillen module over $(\id \# F, \id \# G)$ provided
the following holds.
\begin{enumerate}
\item There is a natural transformation of order two
$(\sigma m) \otimes c \to \sigma (m \otimes c)$.
\item There is a natural transformation of order two
$(\tau n) \otimes d \to \tau (n \otimes d)$.
\item The diagram below commutes.
$$
\xymatrix{
L(\sigma m \otimes c)
\ar[d]
\ar[r] &
L\sigma m \otimes Fc
\ar[d] \\
L\sigma (m \otimes c)
\ar[d] &
\tau Lm \otimes Fc
\ar[d] \\
\tau L(m \otimes c)
\ar[r] &
\tau (Lm \otimes Fc)
} $$
\end{enumerate}
\end{lemma}
\begin{proof}
The first two conditions ensure that $\sigma$ and $\tau$
are enriched functors, so that $\sigma \# M$
and $\tau \# N$ are modules over $\id \# \ccal$ and
$\id \# \dcal$ respectively.
One must check the conditions of
\cite[Propositons 3.5, 3.6, 3.7]{dugshi} to see that we have a Quillen
adjoint module of the skewed categories.
These conditions all hold due to their counterparts
in $M$ and $N$ provided that the natural transformation
$L(m \otimes c) \to Lm \otimes Fc$ induces a map
on the skewed category, which is the third condition.
Thus $(\sigma \# L, \tau \# R)$ is a Quillen adjoint module over
$(\id \# F, \id \# G)$.
\end{proof}

\begin{corollary}
Each of the adjunctions of the skewed module categories
of Proposition \ref{prop:moduleadjunctions}
is a Quillen module over the skewed base categories.
\end{corollary}
\begin{proof}
That the functors of Proposition \ref{prop:moduleadjunctions}
are Quillen modules at the level of involutary
categories follows from \cite[Section 10]{greshi}.
We have worked from a more recent redraft of this paper,
where this section has been altered to take into
account \cite[Propositions 4.7 and 4.8]{dugshi}
which consider Quillen modules of categories of modules
over enriched categories.

It remains to prove that the assumptions of
Lemma \ref{lem:tiwstedadjointmodule}
hold for the three adjunctions of
Proposition \ref{prop:moduleadjunctions}.
This is easy for $(\tilde{\qq}, U')$
and $(D,R')$ since these adjoint pairs
strictly commute with the involutions.
The proof for $(L', \phi^* N)$ is routine.
\end{proof}

\section{Another Involutary Morita Equivalence}\label{sec:anothermorita}
We show that one last Quillen equivalence
is involutary: the Morita equivalence between
$\rightmod D \phi^* N \tilde{\qq} \ecal_{top}$
and $\rightmod \ecal_t$.

\begin{definition}
We define the category $\gspect$
to be $\rightmod D \phi^* N \tilde{\qq} \ecal_{top}$
and we will use $\lambda$ for the involution on this category,
hence we have the skewed category $\lambda \# \gspect$.
As in \cite[Theorem 4.1]{greshi} we define $\mathcal{BC}_t$ to be
cofibrant replacements of the images of the basic cells of $
SO(2)$-spectra under the composite functor from $SO(2)$-spectra to
$\gspect$. The closure of $\mathcal{BC}_t$ under the monoidal product will
be written $\bar{\mathcal{BC}_t}$. Define $\ecal_t$ to be the
full subcategory of $\gspect$ with object set $\bar{\mathcal{BC}_t}$.
The category $\ecal_t$\index{e calt@$\ecal_t$}\label{app:ecalt}
is enriched over differential graded $\qq$-modules.
\end{definition}

Thus the objects of $\mathcal{BC}_t$ have the form
$(\sigma_H)_t = \cofrep D \phi^* N \tilde{\qq} Sp^\Sigma \underhom(-, \sigma_H)$
for $\sigma_H$ a basic cell,
where $\cofrep$ is cofibrant replacement
in the skewed category.

For each $H$ there is a map of order two
$u_H \co (\sigma_H)_t \to \lambda (\sigma_H)_t$ which is induced by
the map $\sigma_H \to \tau \sigma_H$ and the natural
transformations of the involutary functors in the composite.
As with $\ecal_{top}$ we have a ringoid $\lambda \ecal_t$
and conjugation by the $u_H$ gives a map of
ringoids $U \co \ecal_t \to \lambda \ecal_t$.

\begin{lemma}
The category $\rightmod \ecal_t$ is a model
category with involution $(\lambda U)^*$.
\end{lemma}
\begin{proof}
The follows from the construction of
the involution $\rho = (\tau W)^*$ on $\rightmod \ecal_{top}$
in Section \ref{sec:Etopinv}.
\end{proof}

\begin{theorem}
The functors of Theorem \ref{thm:monoidalmorita}
induce an involutary Quillen equivalence
between $\rightmod \ecal_t$ and $\gspect$.
Thus there is a Quillen equivalence
between the skewed categories
$$(\lambda U)^* \# (\rightmod \ecal_t)
\overrightarrow{\longleftarrow}
\lambda \# \gspect$$
and this is an adjunction of
($\id \# dg \qq \leftmod$)-modules.
\end{theorem}
\begin{proof}
This follows from the proof of
Corollary \ref{cor:invSO2morita}.
\end{proof}

\begin{corollary}\label{cor:cyclicsummary}
There is a zig-zag of Quillen equivalences
between $\tau \# \iota^* S_\qq \leftmod$
and $(\lambda U)^* \# (\rightmod \ecal_t)$.
\end{corollary}

\begin{rmk}\label{rmk:nextsteps}
Now we have a zig-zag of involutary equivalences
between cyclic $O(2)$-spectra and
$\rightmod \ecal_t$.
We describe our plan for future work.
The standard category $\mathcal{A}(SO(2))$
is `formed' from copies of the graded ring
$\qq[c] \cong \h^* (\text{B}SO(2))$ with $c$ of degree $2$
see \cite{gre99}.
The inversion map $j \co SO(2) \to SO(2)$, $j(t) = t^{-1}$,
induces a ring map $\qq[c] \to \qq[c]$
which sends $c$ to $-c$.
Thus, we hope to create an
involution $\Upsilon \co \mathcal{A}(SO(2)) \to \mathcal{A}(SO(2))$
based on this ring map.
Following the proof of Corollary \ref{cor:invSO2morita}
we should then be able to prove that
$\mathcal{A}(SO(2))$ is Quillen equivalent to
$\rightmod \ecal_a$ and that this
equivalence is involutary monoidal.
From here we aim to take
the equivalences between
$\ecal_t$ and $\ecal_a$ of \cite{greshi}
and show that these are involutary
(and monoidal according to the outcome of
Remark \ref{rmk:monoidalissue}).
This would complete the classification
of cyclic $O(2)$-spectra in terms of the
skewed category $\Upsilon \# \mathcal{A}(SO(2))$.
\end{rmk}

\appendix

\chapter{List of Model Categories}\label{chp:listmodcat}

We list the model categories used in this thesis. 
We have divided this list into four collections, 
those which are categories of spectra, those we have constructed
as localisations of such categories, categories
of modules over enriched categories and
a miscellany of basic categories. 

\subsection*{Categories of Spectra}
\begin{tabular}{|p{7.8cm}|p{3cm}|p{1.8cm}|}
\hline
Name & Symbol & Page \\
\hline
Symmetric spectra in based simplicial sets
& $Sp^\Sigma$ & \pageref{para:symspec}\\
$Sp^\Sigma$ with the positive model structure
& $Sp^\Sigma_+$ & \pageref{app:possymspec}\\
Symmetric spectra in $\tscr_*$ &
$Sp^\Sigma_+ (\tscr_*)$ &
\pageref{app:gensymspec}, \pageref{app:topsymspec} \\
Symmetric spectra in $dg \qq \leftmod_+$
& $Sp^\Sigma(dg \qq \leftmod_+)$ &
\pageref{app:gensymspec}, \pageref{app:dgQ+symspec}\\
Symmetric spectra in $\sqq \leftmod$
& $Sp^\Sigma(\sqq \leftmod)$ &
\pageref{app:gensymspec}, \pageref{app:simpQsymspec}\\
$G$-equivariant orthogonal spectra & $\GIS$ &
\pageref{para:orthogspec}, \pageref{thm:orthogmodel} \\
$G$-equivariant $S$-modules & $G \mcal$ &
\pageref{para:ekmm}, \pageref{thm:ekmmmodel} \\
$\GIS$ with the positive model structure & $\GIS_+$ &
\pageref{thm:positiveorthogmodel} \\
Modules over $S_\qq$ in $O(2) \mcal$
& $S_\qq \leftmod$ &
\pageref{def:S_qq}, \pageref{sec:SQmod}\\
Modules over $S_H$ in $O(2) \mcal$
& $S_H \leftmod$ &
\pageref{lem:SHobject} \\
Modules over $\iota^* S_\qq$ in $SO(2) \mcal$
& $\iota^* S_\qq \leftmod$ & \pageref{app:SQmod}\\
The involutary category of $\iota^* S_\qq$
& $(\iota^* S_\qq \leftmod, \tau)$ &
\pageref{prop:SQmodinvolutary} \\
Modules over $\h \qq$ in symmetric spectra
& $\h \qq \leftmod$ & \pageref{app:hqmod}\\
\hline
\end{tabular}

\subsection*{Categories of Localised Spectra}

\begin{tabular}{|p{7.8cm}|p{3cm}|p{1.8cm}|}
\hline
Name & Symbol & Page \\
\hline
$\GIS$ localised at $E$,  $L_E \GIS$ & $\GIS_E$  & 
 \pageref{thm:GSlocal} \\
$G \mcal$ localised at $E$, $L_E G \mcal$ & $G \mcal_E$ &
\pageref{thm:GSlocal} \\
$\GIS$ localised at $S^0 \qq$
& $\GIS_\qq$ &
\pageref{sec:GISQ} \\
$\GIS_+$ localised at $S^0 \qq_+$
& $\GIS_\qq^+$ &
\pageref{sec:GISQ} \\
$G \mcal$ localised at $S^0_\mcal \qq$
& $G \mcal_\qq$ &
\pageref{sec:GISQ} \\
$S_\qq \leftmod$ localised at $E \smashprod S_\qq$
& $L_E S_\qq \leftmod$ &
\pageref{prop:rmodlocal}\\
$L_{E \fscr_+} \GIS$ & $\fscr \iscr \sscr$ &
\pageref{def:familymodelcat} \\
$L_{E \widetilde{\fscr}} \GIS$ & $\widetilde{\fscr} \iscr \sscr$ &
\pageref{def:familymodelcat}\\
$L_{E \fscr_+} G \mcal$ &  $\fscr \mcal$ &
\pageref{def:familymodelcat} \\
$L_{E \widetilde{\fscr}} G \mcal$ &  $\widetilde{\fscr} \mcal$ &
\pageref{def:familymodelcat} \\
$G \mcal$ with $\fscr(N)$-model structure
& $G \mcal(N)$ &
\pageref{app:GMNmodel} \\
Cyclic $O(2)$-spectra & $\cscr \mcal_\qq$ &
\pageref{def:CandDspectra} \\
Dihedral $O(2)$-spectra & $\dscr \mcal_\qq$ &
\pageref{def:CandDspectra}\\
Cyclic $S_\qq$-modules & $S_\qq \leftmod(\cscr)$ &
\pageref{prop:cylicSQmodarecyclicspectra}\\
\hline
\end{tabular}

\subsection*{Modules over an Enriched Category}

\begin{tabular}{|p{7.8cm}|p{3cm}|p{1.8cm}|}
\hline
Name & Symbol & Page \\
\hline
Right $\ecal_a^H$-modules in $dg \qq \leftmod$
& $\rightmod \ecal_a^H$ &
\pageref{prop:FiniteAlgMorita} \\
Right $\ecal_{top}^H$-modules in $Sp^\Sigma_+$
& $\rightmod \ecal_{top}^H$ &
\pageref{def:finitegcalH} \\
Right modules over $\ocal$ in $Sp^\Sigma$
& $\rightmod \ocal $ &
\pageref{def:ocalmod}, \pageref{thm:modOmodelcat} \\
Right $\ecal_{top}$-modules in $Sp^\Sigma_+$
& $\rightmod \ecal_{top}$ &
\pageref{thm:so2morita}\\
Right $\tilde{\qq} \ecal_{top}$-modules in $Sp^\Sigma(\sqq \leftmod)$
& $\rightmod \tilde{\qq} \ecal_{top}$ &
\pageref{prop:moduleadjunctions}\\
Right $\phi^* N \tilde{\qq} \ecal_{top}$-modules in $Sp^\Sigma_+(dg \qq \leftmod_+)$
& $\rightmod \phi^* N \tilde{\qq} \ecal_{top}$ &
\pageref{prop:moduleadjunctions} \\
Right $D \phi^* N \tilde{\qq} \ecal_{top}$-modules in $dg \qq \leftmod$
& $\rightmod D \phi^* N \tilde{\qq} \ecal_{top}$ &
\pageref{prop:moduleadjunctions} \\
Right $\ecal_t$-modules in $dg \qq \leftmod$
& $\rightmod \ecal_t$ & \pageref{app:ecalt}
\\
\hline
\end{tabular}

\subsection*{Miscellaneous}

\begin{tabular}{|p{7.8cm}|p{3cm}|p{1.8cm}|}
\hline
Name & Symbol & Page \\
\hline
Based topological spaces
& $\tscr_*$ & \pageref{para:spaces}\\
Based $G$-equivariant topological spaces
& $G \tscr_*$ & \pageref{para:Gspaces}\\
Based simplicial sets
& $\SSET_*$ &
\pageref{para:simpsets}\\
Simplicial $\qq$-modules
& $\sqq \leftmod$ &
\pageref{app:simpQmod}\\
Chain complexes of $R$-modules
& $dg R \leftmod$ &
\pageref{para:dgrmod} \\
Positive chain complexes of $R$-modules
& $dg R \leftmod_+$ &
\pageref{app:dgrmod+} \\
Chain complexes of $\qq G$-modules
& $dg \qq G \leftmod$ &
\pageref{def:QGtensorproduct} \\
A category with involution
& $(\ccal, \sigma)$ &
\pageref{def:catwithinv}
\\
The skewed category of $(\ccal, \sigma)$
& $\sigma \# \ccal$ &
\pageref{def:skewcat} \\
\hline
\end{tabular}

\backmatter
\addcontentsline{toc}{chapter}{Index}
\printindex

\cleardoublepage
\addcontentsline{toc}{chapter}{Bibliography}
\bibliography{davebib}
\bibliographystyle{alpha}

\end{document}